\documentclass{amsart}

\usepackage{amsfonts,mathrsfs}
\usepackage{amsmath,amssymb,amsthm, amscd, bm}
\usepackage{tikz}
\usetikzlibrary{arrows,calc,matrix,topaths,positioning,scopes,shapes,decorations,decorations.markings} 
\usepackage{dsfont}
\usepackage{epsfig}
\usepackage{tikz-cd}
\usepackage{hyperref}
\usepackage{fullpage}
\usepackage[foot]{amsaddr}
\usepackage{adjustbox}
\usepackage{enumitem}

\usepackage[utf8x]{inputenc}

\makeatletter

\renewcommand{\email}[2][]{%
  \ifx\emails\@empty\relax\else{\g@addto@macro\emails{,\space}}\fi%
  \@ifnotempty{#1}{\g@addto@macro\emails{\textrm{(#1)}\space}}%
  \g@addto@macro\emails{#2}%
}
\makeatother

\author{Julien Korinman$^{(1)}$}
\address{${}^{(1)}$ Institut Montpelli\'erain Alexander Grothendieck - UMR 5149 Universit\'e de Montpellier. Place Eug\'ene Bataillon, 34090 Montpellier France}
\email{julien.korinman@gmail.com}
\author{Jun Murakami$^{(2)}$}
\address{${}^{(2)}$Department of Mathematics, Faculty of Science and Engineering, Waseda University,3-4-1 Ohkubo, Shinjuku-ku, Tokyo, 169-8555, Japan}
\email{murakami@waseda.jp}

\subjclass{$57$R$56$, $57$N$10$, $57$M$25$.}

\keywords{Quantum character varieties, Stated skein modules}

\def\restriction#1#2{\mathchoice
              {\setbox1\hbox{${\displaystyle #1}_{\scriptstyle #2}$}
              \restrictionaux{#1}{#2}}
              {\setbox1\hbox{${\textstyle #1}_{\scriptstyle #2}$}
              \restrictionaux{#1}{#2}}
              {\setbox1\hbox{${\scriptstyle #1}_{\scriptscriptstyle #2}$}
              \restrictionaux{#1}{#2}}
              {\setbox1\hbox{${\scriptscriptstyle #1}_{\scriptscriptstyle #2}$}
              \restrictionaux{#1}{#2}}}
\def\restrictionaux#1#2{{#1\,\smash{\vrule height .8\ht1 depth .85\dp1}}_{\,#2}} 

\newcommand{\quotient}[2]{{\raisebox{.2em}{$#1$}\left/\raisebox{-.2em}{$#2$}\right.}}

\newcommand{\Hom}{\operatorname{Hom}}

\newcommand{\SL}{\operatorname{SL}}
\newcommand{\id}{id}
\newcommand{\Span}{\operatorname{Span}}
\newcommand{\End}{\operatorname{End}}

\newcommand{\Cob}{\operatorname{Cob}}
\newcommand{\Alg}{\operatorname{Alg}}
\newcommand{\Gp}{\operatorname{Gp}}
\newcommand{\Top}{\operatorname{Top}}
\newcommand{\Set}{\operatorname{Set}}
\newcommand{\Cat}{\operatorname{Cat}}
\newcommand{\LMod}{\operatorname{LMod}}
\newcommand{\RMod}{\operatorname{RMod}}
\newcommand{\LComod}{\operatorname{LComod}}
\newcommand{\RComod}{\operatorname{RComod}}
\newcommand{\Bimod}{\operatorname{Bimod}}

\newcommand{\BT}{\operatorname{BT}}
\newcommand{\MS}{\operatorname{MS}}
\newcommand{\PP}{\operatorname{P}}
\newcommand{\Spec}{\operatorname{Spec}}
\newcommand{\coinv}{\operatorname{coinv}}

\newcommand{\lt}{\operatorname{lt}}
\newcommand{\Gr}{\operatorname{Gr}}
\newcommand{\TL}{\operatorname{TL}}
\newcommand{\con}{\operatorname{c}}
\newcommand{\Free}{\operatorname{Free}}
\newcommand{\Lan}{\operatorname{Lan}}
\newcommand{\Rep}{\operatorname{Rep}}

\newcommand{\Mod}{\operatorname{Mod}}

\newcommand{\rot}{\operatorname{rot}}
\newcommand{\HT}{\operatorname{HT}}
\newcommand{\hT}{\operatorname{ht}}
\newcommand{\ad}{\operatorname{ad}}
\newcommand{\Ad}{\operatorname{Ad}}
\newcommand{\M}{\boldsymbol{\mathcal{M}}}
\newcommand{\Mun}{\boldsymbol{\mathcal{M}}^{(1)}}

\newcommand{\Char}{\operatorname{Char}}
\newcommand{\Rot}{\operatorname{Rot}}
\newcommand{\Spin}{\operatorname{Spin}}
\newcommand{\SO}{\operatorname{SO}}

\newcommand{\Aut}{\operatorname{Aut}}
\newcommand{\Maps}{\operatorname{Maps}}

\newcommand{\CC}[2]{
\adjustbox{valign=c}{\includegraphics[width=0.5cm]{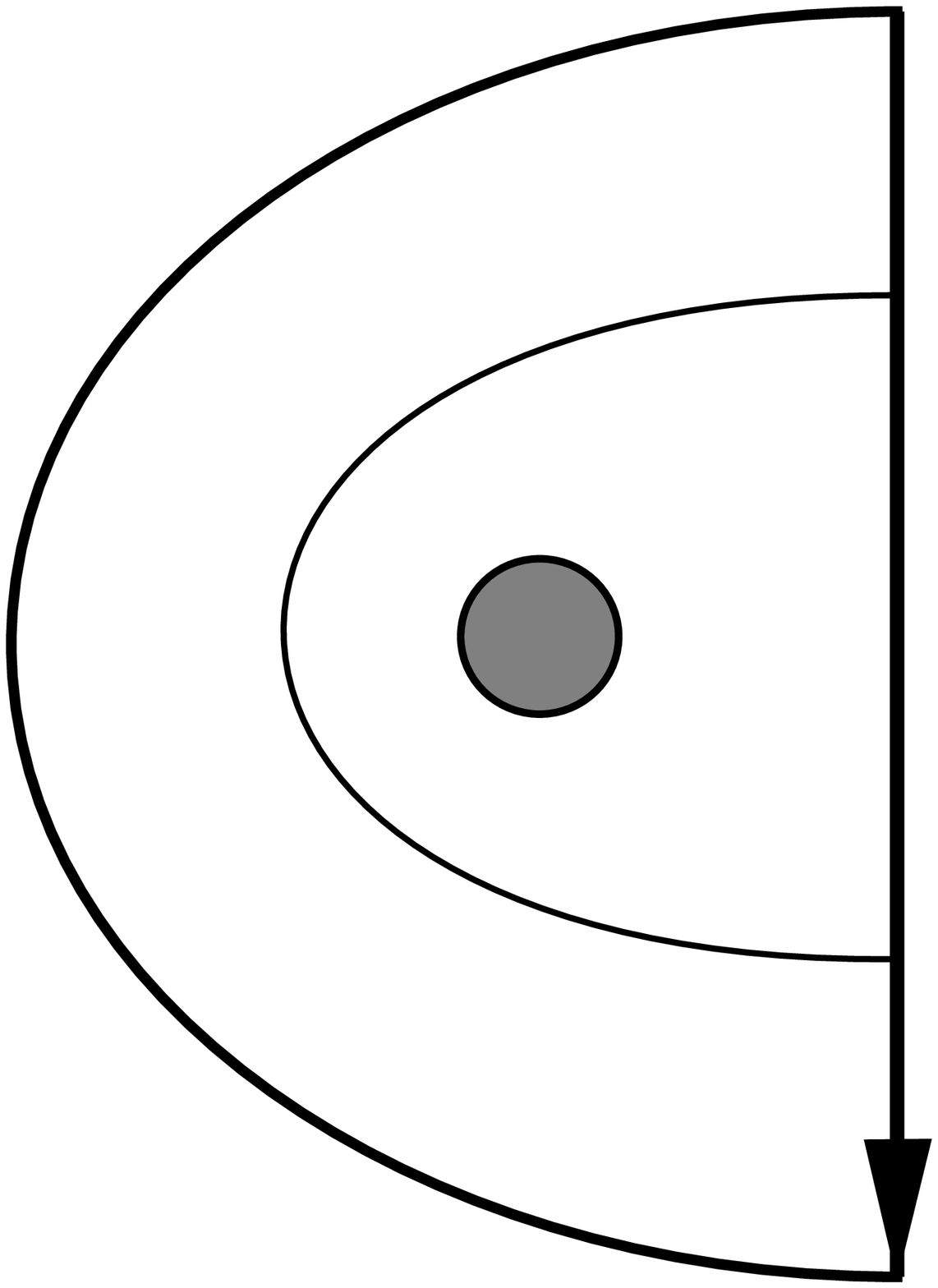}}{}_{#1}^{#2}
}

\newcommand{\CCC}[2]{
\adjustbox{valign=c}{\includegraphics[width=0.5cm]{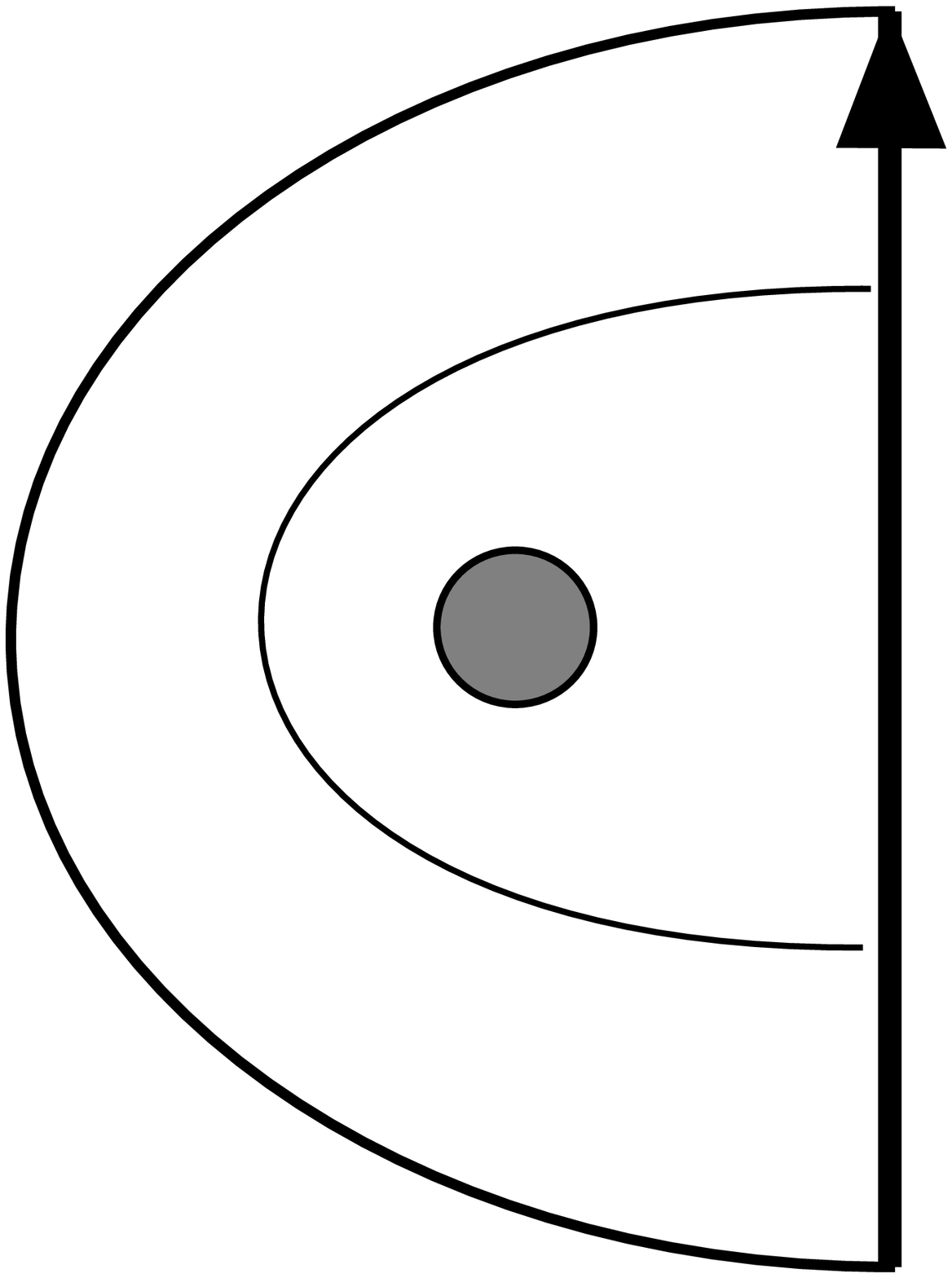}}{}_{#1}^{#2}
}

\newcommand{\crosspos}{
\tikz[baseline=-0.4ex,scale=0.2, >=stealth]{
\draw (-1,-1) -- (1,1);
\draw (-1,1) -- (-0.5,0.5);
\draw (0.5,-0.5) -- (1,-1);
}}

\newcommand{\Crosspos}{
\tikz[baseline=-0.4ex,scale=0.5,>=stealth]{	
\draw [fill=gray!45,gray!45] (-.6,-.6)  rectangle (.6,.6);
\draw[line width=1.2,-] (-0.4,-0.52) -- (.4,.53);
\draw[line width=1.2,-] (0.4,-0.52) -- (0.1,-0.12);
\draw[line width=1.2,-] (-0.1,0.12) -- (-.4,.53);
}}

\newcommand{\crossneg}{
\tikz[baseline=-0.4ex,scale=0.2, >=stealth]{
\draw (-1,1) -- (1,-1);
\draw (-1,-1) -- (-0.5,-0.5);
\draw (0.5,0.5) -- (1,1);
}}

\newcommand{\Crossneg}{
\tikz[baseline=-0.4ex,scale=0.5,>=stealth]{	
\draw [fill=gray!45,gray!45] (-.6,-.6)  rectangle (.6,.6);
\draw[line width=1.2,-] (-0.4,0.53) -- (.4,-.52);
\draw[line width=1.2,-] (-0.4,-0.52) -- (-0.1,-0.12);
\draw[line width=1.2,-] (0.1,0.12) -- (.4,.53);
}}

\newcommand{\heightexch}[3]{
	\begin{tikzpicture}[baseline=-0.4ex,scale=0.5, >=stealth]
	\draw [fill=gray!60,gray!45] (-.7,-.75)  rectangle (.4,.75)   ;
	\draw[#1] (0.4,-0.75) to (.4,.75);
	\draw[line width=1.2] (0.4,-0.3) to (-.7,-.3);
	\draw[line width=1.2] (0.4,0.3) to (-.7,.3);
	\draw (0.65,0.3) node {\scriptsize{$#2$}}; 
	\draw (0.65,-0.3) node {\scriptsize{$#3$}}; 
	\end{tikzpicture}
}

\newcommand{\qplanedouble}[2]{
	\begin{tikzpicture}[baseline=-0.4ex,scale=0.5, >=stealth]
	\draw [fill=gray!60,gray!45] (-.7,-.75)  rectangle (.4,.75)   ;
	\draw[->] (0.4,-0.75) to (.4,.75);
	\draw[->] (-0.7,-0.75) to (-.7,.75);
	\draw[line width=1.2] (0.4,-0.3) to (-.7,-.3);
	\draw[line width=1.2] (0.4,0.3) to (-.7,.3);
	\draw (-1,0.3) node {\scriptsize{$#1$}}; 
	\draw (-1,-0.3) node {\scriptsize{$#2$}}; 
	\end{tikzpicture}
}

\newcommand{\traitalacon}[2]{
	\begin{tikzpicture}[baseline=-0.4ex,scale=0.5, >=stealth]
	\draw [fill=gray!60,gray!45] (-.7,-.75)  rectangle (.4,.75)   ;
	\draw[#1] (0.4,-0.75) to (.4,.75);
	\draw[line width=1.2] (0.4,0) to (-.7,0);
	\draw (0.9,0) node {\scriptsize{$#2$}}; 
	\end{tikzpicture}
}

\newcommand{\traitalacondeux}[1]{
	\begin{tikzpicture}[baseline=-0.4ex,scale=0.5, >=stealth]
	\draw [fill=gray!60,gray!45] (-.7,-.75)  rectangle (.4,.75)   ;
	\draw[->] (0.4,-0.75) to (.4,.75);
	\draw[->] (-0.7,-0.75) to (-.7,.75);
	\draw[line width=1.2] (0.4,0) to (-.7,0);
	\draw (-1,0) node {\scriptsize{$#1$}}; 
	\end{tikzpicture}
}

\newcommand{\traitalacontrois}[2]{
	\begin{tikzpicture}[baseline=-0.4ex,scale=0.5, >=stealth]
	\draw [fill=gray!60,gray!45] (-.7,-.75)  rectangle (.4,.75)   ;
	\draw[->] (0.4,-0.75) to (.4,.75);
	\draw[->] (-0.7,-0.75) to (-.7,.75);
	\draw[line width=1.2] (0.4,0) to (-.7,0);
	\draw (-1,0) node {\scriptsize{$#1$}}; 
	\draw (0.9,0) node {\scriptsize{$#2$}}; 
	\end{tikzpicture}
}

\newcommand{\tresalacon}[2]{
	\begin{tikzpicture}[baseline=-0.4ex,scale=0.5, >=stealth]
	\draw [fill=gray!60,gray!45] (-.7,-.75)  rectangle (.4,.75)   ;
	\draw[#1] (0.4,-0.75) to (.4,.75);
	\draw[line width=1.2] (0.4,0) to (-.7,0);
	\draw (2,0) node {\scriptsize{$#2$}}; 
	\end{tikzpicture}
}

\newcommand{\heightcurve}{
\begin{tikzpicture}[baseline=-0.4ex,scale=0.5]
\draw [fill=gray!20,gray!45] (-.7,-.75)  rectangle (.4,.75)   ;
\draw[-] (0.4,-0.75) to (.4,.75);
\draw[line width=1.2] (-.7,-0.3) to (-.4,-.3);
\draw[line width=1.2] (-.7,0.3) to (-.4,.3);
\draw[line width=1.15] (-.4,0) ++(-90:.3) arc (-90:90:.3);
\end{tikzpicture}
}

\newcommand{\CAP}{
\begin{tikzpicture}[scale=0.5, rotate=90,transform shape]
\draw[line width=0.7] (-.7,-0.3) to (-.4,-.3);
\draw[<-, line width=0.7] (-.7,0.3) to (-.4,.3);
\draw[line width=0.7] (-.4,0) ++(-90:.3) arc (-90:90:.3);
\end{tikzpicture}
}

\newcommand{\heightcurveright}{
\begin{tikzpicture}[baseline=-0.4ex,scale=0.5]
\draw [fill=gray!20,gray!45] (-.7,-.75)  rectangle (.4,.75)   ;
\draw[-] (-0.7,-0.75) to (-.7,.75);
\draw[line width=1.2] (0.1,-0.3) to (.4,-.3);
\draw[line width=1.2] (0.1,0.3) to (.4,.3);
\draw[line width=1.15] (.1,0) ++(90:.3) arc (90:270:.3);
\end{tikzpicture}
}

\newcommand{\heightcurverightdeux}{
\begin{tikzpicture}[baseline=-0.4ex,scale=0.5]
\draw [fill=gray!20,gray!45] (-.7,-.75)  rectangle (.4,.75)   ;
\draw[->] (0.4,-0.75) to (.4,.75);
\draw[->] (-0.7,-0.75) to (-.7,.75);
\draw[line width=1.2] (0.1,-0.3) to (.4,-.3);
\draw[line width=1.2] (0.1,0.3) to (.4,.3);
\draw[line width=1.15] (.1,0) ++(90:.3) arc (90:270:.3);
\end{tikzpicture}
}

\begin{document}

\theoremstyle{plain}
\newtheorem{theorem}{Theorem}[section]
\newtheorem{proposition}[theorem]{Proposition}
\newtheorem{corollary}[theorem]{Corollary}
\newtheorem{lemma}[theorem]{Lemma}
\theoremstyle{definition}
\newtheorem{notations}[theorem]{Notations}
\newtheorem{convention}[theorem]{Convention}
\newtheorem{problem}[theorem]{Problem}
\newtheorem{definition}[theorem]{Definition}
\theoremstyle{remark}
\newtheorem{remark}[theorem]{Remark}
\newtheorem{conjecture}[theorem]{Conjecture}
\newtheorem{example}[theorem]{Example}
\newtheorem{strategy}[theorem]{Strategy}
\newtheorem{question}[theorem]{Question}

\title[Relating quantum character varieties and skein modules]{Relating quantum character varieties and skein modules}

\date{}
\maketitle


\begin{abstract} 
We relate the Kauffman bracket stated skein modules to two independent constructions of quantum representation spaces of Habiro and Van der Veen with the second author. We deduce from this relation a description of the classical limit of stated skein modules, a quantum Van Kampen theorem and a quantum HNN extension theorem for stated skein modules and obtain a new description of the skein modules of mapping tori and  links exteriors. 
\end{abstract}


\tableofcontents

\section{Introduction}

\textit{Main results}
\par 

The $\SL_2$ character variety $\mathcal{X}_{\SL_2}(M)$ of a compact oriented $3$ manifold $M$ admits different quantum deformations  namely the Kauffman-bracket skein module $\mathcal{S}_q(M)$, introduced by Hoste-Przytycki \cite{HP92} and   Turaev \cite{Tu88}  and the quantum character variety introduced by Habiro in \cite{Habiro_QCharVar}. When $M= S^3 \setminus L$ is a link exterior, a third construction of quantum character variety was introduced by Van der Veen and the second author in \cite{MurakamiVdV_QRepSpaces, Murakami_RIMS}. The goal of this paper is to relate all these constructions and deduce new properties of these modules. 

\vspace{2mm}
\par In order to state our results, we now briefly sketch the three constructions of quantum character varieties and refer the reader to Sections \ref{sec2}, \ref{sec3} and \ref{sec6} for details.
By definition, the character variety $\mathcal{X}_{\SL_2}(M)$ is the algebraic quotient of the variety of representations $\mathcal{R}_{\SL_2}(M):= \Hom(\pi_1(M, v), \SL_2)$ by the action of $\SL_2$ by conjugacy. It means that the algebra of regular functions $\mathcal{O}[\mathcal{X}_{\SL_2}(M)]$ is defined as the subalgebra of $\mathcal{O}[\mathcal{R}_{\SL_2}(M)]$ of coinvariant vectors for the $\mathcal{O}[\SL_2]$ coaction. Similarly, all three previously cited constructions of quantum character varieties are obtained as the submodule of coinvariant vectors of a $\mathcal{O}_q\SL_2$ comodule thought as a quantum representation space and relating the three constructions ought to relate these three families of quantum representation spaces. In the case of skein modules, what plays the role of a quantum representation space is the stated skein module introduced in \cite{BonahonWongqTrace, LeStatedSkein, BloomquistLe} and equivalent to the internal skein modules defined in \cite{GunninghamJordanSafranov_FinitenessConjecture} when working over a field. Instead of considering pointed $3$-manifolds, here we consider the category $\mathcal{M}^{(1)}_{\con}$ of connected $1$-marked $3$-manifolds which are pairs $\mathbf{M}=(M, \iota_M)$ where $M$ is a non-closed, connected, compact, oriented $3$-manifold and $\iota_M : \mathbb{D}^2 \hookrightarrow \partial M$ is an oriented embedding of the disc into the boundary of $M$. Morphisms in $\mathcal{M}^{(1)}_{\con}$ are (certain) oriented embeddings. The category $\mathcal{M}^{(1)}_{\con}$ has a natural braided balanced structure and (a restriction of) the Kauffman-bracket stated skein module is a braided balanced functor 
$$ \mathcal{S}_q: \mathcal{M}^{(1)}_{\con} \to {\mathcal{O}_q[\SL_2]}-\RComod$$
where comodules are taken over the ring $k=\mathbb{Z}[q^{\pm 1/4}]$. Consider also the field of rational functions $K:=\mathbb{Q}(q^{1/4})$ and write $\mathcal{S}_q(\mathbf{M})^{rat}:= \mathcal{S}_q(\mathbf{M})\otimes_k K$.
The interpretation of the stated skein module as a quantum representation space is summarized in the

\begin{theorem}\label{theorem0}
\begin{enumerate} Let $\mathbf{M}=(M, \iota_M)\in \mathcal{M}_{\con}^{(1)}$ and consider the associated unmarked $3$ manifold $M$.
\item The module $\mathcal{S}_{+1}(\mathbf{M}):= \mathcal{S}_q(\mathbf{M})\otimes_{q^{1/4}=1}\mathbb{Z}$ has a natural ring structure which is isomorphic to the ring of regular functions of the representation scheme $\mathcal{R}_{\SL_2}(\mathbf{M})$.
\item The inclusion $\mathcal{S}_q(M) \to \mathcal{S}_q(\mathbf{M})^{coinv}$ of the usual skein module into the subset of coinvariant vectors of the stated skein module is surjective and its kernel is included in the torsion submodule of $\mathcal{S}_q(M)$.
\item Suppose that the image of $\iota_M$ lies in a spherical boundary component of $\partial M$. Then every vectors of  $\mathcal{S}^{rat}_q(\mathbf{M})$ are coinvariant, so  $\mathcal{S}^{rat}_q(M) = \mathcal{S}^{rat}_q(\mathbf{M})$.
\end{enumerate}
\end{theorem}

In particular, the first item this theorem re-proves the classical result of Bullock that the skein module at $A=+1$ is isomorphic to the ring of regular functions of the character scheme. Note that we consider here $A=+1$ instead of $A=-1$ (in which case stated skein modules are non commutative algebras).
\vspace{2mm}
\par Let $G$ be a connected reductive complex algebraic group, let $\mathcal{C}_q^G$ be the ribbon category of integrable finite dimensional $\dot{U}_qG$ modules and $\overline{\mathcal{C}_q^G}= \mathcal{O}_qG-\RComod$ be the category of $\mathcal{O}_qG$ (right) comodules, thought as a free cocompletion of $\mathcal{C}_q^G$. Habiro's construction of a quantum representation space makes use of the full subcategory $\BT \subset \mathcal{M}^{(1)}_{\con}$ of elements $(H, \iota_H)$ such that $H$ is homeomorphic to a handlebody. Reinterpreting the constructions in \cite{Kerler_PresTanglesCat, CraneYetter_Categorification}, Habiro showed (after Kerler and Crane-Yetter) in \cite{Habiro_QCharVar} that the genus $1$ handlebody $\mathbf{H}_1\in \BT$ is a Hopf algebra object in $\BT$. Using the work of Kerler \cite{Kerler_PresTanglesCat} and Bobtcheva-Piergallini \cite{BobtchevaPiergallini}, we will associate to any braided quantum group $B_qG$ a braided functor $Q_{B_qG} : \BT \to \overline{\mathcal{C}_q^G}$ sending $\mathbf{H}_1$ to $B_qG$ and the quantum representation space $\Rep_q^G: \mathcal{M}^{(1)}_{\con}\to \overline{\mathcal{C}_q^G}$ is  defined as the left Kan extension $\Rep_q^G:= \Lan_{\iota} Q_{B_qG}$ for the inclusion $\iota: \BT \hookrightarrow \mathcal{M}^{(1)}_{\con}$. The \textit{quantum character variety} is then the submodule $\Char_q^G(\mathbf{M})\subset \Rep_q^G(\mathbf{M})$ of coinvariant vectors for the $\mathcal{O}_qG$-coaction.

\begin{theorem}\label{theorem1} Let $\mathbf{M}\in \mathcal{M}^{(1)}_{\con}$.
 There is an isomorphism $\Psi : \Rep_q^{\SL_2}(\mathbf{M}) \xrightarrow{\cong} \mathcal{S}_q(\mathbf{M})$.  It restricts to a surjective morphism $\Char_q^{\SL_2}(M) \to \mathcal{S}_q(M)$ which becomes an isomorphism  $\Char_q^{\SL_2, rat}(M) \to \mathcal{S}^{rat}_q(M)$ while working over $K$.
\end{theorem}

In other words, the skein modules and Habiro's quantum character varieties are isomorphic when considered over the field $K$ (i.e. neglecting the torsion). The existence of such a relation was conjectured in \cite{Habiro_QCharVar}. 
The main interest in this relation lies in the fact that, by definition of a left Kan extension, the quantum representation space admits the following tensor decomposition 
$$ \Rep_q^G(\mathbf{M}) = \mathbb{Z}[P_M] \otimes_{BT} Q_{B_qG}, $$
where $P_M:= \Hom_{\mathcal{M}_c^{(1)}}(\iota(\cdot), \mathbf{M})$ is the so-called \textit{quantum fundamental group} which only depends on $\mathbf{M}$ (and not on $G$) and $Q_{B_qG}$ only depends on $B_qG$ (and not on $\mathbf{M}$). Therefore, in order to prove relations between the quantum representations spaces of different $\mathbf{M}$ with $G$ fixed it suffices to prove such a relation at the level of the quantum fundamental groups (see the quantum Van Kampen theorem bellow). Similarily, in order to compare the quantum representation spaces of a fixed $\mathbf{M}$ for different $G$, it suffices to compare the associated braided quantum groups (this is how we compute the classical limit of stated skein modules).

\vspace{2mm}
\par As suggested by Habiro in \cite{Habiro_QCharVar}, the quantum representation space admits a Van Kampen type theorem that we now sketch and refer to Section \ref{sec4} for details. Let $\mathbf{M}_1, \mathbf{M}_2 \in \mathcal{M}^{(1)}_{\con}$ and consider a connected, compact, oriented surface $\Sigma$ with a distinguished closed interval $I_{\Sigma}\subset \partial \Sigma$ in its boundary and two oriented embeddings $\phi_1: \Sigma \hookrightarrow \partial M_1$, $\phi_2: \overline{\Sigma}\hookrightarrow \partial M_2$ sending the interval to some subarc of the based discs in such a way that $M_1\cup_{\Sigma}M_2$ becomes an element of $\mathcal{M}_{\con}^{(1)}$. The thickened surface $\Sigma\times [-1,1]$ with the based disc $I_{\Sigma}\times [-1,1]$ defines an element $\mathbf{\Sigma}\in \mathcal{M}^{(1)}$ such that $\mathcal{S}_q(\mathbf{\Sigma})$ is an algebra and such that $\mathcal{S}_q(\mathbf{M}_1)$ and $\mathcal{S}_q(\mathbf{M}_2)$ are left and right $\mathcal{S}_q(\mathbf{\Sigma})$ modules respectively.

\begin{theorem}\label{theorem2}
\par We have (explicit) isomorphisms of $\mathcal{O}_qG$ (resp. $\mathcal{O}_q[\SL_2]$)-comodules:
$$\Rep^G_q(\mathbf{M}_1\cup_{\mathbf{\Sigma}} \mathbf{M}_2) \cong \Rep^G_q(\mathbf{M}_1)\otimes_{\Rep^G_q(\mathbf{\Sigma})}\Rep^G_q(\mathbf{M}_2);  \quad  \mathcal{S}_q(\mathbf{M}_1\cup_{\mathbf{\Sigma}} \mathbf{M}_2) \cong \mathcal{S}_q(\mathbf{M}_1)\otimes_{\mathcal{S}_q(\mathbf{\Sigma})}\mathcal{S}_q(\mathbf{M}_2).$$
\end{theorem}

The proof of Theorem \ref{theorem2} relies on a quantum Van Kampen Theorem \ref{theorem_QVK_FG} on quantum fundamental groups conjectured by Habiro in \cite{Habiro_QCharVar}. In particular, Theorem \ref{theorem0} and Theorem \ref{theorem2} permit to reprove a theorem of Gunningham-Jordan-Safronov in \cite[Corollary 1]{GunninghamJordanSafranov_FinitenessConjecture} (see Section \ref{sec_QVK_Skein} for details). Recently, F.Costantino and T.Q.T.L\^e proved in \cite[Theorem $6.5$]{CostantinoLe_SSkeinTQFT} a theorem similar to Theorem \ref{theorem2} in the $\SL_2$ case where the glued marked $3$-manifolds are allowed to have more than one boundary disc.


\vspace{2mm}
\par Given $\mathbf{M} \in \mathcal{M}^{(1)}_{\con}$ and two embeddings $\phi_1: \Sigma \hookrightarrow \partial M$ and $\phi_2: \overline{\Sigma}\hookrightarrow \partial M$, one can also consider the marked $3$ manifold $\mathbf{M}_{\phi_1 \# \phi_2}$ obtained by gluing the two copies of $\Sigma$ inside $\partial M$ (see Section \ref{sec_selfgluing} for a precise definition). The maps $\phi_1$ and $\phi_2$ endow  $\Rep_q^G(\mathbf{M})$ with a structure of bimodule over $\Rep_q^G(\mathbf{\Sigma})$.
A consequence of the quantum Van Kampen theorem is the 

\begin{theorem}\label{theorem2.5} 
One has $\Rep_q^G(\mathbf{M}_{\phi_1 \# \phi_2}) \cong \mathrm{HH}^0( \Rep_q^G(\mathbf{\Sigma}), \Rep_q^G(\mathbf{M}))$ and $\mathcal{S}_q(\mathbf{M}_{\phi_1 \# \phi_2}) \cong \mathrm{HH}^0( \mathcal{S}_q(\mathbf{\Sigma}), \mathcal{S}_q(\mathbf{M}))$.
\end{theorem}

\par In Theorem \ref{theorem2.5}, the $0$-th Hochschild cohomology group is defined in a braided sense (see Definition \ref{def_HH0}). This theorem should be compared to \cite[Theorem $5.1$]{CostantinoLe_SSkeinTQFT} where the authors obtained a similar gluing theorem where this time they consider a marked $3$-manifold with two boundary discs glued together to give a single one and the Hochschild cohomology group is taken in a non-braided sense.
\par   This theorem can be used to study mapping tori. Consider again a connected, compact, oriented surface $\Sigma$ with  a distinguished closed interval $I_{\Sigma}\subset \partial \Sigma$, an oriented homeomorphism $\phi: \Sigma \xrightarrow{\cong} \Sigma$ preserving $I_{\Sigma}$ and its associated mapping torus $M_{\phi}=\quotient{\Sigma \times [-1,1]}{(x, -1)\sim (\phi(x), +1)}$. A well-known consequence of the Van Kampen theorem is that the fundamental group $\pi_1(M_{\phi})$ is isomorphic to the quotient of the free product $\pi_1(\Sigma)\star \pi_1(S^1)$ by the relation $\phi_*(\gamma)\sim t \gamma t^{-1}$ for $\gamma \in \pi_1(\Sigma)$ and $t$ a generator of $\pi_1(S^1)\cong \mathbb{Z}$, i.e. is a HNN extension. This means that we have a coequalizer in the category of groups:

$$ \begin{tikzcd}
\pi_1(\Sigma) 
\ar[r,shift left=.75ex,"\ad_*"]
  \ar[r,shift right=.75ex,swap,"\iota_1 \circ \phi_*"]
&
\pi_1(\Sigma \wedge S^1) 
\ar[r]
&
\pi_1(M_{\phi}).
\end{tikzcd}$$

As a consequence, we have a right exact sequence of algebras:
$$ \begin{tikzcd} 
 \mathcal{O}[R_{G}(\Sigma)]  
 \ar[rr, "\ad_*-\phi_*\otimes \eta"]
 &{}&
\mathcal{O}[R_{G}(\Sigma)] \otimes \mathcal{O}[G]
\ar[r]
&
\mathcal{O}[R_{G}(M_{\phi})]
\ar[r]
& 0
\end{tikzcd}
$$
or, equivalently, a right exact sequence of vector spaces:
$$ \begin{tikzcd} 
 \mathcal{O}[R_{G}(\Sigma)] \otimes (\mathcal{O}[R_{G}(\Sigma)] \otimes \mathcal{O}[G])
 \ar[rrr, "\mu \circ (\ad_* \otimes \id) - \mu \circ(\iota_1 \phi_*\otimes \id)"]
 &{}&{}&
\mathcal{O}[R_{G}(\Sigma)] \otimes \mathcal{O}[G]
\ar[r]
&
\mathcal{O}[R_{G}(M_{\phi})]
\ar[r]
& 0
\end{tikzcd}
$$

Adapting the preceding discussion to marked $3$-manifolds, we will define  morphisms
 $\Ad_{\Sigma} : \Rep_q^G(\mathbf{\Sigma}) \to \Rep_q^G(\mathbf{\Sigma})\overline{\otimes}B_qG$ and 
 $\Ad_{\Sigma} : \mathcal{S}_q(\mathbf{\Sigma}) \to \mathcal{S}_q(\mathbf{\Sigma})\overline{\otimes}B_q[\SL_2]$ and prove the following quantum analogue:

\begin{theorem}\label{theorem4}
One has right exact sequences
$$ \Rep_q^G(\mathbf{\Sigma}) \overline{\otimes} (\Rep_q^G(\mathbf{\Sigma})\overline{\otimes} B_qG) \xrightarrow{ \mu \circ (\iota_1\phi_*\otimes \id) - \mu^{top} \circ (\Ad_{\Sigma} \otimes \id)} \Rep_q^G(\mathbf{\Sigma})\overline{\otimes} B_qG \to \Rep_q^G(\mathbf{M}_{\phi})\to 0.$$
and 
$$ \mathcal{S}_q(\mathbf{\Sigma}) \overline{\otimes} (\mathcal{S}_q(\mathbf{\Sigma})\overline{\otimes} B_qG) \xrightarrow{ \mu \circ (\iota_1\phi_*\otimes \id) - \mu^{top} \circ (\Ad_{\Sigma} \otimes \id)} \mathcal{S}_q(\mathbf{\Sigma})\overline{\otimes} B_qG \to \mathcal{S}_q(\mathbf{M}_{\phi})\to 0.$$
\end{theorem}
Here $\mu^{top}$ is the twisted opposite product defined by $\mu^{top}:= \mu \circ \psi \circ (\theta \otimes \id)$ where $\psi$ and $\theta$ represent the braiding and the twist in $\mathcal{O}_qG-\RComod$.
What makes Theorem \ref{theorem4} interesting is that, whereas (stated) skein modules of $3$-manifolds are poorly understood, the stated skein algebras of surfaces have been well studied. In particular bases \cite{LeStatedSkein} and finite presentations \cite{KojuPresentationSSkein} of these algebras are well known, so Theorem \ref{theorem4} can be a valuable tool in order to study skein modules of mapping tori.
\vspace{2mm}
\par A second consequence of Theorem \ref{theorem2.5} concerns links exterior. Consider a braid $\beta \in B_n$ whose Markov closure is a link $L\subset S^3$. Denote by $\mathbf{M}_L$ the marked $3$-manifold obtained by removing an open ball from $M_L:= S^3 \setminus N(L)$ and by embedding the base disc in the boundary of the ball. By functoriality, the braid group acts on $\Rep_q^G(\mathbb{D}_n) \cong (B_qG)^{\overline{\otimes}n}$ where $\mathbb{D}_n$ is a marked disc with $n$ subdiscs removed.

\begin{theorem}\label{theorem5}
One has an isomorphism
$$ \Rep_q^G(\mathbf{M}_L) \cong \quotient{ (B_qG)^{\overline{\otimes}n}}{ \left( \mu^{top}(x\otimes y) - \beta_*(x)y \right)}.$$
\end{theorem}

In  \cite{MurakamiVdV_QRepSpaces}, Van der  Veen and that second author associated to a braid $\beta \in B_n$ an algebra $\mathcal{A}_{\beta}$ with a structure of $B_q \SL_2$-comodule defined as 
$$\mathcal{A}_{\beta}:= \quotient{ (B_q\SL_2)^{\overline{\otimes}n}}{ \left( \mu(x\otimes y) - \beta_*(x)y \right)}.$$
Note the similarity with the expression in Theorem \ref{theorem5}.
It is proved  in \cite{MurakamiVdV_QRepSpaces}  that $\mathcal{A}_{\beta}$ is an algebra in $\mathcal{O}_q\SL_2-\RComod$ and that  if two braids $\beta, \beta'$ have the same Markov closure $L\subset S^3$, then  $\mathcal{A}_{\beta}$ and $\mathcal{A}_{\beta'}$ are isomorphic.
Therefore the subalgebra $\mathcal{A}_{\beta}^{coinv}\subset \mathcal{A}_{\beta}$ of coinvariant vectors only depends on $L$ up to isomorphism and was named \textit{quantum character variety} in \cite{MurakamiVdV_QRepSpaces}. A skein reformulation and explicit computations were performed by the second author in \cite{Murakami_RIMS}.
 Even though they are different, Theorem \ref{theorem5} enlights the resemblance between the skein module of $M_L$ and the quantum character variety $\mathcal{A}_{\beta}^{coinv}$. 
Understanding the skein module of a knot exterior is a key feature in order to compute the peripheral ideal of a knot and to find $q$-differential equations satisfied by the Jones polynomials (see \cite{FGL_SkeinApolynomial, GaroufalidisLe_JonesqHolonomic, Garoufalidis_AJ, LeAJ} for details). In a future work, we plan to adapt the techniques developed in \cite{Murakami_RIMS} to deduce from Theorem \ref{theorem5} informations on the skein modules of links exteriors and their peripheral ideal.
\vspace{2mm}
\par 
\textit{Plan of the paper}
\par In Section \ref{sec1} we introduce the braided balanced category $\mathcal{M}^{(1)}_{\con}$ of connected $1$-marked $3$-manifolds, its  subcategory $\BT$ and Habiro's Hopf algebra object in $\BT$. In Section \ref{sec2} we recall the definition of Habiro's quantum representation space functor $\Rep_q^{G}$ and introduce the quantum fundamental group. In Section \ref{sec3} we recall the definition of stated skein modules and relate them to the quantum representation spaces proving Theorem \ref{theorem1}. We then identify the submodule of coinvariant vectors of stated skein modules with the usual skein module. In Section \ref{sec_classical}, we prove that the stated skein module at $q^{1/4}=1$ is isomorphic to the ring of regular functions of the representation scheme thus finishing the proof of Theorem \ref{theorem0}.
In Section \ref{sec4} we prove the quantum Van Kampen theorem for quantum fundamental groups and deduce Theorems \ref{theorem2} and \ref{theorem2.5}.
 Sections \ref{sec5} and \ref{sec6} are devoted to the proofs of Theorems \ref{theorem4} and \ref{theorem5} respectively.  In the appendix, we show how the work of Kerler and Bobtcheva-Piergallini can be used in order to find a finite presentation for $\BT$ and to prove the existence of the functor $Q_{B_qG}$, as conjectured by Habiro.

\vspace{2mm}
\par 
\textit{Acknowledgments.} The authors thank S.Baseilhac, D.Calaque,  F.Costantino, M.De Renzi,  R.Detcherry,  K.Habiro, D.Jordan,  T.Q.T.L\^e ,  A.Quesney and R.Van der Veen for valuable conversations. The first author acknowledges support from the Japanese Society for Promotion of Sciences (JSPS), from the Centre National de la Recherche Scientifique (CNRS) and from the European Research Council (ERC DerSympApp) under the European Union’s Horizon 2020 research and innovation program (Grant Agreement No. 768679).
Part of this work was accomplished during his stay at Waseda University.
 The second author was supported by KAKENHI 20H01803 and 20K20881.

\section{Marked $3$-manifolds}\label{sec1}

\subsection{The categories $\mathcal{M}$ and $\mathcal{M}^{(1)}_c$}\label{sec_M}

\begin{convention}\label{convention_bold}
Let $\Top$ be the category of Haussdorf, locally compact topological spaces and $\Cat_{\Top}$ the category of (small) categories enriched over $\Top$. For $\boldsymbol{\mathcal{C}}\in \Cat_{\Top}$, one can associate its homotopy category $\mathcal{C}=ho(\boldsymbol{\mathcal{C}})\in \Cat$ having the same objects and such that $\Hom_{\mathcal{C}}(x,y):= \pi_0( \Maps_{\boldsymbol{\mathcal{C}}}(x,y))$. In this paper, we will write topological categories using a bold symbol (like $\boldsymbol{\mathcal{C}}$) and write their homotopy categories using the same non-bold symbol (like $\mathcal{C}$).
\end{convention}

\begin{definition}\label{def_MMfd}
Let $\mathbb{D}^2:= \{ (x,y) \in \mathbb{R}^2 | x^2+y^2=1\}$. We call \textit{height} of $(x,y)\in \mathbb{D}^2$ the number $h(x,y):=y$. 

\begin{enumerate}
\item For $n\geq 1$, a  $n$-\textit{marked }$3$-\textit{manifold} is a pair $\mathbf{M}=(M, \iota_M)$ where $M$ is a compact, oriented $3$-manifold and $\iota_M : \mathbb{D}^2 \bigsqcup \ldots \bigsqcup \mathbb{D}^2 \to \partial M$ an oriented embedding of $n$ copies of the disc $\mathbb{D}^2$ in the boundary of $M$. We write $\iota_M^{(1)}, \ldots, \iota_M^{(n)}$ its individual disc embeddings and denote by $\mathbb{D}_M^{(1)}, \ldots, \mathbb{D}_M^{(n)}$ their (pairwise disjoint) images. The \textit{height } of a point  $p=\iota_M^{(i)}(x) \in \mathbb{D}_M^{(i)}$ is $h(p):=h(x)$. 
By convention, a $0$-marked manifold is just $\mathbf{M}=M$ a compact, oriented $3$-manifold (that we will call \textit{unmarked}) and a \textit{marked }$3$-\textit{manifold } is a $n$-marked $3$-manifold for some $n\geq 0$. 
\item An \textit{embedding} $f: \mathbf{M}_1\to \mathbf{M}_2$ of marked $3$-manifolds is an oriented embedding $f:M_1 \to M_2$ of the underlying $3$-manifolds such that: $(1)$ $f$ embeds each marked disc $\mathbb{D}_{M_1}^{(i)}$ into a marked disc $\mathbb{D}_{M_2}^{(j)}$ through an embedding $\mathbb{D}_{M_1}^{(i)}\hookrightarrow \mathbb{D}_{M_2}^{(j)}$ which is height increasing, i.e. $h(x)< h(y)$ implies $h(f(x))< h(f(y))$ and $(2)$ if two discs $\mathbb{D}_{M_1}^{(i)}, \mathbb{D}_{M_1}^{(j)}$ are embedded into the same disc $\mathbb{D}_{M_2}^{(k)}$ then their heights are disjoint, i.e. $h(f(\mathbb{D}_{M_1}^{(i)}))\cap h(f(\mathbb{D}_{M_1}^{(j)}))= \emptyset$. In particular the set of discs of $\mathbf{M}_1$ which are mapped into a given disc of $\mathbf{M}_2$ are totally ordered by their heights.
\item Marked $3$-manifolds with embeddings form a topological category $\boldsymbol{\mathcal{M}}$, where the sets of embeddings are equipped with their compact-open topology. For $n\geq 0$, we denote by $\boldsymbol{\mathcal{M}}^{(n)}$ the full subcategory of $n$-marked $3$-manifolds. Following Convention \ref{convention_bold}, we denote by $\mathcal{M}, \mathcal{M}^{(n)}$ the homotopy categories of $\boldsymbol{\mathcal{M}}, \boldsymbol{\mathcal{M}}^{(n)}$ respectively and by $\mathcal{M}^{(n)}_{\con}$ the full subcategory of connected $n$-marked $3$-manifolds.
\item A $n$-\textit{marked surface} is a pair $\mathbf{\Sigma}=(\Sigma, \mathcal{A})$ where $\Sigma$ is a compact oriented surface and $\mathcal{A}$ an oriented embedding of $n$ copies of $I:=[-1,1]$ into the boundary of $\Sigma$. We associate to $\mathbf{\Sigma}$ an element $\mathbf{\Sigma}\times I \in \mathcal{M}^{(n)}$ by smoothing the corners of $\Sigma\times I$ and identifying $[-1,1]^2$ with $\mathbb{D}^2$. We denote by $\bold{MS} \subset \boldsymbol{\mathcal{M}}$ the full subcategory generated by elements isomorphic to such a thickened marked surface.
\end{enumerate}
\end{definition}

Let $\mathbb{B}^3$ be the unit ball of $\mathbb{R}^3$.
The \textit{bigon} $\mathbb{B}\in \boldsymbol{\mathcal{M}}$ is the ball $\mathbb{B}^3$ with two boundary discs in its boundary. We can think of $\mathbb{B}$ as a thickened disc with two boundary arcs on its boundary (hence the name "bigon"). Similarly, we call \textit{triangle} $\mathbb{T} \in \boldsymbol{\mathcal{M}}$ the ball $\mathbb{B}^3$ with three boundary discs in its boundary. Again, $\mathbb{T}$ can be thought as a thickened disc with three boundary arcs (the edges of the triangle).

\begin{definition}\label{def_operations}
The category of marked $3$-manifolds admits the following three natural operations. 
\begin{enumerate}
\item The \textit{disjoint union} $\bigsqcup$ which endows $\boldsymbol{\mathcal{M}}$ with a symmetric monoidal structure in an obvious way.
\item The \textit{gluing operation}: given $\mathbb{D}^{(i)}, \mathbb{D}^{(j)}$ two distinct boundary discs of $\mathbf{M}$, we denote by $\mathbf{M}_{\mathbb{D}^{(i)}\# \mathbb{D}^{(j)}}$ the marked $3$-manifold obtained from $\mathbf{M}$ by gluing the two discs  $\mathbb{D}^{(i)}, \mathbb{D}^{(j)}$ using $\iota_M^{(j)} \circ (\iota_M^{(i)})^{-1}$.
\item The \textit{fusion operation}:  starting again with $\mathbb{D}^{(i)}, \mathbb{D}^{(j)}$ two distinct boundary discs of $\mathbf{M}$, we denote by $\mathbf{M}_{\mathbb{D}^{(i)}\circledast \mathbb{D}^{(j)}}$ the marked $3$-manifold obtained from $\mathbf{M}\bigsqcup \mathbb{T}$ by gluing $\mathbb{D}^{(i)}$ with the first boundary disc of $\mathbb{T}$ and gluing $\mathbb{D}^{(j)}$ with the second boundary disc of $\mathbb{T}$. In the particular case where $\mathbf{M}=\mathbf{M}_1 \bigsqcup \mathbf{M}_2$ with $\mathbf{M}_1, \mathbf{M}_2 \in \boldsymbol{\mathcal{M}}^{(1)}$ and $\mathbb{D}^{(i)}, \mathbb{D}^{(j)}$ are the unique boundary discs of  $\mathbf{M}_1, \mathbf{M}_2$, we simply write $\mathbf{M}_1\wedge \mathbf{M}_2:= \mathbf{M}_{\mathbb{D}^{(i)}\circledast \mathbb{D}^{(j)}}\in  \boldsymbol{\mathcal{M}}^{(1)}$. Then $\wedge$ endows ${\mathcal{M}}^{(1)}$ with a structure of monoidal category. We also denote by $\iota_1: \mathbf{M}_1 \to \mathbf{M}_1\wedge \mathbf{M}_2$ the map identifying $M_1$ with the union $\mathbb{T}\cup M_1$ inside $M_1\wedge M_2$. The morphism $\iota_2: \mathbf{M}_2\to \mathbf{M}_1\wedge \mathbf{M}_2$ is defined similarly.  
\end{enumerate}
\end{definition}

\begin{convention} We call \textit{braided balanced category} a braided category $\mathcal{C}$ equipped with a compatible twist (i.e. with an automorphism $\theta$ of the identity functor $\id: \mathcal{C}\to \mathcal{C}$ such that $\theta_{V\otimes W}= (\theta_V \otimes \theta_W)c_{W,V}c_{V,W}$). 
Beware that some authors, such as Salvatore-Wahl \cite{SalvatoreWahl_FD2} or Fresse \cite{Fresse_Book} call "ribbon category" what we call braided balanced category. However for quantum topologists, such as Turaev \cite{Tu}, a ribbon category is a braided balanced category equipped with left and right dualities compatible with the braiding and the twist. In this paper we follow Turaev's terminology. 
\end{convention}

By the work in \cite{SalvatoreWahl_FD2} based on \cite{Fiedorowicz} (surveyed in \cite{Fresse_Book}) a braided balanced category can be defined alternatively as a homotopy category $\mathcal{C}=ho(\boldsymbol{\mathcal{C}})$ where $\boldsymbol{\mathcal{C}}$ is an algebra over the framed little discs operad in $\Cat_{\Top}$. In particular, braided balanced categories give rise to locally constant factorization algebras on surfaces (\cite{BenzviBrochierJordan_FactAlg1}).
\par 
The operad of framed little discs $f\mathcal{D}_2$ naturally acts on $\boldsymbol{\mathcal{M}}^{(1)} \in \Cat_{\Top}$ as follows. For $n\geq 1$, an element  $c\in f\mathcal{D}_2(n)$ (a framed little discs configuration) is the data $c=(\iota_1, \ldots, \iota_n, \iota_{out})$ where $\iota_{out}: \mathbb{D}^2 \hookrightarrow \mathbb{R}^2$ is the embedding of an "outer" disc in the plane and $\iota_1, \ldots, \iota_n: \mathbb{D}^2 \hookrightarrow \mathbb{R}^2$ are oriented embeddings of the disc with pairwise disjoint image and such that each disc $\iota_i(\mathbb{D}^2)$ is included in the outer disc $\iota_{out}(\mathbb{D}^2)$. To such a configuration, we associate an element $B_c=(\mathbb{B}^3, \{j_1, \ldots, j_{n+1}\}) \in \boldsymbol{\mathcal{M}}^{(n+1)}$ by identifying $\mathbb{R}^2 \cup \{\infty\}$ with the boundary of the ball $\mathbb{B}^3$ so that each $\iota_i$ defines an embedding $j_i$ of $\mathbb{D}^2$ into the boundary of $\mathbb{B}^3$ and the complementary $\partial \mathbb{B}^3 \setminus \iota_{out}(\mathring{\mathbb{D}}^2)$ of the interior of the outer disc defines the $n+1$-th boundary disc $\mathbb{D}^{(n+1)}$. More precisely, the restriction of $\iota_{out}$ to the boundary of $\mathbb{D}^2$ defines the restriction  $j_{n+1}: \partial \mathbb{D}^2 \cong \partial \mathbb{D}^{(n+1)}$ that we extend canonically to $\mathbb{D}^2$ by first identifying $\partial \mathbb{B}^3\setminus \{0\}$ with $\mathbb{R}^2$ using the stereographic map, so that $\mathbb{D}^{(n+1)}$ is identified with a disc in $\mathbb{R}^2$ centered in $0$, and then applying the (unique) homothety between this disc and $\mathbb{D}^2$ centered in $0$.

 Now for $\mathbf{M}_1, \ldots, \mathbf{M}_n \in \boldsymbol{\mathcal{M}}^{(1)}$ and $c\in f\mathcal{D}_2(n)$, we denote by $\mathfrak{o}(c; \mathbf{M}_1, \ldots, \mathbf{M}_n) \in \Mun$ the marked $3$-manifold obtained by gluing each $\mathbf{M}_i$ to $B_c$ along its $i$-th disc. By construction, the assignation $(\mathbf{M}_1, \ldots, \mathbf{M}_n) \mapsto \mathfrak{o}(c; \mathbf{M}_1, \ldots, \mathbf{M}_n) $ is functorial so $c$ induces a topological functor $\mathfrak{o}(c; \bullet) : {\Mun}^n \to \Mun$ which endows $\Mun$ with a  structure of $f\mathcal{D}_2$-algebra in $\Cat_{\Top}$.
 
 \begin{definition}\label{def_ribbonstr}
 We endow $\mathcal{M}^{(1)}$ with the braided balanced structure coming from its action of the framed little discs operad.
 \end{definition}
 
 \begin{remark}
 \begin{enumerate}
 \item By construction, the monoidal structure on $\mathcal{M}^{(1)}$  underlying the ribbon structure of Definition \ref{def_ribbonstr} coincides with the monoidal structure $\wedge$ of Definition \ref{def_operations}.
 \item Consider the "forgetful" functor $F: \mathcal{M}^{(1)} \to \Top^{\bullet}$ to the category of pointed topological spaces sending $(M, \iota_M)$ to the pointed space $(M, \iota_M(0))$. By definition, $F$ is lax monoidal, i.e. $F(\mathbf{M}_1\wedge \mathbf{M}_2)\cong F(\mathbf{M}_1) \wedge F(\mathbf{M}_2)$ for the wedge product of $\Top^{\bullet}$. As we shall review, the quantum fundamental group and quantum representations space are monoidal functors valued in $\mathcal{M}^{(1)}_{\con}$ constructed by analogy with the classical fundamental group $\pi_1: \Top^{\bullet} \to \Gp$ and quantum representation space $\mathcal{O}[\Hom(\pi_1(\bullet), \SL_2)] : \Top^{\bullet} \to \Alg$. However one should be really careful when using this analogy: whereas the wedge product $\wedge$ is a coproduct in $\Top^{\bullet}$, the product $\wedge$ in $\mathcal{M}^{(1)}_{\con}$ is not a coproduct. Indeed, due to the fact that we consider  isotopy classes of embeddings in $\mathcal{M}^{(1)}_{\con}$, rather than just continuous maps as in $\Top^{\bullet}$, its is no longer true that two isotopy classes of embeddings $f: \mathcal{M}_1\to \mathcal{M}_3$ and $g:\mathcal{M}_2 \to \mathcal{M}_3$ induce an isotopy class of embedding $\mathcal{M}_1 \wedge \mathcal{M}_2 \to \mathcal{M}_3$.
 \end{enumerate}
 \end{remark}

\subsection{Bottom tangles}\label{sec_BT}

\begin{definition}\label{def_BTCat}
For each $n\geq 0$, let $\mathbb{D}_n$ be the disc $\mathbb{D}^2$ with $n$ pairwise disjoint open subdiscs removed and fix a boundary arc $a\subset \partial \mathbb{D}^2$ so that $(\mathbb{D}_n, \{a\})$ becomes a marked surface.
We set  $\mathbf{H}_n:=  (\mathbb{D}_n, \{a\}) \times [-1, 1]\in \Mun_c$. It is a genus $n$ handlebody with one boundary disc in its component.
We denote by $\mathbf{BT}\subset \Mun_c$ the full subcategory generated by elements isomorphic to some $\mathbf{H}_n$. We will often write $\mathbf{BT}(n,m):= \Maps_{\Mun_c}(\mathbf{H}_n, \mathbf{H}_m)$.
\end{definition}

For $\mathbf{M}\in \Mun$ and $n\geq 0$, Habiro introduced in \cite{Habiro_QCharVar} an alternative description of the set $\Hom_{\mathcal{M}_c^{(1)}}(\mathbf{H}_n, \mathbf{M})$ by means of bottom tangles that we now introduce.

\begin{definition}\label{def_tangles}
\begin{enumerate}
\item For $\mathbf{M}\in \M$, a \textit{tangle} in $\mathbf{M}$  is a  compact framed, properly embedded $1$-dimensional manifold $T\subset M$ such that each point $p \in \partial T$ lies in the interior of some boundary disc $\mathbb{D}_M^{(i)}$ and has framing parallel to the height direction pointing towards the increasing height direction.
 Here, by framing, we refer to a section of the unitary normal bundle of $T$. Moreover,  for a boundary disc $\mathbb{D}_M^{(i)}$ we impose that no two points of $\partial_i T:= \partial T \cap \mathbb{D}_M^{(i)}$  have the same height, hence the set $\partial_i T$ is totally ordered by the heights. Two tangles are isotopic if they are isotopic through an isotopy of tangles that preserves the boundary height orders. By convention, the empty set is a tangle only isotopic to itself.
\item For $\mathbf{M}\in \Mun$, a tangle $T\subset M$ is called a \textit{bottom tangle} if $(1)$ $T$ does not have any closed component and $(2)$ if $T=T_1\cup \ldots \cup T_n$ are the connected components of $T$ then, up to reindexing, $i<j$ implies that $h(\partial T_i)< h(\partial T_j)$, i.e. the heights of both points of $\partial T_i$ are smaller than the heights of both points of $\partial T_j$. So the connected components are totally ordered by their heights. Figure \ref{fig_BTexemple} illustrates such a bottom tangle. For $n\geq 1$, a $n$-bottom tangle is a bottom tangle with $n$ connected components and, by convention, the only $0$-bottom tangle is the empty tangle.
\item We denote by $\mathbf{P}_n(\mathbf{M})$ the space of $n$-bottom tangles in $\mathbf{M}$ and $\PP_n(\mathbf{M}):= \pi_0(\mathbf{P}_n(\mathbf{M}))$.
\item For $n\geq 1$, the \textit{trivial bottom tangle} of $\mathbf{H}_n$ is the $n$-bottom tangle $T_n$ drawn in Figure \ref{fig_BTexemple} such that $H_n$ retracts on $T_n$ and the framing points towards the direction of $1$ in $H_n=\mathbb{D}_n \times [-1,1]$. The \textit{trivial bottom tangle} of $\mathbf{H}_0$ is the empty tangle.
\end{enumerate}
\end{definition}

\begin{figure}[!h] 
\centerline{\includegraphics[width=8cm]{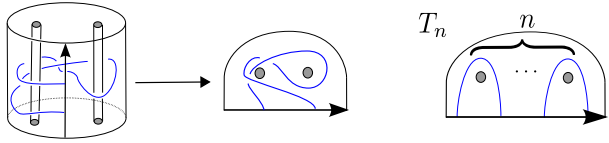} }
\caption{On the left: a $1$-bottom tangle in $H_2=\mathbb{D}_2\times I$ and its planar diagram projection in $\mathbb{D}_2$. The arrow depicts the height order and we use the blackboard framing. On the right: the trivial bottom tangle $T_n$.} 
\label{fig_BTexemple} 
\end{figure}

\begin{definition}\label{def_theta}
For $n\geq 0$ and $\mathbf{M}\in \Mun_c$, let  $\theta: \Hom_{\Mun_c}(\mathbf{H}_n, \mathbf{M}) \to \mathbf{P}_n(\mathbf{M})$ be the continuous map sending an embedding $f: \mathbf{H}_n\to \mathbf{M}$ to the image $\theta(f):=f(T_n)$ of the trivial bottom tangle by $f$.
\end{definition}

\begin{lemma}\label{lemma_BT}
The map  $\theta: \Hom_{\Mun_c}(\mathbf{H}_n, \mathbf{M}) \xrightarrow{\sim} \mathbf{P}_n(\mathbf{M})$ is an equivalence of homotopy.
\end{lemma}

\begin{proof}
The lemma follows from the facts that $\theta$ is clearly surjective and that the fibers $\theta^{-1}(T)$ are contractible. 

\end{proof}

In particular, we get a bijection $\theta_* : \Hom_{\mathcal{M}^{(1)}}(\mathbf{H}_n, \mathbf{M}) \xrightarrow{\cong} P_n(\mathbf{M})$. Note that $\mathbf{H}_a \wedge \mathbf{H}_b \cong \mathbf{H}_{a+b}$ so $\BT$ is a PROP and we can draw the morphisms in $\BT(n,m)$ as $n$-bottom tangles in $\mathbf{H}_m$ as in Figure \ref{fig_BTcomposition}. 

\begin{figure}[!h] 
\centerline{\includegraphics[width=8cm]{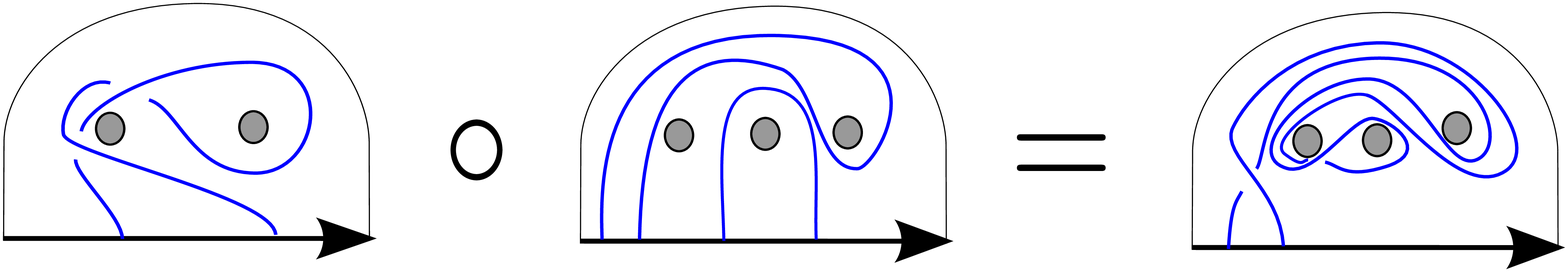} }
\caption{An illustration of the composition law $\circ : \BT(1,2)\times \BT(2,3) \to \BT(1,3)$ in terms of bottom tangles.} 
\label{fig_BTcomposition} 
\end{figure} 

 Let us now briefly introduce a different PROP $\mathfrak{bt}$, defined by Habiro in \cite{Habiro_BottomTangles}, to which we refer for further details. Let ${\mathcal{T}}$ be the category of tangles whose objects are words in the generators $\uparrow$ and $\downarrow$ (for instance $w=\uparrow \uparrow \downarrow \uparrow \downarrow$) and whose morphisms are isotopy classes of framed oriented tangles $T:w\to w'$ in $\mathbb{D}^2 \times I $ such that the endpoints at top are prescribed by $w$ and the endpoints at the bottom are prescribed by $w'$ (see Figure \ref{fig_BTaction} (a) for an instance of morphism between $\downarrow\uparrow\downarrow\uparrow\downarrow \uparrow$ and $\downarrow\uparrow\downarrow\uparrow$). Composition is given by vertical pasting and the monoidal structure is given by horizontal pasting (see \cite{Habiro_BottomTangles} for details). Let $b:= \downarrow \uparrow \in {\mathcal{T}}$, $\eta_b:= \CAP \in {\mathcal{T}}(\mathbf{1}, b)$ and $\eta_n := \eta_b^{\otimes n} \in {\mathcal{T}}(\mathbf{1}, b^{\otimes n})$ for $n\geq 0$. Set 
 $$ \mathfrak{bt}(n,m) := \{ T \in \mathcal{T}(b^{\otimes n}, b^{\otimes m}) | T \eta_n \sim_h \eta_m \}, $$
 where $\sim_h$ means an isotopy that allows the change of crossings $\crosspos\leftrightarrow \crossneg$ (Part (a) of Figure \ref{fig_BTaction} gives an example of such morphism).

\begin{figure}[!h] 
\centerline{\includegraphics[width=8cm]{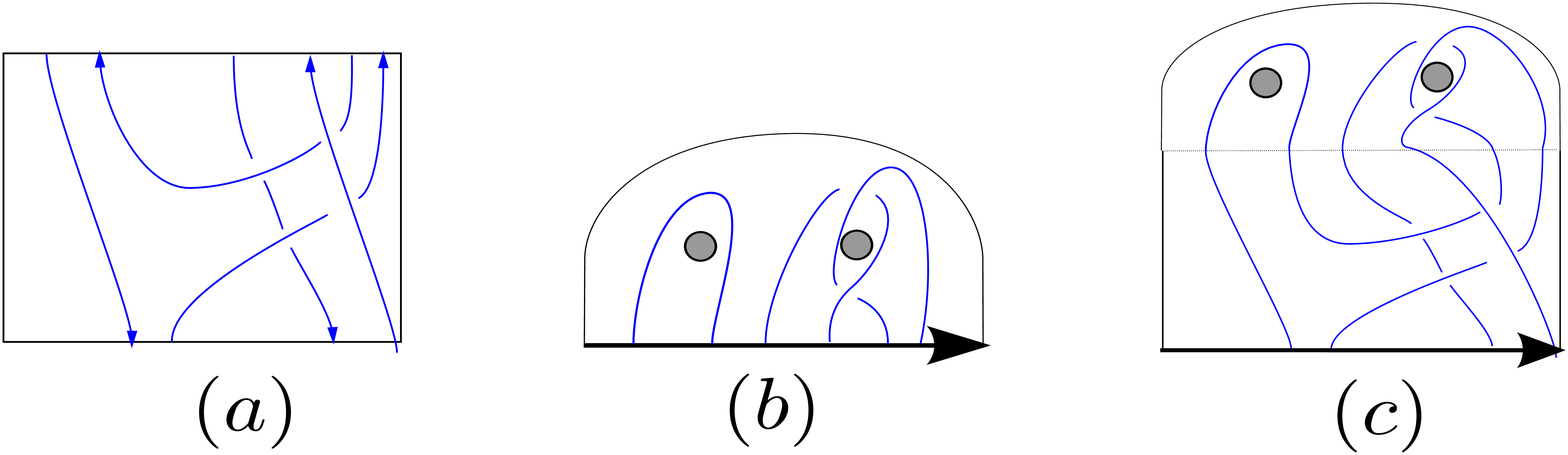} }
\caption{$(a)$ A morphism  $T \in \mathfrak{bt}(3,2)$. $(b)$ A bottom tangle $T' \in P_3(\mathbf{H}_2)$. $(c)$ The resulting bottom tangle $T' \cdot T \in P_2(\mathbf{H}_2)$.} 
\label{fig_BTaction} 
\end{figure} 
 
 \begin{definition}\label{def_bt}
 We denote by $\mathfrak{bt}$ the subcategory of $\mathcal{T}$ whose objects are elements $b^{\otimes n}$ for $n\geq 0$ and whose morphisms are the $ \mathfrak{bt}(n,m)$. 
 
 \par 
 For $T \in \mathfrak{bt}(n,m)$ and $T' \in P_n(\mathbf{M})$ we denote by $T'\cdot T \in P_m(\mathbf{M})$ the bottom tangle obtained by gluing $T$ and $T'$ as in Figure \ref{fig_BTaction}.
 \end{definition}

\subsection{The Crane-Habiro-Kerler-Yetter braided Hopf algebra}\label{sec_HopfAlg}

\begin{definition} Let $\mathcal{C}$ be a braided category. A \textit{dual BP Hopf algebra object  in } $\mathcal{C}$ is an object $H\in \mathcal{C}$, together with morphisms $\mu : H\otimes H \to H$, $\eta: \mathds{1}\to H$, $\Delta: H\to H \otimes H$, $\epsilon: H \to \mathds{1}$, $S^{\pm 1}:H\to H$, $\theta^{\pm 1}: H \to \mathds{1}$ such that, writing $B:=  \theta(\theta^{-1} \otimes  \mu \otimes \theta^{-1})(\Delta\otimes \Delta) : H\otimes H \to \mathds{1}$ and $B^-:= \theta^{-1}(\theta\otimes  \mu\circ c_{H,H} \otimes  \theta )(\Delta\otimes \Delta) : H\otimes H \to \mathds{1}$ , then the morphisms satisfy the relations of Figure \ref{fig_BPHopf}.
\end{definition}

\begin{figure}[!h] 
\centerline{\includegraphics[width=14cm]{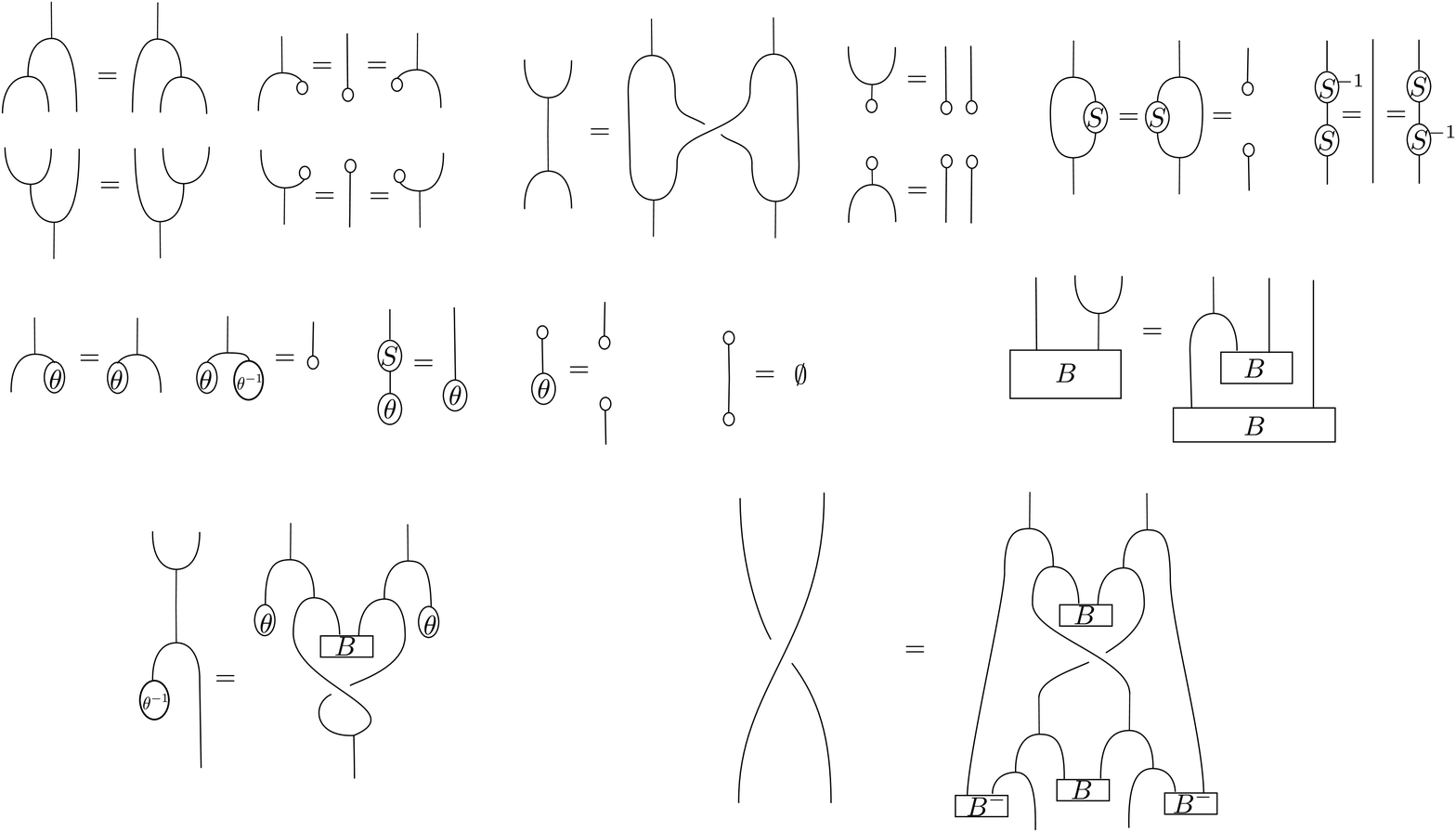} }
\caption{The relations defining a dual BP Hopf algebra. Diagrams are read from top to bottom. The first line are the relations for $H$ for being a (braided) Hopf algebra. The second line  asserts that $\theta$ is a cotwist and that the pairing $B$ is compatible with the product. The third line are the so-called BP relations.} 
\label{fig_BPHopf} 
\end{figure}

\begin{definition}\label{def_HabiroHopf}(Habiro)
We endow $\mathbf{H}_1 \in \BT$ with a structure of dual BP Hopf algebra object   $\mathcal{H} = (\mathbf{H_1}, \mu, \eta, \Delta, \epsilon, S, \theta)$  in the braided category $(\BT, \wedge, c_{\cdot, \cdot})$ with the structure morphisms drawn in Figure \ref{fig_BTHopfAlg}. The braiding $\Psi:= c_{\mathbf{H}_1, \mathbf{H}_1}$ is also depicted. It has a structure of right comodule over itself (in the braided sense) with comodule map 
$$\ad:=(\id \wedge \mu)\circ (\Psi \wedge \id) \circ (S\wedge \id \wedge \id) \circ \Delta^{(2)}= \adjustbox{valign=c}{\includegraphics[width=1.5cm]{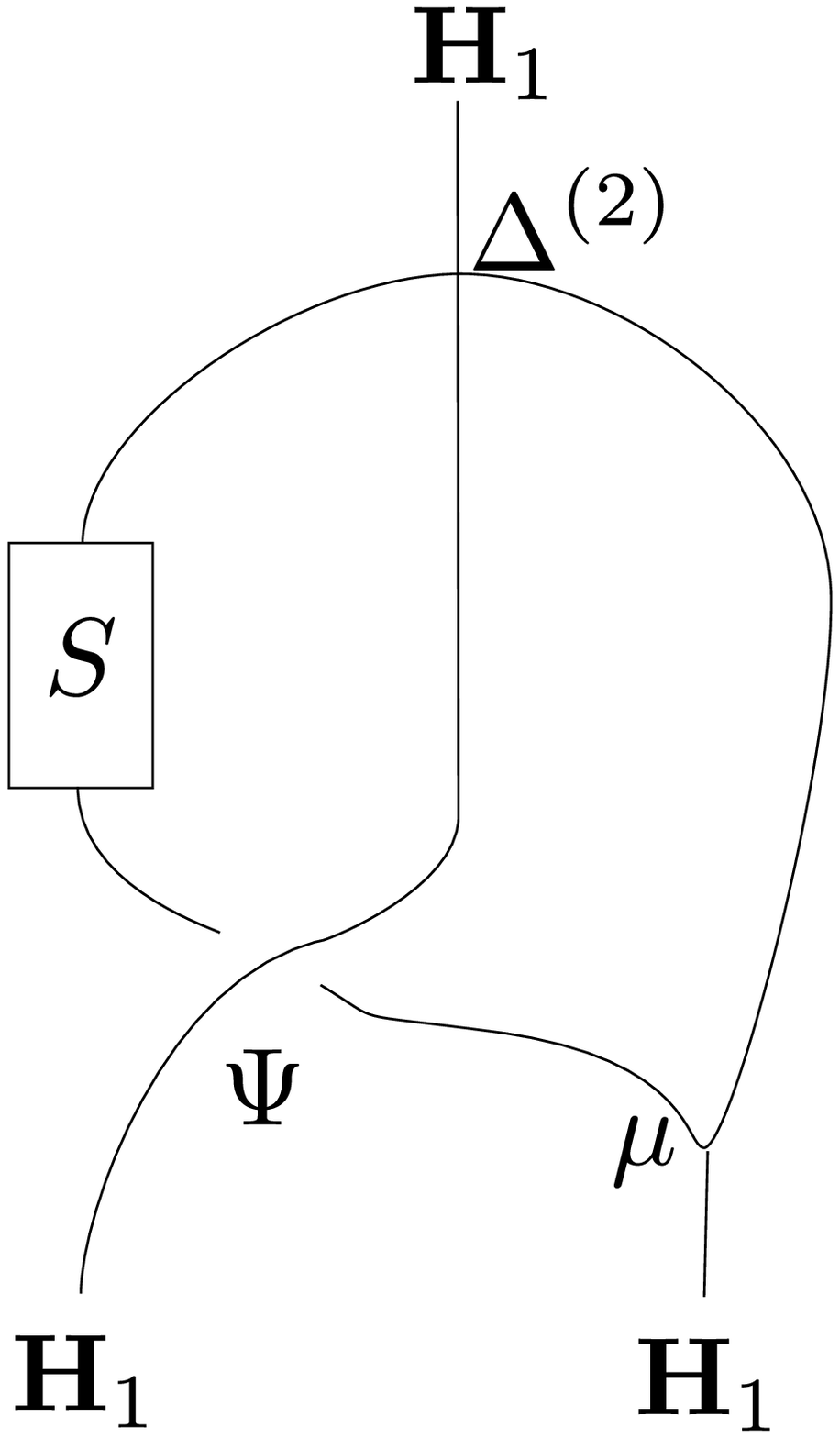}}.$$ 
\begin{figure}[!h] 
\centerline{\includegraphics[width=14cm]{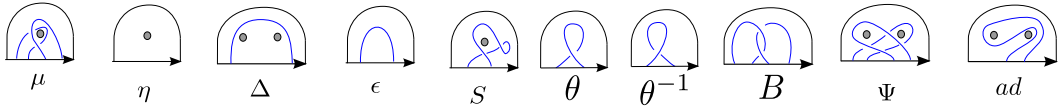} }
\caption{The product $\mu$, unit $\eta$, coproduct $\Delta$, counit $\epsilon$, antipode $S$, cotwist $\theta$ and its inverse $\theta^{-1}$, the associated pairing $B$,  the braiding $\Psi$ and the right comodule map $\ad$ in $\mathcal{H}$.} 
\label{fig_BTHopfAlg} 
\end{figure} 
\end{definition}

As illustrated in Figure \ref{fig_commutativity}, the Hopf algebra $\mathcal{H}$ is braided commutative in the sense that 
$$(\id \wedge \mu) (\ad\otimes \id) =(\id \wedge \mu)(\Psi\wedge \id)(\id\wedge \ad)\Psi.$$

\begin{figure}[!h] 
\centerline{\includegraphics[width=12cm]{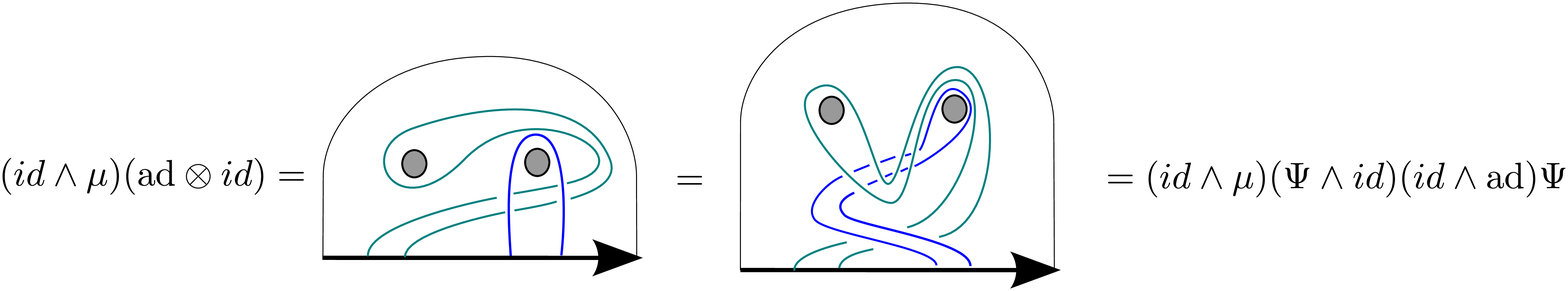} }
\caption{An illustration of the braided commutativity of $\mathcal{H}$.} 
\label{fig_commutativity} 
\end{figure}

\begin{remark}\label{remark_CYK}
Let $\Cob_{123}$ be the monoidal $2$-category with objects the oriented closed $1$-manifolds, with $1$-morphisms the oriented compact surface cobordisms and with $2$-morphisms the oriented compact $3$-manifolds with corners (see e.g. \cite{BDSV_ETQFT2} for details). The circle $\mathbb{S}^1$ and the empty set $\emptyset$ are objects in $\Cob_{123}$ and $\Hom_{\Cob_{123}}(\mathbb{S}^1, \emptyset)$ is a category whose objects are surfaces $\Sigma$ bounding $\partial \Sigma \cong \mathbb{S}^1$ and morphisms are $3$-manifolds with corners between two such surfaces. Let $\mathcal{C}^{CYK}\subset \Hom_{\Cob_{123}}(\mathbb{S}^1, \emptyset)$ be the subcategory with connected surfaces and connected $3$-manifolds. The category $\mathcal{C}^{CYK}$ has a braided balanced structure where the action of the framed little discs operad is given by gluing the inner discs of a little discs configuration to the boundaries of a family of such $\Sigma \in \mathcal{C}^{CYK}$. In particular, the monoidal structure is given by gluing two surfaces $\Sigma_1, \Sigma_2$ with circle boundary along a pair of pants to get a new one. 
Kerler \cite{Kerler_QInv} and Crane-Yetter \cite{CraneYetter_Categorification} independently noticed that the one-holed torus $\Sigma_{1,1} \in \mathcal{C}^{CYK}$ has a structure of Hopf algebra object, say $\mathcal{H}^{CYK}$, in $\mathcal{C}^{CYK}$. There is a braided balanced functor
$$ \partial : \mathcal{M}_{\con}^{(1)} \to \mathcal{C}^{CYK}$$ 
sending $\mathbf{M}$ to $\partial M \setminus \mathring{\mathbb{D}}_M$ and sending an embedding $f: \mathbf{M}_1 \to \mathbf{M}_2$ to the cobordism with corners $M_2 \setminus f(M_1)$.  The  restriction of $\partial$ to  $\BT$ is faithful  and the Crane-Yetter-Kerler Hopf algebra $\mathcal{H}^{CYK}$ is the image by $\partial$ of Habiro's Hopf algebra $\mathcal{H}$. This explains the origin of this Hopf algebra. Note that, due to the fact that $\mathcal{C}^{CYK}$ has more morphisms than $\mathcal{M}^{(1)}_{\con}$ ($\partial$ is not full), the Hopf algebra $\mathcal{H}^{CYK}$ admits a richer algebraic structure than $\mathcal{H}$ since it admits also an integral and a cointegral (see Appendix \ref{appendix_presentation_BT} for details).

\end{remark}

The following theorem is a direct consequence of the deep work of Kerler \cite{Kerler_PresTanglesCat, Kerler_AlgCobordisms} and Bobtcheva-Piergallini \cite{BobtchevaPiergallini} on the category $\mathcal{C}^{CYK}$, whose proof is postponed to the Appendix \ref{appendix_presentation_BT}.

\begin{theorem}\label{theorem_functorBT} Let $\mathcal{C}$ be a braided category and $H$ a dual BP Hopf algebra object  in $\mathcal{C}$. Then there exists a unique braided functor $Q_H: \BT \to \mathcal{C}$ sending  $\mathbf{H}_1$ to $H$ and preserving their structure morphisms $( \mu, \eta, \Delta, \epsilon, S, \theta)$.
\end{theorem}

\section{Habiro's quantum representation spaces}\label{sec2}

\subsection{Categorical preliminaries}\label{sec_Cat}
We introduce a bit of (standard) categorical terminology.
\par \textit{Free cocompletions}
 \par Let $(\mathcal{E}, \otimes)$ be a cocomplete symmetric monoidal category, which concretely will be either $\Set$, $\Mod_k$ or $\Top$, for $k$ is a commutative unital ring. Note that these three categories are related by monoidal functors $\pi_0 : \Top \to \Set$ and $k[\cdot] : \Set \to \Mod_k$, where the latter functor sends a set $X$ to the $k$-module $k[X]$ freely generated by $X$.
 
 Let $\mathcal{C}$ be a category enriched over $\mathcal{E}$. A \textit{free cocompletion} is a pair $(i, \Free(\mathcal{C}))$ where $\Free(\mathcal{C})$ is a cocomplete category and $i:\mathcal{C} \to \Free(\mathcal{C})$ is a fully faithful functor. The pair is required to be universal among such pairs in the sense that for any fully faithful functor $j: \mathcal{C}\to \mathcal{D}$ to a cocomplete category $\mathcal{D}$, there exists a continuous functor (i.e. a functor that commutes with colimits) $k: \Free(\mathcal{C}) \to \mathcal{D}$ such that $j=k\circ i$.
  Such a free cocompletion  is given by the category $\widehat{\mathcal{C}}$ of functors $F: \mathcal{C}^{op}\to \mathcal{E}$ with natural transformations as morphisms. The fully faithful functor (Yoneda embedding)  $\widehat{\bullet} : \mathcal{C}\to \widehat{\mathcal{C}}$ sends an object $x\in \mathcal{C}$ to $\widehat{x} := \Hom_{\mathcal{C}}(\bullet, x)$. If $(\mathcal{C}, \wedge)$ is symmetric monoidal, then $\widehat{\mathcal{C}}$ receives a symmetric monoidal structure by the \textit{Day convolution product} $\otimes_D$ defined by the coend
 $$ F\otimes_D G (x) := \int^{a,b \in \mathcal{C}} F(a)\otimes G(b) \otimes \Hom_{\mathcal{C}}( x, a\wedge b). $$
 The Day convolution is designed such that $\widehat{\bullet} : \mathcal{C}\to \widehat{\mathcal{C}}$ is symmetric monoidal (see \cite{DayThesis}). 
 
 \vspace{2mm}
 \par \textit{Left and right modules}
 \par Let $\mathcal{C}, \mathcal{D}$ be categories enriched over $\Mod_k$, for $k$ a commutative unital ring. A \textit{left} $\mathcal{C}$-\textit{module} is a functor $F: \mathcal{C}^{op}\to \Mod_k$, a  \textit{right} $\mathcal{C}$-\textit{module} is a functor $F: \mathcal{C}\to \Mod_k$
and a $\mathcal{C}-\mathcal{D}$ bimodule is a functor $F: \mathcal{C}^{op}\times \mathcal{D} \to \Mod_k$.  We denote by $\LMod(\mathcal{C}), \RMod(\mathcal{C}), \Bimod(\mathcal{C}, \mathcal{D})$ the categories of left, right and bi-modules respectively.

 Let $\mathbf{k}$ be the category with a single object whose set of endomorphism is $k$. Note that  a left $\mathcal{C}$-module is the same as a $\mathcal{C}-\mathbf{k}$-bimodule, a right $\mathcal{C}$-module is the same as a $\mathbf{k}-\mathcal{C}$-bimodule and a $\mathbf{k}-\mathbf{k}$-bimodule is just an element of $\Mod_k$.

  The \textit{convolution product} $\bullet \otimes_{\mathcal{C}} \bullet : \Bimod(\mathcal{A}, \mathcal{B})\times \Bimod(\mathcal{B}, \mathcal{C}) \to \Bimod(\mathcal{A}, \mathcal{C})$ sends $F: \mathcal{A}^{op}\times \mathcal{B} \to \Mod_k$ and $G: \mathcal{B}^{op}\times \mathcal{C}\to \Mod_k$ to $ F\otimes_{\mathcal{C}} G$ defined by:
  $$ F\otimes_{\mathcal{C}} G(a,b) = \int^{x \in \mathcal{C}} F(a,x)\otimes G(x,b).$$
  
  \begin{lemma}\label{lemma_tenstens} Let $(\mathcal{C}, \wedge)$ be a symmetric monoidal category enriched over $\Mod_k$, $F_1, F_2: \mathcal{C} \to \Mod_k$ two right $\mathcal{C}$-modules and $G: \mathcal{C}^{op}\to \Mod_k$ a left module which is monoidal. Then 
  $$ (F_1\otimes_D F_2) \otimes_{\mathcal{C}} G \cong (F_1\otimes_{\mathcal{C}} G) \otimes_k (F_2\otimes_{\mathcal{C}} G).$$
  \end{lemma}
  
  \begin{proof}
  Unfolding the definitions, one has:
  \begin{multline*}
  (F_1\otimes_D F_2) \otimes_{\mathcal{C}} G = \int^{x \in \mathcal{C}} (F_1\otimes_D F_2)(x) \otimes_k G(x) = \int^{a,b,x \in \mathcal{C}} F_1(a)\otimes_k F_2(b)\otimes_k \Hom_{\mathcal{C}}(x, a\wedge b) \otimes_k G(x)
 \\ = \int^{a,b \in \mathcal{C}} F_1(a)\otimes_k F_2(b) \otimes_k \left( \int^{x\in \mathcal{C}} \Hom_{\mathcal{C}}(x, a\wedge b) \otimes_k G(x)\right) \cong  \int^{a,b \in \mathcal{C}} F_1(a)\otimes_k F_2(b) \otimes_k G(a\wedge b) 
 \\ =   
  \int^{a,b \in \mathcal{C}} F_1(a)\otimes_k F_2(b) \otimes_k G(a)\otimes_k G(b) \cong \left( \int^{a\in \mathcal{C}} F_1(a) \otimes_k G(a) \right) \otimes_k \left( \int^{b\in \mathcal{C}} F_2(b)\otimes_k G(b)\right) \\ = (F_1\otimes_{\mathcal{C}} G) \otimes_k (F_2\otimes_{\mathcal{C}} G)
  \end{multline*}
  \end{proof}

 \vspace{2mm}
 \par \textit{Left Kan extensions}
 \par Let $\mathcal{A}, \mathcal{B}, \mathcal{C}$ be three categories enriched over $\mathcal{E}$ and $F: \mathcal{A}\to \mathcal{B}$, $G:\mathcal{A}\to \mathcal{C}$ two functors. A \textit{left Kan extension} is a pair $(L, \eta)$ where $L :\mathcal{B} \to \mathcal{C}$ is a functor and $\eta : X \to LF$ a natural transformation visualized as:
 $$
 \begin{tikzcd}
 {} & \mathcal{B}\ar[rd,dotted,  "L"] & {} \\
 \mathcal{A}  \ar[ru, "F"] \ar[rr, "G"{below} ,""{name=D , inner sep=1pt}] &{}& \mathcal{C}
 \arrow[Rightarrow, from=D, to= 1-2, "\eta"]
 \end{tikzcd}
 $$ 
 such that $(L, \eta)$ is universal among these pairs in the sense that if $M:\mathcal{B}\to \mathcal{C}$ is another functor and $\mu : G \to MF$ a natural transformation then there exists a unique natural transformation $\sigma: L\to M$ such that the following diagram commutes:
 $$
 \begin{tikzcd}
 {} & LF \ar[rd, "\sigma_F"] & {} \\
 G \ar[ru, "\eta"] \ar[rr, "\mu"] &{}& MF
 \end{tikzcd}
 $$
 where $\sigma_F: LF \to MF$ is the natural transformation defined by $\sigma_F(x):= \sigma(F (x))$. A left Kan extension is unique up to unique isomorphism so we call $L$ the left Kan extension and denote it by $L=\Lan_F G$. Explicitly,  the left Kan extension can be defined by the formula:
 $$ \Lan_FG (b) = \int^{a \in \mathcal{A}} \Hom_{\mathcal{B}}(F(a), b)\otimes G(a).$$
 
 \begin{lemma}\label{lemma_KAN}[\cite[Proposition $3.3.6$]{KashiwaraSchapira_Book}] Left adjoint functors preserve left Kan extensions. Said differently, if $F: \mathcal{A}\to \mathcal{B}$, $G:\mathcal{A}\to \mathcal{C}$ are two functors, $L=\Lan_FG$ and $H:\mathcal{C} \to \mathcal{D}$ admits a right adjoint, then $H\circ L = \Lan_{H\circ F}H\circ G$.
\end{lemma} 
 
\subsection{Half twists}\label{sec_HT}

As we shall review in the next subsection, it is well known that the quantum enveloping algebras $U_qG$ are (topological) ribbon Hopf algebras, so dually, the quantum groups $\mathcal{O}_qG$ are coribbon Hopf algebras. Less known is the fact, discovered by Kamnitzer-Snyder-Tingley \cite{KamnitzerTingley, SnyderTingley_HalfTwist} by revisiting the original works in  \cite{LevandorskiiSoibelman, KirillovReshetikhin}, that quantum enveloping algebras admit an additional structure of half-ribbon Hopf algebras so that $\mathcal{O}_qG$ are half coribbon Hopf algebras. This additional half-twist structure plays an important role in the study (and definition) of stated skein modules and algebras so we recall here the main definitions.

\begin{definition}\label{def_HT} Let $k$ be a commutative unital ring.
\begin{enumerate}
\item
(\cite{SnyderTingley_HalfTwist})
A \textit{half-ribbon Hopf algebra} is a ribbon Hopf algebra $(H, R,\theta^{-1})$ in $\Mod_k$ with an element $\omega \in H$, named \textit{half-twist}, such that $(1)$ $\omega$ is invertible, $(2)$ $\omega^2= \theta$ and $(3)$ $\Delta(\omega) = (\omega \otimes \omega) R$. When the $R$-matrix, the twist and the half twist live in some pro-finite completion of $H$, we call it a \textit{topological half-ribbon Hopf algebra}.
\item (\cite{Haioun_Sskein_FactAlg})
A \textit{half-coribbon Hopf algebra} is a coribbon Hopf algebra $(A, \mathbf{r}, \Theta^{-1})$ with a linear map $t: A \to k$ such that $(1)$ $t$ admits an inverse $t^{-1}$ for the convolution $\star$ (i.e. $t\star t^{-1}= t^{-1} \star t = \eta \circ \epsilon$), $(2)$ $t\star t = \Theta$ and $(3)$ $t(xy)=\sum t(x_{(1)}) t(y_{(1)}) r(x_{(2)}\otimes y_{(2)})$ for all $x,y \in A$.
\item Let $(A, \mathbf{r}, t)$ be a half-coribbon Hopf algebra and consider the involutive  morphism $C_t : A\to A$, $C_t : x \mapsto \sum t(x_{(1)}) x_{(2)} t^{-1}(x_{(3)})=\sum t^{-1}(x_{(1)})x_{(2)} t(x_{(3)})$. It is a straightforward consequence of the definition (obtained by dualizing the proof of \cite[Proposition $4.2$]{SnyderTingley_HalfTwist}) that $C_t$ is an isomorphism of co-algebras and an anti-morphism of algebras. The \textit{rotation isomorphism} is the automorphism of algebras $\rot : A\to A$ (and anti-morphism of co-algebras) defined by 
$$\rot := C_t \circ S.$$
Define an isomorphism $\Rot: \LComod_A \to \RComod_A$ sending a left $A$-comodule $(V, \Delta_V^L)$ to the right $A$-comodule $(V, \Delta_V^R)$ with the comodule map 
$$ \Delta_V^R : V \xrightarrow{ \Delta_V^L} A \otimes V \xrightarrow{\rot \otimes \id} A \otimes V \xrightarrow{\tau_{A,V}} V\otimes A.$$
Here $\tau_{A,V}$ is the flip ($\tau_{A,V}(x\otimes y) = y\otimes x$). 
\item The half-coribbon structure of a half-coribbon Hopf algebra $A$ can be encoded categorically in $\RComod_A$ by the pair $(\HT, \hT)$ where $\HT : \RComod_A \to \RComod_A$ is the functor sending a comodule $V= (V, \Delta_V^R)$ to $\HT(V)=(V, (C_t \otimes \id)\otimes \Delta_V^R)$ and $\hT: \id \Rightarrow \HT$ is the natural isomorphism defined by the isomorphisms $\hT_V : V \to \HT(V)$, $\hT_V(x)= \sum x_{(1)} \otimes t(x_{2})$. 
 \item The category $\RComod_A$  has a structure of braided category where for $(V, \Delta_V)$ and $(W, \Delta_W)$ two right $A$-comodules, the braiding is given by: 
$$
 c_{V,W} : V\otimes W \xrightarrow{\tau_{V,W}} W\otimes V \xrightarrow{\Delta_W\otimes \Delta_V} W\otimes A \otimes V \otimes A \xrightarrow{\id_W \otimes \tau_{A,V} \otimes \id_A}  W \otimes V \otimes A\otimes A  \xrightarrow{\id_W \otimes \id_V \otimes r} W\otimes V.
$$
Here again $\tau$ is the flip sending $x\otimes y$ to $y\otimes x$.
\item For $(V, \Delta_V)$ and $(W, \Delta_W)$ in $\RComod_A$, their  \textit{braided tensor product}  is the element $V\overline{\otimes} W \in \RComod_A$, whose underlying $k$-module is $V\otimes_k W$ with comodule map:
$$ \Delta_{V\overline{\otimes} W}:  V\otimes W \xrightarrow{\Delta_V \otimes \Delta_W} V \otimes A \otimes W \otimes A \xrightarrow{\id_V \otimes \tau_{A \otimes W} \otimes \id_A} V \otimes W \otimes A \otimes A \xrightarrow{\id_V\otimes \id_W\otimes \mu} V\otimes W \otimes A.$$
 If, moreover, $(V, \mu_V, \eta_V)$ and $(W, \mu_W, \eta_W)$ are algebra objects in $\RComod_A$, then $V\overline{\otimes}W$ is also an algebra object with  unit  $\eta_V \otimes \eta_W$ and, identifying $V$ with $V\otimes 1 \subset V\otimes W$ and $W$ with $1 \otimes W\subset V\otimes W$, the product is characterized by:
$$ \mu_{V\overline{\otimes} W} (x\otimes y) = \left\{ 
\begin{array}{ll}
\mu_V(x\otimes y) & \mbox{, if }x,y\in V; \\
\mu_W(x \otimes y) & \mbox{, if }x,y \in W; \\
x\otimes y & \mbox{, if }x\in V, y\in W; \\
c_{V,W}(x\otimes y) & \mbox{, if }x\in W, y\in V.
\end{array} \right.
$$
\end{enumerate}
\end{definition}

\subsection{Quantum groups and their braided transmutations}\label{sec_QG}

Let $G$ be a connected complex algebraic reducible group and fix a  maximal torus $H$ and a positive Weyl chamber.  Consider the  weight lattice $X_G\subset \mathfrak{h}_{\mathbb{R}}^*$, the subset $X_G^+\subset X_G$ of positive roots, the subset $\Delta$ of simple roots and the invariant pairing $(\cdot, \cdot): \mathfrak{h}_{\mathbb{R}}^* \otimes \mathfrak{h}_{\mathbb{R}}^* \to \mathbb{R}$ normalized such that the shortest root $\alpha\in \Delta$ has norm $(\alpha, \alpha)=2$. Denote by $\rho$ the element such that $(\rho, \alpha_i)=\frac{(\alpha_i, \alpha_i)}{2}$ for all $\alpha_i \in \Delta$.
Let 
$$ l=l_G = \min \{ n\geq 1 | \frac{(\lambda, \lambda)}{2}, (\lambda, \rho) \in \frac{1}{n}\mathbb{Z}, \forall \lambda \in X_G \}.$$
and consider the ring $k=k_G:= \mathbb{Z}[q^{\pm \frac{1}{l}}]$, where $q^{\frac{1}{l}}$ is an invertible formal variable whose $l$-th power will be denoted by $q$. Let $\dot{U}_qG$ be Lusztig's quantum enveloping algebra over $k_G$ as defined in \cite{Lusztig_Book} and denote by $\mathcal{C}_q^G$ the category of finite dimensional integrable $\dot{U}_qG$ modules over the ring $k$.
 The category $\mathcal{C}_q^G$ is semi-simple with simple objects the irreducible representations $V_{\lambda}$ with highest weight $\lambda \in X_G^+$ and, by definition of being integrable of finite rank, a module $V\in \mathcal{C}_q^G$ is a direct sum of finitely such $V_{\lambda}$.
 The Hopf algebra $\dot{U}_qG$ has a natural pro-finite completion $\widetilde{U}_qG=\int_{V\in \mathcal{C}_q^G} V^{*}\otimes V \cong \oplus_{\lambda \in X_G^+} \End(V_{\lambda})$ (see for instance \cite{KamnitzerTingley, Sawin, Negron_Frobenius}) and it is well known that $\dot{U}_qG$ has a structure of topological ribbon Hopf algebra in the sense that the $R$-matrix $R$ and twist $\theta$ are defined on the completions  $\widetilde{U}_qG$ and $\widetilde{(U_qG)^{\otimes 2}}$ (see \cite[Section $1$]{KamnitzerTingley} for details on this terminology). As a result, $\mathcal{C}_q^G$ has a structure of ribbon category. 
Moreover, it is showed in \cite{SnyderTingley_HalfTwist}  that $\dot{U}_qG$ can be enhanced with an additional structure of topological half-ribbon Hopf algebra, provided that we choose the correct twist for its ribbon structure. The half-twist plays a crucial role in the construction of stated skein modules and in Costantino-L\^e's reinterpretation in Section \ref{sec_skein_transmutation} of Majid's transmutation theory, so we briefly recall it following \cite{SnyderTingley_HalfTwist} to which we refer for further details.

Here $\theta$ represents the positive twist and, for historical reasons, most authors defined the ribbon structure by considering $\theta^{-1}$ (denoted $v$ for instance in \cite{Kassel}). 
The completion  $\widetilde{U}_qG$ is isomorphic to the product $\oplus_{\lambda \in X_G^+} \End(V_{\lambda})$, where $V_{\lambda}$ is the simple module of highest weight $\lambda$. In most textbooks, the authors endow $\dot{U}_qG$ with the twist $\theta_0\in \widetilde{U}_qG$ defined on a weight vector $v$ of weight $wt(v)$ by $\theta_0 v = q^{(wt(v), wt(v))+2(wt(v), \rho)}v$. However, as noticed in \cite{SnyderTingley_HalfTwist}, such a twist might not admit any half-twist so it is wise to modify it slightly as follows. Let $\varepsilon \in \widetilde{U}_qG$ be  defined on $V_{\lambda}$ by $\varepsilon_{\lambda}= (-1)^{(2\lambda, \rho^{\vee})} \id_{V_{\lambda}}$ and set $\theta:= \varepsilon \theta_0$. Then $\theta$ is a (topological) twist as well and admits two different half-twists. Let $L$ and $J$ be the elements in $\widetilde{U}_qG$ defined on a weight vector $v$ of weight $wt(v)$ by 
$$ L v = q^{(wt(v), wt(v))/2} v,  \quad J v = q^{(wt(v), wt(v))/2 +(wt(v), \rho)}v.$$
Let $\mathcal{T}_{w_0}\in \widetilde{U}_qG$ be the braid group operator corresponding to the longest element $w_0 \in W_G$ in the Weyl group of $G$. 

\begin{definition}
The half twists $\omega_1, \omega_2 \in \widetilde{U}_qG$ are the operators 
$$ \omega_1 := L \mathcal{T}_{w_0}, \quad \omega_2 := J \mathcal{T}_{w_0}.$$
\end{definition}

It is proved in \cite{KamnitzerTingley,SnyderTingley_HalfTwist} that  $ \omega_2$ is a half twists with square $\theta$ which endows $\dot{U}_qG$ with a structure of topological half-ribbon Hopf algebra. It is an easy exercise (left to the reader) that $\omega_1$ is a half-twist as well with the same square. Note that the half-twist $\omega_1$ appeared in the original pioneered of Levandorskii-Soidelman in \cite{LevandorskiiSoibelman} who proved that $\Delta(\omega_1) = (\omega_1 \otimes \omega_1) R$

\begin{definition}
The quantum group $\mathcal{O}_qG$ is the half coribbon Hopf algebra defined as $k$-module as 
$$ \mathcal{O}_qG:= \int^{V\in \mathcal{C}_q^G} V^*\otimes V$$
so that its elements are classes of matrix coefficients $[v^*\otimes v]$, for $V\in \mathcal{C}_qG$ and $v\in V$, $v^*\in V^*$. The product is $\mu([v_1^* \otimes v_1] \otimes [v_2^* \otimes v_2]) := [(v_1^*\otimes v_2^*) \otimes (v_1\otimes v_2)]$. The unit is $\eta(1):=[1^*\otimes 1]$ where $1\in k$ and $1^*\in k^*$ such that $1^* (1)=1$. The coproduct is $\Delta([v^*\otimes v]) = \sum_{i\in I} [v^* \otimes v_i] \otimes [v_i^* \otimes v]$, where $v\in V, v^* \in V^*$, $\{v_i\}_{i\in I}$ is a basis of $V$ and $\{v_i^*\}$ the dual basis of $V^*$. The counit is $\epsilon ([v^*\otimes v]) := \left< v^*, v\right>$. The antipode is $S([v^* \otimes v]):= [\alpha_V(v) \otimes v^*]$, where $v\in V$, $\alpha_V : V \xrightarrow{\cong} (V^*)^*$ is the isomorphism induced by the image of the rigid structure of $\mathcal{C}_qG$, i.e. by the charmed element of $\dot{U}_qG$. The co-R-matrix is $\mathbf{r}([v_1^*\otimes v_1] \otimes [v_2^* \otimes v_2]):= \left< v_2^* \otimes v_1^* , c_{V_1, V_2} (v_1\otimes v_2)\right>$. The inverse of the co-twist is $\Theta([v^* \otimes v]):= \left< v^* , \theta_V(v) \right>$ and it has two possible  co-half twists $t_1, t_2$ defined by $t_j([v^* \otimes v]) :=  \left< v^* , (\omega_j)_V(v) \right>$.
\end{definition}
For instance, when $G=\SL_2$, then $k_{\SL_2}=\mathbb{Z}[q^{\pm \frac{1}{4}}]$. The twist $\theta_0$ is the one defining the original Jones polynomial whereas the twist $\theta$ that we choose is the one defining the Kauffman-bracket (see \cite{Tingley_MinusSign} for details) so that $\mathcal{C}_q^{\SL_2}$ is equivalent to the Cauchy closure of the Temperley-Lieb category. The matrix coefficients $a,b,c,d$ of the standard $2$-dimensional representation of $\dot{U}_q\SL_2$ generate $\mathcal{O}_q[\SL_2]$ with relations:
$$
ab = q^{-1}ba, \quad  ac=q^{-1}ca,
\quad  db = q bd, \quad dc=q cd,
\quad ad=1+q^{-1}bc, \quad  da=1 + q bc, 
\quad bc=cb.
$$
It has a Hopf algebra structure characterized by the formulas 

 \begin{equation*}
 \begin{pmatrix} \Delta (a) & \Delta (b) \\ \Delta(c) & \Delta(d) \end{pmatrix} 
 = 
 \begin{pmatrix} a & b \\ c & d \end{pmatrix} 
 \otimes 
 \begin{pmatrix} a & b \\ c & d \end{pmatrix} 
\quad
 \begin{pmatrix} \epsilon(a) & \epsilon(b) \\ \epsilon(c) & \epsilon(d) \end{pmatrix} =
\begin{pmatrix} 1 &0 \\ 0& 1 \end{pmatrix}  
\quad
\begin{pmatrix} S(a) & S(b) \\ 	S(c) & S(d) \end{pmatrix} 
	= 
	\begin{pmatrix} d & -q b \\ -q^{-1}c & a \end{pmatrix} .
 \end{equation*}
 and the half-coribon structure is characterized by
$$ r \left(  \begin{pmatrix} a & b \\ c & d \end{pmatrix} 
 \otimes 
 \begin{pmatrix} a & b \\ c & d \end{pmatrix} \right)
= \begin{pmatrix} q^{1/2} & 0 & 0 & 0 \\ 0 & 0 &q^{-1/2} & 0 \\ 0 & q^{-1/2} & q^{1/2}-q^{-3/2} & 0 \\ 0 & 0 & 0 & q^{1/2} \end{pmatrix}=: \mathscr{R}
\quad
\Theta  \begin{pmatrix} a & b \\ c & d \end{pmatrix} = -q^{3/2}\begin{pmatrix} 1 &0 \\ 0& 1 \end{pmatrix}.$$
The two co-half twists are: 
$$ t_1  \begin{pmatrix} a & b \\ c & d \end{pmatrix} = \begin{pmatrix} 0 &-q^{5/4} \\  q^{1/4}& 0 \end{pmatrix}, \quad t_2  \begin{pmatrix} a & b \\ c & d \end{pmatrix} = \begin{pmatrix} 0 & q^{3/4} \\ -q^{3/4}& 0 \end{pmatrix}.
$$
The rotation isomorphism associated to $t_1$ is $\rot \begin{pmatrix} a & b \\ c & d \end{pmatrix} = \begin{pmatrix} a & c \\ b & d \end{pmatrix}$.
Note that the relations defining the algebra structure of $\mathcal{O}_q\SL_2$ can be rewritten, using the notation $M:= \begin{pmatrix} a & b \\ c & d \end{pmatrix}$ in the following compact form:
$$ \mathscr{R} (M\odot M) = (M \odot M) \mathscr{R}, \quad ad-q^{-1}bc=1$$ 
, where $\odot$ is the Kronecker product.
\par Denote by $\overline{\mathcal{C}_q^G}$ the category of (possibly infinite dimensional)  locally finite $\dot{U}_qG$ modules and let $i: \mathcal{C}_q^G \to \overline{\mathcal{C}_q^G}$ be the inclusion functor. Equivalently, $\overline{\mathcal{C}_q^G}$ is isomorphic to the category of $\mathcal{O}_qG$-comodules. Let $\underline{\dot{U}_qG}$ be the category with one element, say $pt$, whose endomorphisms ring is $\dot{U}_qG$ so that the category of left modules over $\underline{\dot{U}_qG}$ is canonically isomorphic to  $\overline{\mathcal{C}_q^G}$. Let $E$ be the $\underline{\dot{U}_qG}-\mathcal{C}_q^G$ bimodule sending $(pt, V)$ to $V$ and $E'$ the $\mathcal{C}_q^G-\underline{\dot{U}_qG}$ bimodule sending $(pt, V)$ to $V^*$. We also consider the field of fraction $K_G:= \mathbb{Q}(q^{1/n_G})$ and denote by $\mathcal{C}_q^{G, rat}$ (resp. $\overline{\mathcal{C}_q^{G,rat}}$) the category of finite dimensional (resp. locally finite) $\dot{U}_qG\otimes_{k_G} K_G$ modules.

\begin{lemma}\label{lemma_cocompletion}
\begin{enumerate}
\item $\overline{\mathcal{C}_q^{G, rat}}$ is semi-simple with simple objects the $V_{\lambda}$, $\lambda\in X_G^+$.
\item The functors $$E\otimes_{\underline{\dot{U}_qG}} \bullet: \LMod(\underline{\dot{U}_qG})=\overline{\mathcal{C}_q^G} \to \LMod(\mathcal{C}_q^G)=\widehat{\mathcal{C}_q^G}$$ and $$E'\otimes_{\mathcal{C}_q^G}\bullet : \LMod(\mathcal{C}_q^G)=\widehat{\mathcal{C}_q^G} \to  \LMod(\underline{\dot{U}_qG})=\overline{\mathcal{C}_q^G}$$ are mutual inverse equivalence of categories.
\end{enumerate}
Therefore $\mathcal{C}_q^G$ and $\underline{\dot{U}_qG}$ are Morita equivalent and the pair $(\overline{\mathcal{C}_q^G}, i)$ is a free cocompletion of $\mathcal{C}_q^G$. 
\end{lemma}

\begin{proof}
$(1)$ That $\mathcal{C}_q^{G, rat}$ is semi-simple with simple objects the $V_{\lambda}$ is a classical result (see e.g. \cite{Lusztig_Book}). By \cite[Theorem $2.1.3$ (b)]{Sweedler_HopfAlgebras}, every cyclic $\mathcal{O}_qG$-comodule is finite dimensional so $\overline{\mathcal{C}_q^{G, rat}}$ is semi-simple as well with the same simple objects.
\par $(2)$ We need to prove that $E\otimes_{\underline{\dot{U}_qG}} E' \cong \mathcal{C}_q^G$ as a $\mathcal{C}_q^G$-bimodule and that $E'\otimes_{\mathcal{C}_q^G} E\cong \mathcal{O}_qG$ as a $\dot{U}_qG$-bimodule. By definition, for $V,W \in \mathcal{C}_q^G$, then
 $$E\otimes_{\underline{\dot{U}_qG}} E'(V,W)=  V\otimes_{\dot{U}_qG} W^*\cong \Hom_{\mathcal{C}_q^G}(W,V)$$
 and $$ E'\otimes_{\mathcal{C}_q^G} E (pt, pt) = \int^{V\in \mathcal{C}_q^G} V^*\otimes V= \mathcal{O}_qG.$$
 This proves the second assertion.
\end{proof}

In \cite{Majid_BQG_Rank}, Majid introduced another quantum group $B_qG$ named the \textit{braided quantum group} obtained from $\mathcal{O}_qG$ by the so-called transmutation procedure. The following definition of $B_qG$ only depends on the coribbon structure of $\dot{U}_qG$ and not on the half-twists.

\begin{definition}\label{def_transmutation} Let $(H,\mu, \eta, \Delta, \epsilon, S)$ be a coribbon Hopf algebra.
\begin{enumerate}
\item \cite[Example $9.4.10$]{Majid_QGroups} The \textit{transmutation} of $H$ is the Hopf algebra object $BH$ in the braided category $H-\RComod$ whose underlying right $H$-comodule is $(H, \Ad)$ with the (right) adjoint coaction $\Ad : BH \to BH \otimes H$
$$\Ad (x) := \sum x_{(2)} \otimes S(x_{(1)})x_{(3)},$$
 with coproduct $\Delta$, unit $\eta$, counit $\epsilon$, cotwist $\theta$ and with the modified product and antipode given by:
\begin{equation}\label{eq_transmuted_prod}
\underline{\mu} (x\otimes y) := \sum x_{(2)}y_{(2)} r( S(x_{(1)}) x_{(3)} \otimes S(y_{(1)})) 
\end{equation}
\begin{equation}\label{eq_transmuted_antipod}
 \underline{S} (x) = \sum S(x_{(2)})r((S^{2}(x_{(3)}) S(x_{(1)}) \otimes x_{(4)}))
 \end{equation}
The \textit{braided quantum adjoint coaction} $\Ad^B : BH \to  BH \overline{\otimes} BH$ is defined by
$$ \Ad^B :=  \sum (\id \otimes \mu) (c_{H,H} \otimes \id) (S(x_{(1)}) \otimes x_{(2)} \otimes x_{(3)}). $$
This equips $BH$ with a structure of right $BH$-comodule object in the category $H-\RComod$ (see \cite{Majid_BraidedHopfAlg} for details on this notion).
The non braided and braided quantum adjoint coaction are depicted as:
$$\Ad= \adjustbox{valign=c}{\includegraphics[width=1.5cm]{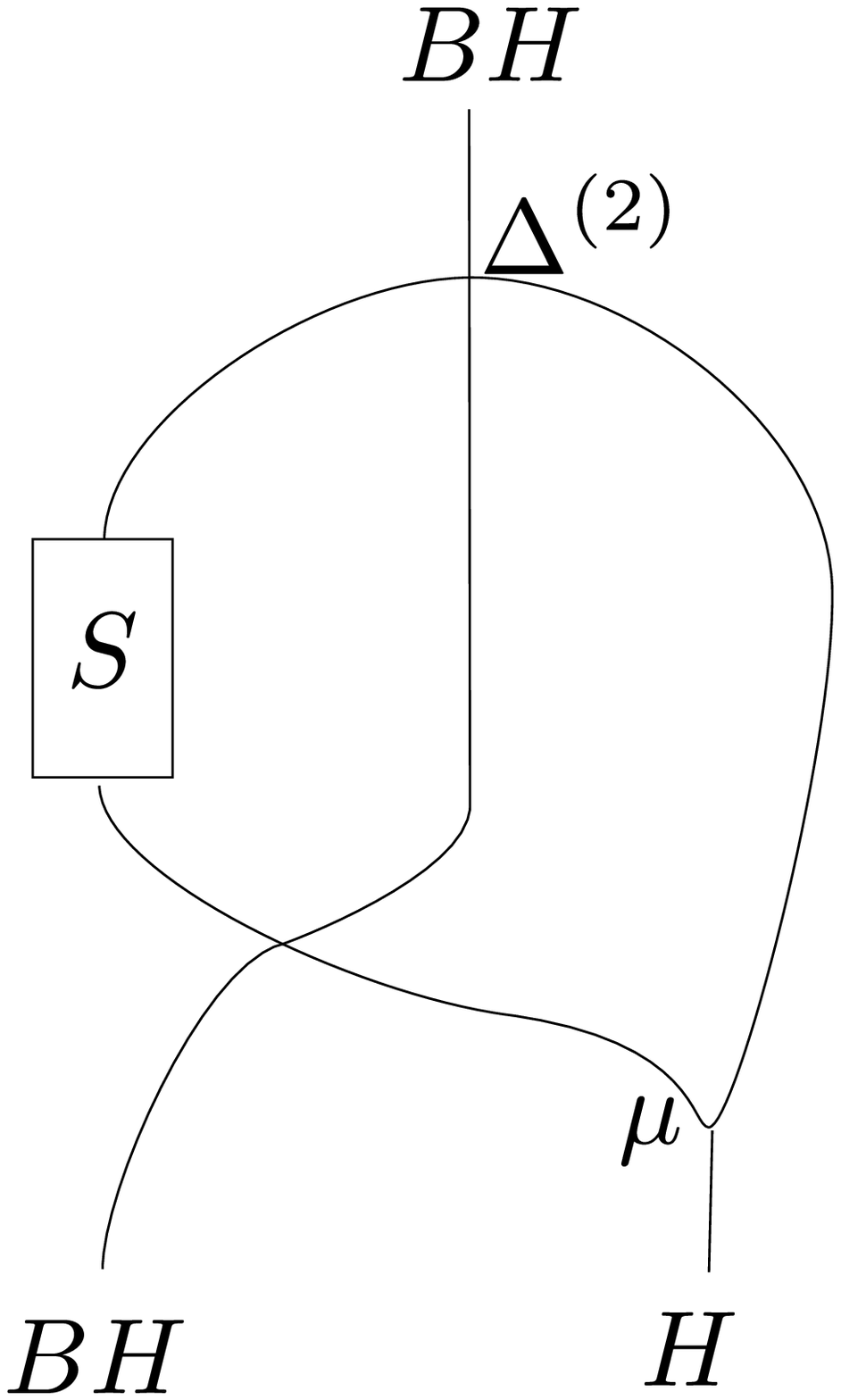}}, \quad \Ad^B =  \adjustbox{valign=c}{\includegraphics[width=1.5cm]{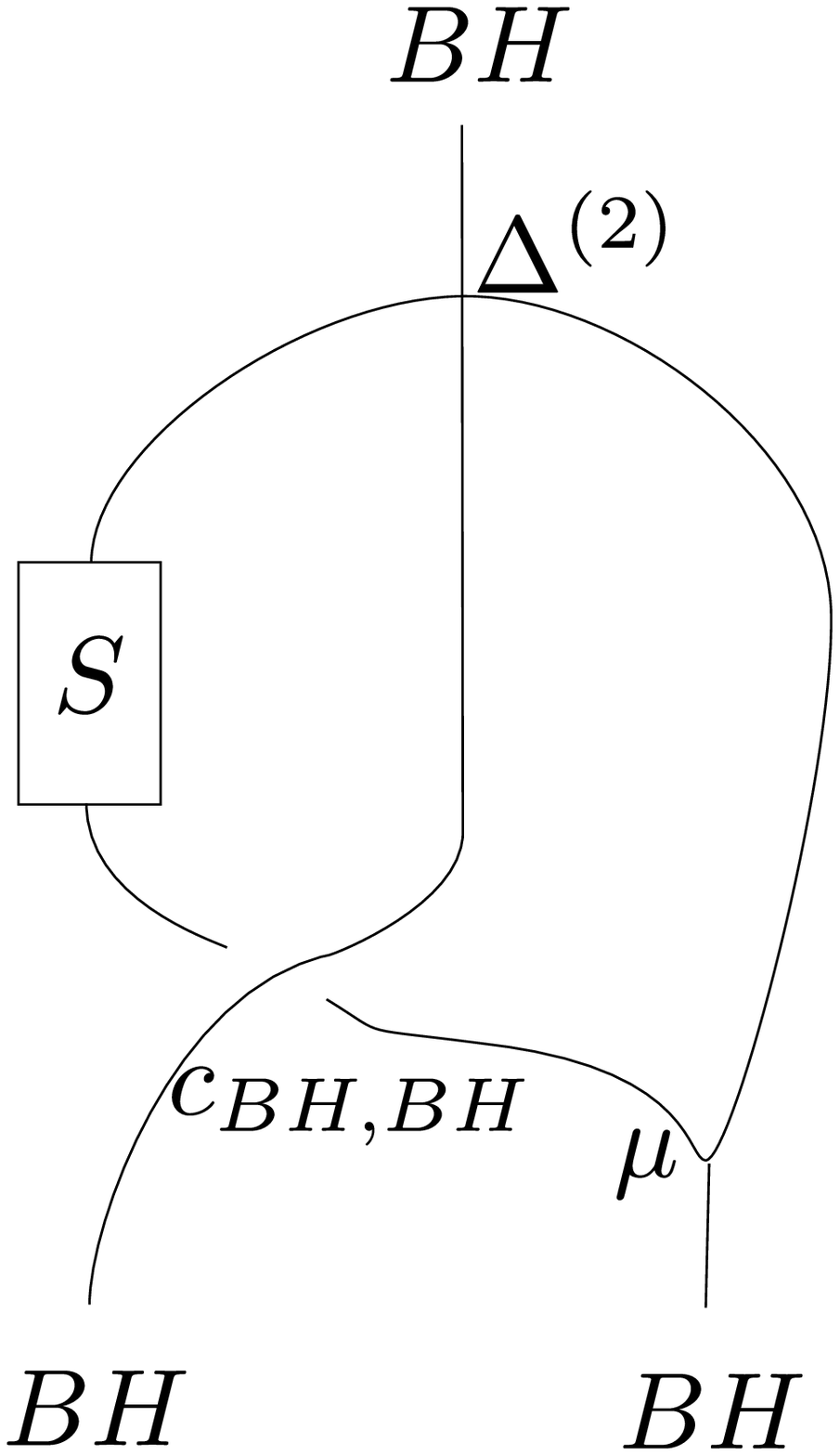}}.$$ 
\item The transmutation of $\mathcal{O}_qG$ is called the \textit{braided quantum group} $B_qG$.
\end{enumerate}
\end{definition}

The main property of the transmutation $BH$ is its \textit{braided commutativity}, that is the equality:
$$ (\id \otimes \mu) (c_{H,H} \otimes \id)(\id \otimes \Ad)c_{H,H}= (\id \otimes \mu)(\Ad\otimes \id).$$
So the transmutation procedure transforms the non-commutative cobraided Hopf algebra $\mathcal{O}_qG$ into a (braided) commutative Hopf algebra object $B_qG$ in the braided category ${\mathcal{O}_qG}-\RComod$.

\vspace{2mm}
\par For instance, when $G=\SL_2$, as developed in  \cite[Example $4.3.4$]{Majid_QGroups}, the algebra structure of $B_q\SL_2$ is given by the same generators $a,b,c,d$ than $\mathcal{O}_q[\SL_2]$ but with relations: 
\begin{align*}
& ba=q^{2} ab, \quad ca=q^{-2}ac, \quad da=ad, \quad bc=cb+(1-q^{-2})a(d-a) \\
&db=bd+(1-q^{-2})ab, \quad cd=dc+(1-q^{-2})ca, \quad, ad-q^{2}cb=1.
\end{align*} 
Using the same notations as before, and writing $\mathds{1}_2= \begin{pmatrix} 1 & 0 \\ 0 & 1 \end{pmatrix}$, the relations can be re-written in the compact form
\begin{equation}\label{eq_BSL2}
\mathscr{R}(\mathds{1}_2\odot M) \mathscr{R}(\mathds{1}_2 \odot M)= (\mathds{1}_2\odot M) \mathscr{R}(\mathds{1}_2\odot M) \mathscr{R}, \quad  ad-q^2cb=1.
\end{equation}

That the transmutation $B_qG$ is indeed a braided Hopf algebra object in ${\mathcal{O}_qG}-\RComod$ and that the cotwist $\theta$ satisfies the axioms of Figure \ref{fig_BPHopf} was proved by Kerler in \cite{Kerler_PresTanglesCat} and  \cite[Section $5.2$]{Kerler_AlgCobordisms} with the exception of the two last relations at the bottom of Figure  \ref{fig_BPHopf}  which were discovered latter by Bobtcheva-Piergallini in \cite{BobtchevaPiergallini} (where they are denoted by $(r8)$, $(r9)$). That $B_qG$ indeed satisfies these two relations is proved using the same computations than in the proof of \cite[Proposition $7.3$]{BeliakovaDeRenzi_KLFunctor}; we leave the details of this computation to the reader.

\begin{definition}
We denote by $Q_{B_qG}: \BT \to {\mathcal{O}_qG}-\RComod=\overline{\mathcal{C}_q^G}$ the braided functor obtained from Theorem \ref{theorem_functorBT}.
\end{definition}

\subsection{Quantum representation spaces and quantum fundamental groups}\label{sec_QRep}

\begin{definition}(Habiro \cite{Habiro_QCharVar}) 
\begin{enumerate}
\item
 The \textit{quantum representation functor} is the functor $\Rep_q^G : \mathcal{M}^{(1)}_{\con} \to \overline{\mathcal{C}_q^G}$ defined as the left Kan extension $\Rep_q^G := Lan_{i} Q_{B_qG}$ lying the diagram
 $$
 \begin{tikzcd}
 {} & \mathcal{M}^{(1)}_{\con}\ar[rd,  "\Rep_q^G"] & {} \\
 \BT  \ar[ru, "i"] \ar[rr, "Q_{B_qG}"] &{}& \overline{\mathcal{C}_q^G}
 \end{tikzcd}
 $$ 
where $i : \BT \to \mathcal{M}^{(1)}_{\con}$ is the inclusion functor. For $\mathbf{M}\in \mathcal{M}^{(1)}_{\con}$, the comodule $\Rep_q^G(\mathbf{M})$ is called its \textit{quantum representation space}. The subspace $\Char_q^G(\mathbf{M}) \subset \Rep_q^G(\mathbf{M})$ of coinvariant vectors is called the \textit{quantum character variety}.
\item Let $\widehat{\BT}$ be the category of functors $\BT^{op} \to \Set$. For $\mathbf{M}\in \mathcal{M}^{(1)}_{\con}$, the \textit{quantum fundamental group} is the functor $P_M := \Hom_{\mathcal{M}_{\con}^{(1)}}( i(\bullet), \mathbf{M}) \in \widehat{BT}$. In particular, by Lemma \ref{lemma_BT},  $P_M(\mathbf{H}_n)$ is identified with the set $P_n(M)$ of isomorphism classes of $n$-bottom tangles in $\mathbf{M}$. We denote by $P_{\bullet} : \mathcal{M}^{(1)}_c \to \widehat{\BT}$ the functor sending $\mathbf{M}$ to $P_M$,
\item We denote by $\widehat{\BT}_k$ the category of functors $\BT^{op} \to \Mod_k$ and $k[P_M] \in \widehat{\BT}_k$ the composition of $P_M$ with the monoidal functor $\Set \to \Mod_k$ sending a set $S$ to the $k$-module $k[S]$ freely generated by $S$. 
\end{enumerate}
\end{definition}

Note that, in the terminology of Section \ref{sec_Cat},  the functor $Q_{B_qG}$ is a right $\BT$-module whereas $k[P_M]$ is a left $\BT$-module.

\begin{lemma}\label{lemma_QFG_QRS} For $\mathbf{M}\in \mathcal{M}^{(1)}_{\con}$, one has an isomorphism of $k$-modules:
$$ \Rep_q^G(\mathbf{M}) \cong k[P_M] \otimes_{\BT} Q_{B_qG}.$$
\end{lemma}

\begin{proof} This is an immediate consequence of the explicit expression of the left Kan extension given in Section \ref{sec_Cat}.
\end{proof}

The functor $P_{\bullet} : \mathcal{M}^{(1)}_{\con}\to \widehat{\BT}$ can be thought as a "quantum" analogue of the fundamental group functor $\pi_1: \Top^{\bullet} \to \Gp$. The philosophy promoted by Habiro in \cite{Habiro_QCharVar} is that many results about $\pi_1$ should extend to $P_{\bullet}$ thus permitting to extend results about the classical representation spaces and character varieties to the quantum ones using Lemma \ref{lemma_QFG_QRS}. Our first example of success of this philosophy is the:

\begin{lemma}\label{lemma_monoidal} The quantum fundamental group functor $P_{\bullet} : (\mathcal{M}^{(1)}_{\con}, \wedge) \to (\widehat{\BT}, \otimes_D)$ is lax monoidal, i.e. for $\mathbf{M}_1, \mathbf{M}_2 \in \mathcal{M}^{(1)}_{\con}$, then $P_{M_1\wedge M_2}$ is isomorphic to the Day convolution product $P_{M_1} \otimes_D P_{M_2}$. Therefore $\Rep_q^G(\mathbf{M}_1\wedge \mathbf{M}_2) \cong \Rep_q^G(\mathbf{M}_1) \overline{\otimes} \Rep_q^G(\mathbf{M}_2)$.
\end{lemma}

\begin{proof} Let $\mathbf{M}_1, \mathbf{M}_2 \in \mathcal{M}^{(1)}_{\con}$ and let us define bijections
 $$ f^{(n)} : P_{M_1}\otimes_D P_{M_2} (\mathbf{H}_n) \xrightarrow{\cong} P_{M_1\wedge M_2}(\mathbf{H}_n), \quad n\geq 0 $$
 which induce an isomorphism $f : P_{M_1}\otimes_D P_{M_2} \xrightarrow{\cong} P_{M_1\wedge M_2}$. When $n=0$, both $P_{M_1}\otimes_D P_{M_2} (\mathbf{H}_0)$ and $P_{M_1\wedge M_2}(\mathbf{H}_0)$ have only one point so the definition of $f^{(0)}$ is obvious and we assume that $n\geq 1$. Using the isomorphism $\theta_*$ of Lemma \ref{lemma_BT}, we identify the sets $P_M(\mathbf{H}_n)$ with the sets $P_n(M)$ of $n$-bottom tangles in $\mathbf{M}$ and to a bottom tangle $T\in P_n(M)$ we denote by $\varphi_T : H_n \to M$ the associated embedding in $ P_M(\mathbf{H}_n)$  ($\varphi_T$ is only well defined up to isotopy). For $a,b \geq 0$, let 
 $$ V_{a,b}^{(n)}:= P_a(M_1) \times P_b(M_2) \times \BT(n, a+b).$$
 Recall that $M_1\wedge M_2=M_1\cup \mathbb{T} \cup M_2$ is obtained from $M_1\bigsqcup M_2\bigsqcup \mathbb{T}$ by gluing the boundary disc $\mathbb{D}_{M_1}$ to the disc $e_1$ of $\mathbb{T}$  and gluing $\mathbb{D}_{M_2}$ to $e_2$. Let $\mathbb{D}_1, \mathbb{D}_2 \subset M_1\wedge M_2$ be the images of  $\mathbb{D}_{M_1}$ and $\mathbb{D}_{M_2}$ respectively through the quotient map.
 Similarly, $\mathbf{H}_{a}\wedge \mathbf{H}_b = \mathbf{H}_a \cup \mathbb{T} \cup \mathbf{H}_b$.

 For $(T_1, T_2, T) \in V_{a,b}^{(n)}$, consider the embedding $\varphi_{T_1, T_2} : \mathbf{H}_a \wedge \mathbf{H}_b \to M_1 \wedge M_2$, whose restriction to $\mathbb{T}$ is the identity and whose restrictions to $\mathbf{H}_a, \mathbf{H}_b$ are $\varphi_{T_1}$ and $\varphi_{T_2}$ and define a function $f_{a,b}^{(n)} : V_{a,b}^{(n)} \to P_n(M_1\wedge M_2)$ by $f_{a,b}^{(n)}(T_1, T_2, T):= \varphi_{T_1, T_2} (T)$.

 By definition of the Day convolution as a coend, we have 
 $$ P_{M_1}\otimes_D P_{M_2} (\mathbf{H}_n) = \quotient{\left( \oplus_{a,b\geq 0} V_{a,b}^{(n)} \right)}{\sim}, $$
 where for every $\alpha\in \BT(a',a)$, $\beta\in \BT(b',b)$ and $(T_1,T_2,T)\in V_{a',b'}^{(n)}$ we set 
 $$ (T_1 \circ \alpha, T_2  \circ \beta, T) \sim (T_1, T_2,  \varphi_{\alpha, \beta}(T)). $$
Since $f_{a,b}^{(n)}  (T_1 \circ \alpha, T_2  \circ \beta, T) = f_{a',b'}^{(n)}(T_1, T_2,  \varphi_{\alpha, \beta}(T))$, the maps $f_{a,b}^{(n)}$ induce a map $f^{(n)}: P_{M_1}\otimes_D P_{M_2} (\mathbf{H}_n) \to P_{M_1\wedge M_2}(\mathbf{H}_n)$ illustrated in Figure \ref{fig_wedgeproduct}. It is a straightforward consequence of the definitions that the family $(f^{(n)})_{n\geq 0}$ form a natural transformation $f: P_{M_1}\otimes_D P_{M_2} \to P_{M_1\wedge M_2}$. 

\begin{figure}[!h] 
\centerline{\includegraphics[width=11cm]{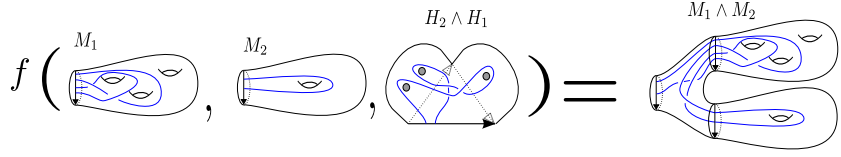} }
\caption{An illustration of the isomorphism $f$.} 
\label{fig_wedgeproduct} 
\end{figure}

To prove that $f$ is an isomorphism, let us define an inverse map $g^{(n)}$ to $f^{(n)}$. 
 A $n$-bottom tangle $T\subset M_1\wedge M_2$ is said in \textit{standard position} if $(1)$ it intersects $\mathbb{D}_1 \cup \mathbb{D}_2$ transversally such that at each intersection point in $T \cap \mathbb{D}_i$ the framing points towards the height direction and $(2)$ the points of $T\cap \mathbb{D}_i$ have pairwise distinct heights and for two connected components $T^{(j)}, T^{(k)}$ of $T$, either all points of $T^{(j)}\cap \mathbb{D}_i$ have smaller heights than the points of  $T^{(k)}\cap \mathbb{D}_i$ or they all have higher heights. 
 These conditions insure that when we cut $T$ into $T=T_{M_1} \cup T_{M_2} \cup T_{\mathbb{T}}$, with $T_{M_i} := T\cap M_i$ and $T_{\mathbb{T}}:= T\cap \mathbb{T}$, then  $T_{M_i}$ is a bottom tangle in $M_i$. Let $2n_i$ be the cardinal of $T\cap \mathbb{D}_i$ (so that $T_{M_i}$ is a $n_i$ bottom tangle) and let $\widehat{T_{\mathbb{T}}}\in \BT(n, n_1+n_2)$ be the $n$ bottom tangle in $\mathbf{H}_{n_1}\wedge \mathbf{H}_{n_2}$ obtained  from $T_{\mathbb{T}}\subset \mathbb{T}$ by gluing the handlebody $H_{n_1}$ with the trivial bottom tangle $T_{n_1}$ of Figure \ref{fig_BTexemple} to the boundary disc $e_1$ of $\mathbb{T}$ and gluing $H_{n_2}$ with $T_{n_2}$ to $e_2$.  Set 
 $$g^{(n)} (T) := (T_{M_1}, T_{M_2}, \widehat{T_{\mathbb{T}}}) \in V_{n_1, n_2}^{(n)}.$$
 To prove that $g^{(n)}$ induces a map $g^{(n)}:  P_{M_1\wedge M_2}(\mathbf{H}_n) \to P_{M_1}\otimes_D P_{M_2} (\mathbf{H}_n)$, 
  we need to show that if $T\cong T'$ are two isotopic bottom tangles in good positions then $g^{(n)}(T) \sim g^{(n)}(T')$.
  If $T$ and $T'$ are isotopic, then we can pass from $T$ to $T'$ by a finite sequence of these two elementary moves:
 \begin{enumerate}
 \item Perform an isotopy inside $M_1, M_2$ or $\mathbb{T}$ whose restriction to $\mathbb{D}_i$ is the identity.
 \item  Pass a tangle of $\mathfrak{bt}$ through $\mathbb{D}_1$ or $\mathbb{D}_2$ as illustrated in Figure \ref{fig_MoveII}. 
 \end{enumerate}
 
 \begin{figure}[!h] 
\centerline{\includegraphics[width=11cm]{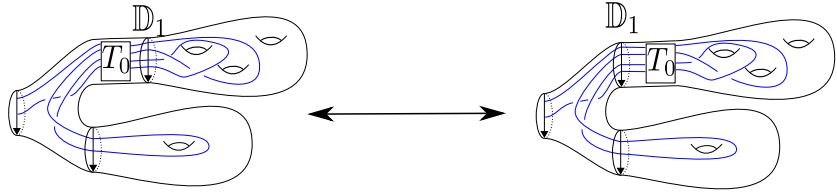} }
\caption{Passing a tangle of $\mathfrak{bt}(4,4)$ through $\mathbb{D}_1$.} 
\label{fig_MoveII} 
\end{figure}

If $T, T'$ differ by an n isotopy inside $M_1, M_2$ or $\mathbb{T}$, then $g^{(n)}(T)=g^{(n)}(T')$. Now suppose that there exists a tangle $T_0 \in \mathfrak{bt}(n_1, n_1')$ such that $T_{M_2}=T'_{M_2}$, $T_{M_1}=T'_{M_1}\cdot T_0$ and $T'_{\mathbb{T}}= T_0 \cdot T'_{\mathbb{T}}$ as in Figure \ref{fig_MoveII}. Let $\alpha_{T_0} \in \BT(n_1, n_1')$ be the bottom tangle $\alpha_{T_0}:= T_{n_1}\cdot T_0$. Then $g^{(n)}(T)= (T_{M_1}' \circ \alpha_{T_0}, T_{M_2}, T_{\mathbb{T}})$ and $g^{(n)}(T')=(T_{M_1}', T_{M_2}, \varphi_{\alpha_{T_0}, \id} (T_{\mathbb{T}}))$ so $g^{(n)}(T)\sim g^{(n)}(T')$ and the map  $g^{(n)}$ is well defined. That $g^{(n)}$ is the inverse of $f^{(n)}$ is an easy consequence of the definition, therefore $f : P_{M_1}\otimes_D P_{M_2} \xrightarrow{\cong} P_{M_1\wedge M_2}$ is an isomorphism. Tensoring on the right by $\bullet \otimes_{BT} Q_{B_qG}$ and using Lemmas \ref{lemma_QFG_QRS} and \ref{lemma_tenstens}, we get an isomorphism $\Rep_q^G(\mathbf{M}_1\wedge \mathbf{M}_2) \cong \Rep_q^G(\mathbf{M}_1) \overline{\otimes} \Rep_q^G(\mathbf{M}_2)$.

\end{proof}

\section{Stated skein modules and algebras}\label{sec3}

\subsection{The functor $\mathcal{S}_q$ and its properties}

\subsubsection{Stated skein modules}

Let $V\in \mathcal{C}_q^{\SL_2}$ be the standard $2$-dimensional representation $\rho : U_q \mathfrak{sl}_2 \rightarrow \mathrm{End}(V)$ of $U_q\mathfrak{sl}_2$, where $V$ has  basis $(v_+, v_-)$ and 
$$ \rho(E) = \begin{pmatrix} 0 & 1 \\ 0 & 0 \end{pmatrix}, \quad \rho(F) = \begin{pmatrix} 0 & 0 \\ 1 & 0 \end{pmatrix}, \quad \rho(K) = \begin{pmatrix} q & 0 \\ 0 & q^{-1} \end{pmatrix}.
$$
Here we work over the ring $k=k_{\SL_2}= \mathbb{Z}[q^{\pm 1/4}]$. We will use the notation $A:= q^{1/2}$ and $A^{1/2}:=q^{1/4}$. Recall from Definition \ref{def_HT} the isomorphism $\hT : V \xrightarrow{\cong} V$ defined by $\hT \begin{pmatrix} v_+ & v_- \end{pmatrix} := \begin{pmatrix} v_+ & v_- \end{pmatrix} \begin{pmatrix} 0 & -A^{5/2} \\ A^{1/2} & 0 \end{pmatrix}=\begin{pmatrix} A^{1/2} v_- & -A^{5/2} v_+\end{pmatrix}$. Its dual is characterized by
$\hT_* \begin{pmatrix} v_+ \\ v_- \end{pmatrix} := \begin{pmatrix} 0 & -A^{5/2} \\ A^{1/2} & 0 \end{pmatrix}  \begin{pmatrix} v_+ \\ v_- \end{pmatrix} = \begin{pmatrix} -A^{5/2} v_- \\ A^{1/2} v_+\end{pmatrix}$. 

For $\mathbf{M}=(M, \iota_M) \in \mathcal{M}$, a \textit{stated tangle} is a pair $(T,s)$ where $T\subset M$ is a tangle (in the sense of definition \ref{def_tangles}) and $s: \partial T \to V$ a map. When $\mathbf{M}=\mathbf{\Sigma}\times I$ is a thickened surface, we will represent a stated tangle by drawing its 2-dimensional projection diagram and draw an arrow on each boundary arc $a$ of $\mathbf{\Sigma}$ to represent the height order of $\partial T \cap a$ as in Figure \ref{fig_statedtangle}. When a point in $\partial T$ has a state $v_+$ or $v_-$, we will only draw a $+$ or $-$ in front of it for simplicity.
Recall from Definition \ref{def_tangles} that isotopies are required to preserve these height orders.
\begin{figure}[!h] 
\centerline{\includegraphics[width=6cm]{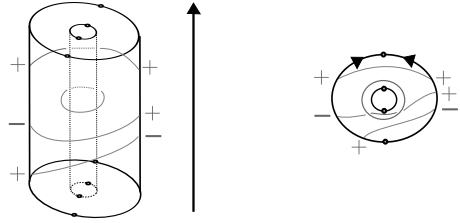} }
\caption{On the left: a stated tangle. On the right: its associated diagram. The arrows represent the height orders. } 
\label{fig_statedtangle} 
\end{figure} 

\begin{definition}(\cite{BonahonWongqTrace, LeStatedSkein, BloomquistLe})\label{def_skein}
\begin{enumerate}
\item The (Kauffman-bracket) \textit{stated skein module}  $\mathcal{S}_q(\mathbf{M})$ is the quotient of the free $k$-module generated by stated tangles by the submodule spanned by isotopy and by the following skein relations:

  	\begin{equation*} 
\begin{tikzpicture}[baseline=-0.4ex,scale=0.5,>=stealth]	
\draw [fill=gray!45,gray!45] (-.6,-.6)  rectangle (.6,.6)   ;
\draw[line width=1.2,-] (-0.4,-0.52) -- (.4,.53);
\draw[line width=1.2,-] (0.4,-0.52) -- (0.1,-0.12);
\draw[line width=1.2,-] (-0.1,0.12) -- (-.4,.53);
\end{tikzpicture}
=A
\begin{tikzpicture}[baseline=-0.4ex,scale=0.5,>=stealth] 
\draw [fill=gray!45,gray!45] (-.6,-.6)  rectangle (.6,.6)   ;
\draw[line width=1.2] (-0.4,-0.52) ..controls +(.3,.5).. (-.4,.53);
\draw[line width=1.2] (0.4,-0.52) ..controls +(-.3,.5).. (.4,.53);
\end{tikzpicture}
+A^{-1}
\begin{tikzpicture}[baseline=-0.4ex,scale=0.5,rotate=90]	
\draw [fill=gray!45,gray!45] (-.6,-.6)  rectangle (.6,.6)   ;
\draw[line width=1.2] (-0.4,-0.52) ..controls +(.3,.5).. (-.4,.53);
\draw[line width=1.2] (0.4,-0.52) ..controls +(-.3,.5).. (.4,.53);
\end{tikzpicture}
\hspace{.5cm}
\text{ and }\hspace{.5cm}
\begin{tikzpicture}[baseline=-0.4ex,scale=0.5,rotate=90] 
\draw [fill=gray!45,gray!45] (-.6,-.6)  rectangle (.6,.6)   ;
\draw[line width=1.2,black] (0,0)  circle (.4)   ;
\end{tikzpicture}
= -(q+q^{-1}) 
\begin{tikzpicture}[baseline=-0.4ex,scale=0.5,rotate=90] 
\draw [fill=gray!45,gray!45] (-.6,-.6)  rectangle (.6,.6)   ;
\end{tikzpicture}
;
\end{equation*}

\begin{equation*} 
\begin{tikzpicture}[baseline=-0.4ex,scale=0.5,>=stealth]
\draw [fill=gray!45,gray!45] (-.7,-.75)  rectangle (.4,.75)   ;
\draw[->] (0.4,-0.75) to (.4,.75);
\draw[line width=1.2] (0.4,-0.3) to (0,-.3);
\draw[line width=1.2] (0.4,0.3) to (0,.3);
\draw[line width=1.1] (0,0) ++(90:.3) arc (90:270:.3);
\draw (0.65,0.3) node {\scriptsize{$i$}}; 
\draw (1.2,-0.3) node {\scriptsize{$ht(j)$}}; 
\end{tikzpicture}
=
\begin{tikzpicture}[baseline=-0.4ex,scale=0.5,>=stealth]
\draw [fill=gray!45,gray!45] (-.7,-.75)  rectangle (.4,.75)   ;
\draw[->] (0.4,-0.75) to (.4,.75);
\draw[line width=1.2] (0.4,-0.3) to (0,-.3);
\draw[line width=1.2] (0.4,0.3) to (0,.3);
\draw[line width=1.1] (0,0) ++(90:.3) arc (90:270:.3);
\draw (1.2,0.3) node {\scriptsize{$ht_*(i)$}}; 
\draw (0.65,-0.3) node {\scriptsize{$j$}}; 
\end{tikzpicture}
=\delta_{i,j}
 , \quad \forall i,j \in \{-, +\};
\end{equation*}

\begin{equation*}
\heightcurve 
=
\sum_{i=\pm} 
\begin{tikzpicture}[baseline=-0.4ex,scale=0.5, >=stealth]
	\draw [fill=gray!60,gray!45] (-.7,-.75)  rectangle (.4,.75)   ;
	\draw[->] (0.4,-0.75) to (.4,.75);
	\draw[line width=1.2] (0.4,-0.3) to (-.7,-.3);
	\draw[line width=1.2] (0.4,0.3) to (-.7,.3);
	\draw (0.65,0.3) node {\scriptsize{$i$}}; 
	\draw (1.2,-0.3) node {\scriptsize{$ht_*(i)$}}; 
	\end{tikzpicture}
 = \sum_{i=\pm} 
 \begin{tikzpicture}[baseline=-0.4ex,scale=0.5, >=stealth]
	\draw [fill=gray!60,gray!45] (-.7,-.75)  rectangle (.4,.75)   ;
	\draw[->] (0.4,-0.75) to (.4,.75);
	\draw[line width=1.2] (0.4,-0.3) to (-.7,-.3);
	\draw[line width=1.2] (0.4,0.3) to (-.7,.3);
	\draw (1.1,0.3) node {\scriptsize{$ht(i)$}}; 
	\draw (0.65,-0.3) node {\scriptsize{$i$}}; 
	\end{tikzpicture}
.
\end{equation*}

\begin{equation*} 
\tresalacon{->}{\alpha_1v_1 + \alpha_2 v_2} = \alpha_1 \traitalacon{->}{v_1} + \alpha_2 \traitalacon{->}{v_2}, \quad \forall \alpha_1, \alpha_2 \in k, v_1, v_2 \in V.
\end{equation*}

\item For $f: \mathbf{M}_1 \to \mathbf{M}_2$ an embedding of marked $3$-manifolds, we denote by $f_*: \mathcal{S}_q(\mathbf{M}_1) \to \mathcal{S}_q(\mathbf{M}_2)$ the linear map sending $[T, s]$ to $[f(T), s\circ f^{-1}]$. We thus get a functor $$\mathcal{S}_q : \mathcal{M}\to \Mod_k.$$
\item For $\mathbf{\Sigma}\times I \in \MS$ a thickened marked surface, $\mathcal{S}_q(\mathbf{\Sigma}):= \mathcal{S}_q(\mathbf{\Sigma}\times I)$ has an algebra structure where the product of two classes of stated tangles $[T_1,s_1]$ and $[T_2,s_2]$ is defined by  isotoping $T_1$ and $T_2$  in $\Sigma\times (0, 1) $ and $\Sigma\times (-1, 0)$ respectively and then setting $[T_1,s_1]\cdot [T_2,s_2]=[T_1\cup T_2, s_1\cup s_2]$. Figure \ref{fig_product} illustrates this product. So we get, by restriction,  a functor $$\mathcal{S}_q: \MS \to \Alg_k.$$
\end{enumerate}
\end{definition}

\begin{figure}[!h] 
\centerline{\includegraphics[width=8cm]{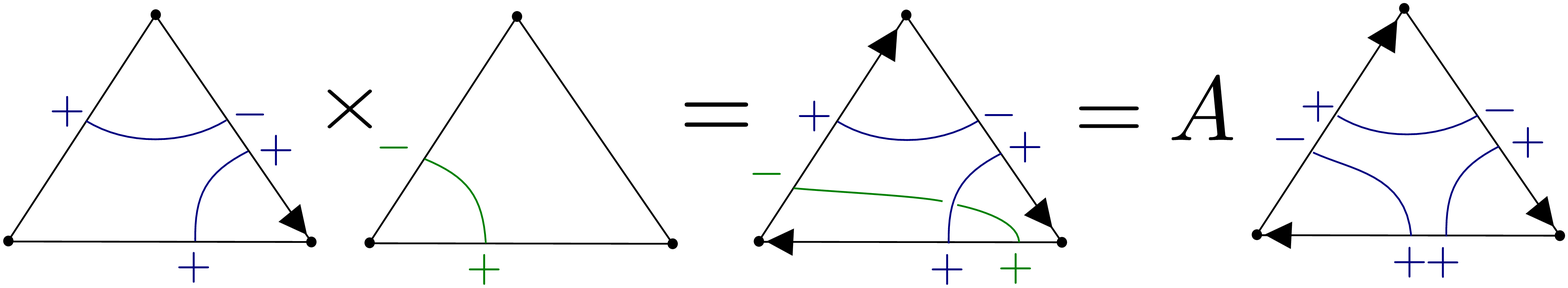} }
\caption{An illustration of the product in stated skein algebras.} 
\label{fig_product} 
\end{figure}

\begin{remark}\label{remark_skeinrelations}
An immediate consequence of the definition (detailed in \cite{LeStatedSkein}) are the following skein relations:

\begin{equation*} 
\begin{tikzpicture}[baseline=-0.4ex,scale=0.5,>=stealth]
\draw [fill=gray!45,gray!45] (-.7,-.75)  rectangle (.4,.75)   ;
\draw[->] (-0.7,-0.75) to (-.7,.75);
\draw[line width=1.2] (-0.7,-0.3) to (-0.3,-.3);
\draw[line width=1.2] (-0.7,0.3) to (-0.3,.3);
\draw[line width=1.15] (-.4,0) ++(-90:.3) arc (-90:90:.3);
\draw (-0.9,0.3) node {\scriptsize{$i$}}; 
\draw (-1.7,-0.3) node {\scriptsize{$\hT^{-1}(j)$}}; 
\end{tikzpicture}
=
\begin{tikzpicture}[baseline=-0.4ex,scale=0.5,>=stealth]
\draw [fill=gray!45,gray!45] (-.7,-.75)  rectangle (.4,.75)   ;
\draw[->] (-0.7,-0.75) to (-.7,.75);
\draw[line width=1.2] (-0.7,-0.3) to (-0.3,-.3);
\draw[line width=1.2] (-0.7,0.3) to (-0.3,.3);
\draw[line width=1.15] (-.4,0) ++(-90:.3) arc (-90:90:.3);
\draw (-1.7,0.3) node {\scriptsize{$\hT^{-1}(i)$}}; 
\draw (-0.9,-0.3) node {\scriptsize{$j$}}; 
\end{tikzpicture}
=\delta_{i,j}
 , \quad \forall i,j \in \{-, +\};
\end{equation*}

\begin{equation*}
\heightcurveright
= \sum_{i= \pm} 
 \hspace{.2cm} 
	\begin{tikzpicture}[baseline=-0.4ex,scale=0.5, >=stealth]
	\draw [fill=gray!60,gray!45] (-.7,-.75)  rectangle (.4,.75)   ;
	\draw[->] (-0.7,-0.75) to (-0.7,.75);
	\draw[line width=1.2] (0.4,-0.3) to (-.7,-.3);
	\draw[line width=1.2] (0.4,0.3) to (-.7,.3);
	\draw (-1,0.3) node {\scriptsize{$i$}}; 
	\draw (-1.8,-0.3) node {\scriptsize{$\hT^{-1}_*(i)$}}; 
	\end{tikzpicture}
= \sum_{i= \pm} 
 \hspace{.2cm} 
\begin{tikzpicture}[baseline=-0.4ex,scale=0.5, >=stealth]
	\draw [fill=gray!60,gray!45] (-.7,-.75)  rectangle (.4,.75)   ;
	\draw[->] (-0.7,-0.75) to (-0.7,.75);
	\draw[line width=1.2] (0.4,-0.3) to (-.7,-.3);
	\draw[line width=1.2] (0.4,0.3) to (-.7,.3);
	\draw (-1.8,0.3) node {\scriptsize{$\hT^{-1}(i)$}}; 
	\draw (-1,-0.3) node {\scriptsize{$i$}}; 
	\end{tikzpicture}
.
\end{equation*}

\end{remark}

\subsubsection{Splitting morphisms and comodule structures}

The stated skein functor has a good behavior for all three operations that we defined on $\mathcal{M}$. First, it is an immediate consequence of the definition that $\mathcal{S}_q : (\mathcal{M}, \bigsqcup) \to (\Mod_k ,\otimes_k)$ is monoidal. Next for $a,b$ two distinct boundary discs of $\mathbf{M} \in \mathcal{M}$, there is a linear map 
$$ \theta_{a \# b} : \mathcal{S}_q( \mathbf{M}_{a \# b}) \to \mathcal{S}_q(\mathbf{M}), $$
named \textit{splitting morphism} defined in \cite{LeStatedSkein, BloomquistLe} defined as follows. Let $c \subset M_{a \# b} $ be the common image of $a$ and $b$. For $[T_0,s_0]\in \mathcal{S}_q ( \mathbf{M}_{a \# b})$, isotope $T_0$ such that it intersects $c$ transversally and such that the framing of every point of $T_0\cap c$ points in the height direction.
Let $T\subset M$ be the framed tangle obtained by cutting $T_0$ along $c$. 
Any two states $s_a : \partial_a T \rightarrow \{-,+\}$ and $s_b : \partial_b T \rightarrow \{-,+\}$ give rise to a state $(s_a, s_0, s_b)$ on $T$. 
Both the sets $\partial_a T$ and $\partial_b T$ are in canonical bijection with the set $T_0\cap c$ by the map quotient map $M\to M_{a \# b}$. Hence the two sets of states $s_a$ and $s_b$ are both in canonical bijection with the set $\mathrm{St}(c):=\{ s: c \cap T_0 \rightarrow \{-,+\} \}$. 

\begin{definition}\label{def_gluing_map}(\cite{LeStatedSkein, BloomquistLe})
The \textit{splitting morphism} $\theta_{a \# b} : \mathcal{S}_q( \mathbf{M}_{a \# b}) \to \mathcal{S}_q(\mathbf{M})$ is the linear map given, for any $(T_0, s_0)$ as above, by: 
$$ \theta_{a \# b} \left( [T_0,s_0] \right) := \sum_{s \in \mathrm{St}(c)} [T, (s, s_0 , s) ].$$
\end{definition}

\begin{figure}[!h] 
\centerline{\includegraphics[width=8cm]{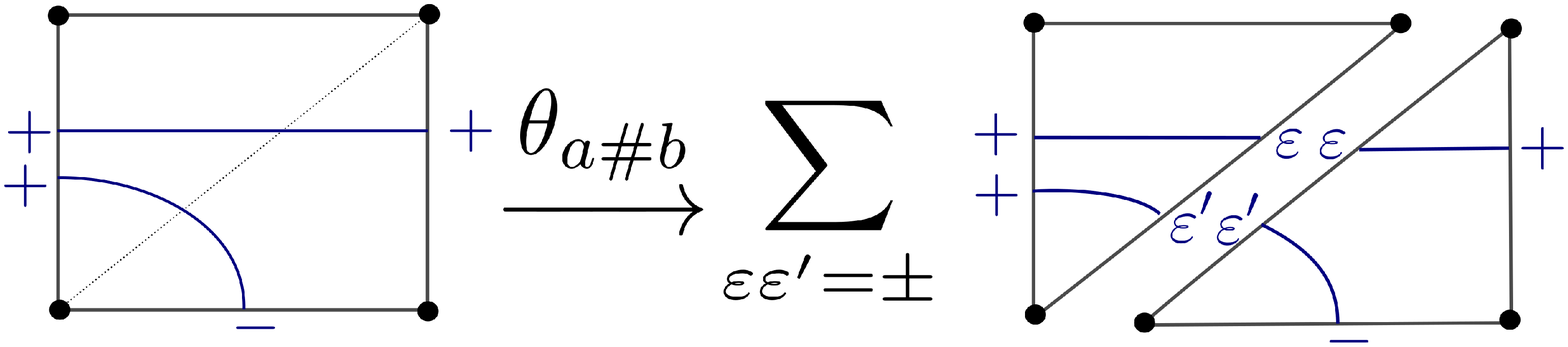} }
\caption{An illustration of the splitting morphism $\theta_{a\#b}$.} 
\label{fig_gluingmap} 
\end{figure} 

\begin{theorem}(\cite[Theorem $3.1$]{LeStatedSkein})
 When $\mathbf{M}$ is a thickened surface, then $\theta_{a \# b}$ is an injective morphism of algebras.
\end{theorem}

Recall that the bigon $\mathbb{B}$ is a thickened disc with two boundary arcs, say $a_L$ and $a_R$.
For $\varepsilon, \varepsilon' \in \{-, +\}$, let $\alpha_{\varepsilon \varepsilon'}\in \mathcal{S}_q(\mathbb{B})$ be the class of the stated arc with arc $\alpha$ connecting $a_L$ to $a_R$ with state $\varepsilon$ on $\alpha\cap a_L$ and state $\varepsilon'$ on $\alpha\cap a_R$, i.e. $\alpha_{\varepsilon \varepsilon'}= \adjustbox{valign=c}{\includegraphics[width=1.3cm]{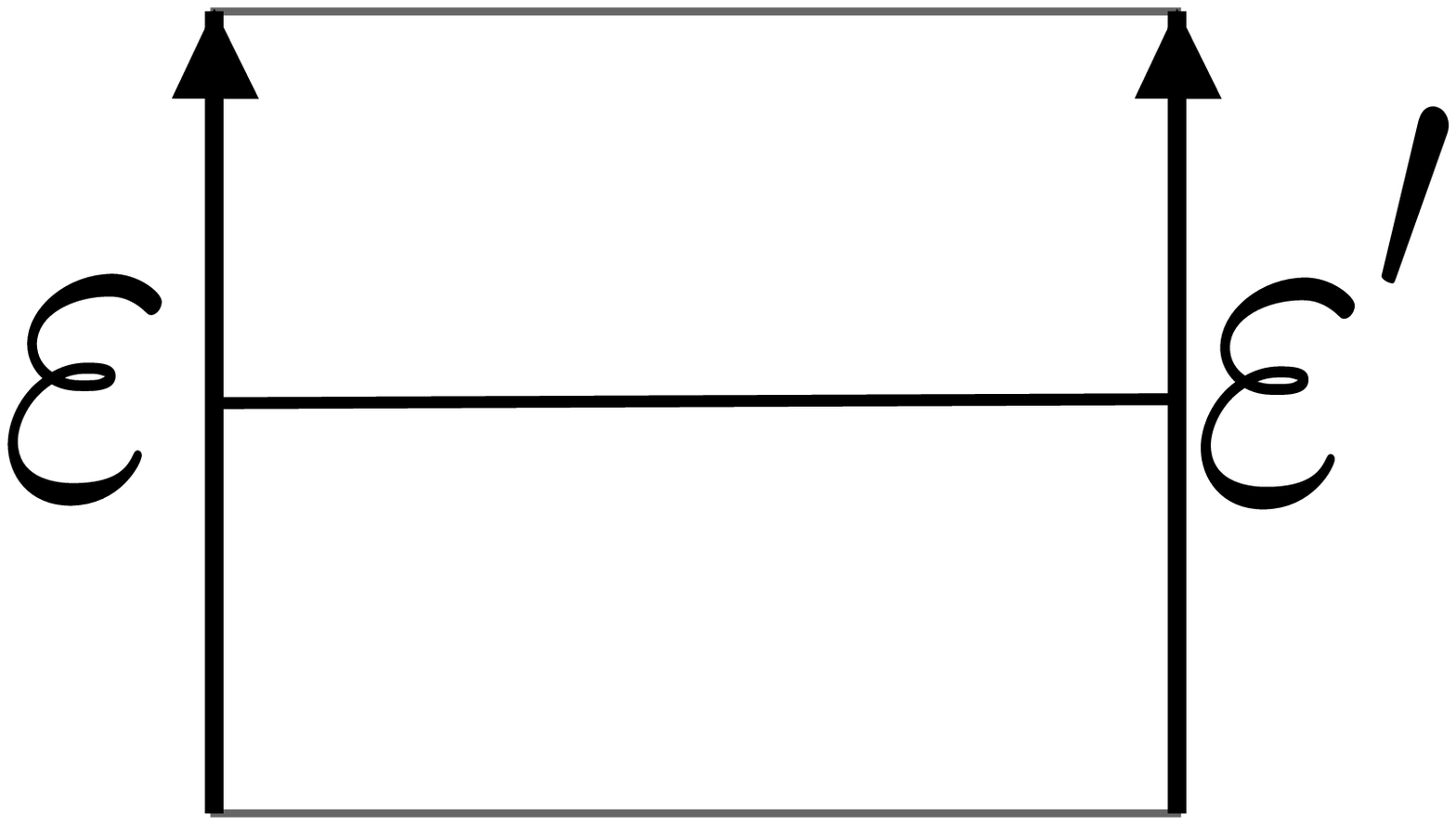}}$.
 By gluing two copies $\mathbb{B} \bigsqcup \mathbb{B}'$ of the bigon together, by identifying $a_R$ with $a_L'$, we get another bigon. The splitting morphism
$$\Delta := \theta_{a_R \# a_L'} : \mathcal{S}_q(\mathbb{B})^{\otimes 2} \to \mathcal{S}_q(\mathbb{B}), $$
endows $\mathcal{S}_q(\mathbb{B})$ with a Hopf algebra structure with coproduct $\Delta$ and counit $\epsilon \begin{pmatrix} \alpha{++} & \alpha_{+-} \\ \alpha_{-+} & \alpha_{--} \end{pmatrix} = \begin{pmatrix} 1&0 \\ 0&1\end{pmatrix}$. 
Moreover, $\mathcal{S}_q(\mathbb{B})$ has a structure of cobraided Hopf algebra where the co-R matrix  $\mathbf{r}: \mathcal{S}_q(\mathbb{B})^{\otimes 2} \to k$  is defined by the formula
  $$ \mathbf{r} \left( \adjustbox{valign=c}{\includegraphics[width=1.7cm]{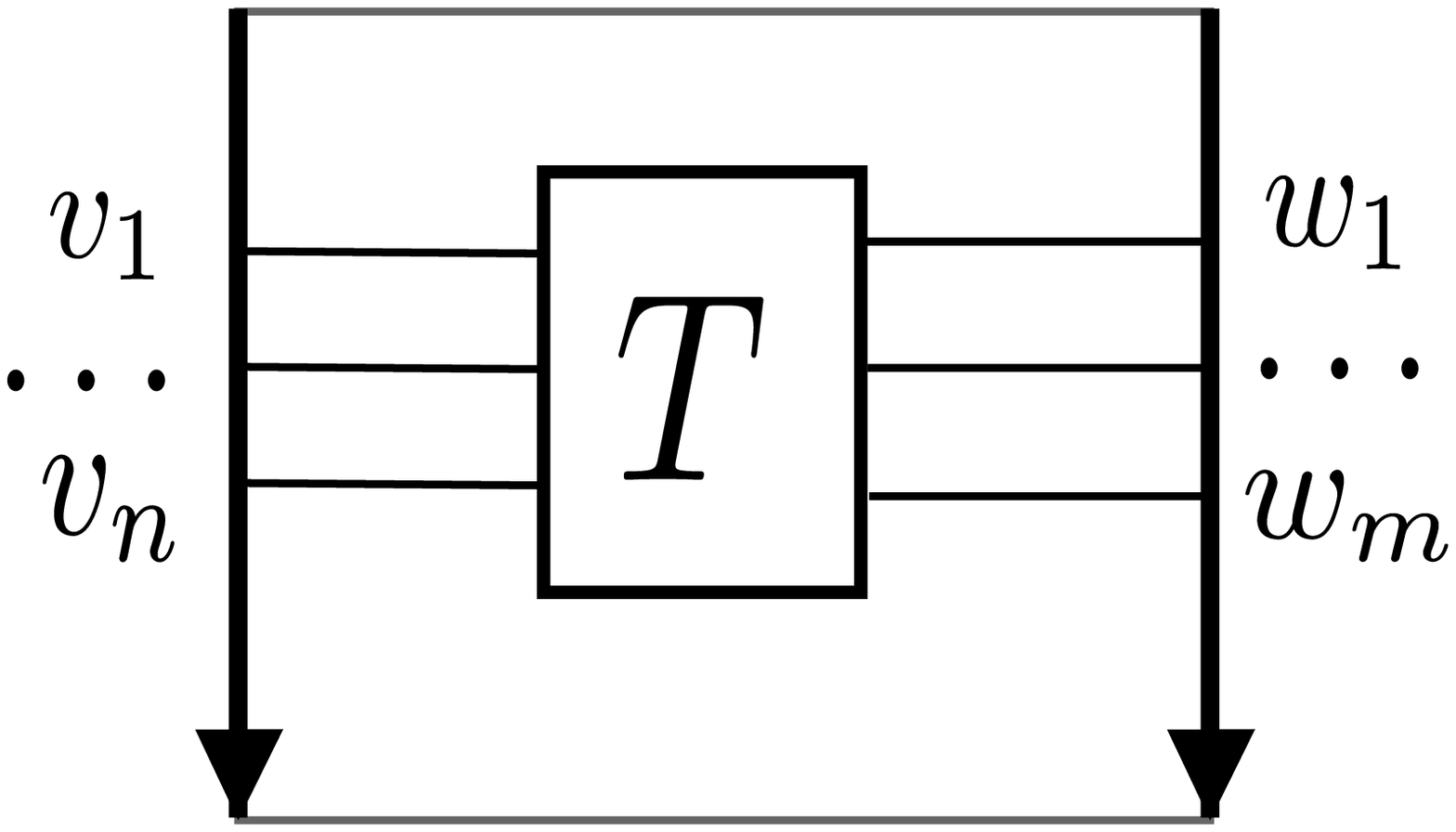}} \otimes\adjustbox{valign=c}{\includegraphics[width=1.7cm]{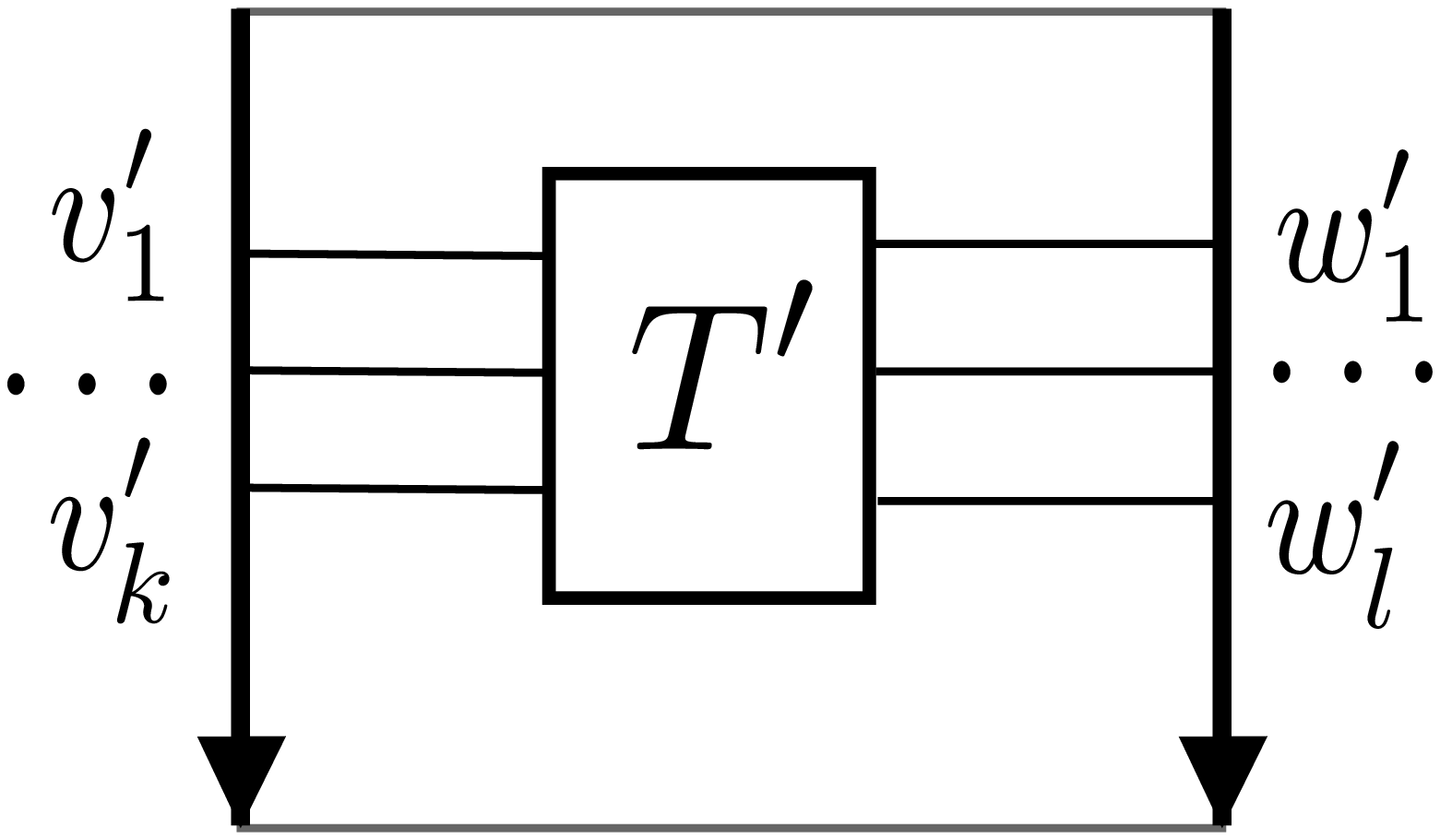}}  \right)
  := \epsilon\left(  \adjustbox{valign=c}{\includegraphics[width=3cm]{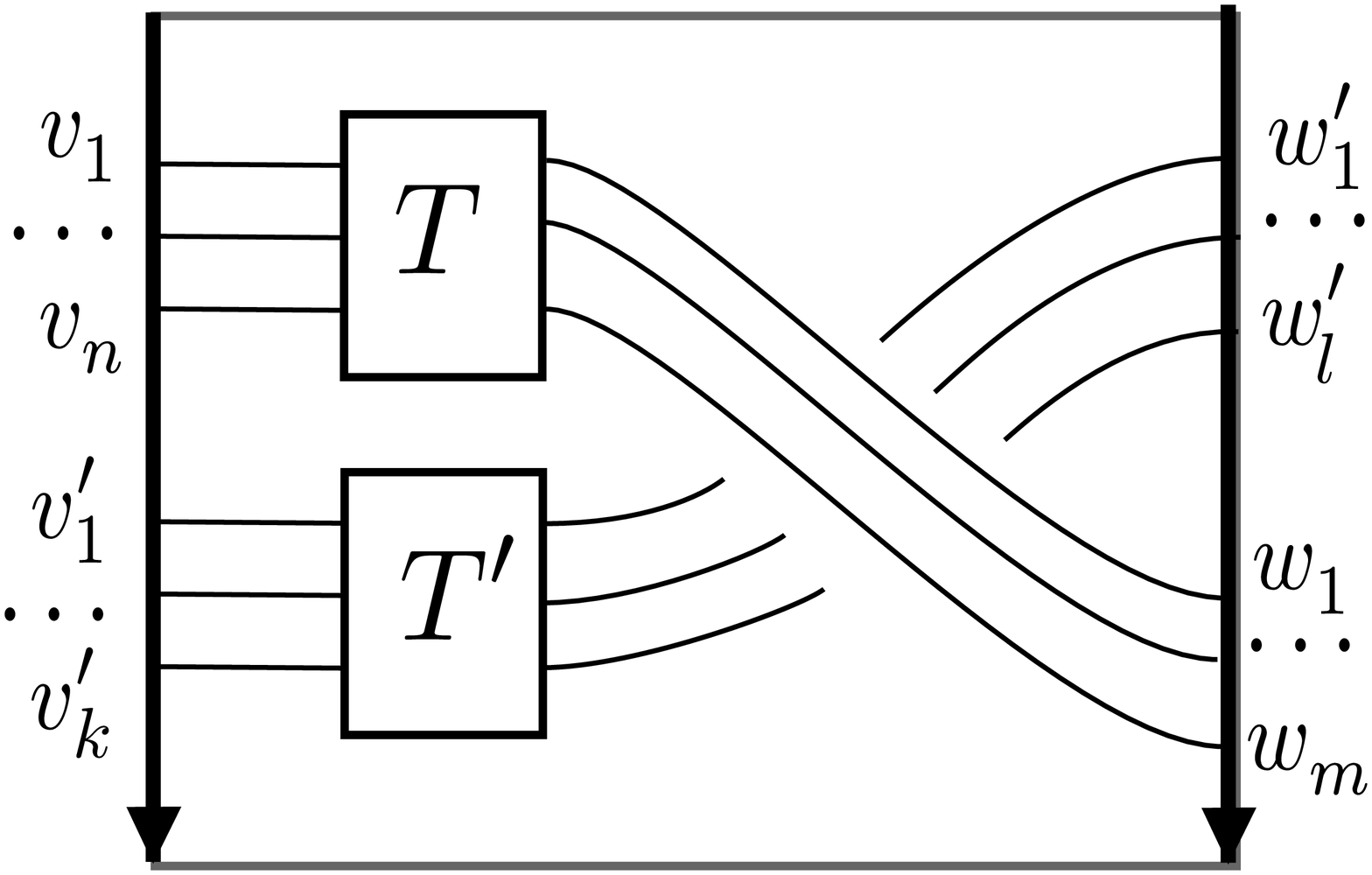}} \right),$$
    The coribbon structure is given by the (positive) co-twist
  $$\Theta \left(\adjustbox{valign=c}{\includegraphics[width=1.7cm]{TangleHTwistL.eps}}  \right) = \epsilon \left( \adjustbox{valign=c}{\includegraphics[width=2cm]{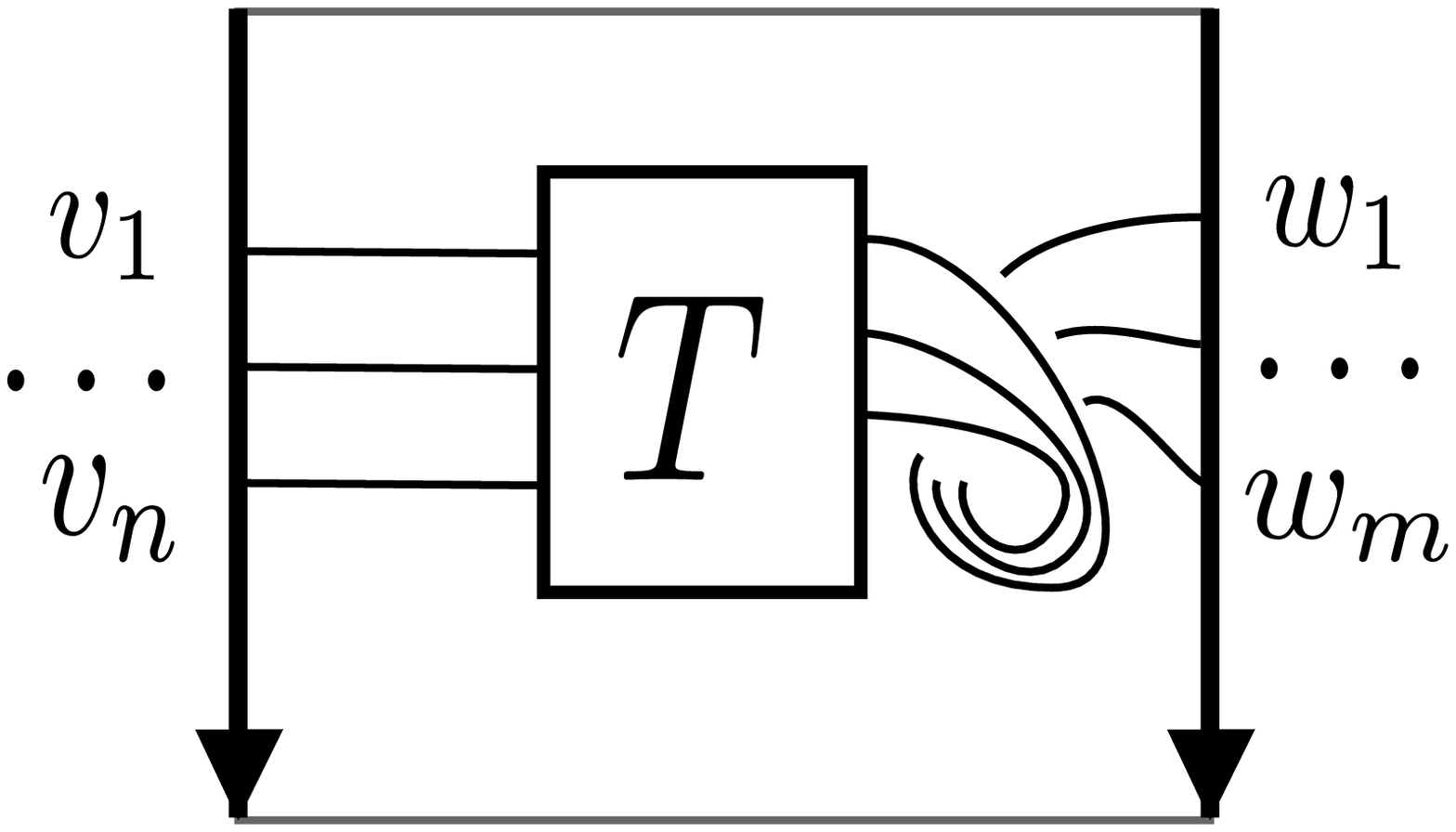}} \right).$$

\par 
  The  half-coribbon structure on $\mathcal{S}_q(\mathbb{B})$ is defined by the half-twist $t: \mathcal{S}_q(\mathbb{B})\to k$:
$$ t \left(  \adjustbox{valign=c}{\includegraphics[width=1.7cm]{TangleHTwistL.eps}} \right) = \varepsilon \left(  \adjustbox{valign=c}{\includegraphics[width=2cm]{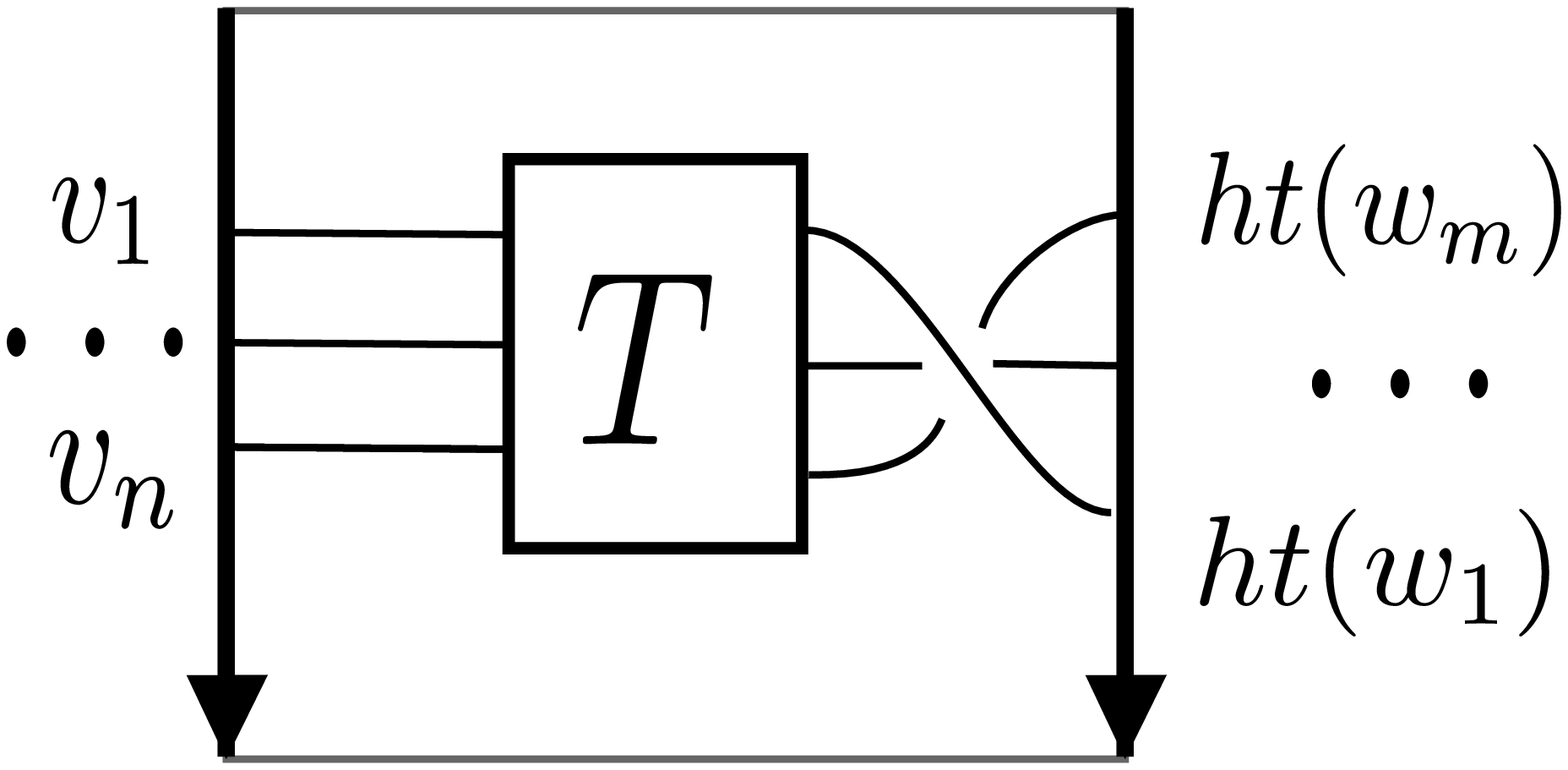}} \right)= \varepsilon \left(  \adjustbox{valign=c}{\includegraphics[width=2cm]{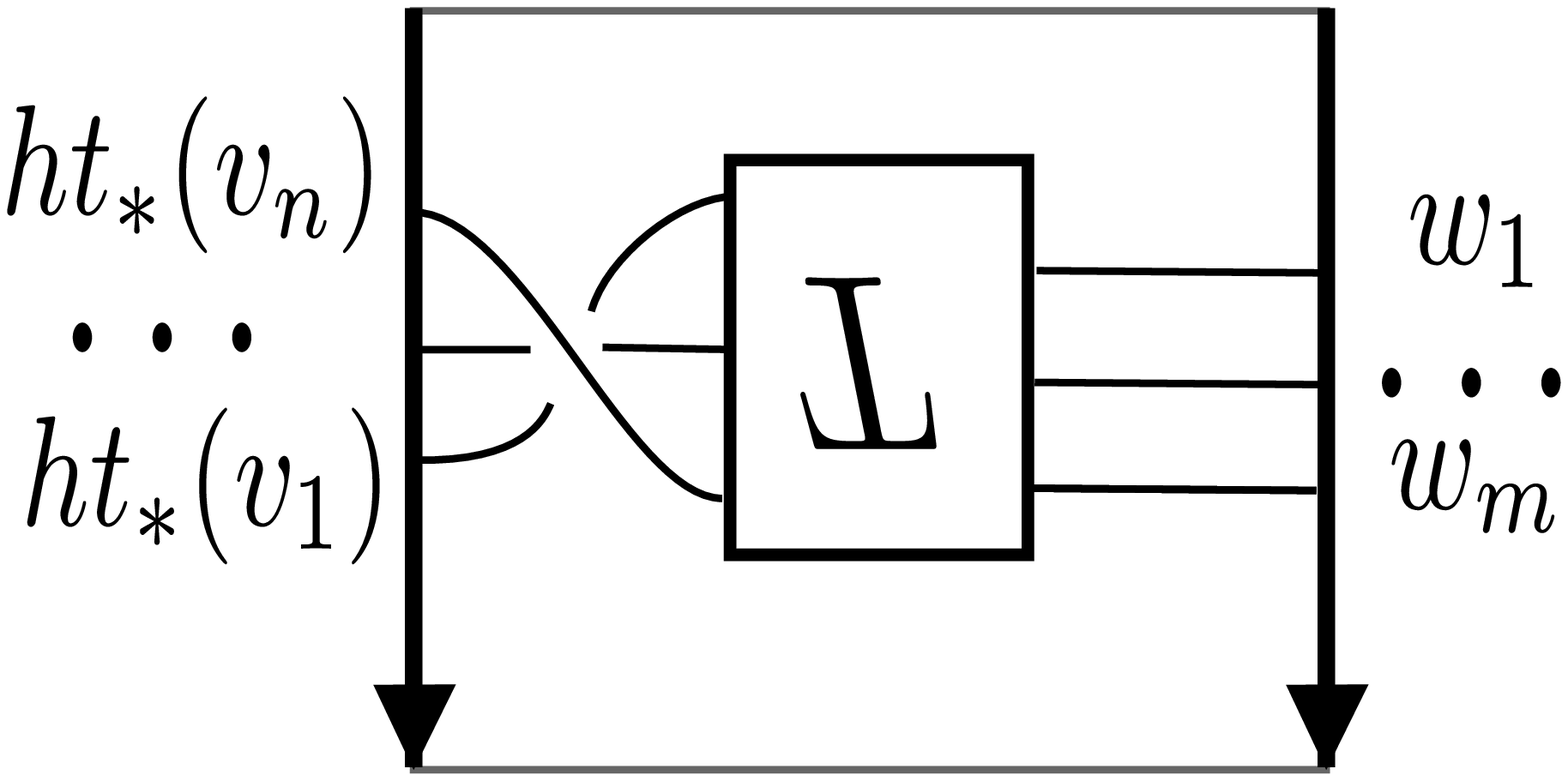}} \right).$$

The rotation operator $\rot : \mathcal{S}_q(\mathbb{B}) \to \mathcal{S}_q(\mathbb{B})$ can be visualized as a $90$ degrees rotation (hence the name) as follows:
$$ \rot \left(  \adjustbox{valign=c}{\includegraphics[width=1.5cm]{TangleHTwistL.eps}} \right) = \left(  \adjustbox{valign=c}{\includegraphics[width=1.5cm]{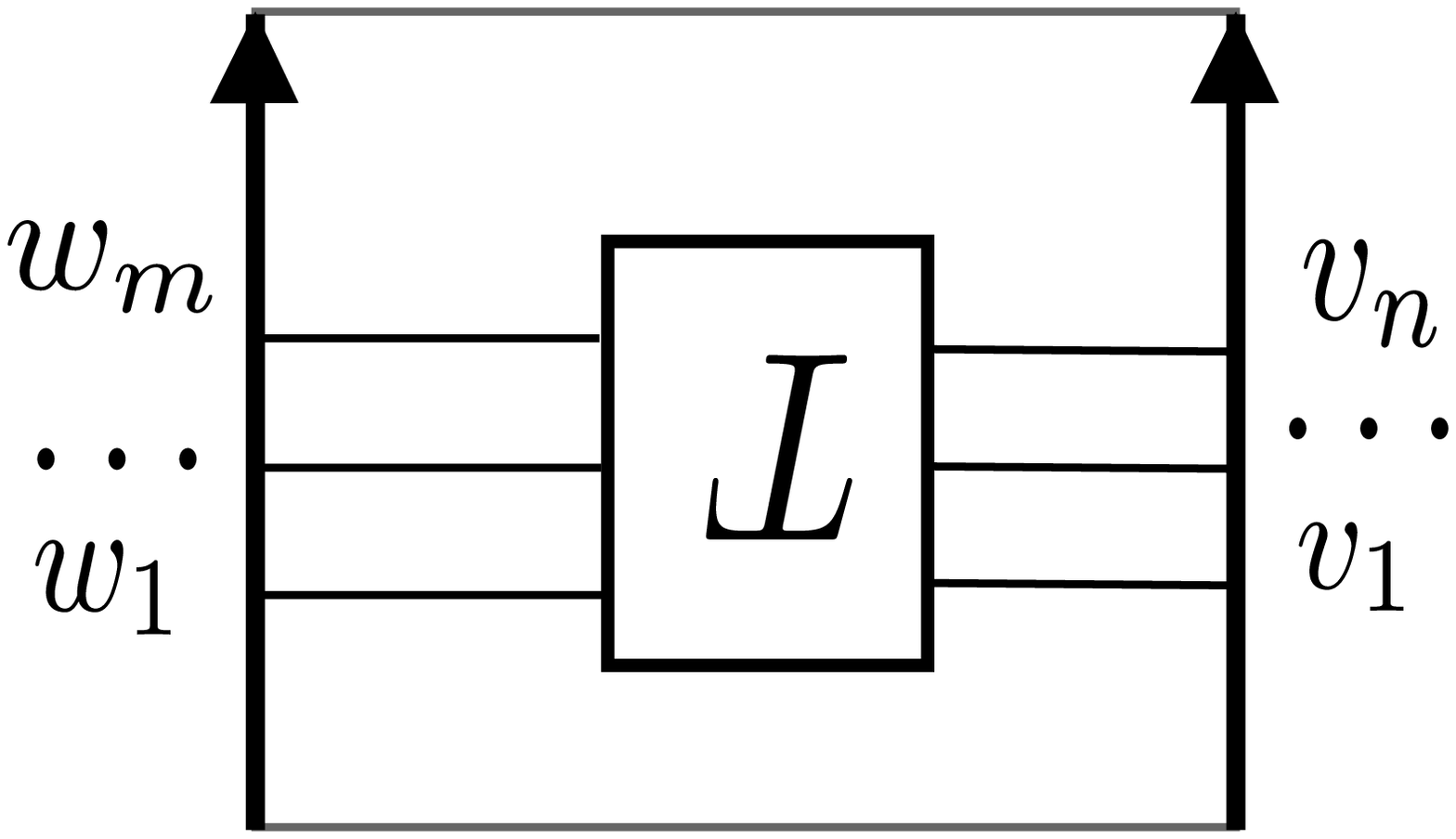}} \right).$$
Therefore, the antipode can be depicted graphically as:
$$ S \left(  \adjustbox{valign=c}{\includegraphics[width=1.5cm]{TangleHTwistL.eps}} \right) = \adjustbox{valign=c}{\includegraphics[width=2.5cm]{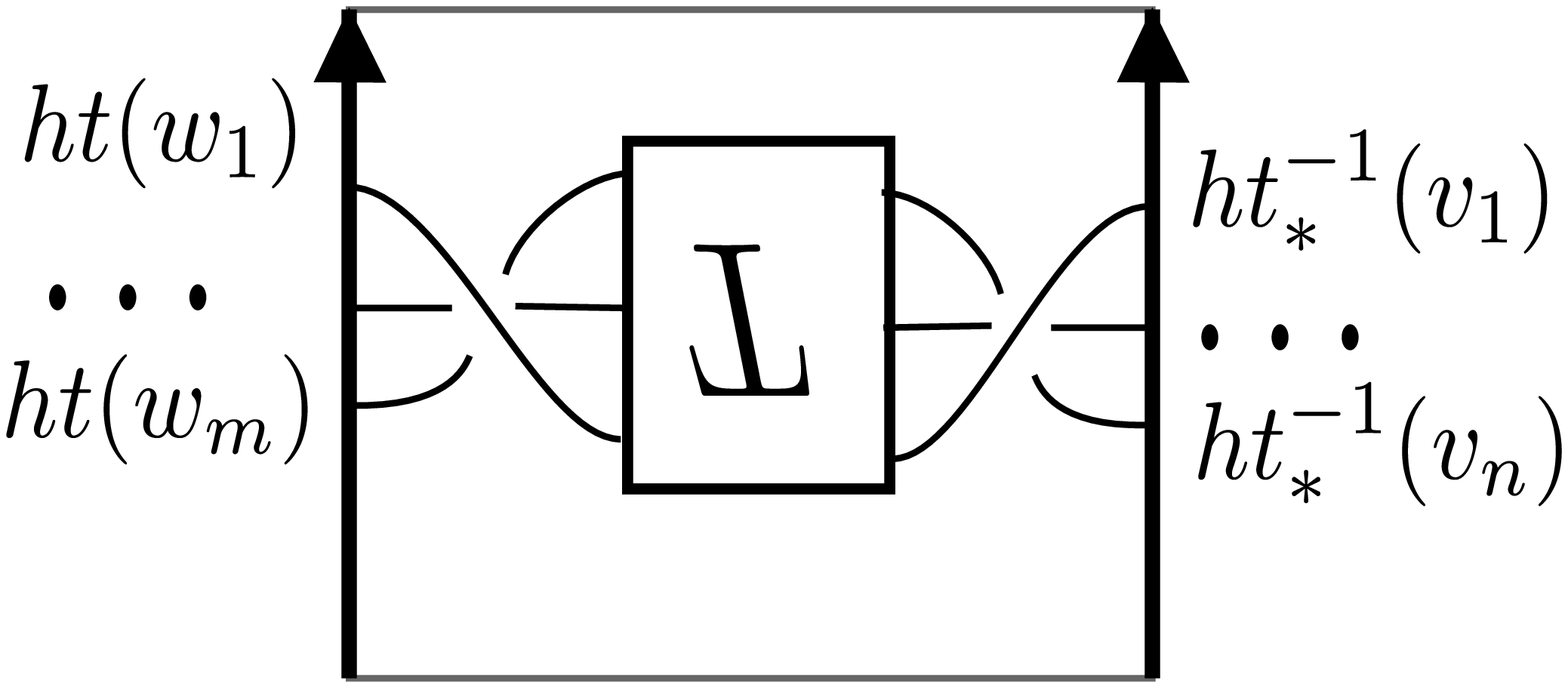}} = \adjustbox{valign=c}{\includegraphics[width=2.7cm]{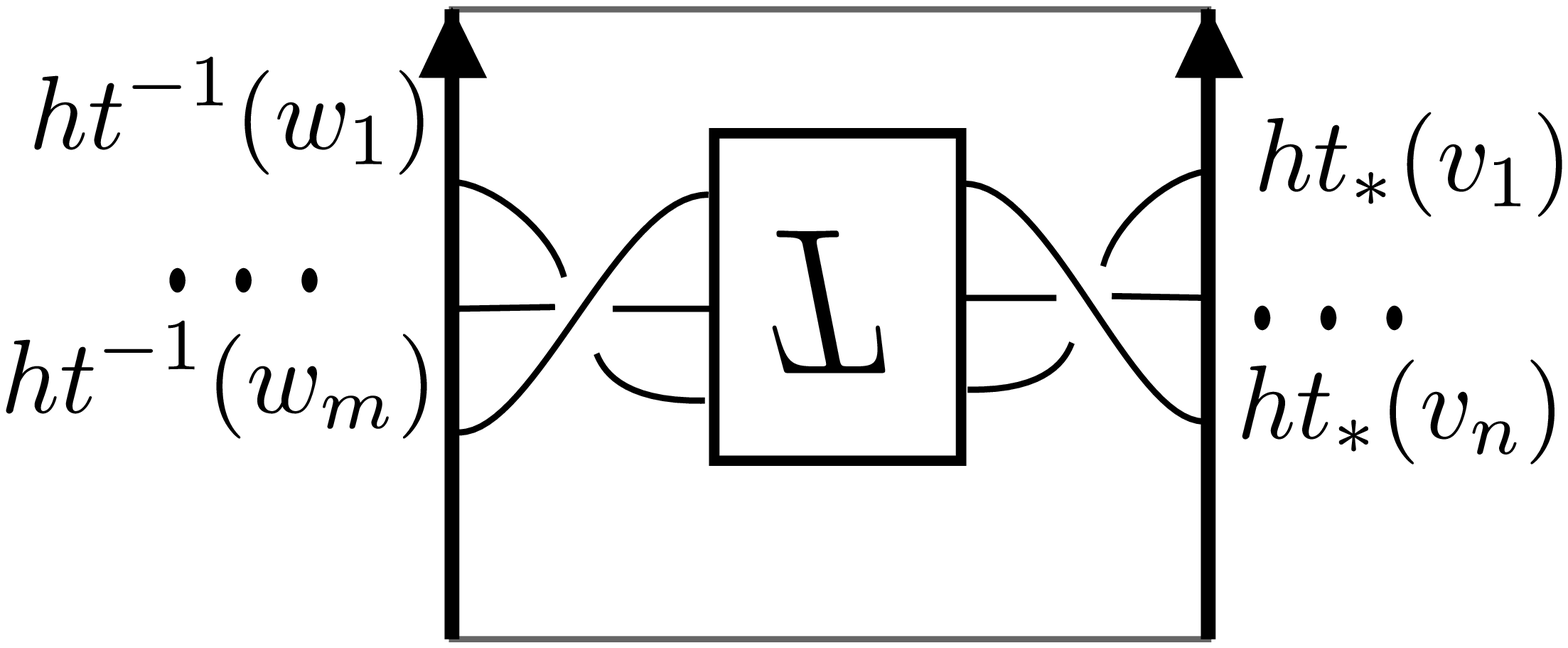}}= \adjustbox{valign=c}{\includegraphics[width=2.7cm]{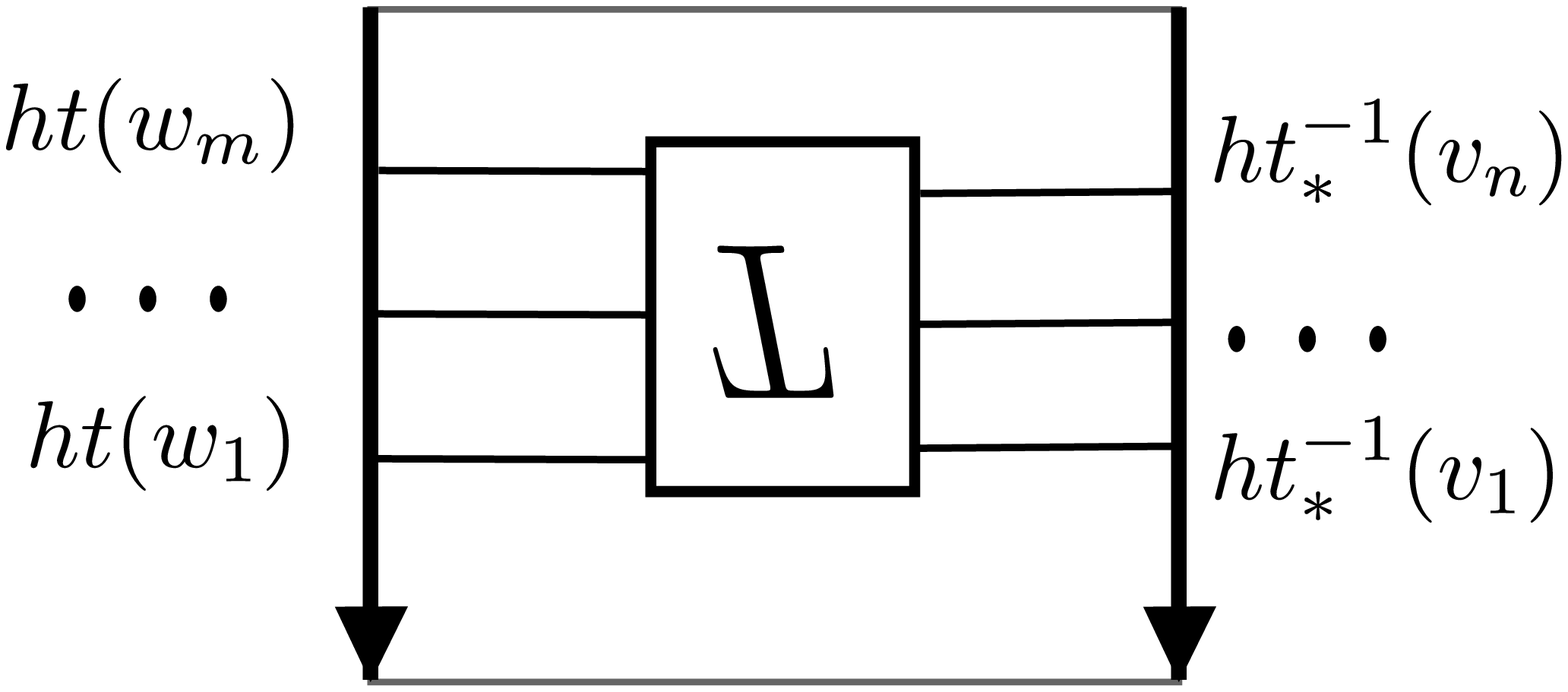}}.$$

\begin{theorem}\label{theorem_bigon}(\cite{LeStatedSkein}, \cite{KojuQuesneyClassicalShadows}, \cite{CostantinoLe19}, \cite{Haioun_Sskein_FactAlg})
There exists an isomorphism of half-coribbon Hopf algebras $\mathcal{S}_q(\mathbb{B})\cong \mathcal{O}_q[\SL_2]$ sending the generators $\alpha_{++}, \alpha_{+-}, \alpha_{-+}, \alpha_{--}$ to $a,b,c,d$ respectively, where $\mathcal{O}_q[\SL_2]$ is equipped with the  half-coribbon structure given by $t_1$.
\end{theorem}

Let $\mathbf{M} \in \mathcal{M}$ and $b$ a boundary disc of $\mathbf{M}$. By gluing the bigon $\mathbb{B}$ to $\mathbf{M}$ while identifying $a_R$ with $b$ we get a marked $3$-manifold isomorphic to $\mathbf{M}$ so, identifying $\mathcal{S}_q(\mathbb{B})$ with $\mathcal{O}_q[\SL_2]$ using the isomorphism of Theorem \ref{theorem_bigon},  the splitting morphism
$$\Delta_c^L := \theta_{a_R \# c} : \mathcal{S}_q(\mathbf{M}) \to \mathcal{O}_q[\SL_2] \otimes \mathcal{S}_A(\mathbf{M})$$
endows $\mathcal{S}_q(\mathbf{M})$ with a structure of left $\mathcal{O}_q[\SL_2]$-comodule. Similarly, while gluing $b$ with $a_L$ we get left comodule map:
$$ \Delta_c^R := \theta_{c \# a_L} \mathcal{S}_q(\mathbf{M}) \to \mathcal{S}_A(\mathbf{M}) \otimes \mathcal{O}_q[\SL_2].$$
The comodule map $\Delta_c^L$ is depicted in Figure \ref{fig_comodule}. Note that the two comodules are related by the functor $\rot_*: \LComod_{\mathcal{O}_q[\SL_2]} \to \RComod_{\mathcal{O}_q[\SL_2]}$, i.e. $\Delta_c^R = \tau \circ (\rot \otimes \id) \circ \Delta_c^L$.

Therefore, the stated skein functor restricts to functors:
$$ \mathcal{S}_q : \mathcal{M}^{(n)} \to (\mathcal{O}_q[\SL_2])^{\otimes n}-\RComod.$$

\begin{figure}[!h] 
\centerline{\includegraphics[width=12cm]{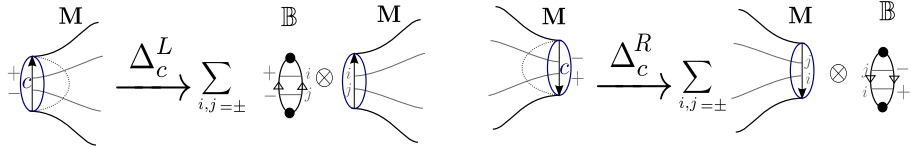} }
\caption{An illustration of the comodule maps $\Delta_c^L, \Delta_c^R$.} 
\label{fig_comodule} 
\end{figure}

\subsubsection{The quantum fusion operation}\label{sec_qfusion}

\begin{definition}\label{def_fusion_quantique}
Let $(A,\mu_A, \eta, \Delta, \epsilon_A, S, \mathbf{r})$  be a cobraided Hopf algebra and $C \in \mathrm{RComod}_{A \otimes A}$ with comodule map $\Delta_{A\otimes A} : C \to C\otimes A \otimes A$.
Write $\Delta^1:= (\id \otimes \epsilon \otimes \id) \circ \Delta_{A\otimes A}: C\to C\otimes A$ and $\Delta^2: (\epsilon \otimes \id \otimes \id) \circ \Delta_{A\otimes A}: C \to C \otimes A$.
 The \textit{quantum fusion} ${C}_{1\circledast 2}$ is the comodule in $\mathrm{RComod}_A$ where 
 \begin{enumerate}
 \item $C_{1\circledast 2}= C$ as a $k$-module;
 \item The comodule map is $\Delta_A:= (\id \otimes \mu_A ) \circ \Delta_{A\otimes A} $.
 \end{enumerate}
 Moreover if $(C, \mu, \epsilon)$ is an algebra object in $\mathrm{RComod}_{A\otimes A}$, its quantum fusion  has a structure of algebra object  $(C_{1\circledast 2}, \mu_{1\circledast 2}, \epsilon_{1\circledast 2})$ in $\mathrm{RComod}_A$ where $\epsilon_{1 \circledast 2}=\epsilon$ and the product is the composition
 $$ \mu_{1 \circledast 2} : C\otimes C \xrightarrow{\Delta_1\otimes \Delta_2} C\otimes A \otimes C \otimes A \xrightarrow{\id \otimes \tau_{A,C} \otimes \id} C\otimes C \otimes A \otimes A \xrightarrow{\mu \otimes \mathbf{r} } C.$$
 \end{definition}
 
 \begin{remark}
 \begin{enumerate}
 \item If $V$ and $W$ are two algebra objects in $\mathrm{RComod}_A$, then $V\otimes_k W$ is an algebra object in $\mathrm{RComod}_{A\otimes A}$ and its quantum fusion $(V\otimes_k W)_{1 \circledast 2}$ is the braided tensor product $V\overline{\otimes} W$  of Definition \ref{def_transmutation}.
 \item When $A$ is a deformation quantization of some Poisson-Lie group $G$ and $C$ is the deformation quantization of some (smooth) Poisson $G$-variety, it is proved in \cite[Theorem $13$]{KojuMCGRepQT} that, at the semi-classical level, the quantum fusion operation recovers Alekseev-Malkin's (classical) fusion operation as defined in \cite{AlekseevMalkin_PoissonLie} (see \cite[Section $4.3$]{KojuSurvey} for details). This explains the name "quantum fusion operation".
 \end{enumerate}
 \end{remark}
 
 Now consider a marked surface $\mathbf{M}_{a\circledast b}$ obtained by fusioning two boundary disc $a$ and $b$ of $\mathbf{M}$. Recall that $\mathbb{T}$ is a ball with three boundary discs, say $e_1, e_2, e_3$ and  that  $\mathbf{M}_{a\circledast b}$ is obtained from $\mathbf{M}\bigsqcup \mathbb{T}$ by gluing $a$ to $e_1$ and $b$ to $e_2$. 
 Define a linear map $\Psi_{a\circledast b} : \mathcal{S}_q(\mathbf{M}) \to \mathcal{S}_q(\mathbf{M}_{a\#b})$ by $\Psi_{a\circledast b} ([T,s]^{\mathfrak{o}}):= [T',s']$ where $(T',s')$ is obtained from $(T,s)$ by gluing to each point of $T\cap a$ a straight line in $\mathbb{T}$ between $e_1$ and $e_3$ and by gluing to each point of $T\cap b$ a straight line in $\mathbb{T}$ between $e_2$ and $e_3$ as illustrated in Figure \ref{fig_fusion}. 
 
  \begin{figure}[!h] 
\centerline{\includegraphics[width=6cm]{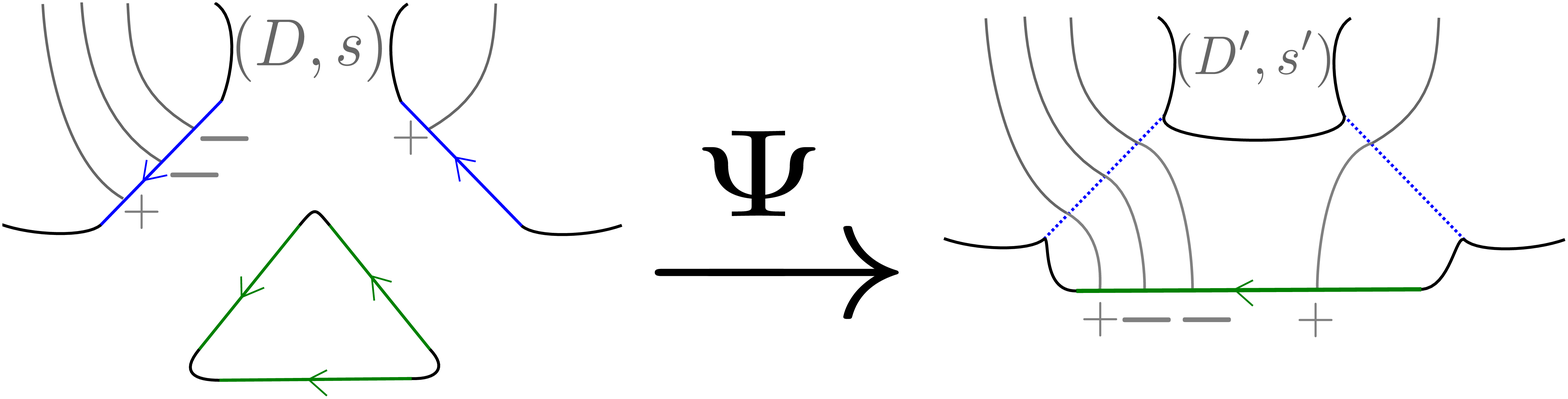} }
\caption{An illustration of $\Psi_{a\circledast b}$.} 
\label{fig_fusion} 
\end{figure}

 \begin{theorem}\label{theorem_fusion}(\cite{CostantinoLe19, Higgins_SSkeinSL3, LeSikora_SSkein_SLN, KojuMCGRepQT}) The linear map $\Psi_{a\circledast b}$ is an isomorphism of $k$-modules which identifies $\mathcal{S}_q(\mathbf{M}_{a\circledast b})$  with the quantum fusion $\mathcal{S}_q(\mathbf{M})_{a\circledast b}$. If $\mathbf{M}$ is a thickened marked surface, then $\Psi_{a\circledast b}$ is an isomorphism of algebras.
 \end{theorem}
 
Theorem \ref{theorem_fusion} was proved by Costantino-L\^e in \cite[Theorem $4.13$]{CostantinoLe19} in the particular case where $\mathbf{M}=\mathbf{\Sigma}_1 \times I \bigsqcup \mathbf{\Sigma}_2 \times I$ is the disjoint union of two thickened marked surfaces with $a$ in $\mathbf{\Sigma}_1$ and $b$ in $\mathbf{\Sigma}_2$. Another proof was proposed by Higgins in \cite{Higgins_SSkeinSL3} (for the $\SL_3$ stated skein algebras). As proved independently in \cite[Proposition $7.6$]{LeSikora_SSkein_SLN} and \cite[Theorem $2.7$]{KojuMCGRepQT}, Higgins' proof extends word-by-word to the more general context of Theorem \ref{theorem_fusion}.

\begin{corollary}
The functors $\mathcal{S}_q : (\mathcal{M}^{(1)}, \wedge) \to (\overline{\mathcal{C}_q^{\SL_2}},  \overline{\otimes})$ and $\mathcal{S}_q: (\MS^{(1)}, \wedge) \to (\Alg(\overline{\mathcal{C}_q^{\SL_2}}), \overline{\otimes})$ are lax monoidal.
\end{corollary}

\subsection{Costantino-L\^e's skein interpretation of the transmutation}\label{sec_skein_transmutation}

Since $\mathcal{S}_q : \mathcal{M}^{(1)}_{\con} \to \overline{\mathcal{C}_q^{\SL_2}}$ is monoidal, the image by $\mathcal{S}_q$ of the Hopf algebra object $\mathcal{H}$ is a Hopf algebra object $\mathcal{S}_q(\mathbf{H}_1)$ in $\overline{\mathcal{C}_q^{\SL_2}}$. Clearly, its algebra structure is the same as the one in Definition \ref{def_skein}. We still denote by $\ad: \mathcal{S}_q(\mathbf{H}_1) \to \mathcal{S}_q(\mathbf{H}_2)$ the image of $\ad$ by $\mathcal{S}_q$ turning $\mathcal{S}_q(\mathbf{H}_1)$ into a comodule over itself in the braided sense. Define a linear isomorphism $f: \mathcal{S}_q(\mathbb{B}) \xrightarrow{\cong} \mathcal{S}_q(\mathbf{H}_1)$, with inverse $f^{-1}$, by the formula:
$$ f \left(  \adjustbox{valign=c}{\includegraphics[width=1.7cm]{TangleHTwistL.eps}} \right) = \adjustbox{valign=c}{\includegraphics[width=1.7cm]{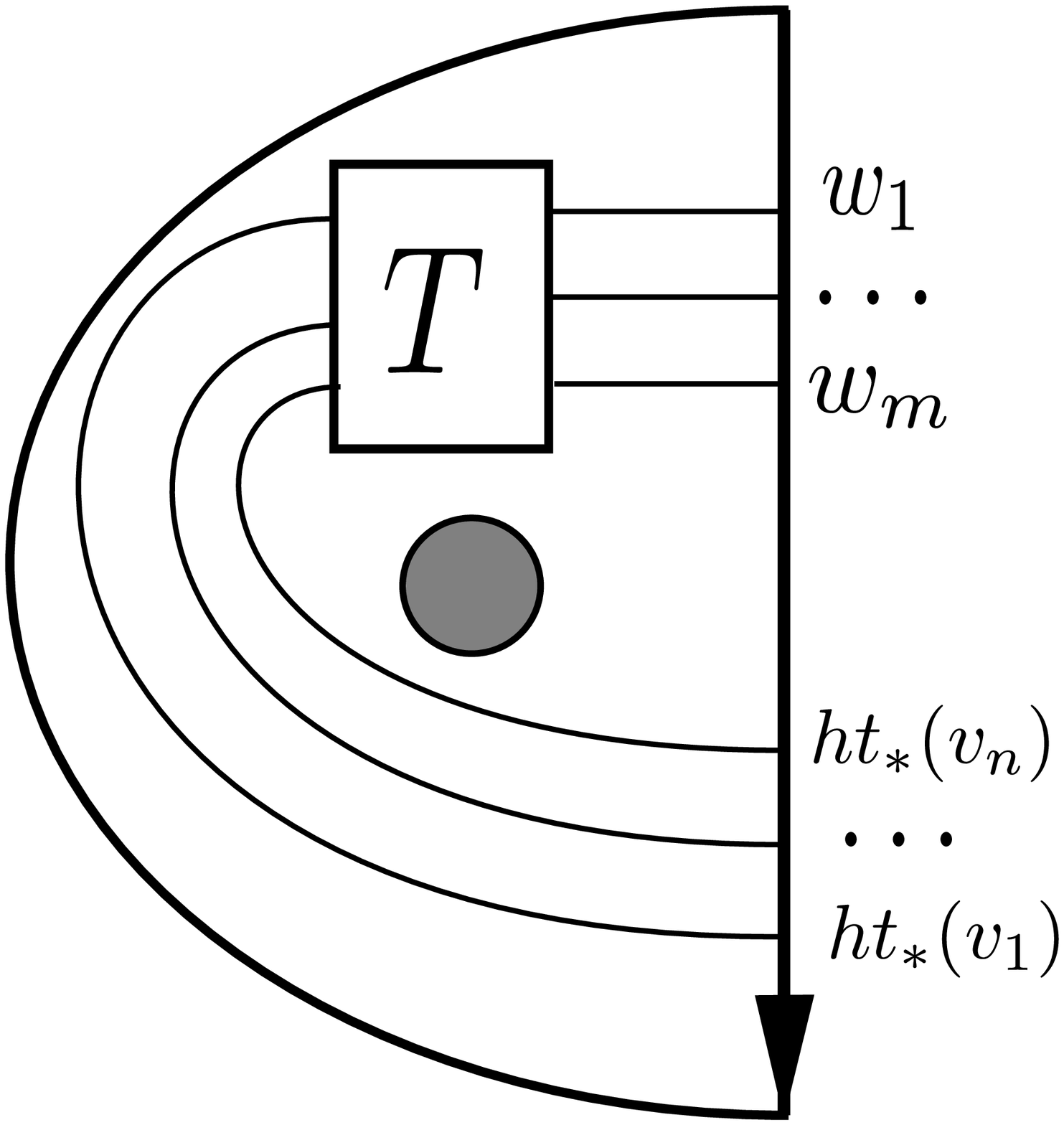}}, \quad f^{-1}\left( \adjustbox{valign=c}{\includegraphics[width=1.7cm]{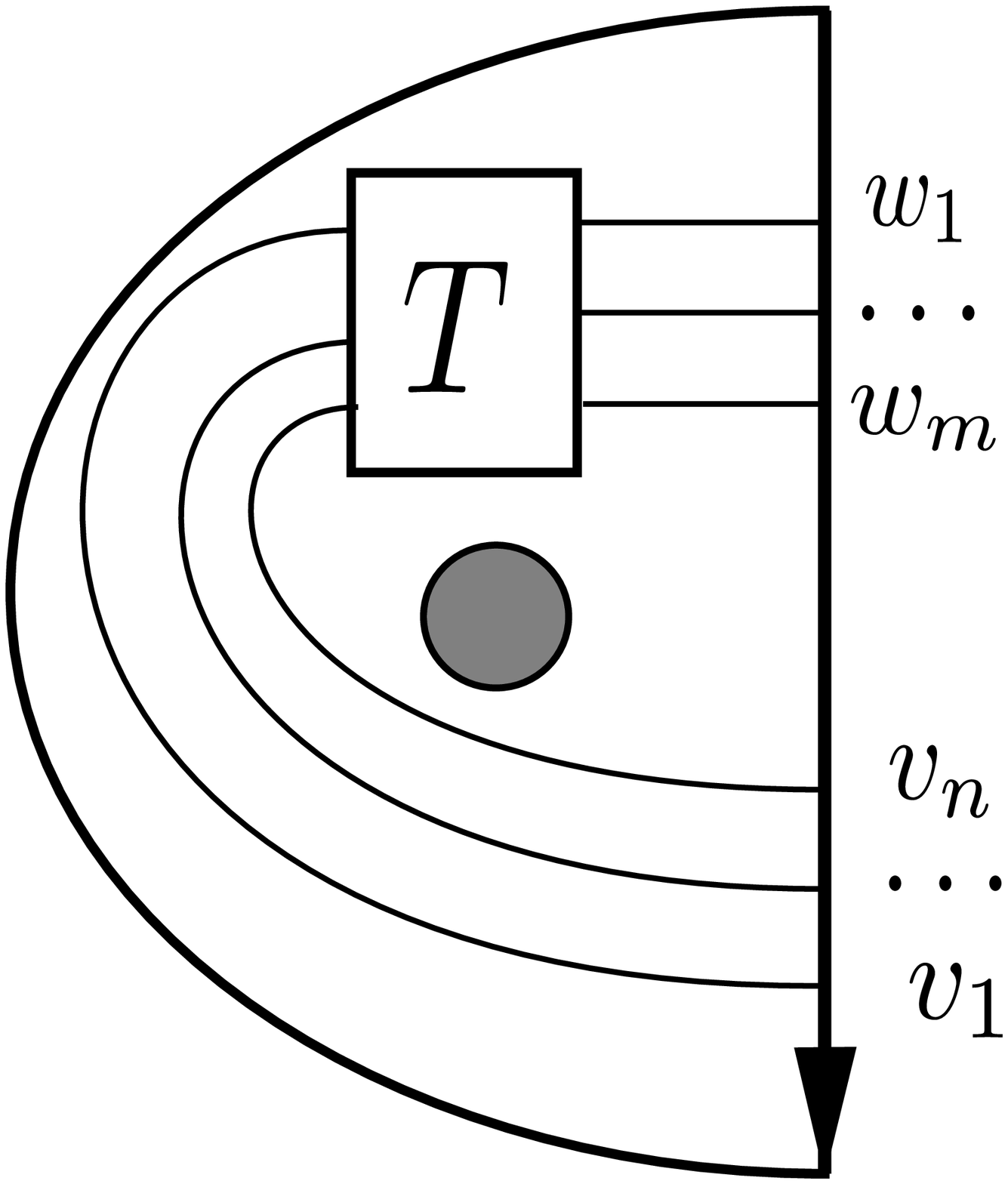}} \right) =  \adjustbox{valign=c}{\includegraphics[width=2.1cm]{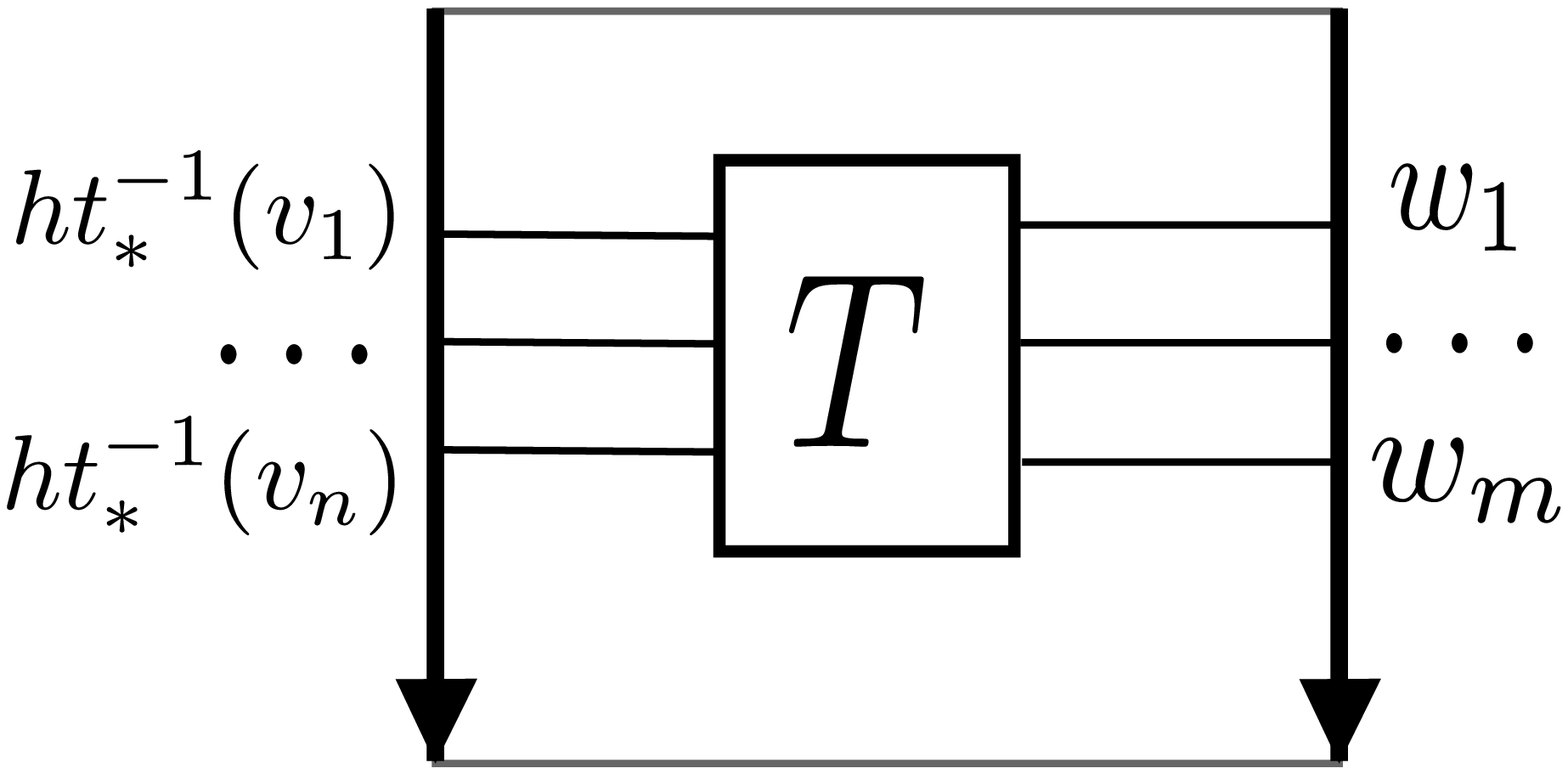}} .$$
Clearly, $f$ sends skein relations in $\mathbb{B}$ to skein relations in $\mathbf{H}_1$ so $f$ is well defined. The following is a slight reformulation of Costantino-L\^e's skein interpretation of the transmutation in \cite{CostantinoLe19}. Since our conventions are different and since the details of the proof will be crucial in the rest of the paper (and are left to the reader in \cite{CostantinoLe19}) we reformulate slightly differently and reprove their result.

\begin{theorem}\label{theorem_skein_transmutation}
The linear isomorphism $f$ is an isomorphism $f: B \mathcal{S}_q(\mathbb{B}) \xrightarrow{\cong} \mathcal{S}_q(\mathbf{H}_1)$ of Hopf algebra objects in $\overline{\mathcal{C}_q^{\SL_2}}$ between the transmutation of $\mathcal{S}_q(\mathbb{B})$ and $\mathcal{S}_q(\mathbf{H}_1)$. Moreover, $f$ intertwines the braided adjoint coaction $\Ad^B$ and $\ad$. 
\end{theorem}

\begin{corollary}\label{coro_skein_transmutation} (\cite{CostantinoLe19, Faitg_LGFT_SSkein})
One has an isomorphism $\widetilde{f} : B_q\SL_2 \xrightarrow{\cong} \mathcal{S}_q(\mathbf{H}_1)$ of Hopf algebra objects in $\overline{\mathcal{C}_q^{\SL_2}}$ characterized by:

$$ \widetilde{f} \begin{pmatrix} a & b \\ c & d \end{pmatrix} = \begin{pmatrix} 0 & -A^{5/2} \\ A^{1/2} & 0 \end{pmatrix} \begin{pmatrix} \CC{+}{+} & \CC{+}{-} \\ \CC{-}{+}& \CC{-}{-}\end{pmatrix}= \begin{pmatrix} -A^{5/2}\CC{-}{+} & -A^{5/2}\CC{-}{-} \\
A^{1/2}\CC{+}{+} & A^{1/2}\CC{+}{-} \end{pmatrix} .$$

So one gets an algebra isomorphism $(B_q\SL_2)^{\overline{\otimes} n } \cong \mathcal{S}_q(\mathbf{H}_n)$ and the two functors  $\restriction{\mathcal{S}_q}{\BT} : \BT \to \overline{\mathcal{C}_q^{\SL_2}}$ and $Q_{B_q\SL_2} :\BT \to \overline{\mathcal{C}_q^{\SL_2}}$ are isomorphic.
\end{corollary}

\begin{remark}
\begin{enumerate}
\item (Comparison with Costantino-L\^e)
Let $\Theta: \mathcal{S}_q (\mathbf{H}_1) \to \mathcal{S}_{q^{-1}}(\mathbf{H}_1)$ be the (involutive) isomorphism sending a class $[T,s]$ for which $s:\partial T \to \{v_-, v_+\}$ is valued in the standard basis,  to $[\varphi(T), s\circ \varphi^{-1}]$ , where $\varphi: \mathbb{D}_1\times [-1,1] \cong \mathbb{D}_1 \times [-1, 1]$ sends $(x,t)$ to $(x, 1-t)$. Then $\Theta$ is an anti-morphism of algebras, i.e. $\Theta(xy)=\Theta(y)\Theta(x)$ (see \cite{LeStatedSkein} where $\Theta$ is called \textit{reflexion involution} for details). Recall from Definition \ref{def_HT} the (involutive) anti-morphism of algebras $C_t: \mathcal{S}_q(\mathbb{B}) \to \mathcal{S}_q(\mathbb{B})$. The composition $\Theta \circ f \circ C_t$ is an algebra morphism between $B \mathcal{S}_q(\mathbb{B})$ and $\mathcal{S}_{q^{-1}}(\mathbf{H}_1)$. Changing $q$ to $q^{-1}$, we get an isomorphism $g: B_{q^{-1}}\SL_2 \to \mathcal{S}_q(\mathbf{H}_1)$ characterized by the formula
$$ g \begin{pmatrix} a & b \\ c & d \end{pmatrix} = \begin{pmatrix} \CCC{+}{+} & \CCC{+}{-} \\ \CCC{-}{+}& \CCC{-}{-}\end{pmatrix} \begin{pmatrix} 0 & -A^{5/2} \\ A^{1/2} & 0 \end{pmatrix} = \begin{pmatrix} A^{1/2}\CCC{+}{-} & -A^{5/2}\CCC{+}{+} \\
A^{1/2}\CCC{-}{-} & -A^{5/2}\CCC{-}{+} \end{pmatrix} .$$
Costantino and L\^e proved in \cite[Proposition $4.17$]{CostantinoLe19} that $g$ is a surjective morphism of Hopf algebras (there is a typo in \cite[Proposition $4.17$]{CostantinoLe19} where $-A^{5/2}$ were incidentally replaced by $-A^{3/2}$). The authors also claim the injectivity of $g$ without giving any argument.
\item (Comparison with Faitg) In \cite{Faitg_LGFT_SSkein}, Faitg considered an algebra denoted by $\mathcal{L}_{0,1}$ named \textit{quantum moduli algebra} which  appeared originally in \cite{AlekseevGrosseSchomerus_LatticeCS1,AlekseevGrosseSchomerus_LatticeCS2, BuffenoirRoche, BuffenoirRoche2} and proved in \cite[Lemma $5.6$]{Faitg_LGFT_SSkein} the existence of an isomorphism of algebras $\mathcal{S}_q(\mathbf{H}_1)^{op} \cong \mathcal{L}_{0,1}$, where $ \mathcal{S}_q(\mathbf{H}_1)^{op}$ is the skein algebra with opposite product (tangle are stacked from bottom to top in  \cite{Faitg_LGFT_SSkein} instead of top to bottom). By comparing the definition of $\mathcal{L}_{0,1}$ in \cite[Equation $51$]{Faitg_LGFT_SSkein} with Equation \eqref{eq_BSL2} in Section \ref{sec_QG}, one sees that $\mathcal{L}_{0,1}$ is canonically isomorphic to $B_q\SL_2$ and that Faitg isomorphism corresponds to the isomorphism $h:=\widetilde{f}\circ C_t: B_{q}\SL_2 \cong \mathcal{S}_q(\mathbf{H}_1)^{op}$. To prove that $h$ is isomorphism, Faitg constructed an explicit inverse (see also \cite{KojuPresentationSSkein} where an alternative proof is presented using PBW bases).
\item The two above items imply that Theorem \ref{theorem_skein_transmutation} is simply a reformulation of the theorems of Costantino-L\^e and Faitg. However the previous proofs are made by blind computations (left to the reader in \cite{CostantinoLe19}), using the knowledge of explicit vectorial bases for stated skein algebras. The purpose of this subsection is to provide a more conceptual proof by directly reinterpreting Majid's general formulas \eqref{eq_transmuted_prod} and \eqref{eq_transmuted_antipod} of Definition \ref{def_transmutation} in the skein framework. This idea is obviously present in Costantino and L\^e's work and no originality is claimed in this section. The advantage of the above proof is that it generalizes to arbitrary group $G$. Indeed, the stated skein functor can be generalized to every reducible algebraic group $G$ (the case detailed in this paper is $G=\SL_2$) as $\mathcal{S}_q^{G} : \mathcal{M}\to \Mod_k$ where $k=k_G$. The (easy) case $G=\mathbb{C}^*$ is defined in \cite{KojuQuesneyQNonAb}, the $\SL_3$ case was studied by Higgins in \cite{Higgins_SSkeinSL3}, the $\SL_n$ case was studied by L\^e-Sikora \cite{LeSikora_SSkein_SLN} and the general case will appear in the next forthcoming paper \cite{KojuTannakianSSkein_ToAppear}. Unlike the $\SL_2$ case, finding explicit bases for stated skein algebras in general is quite difficult (see \cite{Higgins_SSkeinSL3} in the $\SL_3$ case), hence the necessity of finding bases-independent proofs in order to work for a general group. The proofs of Theorems \ref{theorem_bigon} ($\mathcal{S}_q^G(\mathbb{B}) \cong \mathcal{O}_qG$) and \ref{theorem_fusion} extend to any group, so does the above proof of Theorem \ref{theorem_skein_transmutation}.  Even though we only discuss the $\SL_2$ case, the whole paper is designed so that all results and proofs extend straightforwardly to every (connected reducible complex) group $G$.
\end{enumerate}
\end{remark}

\begin{notations} In order to simplify the computations, a generic stated tangle $\adjustbox{valign=c}{\includegraphics[width=1.5cm]{TangleHTwistL.eps}}$ in the bigon will simply be denoted by $v\adjustbox{valign=c}{\includegraphics[width=1cm]{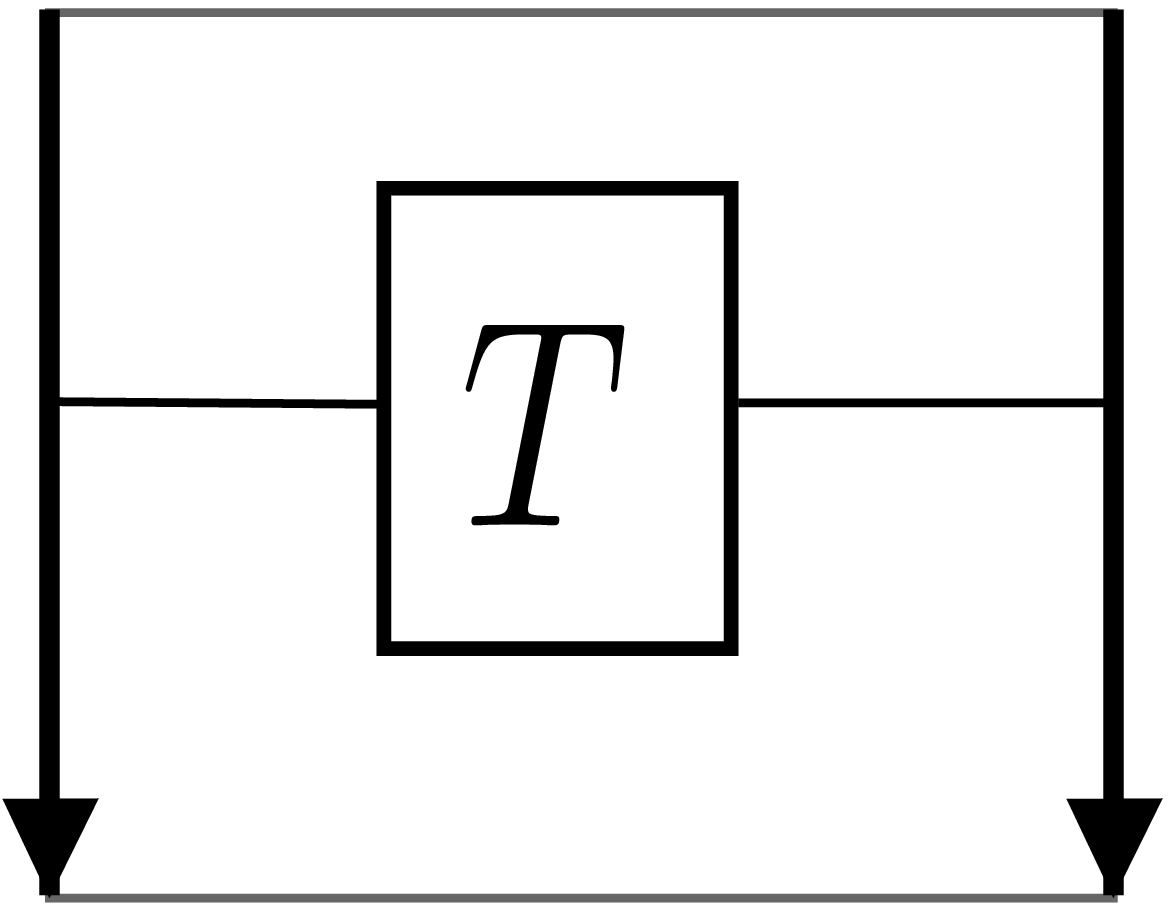}}w$, where it is understood $T$ has arbitrary numbers $n,m\geq 0$ of left and right endpoints and that $v \in V^n$ and $w\in V^m$. 
Its image $ \adjustbox{valign=c}{\includegraphics[width=1.7cm]{TangleMonogon.eps}}$ by $f$ will be denoted by $\CC{\hT_*(v)}{w}$, where for $v=(v_1, \ldots, v_n) \in V^n$ we wrote $\hT_*(v)=(\hT_*(v_1), \ldots, \hT_*(v_n))$.
\end{notations}

\begin{proof}[Proof of Theorem \ref{theorem_skein_transmutation}]
Let $\underline{\mu}, \underline{S}$ be the transmuted products and antipode of $B \mathcal{S}_q(\mathbb{B})$ as in Definition \ref{def_transmutation}.
\vspace{2mm}
 \par $f$ \textit{is an isomorphism of $\mathcal{S}_q(\mathbb{B})$-comodule:}
\\ Let us prove that $(f\otimes \id)\circ \Ad \circ f^{-1}$ coincides with the (right) comodule structure of $\mathcal{S}_q(\mathbf{H}_1)$. Let $x= \adjustbox{valign=c}{\includegraphics[width=0.5cm]{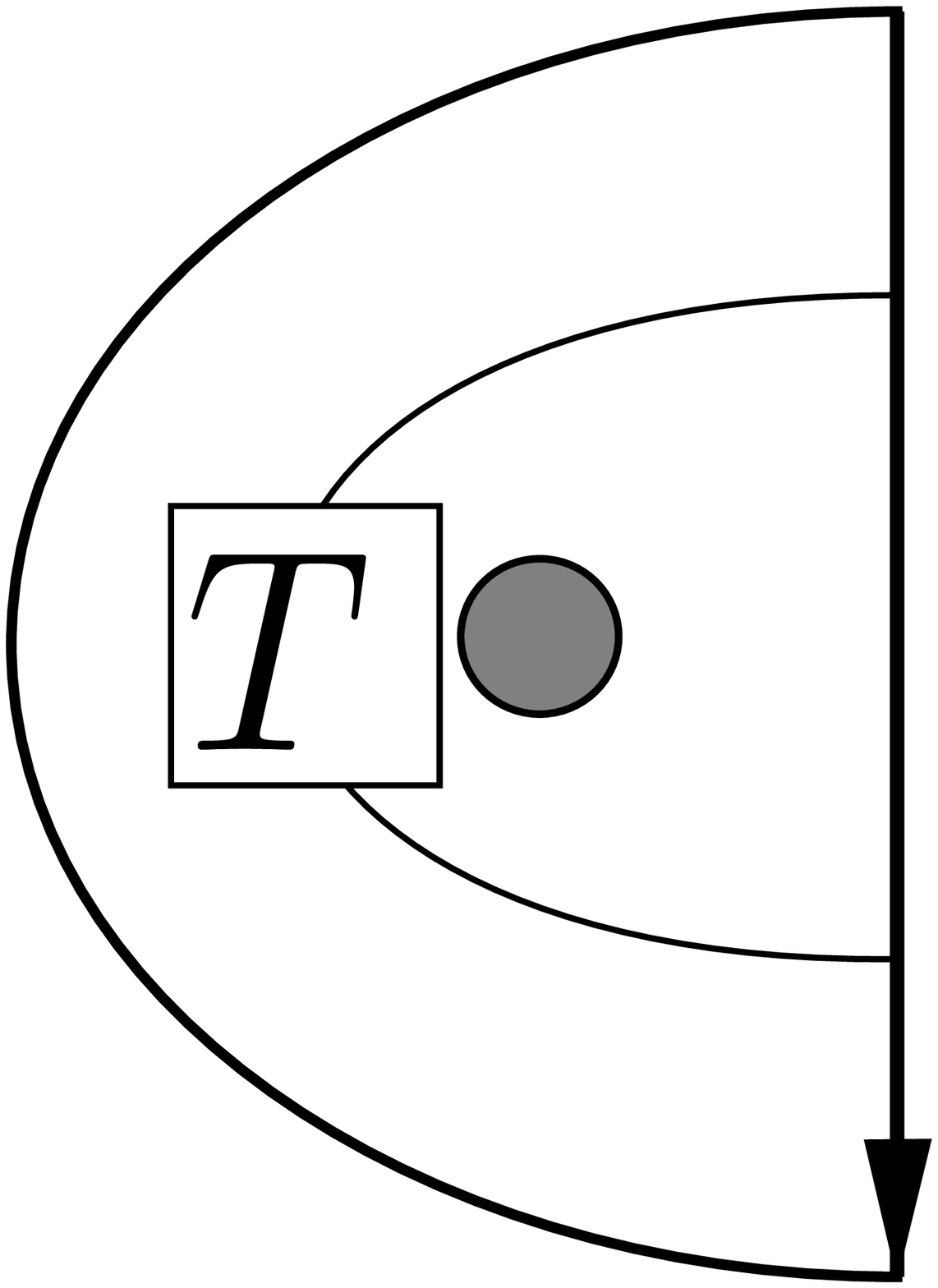}}{}^{b}_{a}$ so that $f^{-1}(x)=\hT_{*}^{-1}(a)\adjustbox{valign=c}{\includegraphics[width=1cm]{TangleSimple.eps}}b$. By definition,  $\Ad= (\id\otimes \mu)(\tau \otimes \id) (S\otimes \id \otimes \id) \Delta^{(2)}$. Now
$$ \Delta^{(2)} (f^{-1}(x)) = \sum_{ij} \hT_*^{-1}(a) \adjustbox{valign=c}{\includegraphics[width=1cm]{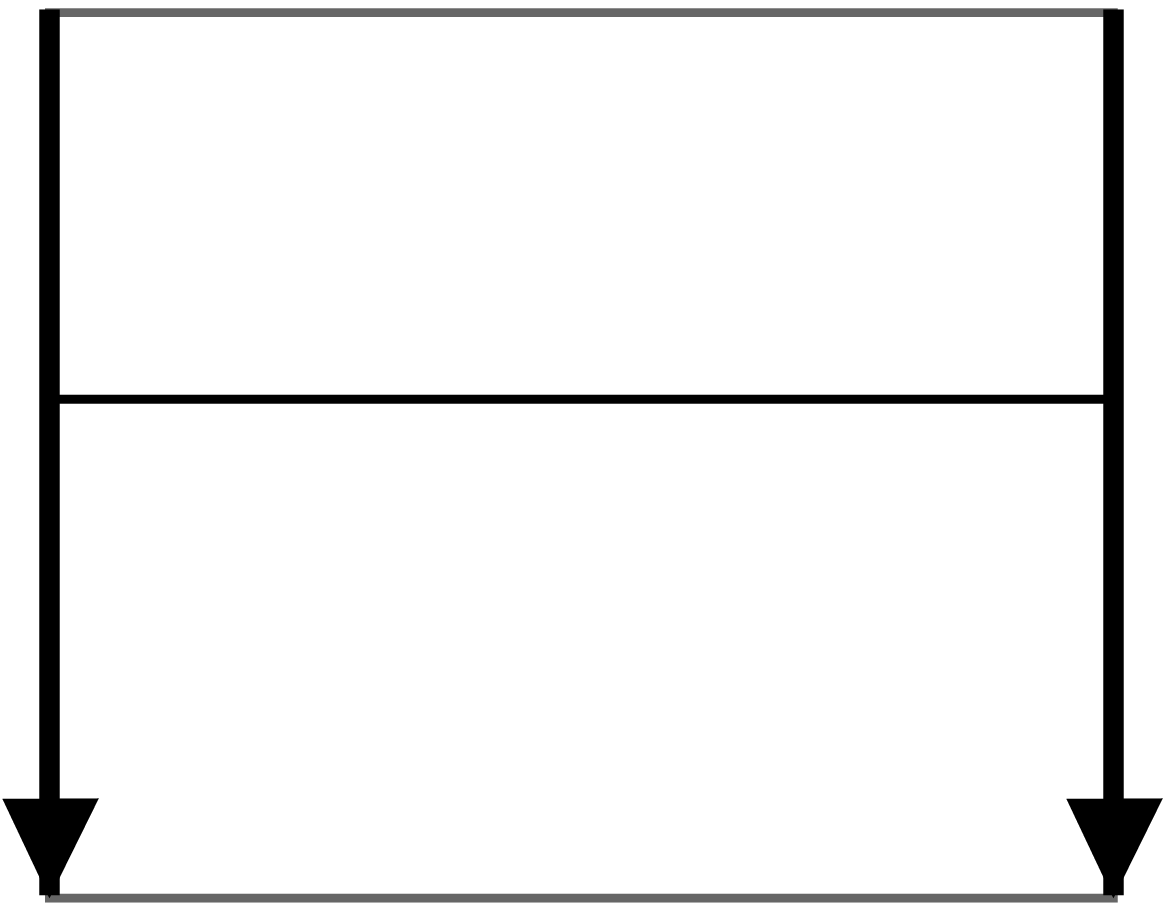}} i \otimes i \adjustbox{valign=c}{\includegraphics[width=1cm]{TangleSimple.eps}} j \otimes j \adjustbox{valign=c}{\includegraphics[width=1cm]{TangleVSimple.eps}} b.$$
So 
$$(S\otimes \id \otimes \id)  \Delta^{(2)} (f^{-1}(x)) = \sum_{ij}\hT^{-1}(i) \adjustbox{valign=c}{\includegraphics[width=1cm]{TangleVSimple.eps}} a \otimes i \adjustbox{valign=c}{\includegraphics[width=1cm]{TangleSimple.eps}} j \otimes j \adjustbox{valign=c}{\includegraphics[width=1cm]{TangleVSimple.eps}} b, $$
from which we get
$$\Ad (f^{-1}(x)) = \sum_{ij}  i \adjustbox{valign=c}{\includegraphics[width=1cm]{TangleSimple.eps}} j \otimes {}_{\hT^{-1}(i)}^{j}\adjustbox{valign=c}{\includegraphics[width=1cm]{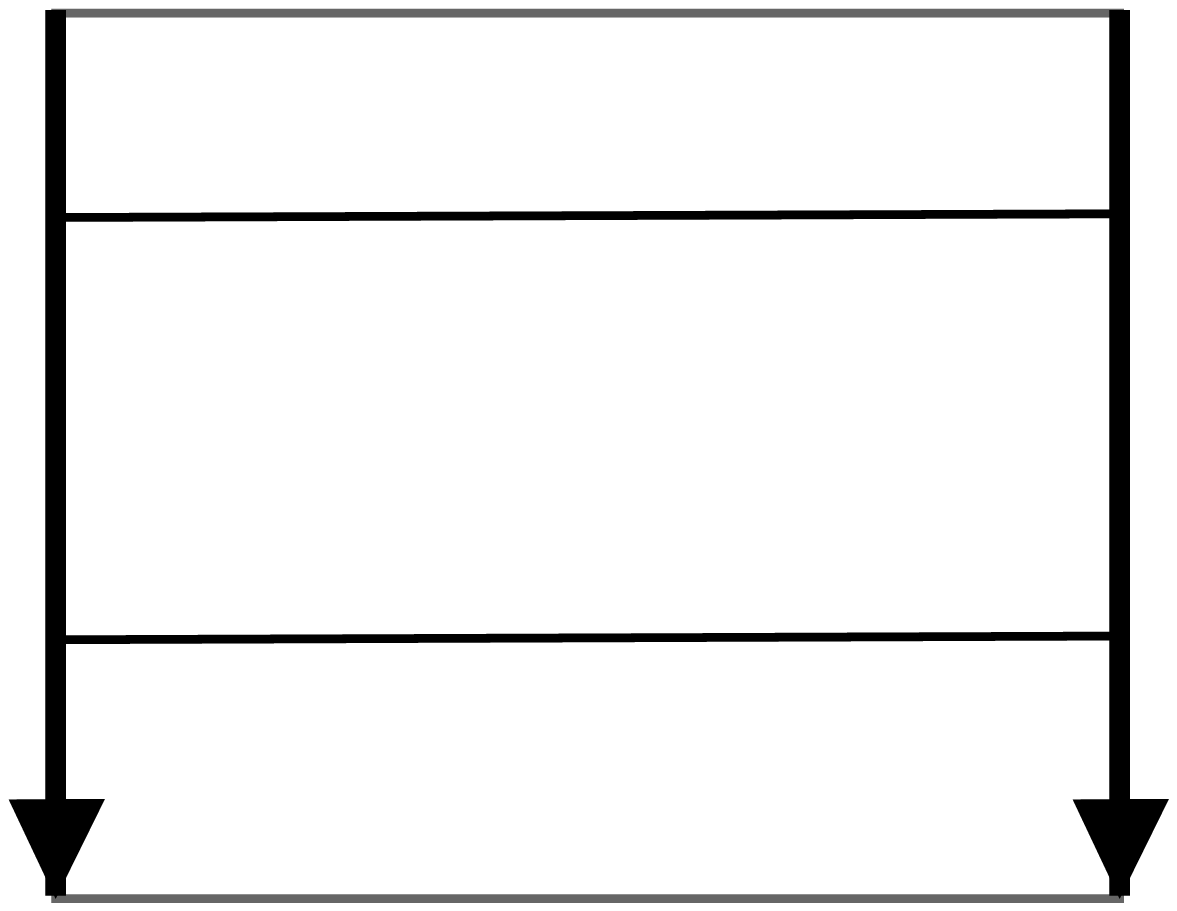}}{}_a^b, $$
so
$$ (f\otimes \id)\circ \Ad \circ f^{-1} (x) = \sum_{ij}\adjustbox{valign=c}{\includegraphics[width=0.8cm]{TangleMonogonSimpleT.eps}}{}_{\hT_*(i)}^j   \otimes {}_{\hT^{-1}(i)}^{j}\adjustbox{valign=c}{\includegraphics[width=1cm]{TangleDouble.eps}}{}_a^b= \sum_{ij}\adjustbox{valign=c}{\includegraphics[width=0.8cm]{TangleMonogonSimpleT.eps}}{}_{i}^j   \otimes {}_{i}^{j}\adjustbox{valign=c}{\includegraphics[width=1cm]{TangleDouble.eps}}{}_a^b.$$
Thus we recover the comodule structure of $\mathcal{S}_q(\mathbf{H}_1)$.

 \par $f$ \textit{is an isomorphism of algebras:}
\\ Let us prove that $f\circ \underline{\mu} \circ (f^{-1} \otimes f^{-1})$ coincides with the product of $\mathcal{S}_q(\mathbf{H}_1)$. Let $x= \adjustbox{valign=c}{\includegraphics[width=0.5cm]{TangleMonogonSimpleT.eps}}{}^{b}_{a}$ and $y=\adjustbox{valign=c}{\includegraphics[width=0.5cm]{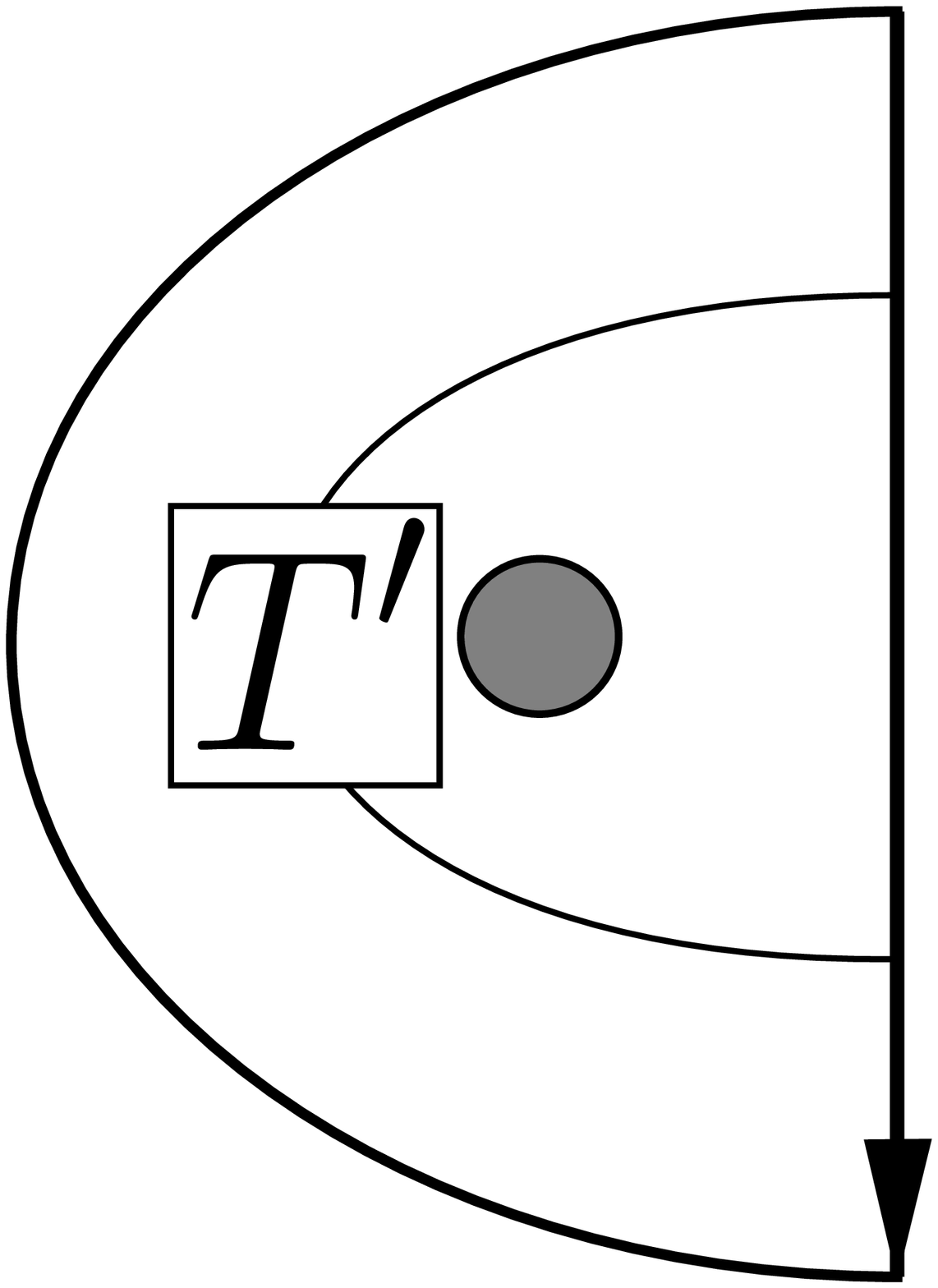}}{}^{d}_{c}$ so that $f^{-1}(x)\otimes f^{-1}(y)= \hT_{*}^{-1}(a)\adjustbox{valign=c}{\includegraphics[width=1cm]{TangleSimple.eps}}b \otimes \hT_{*}^{-1}(c)\adjustbox{valign=c}{\includegraphics[width=1cm]{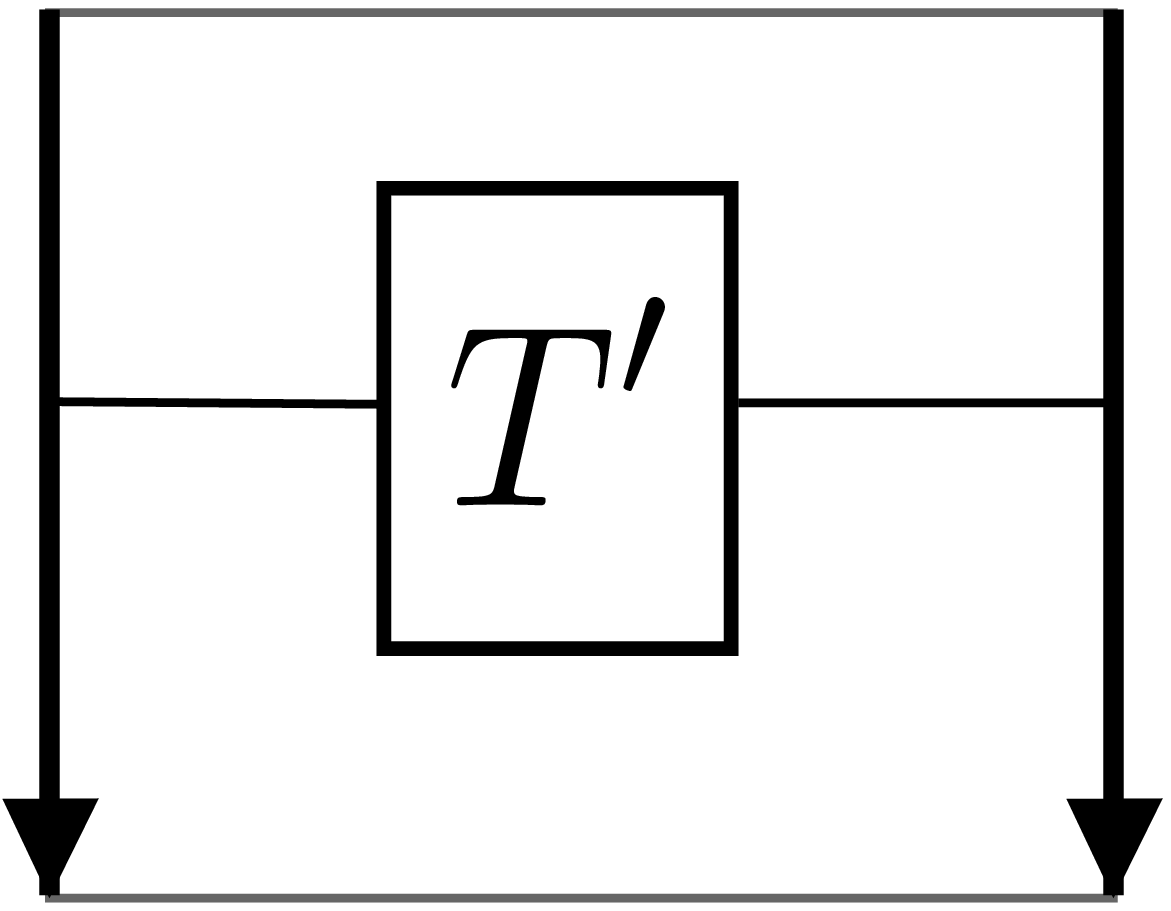}}d$. Equation \eqref{eq_transmuted_prod} in Definition \ref{def_transmutation} reads:
$$ \underline{\mu} \left( f^{-1}(x)\otimes f^{-1}(y) \right) = \sum_{i,j,k}  \adjustbox{valign=c}{\includegraphics[width=1.5cm]{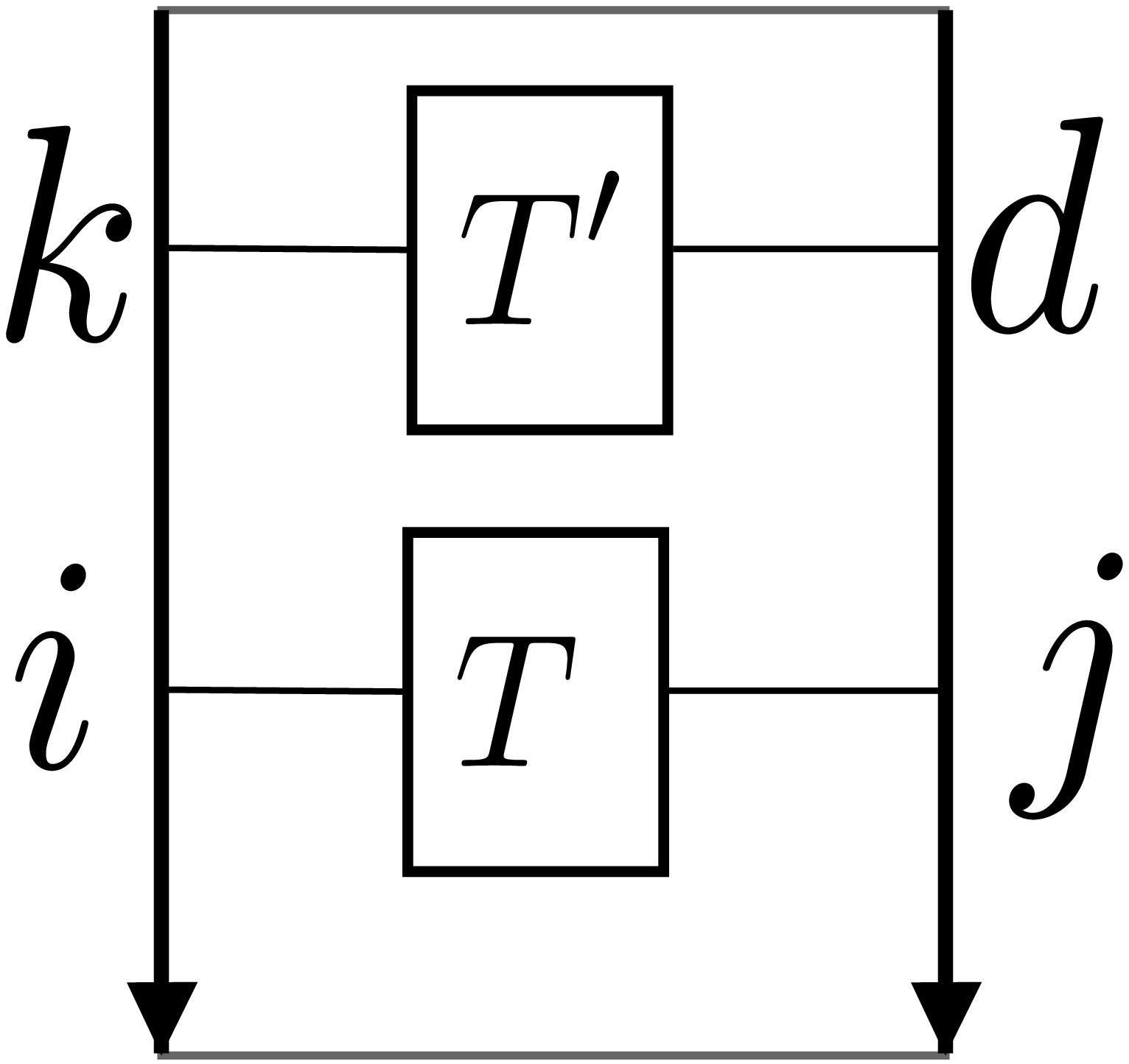}}\quad \epsilon \left(  \adjustbox{valign=c}{\includegraphics[width=3cm]{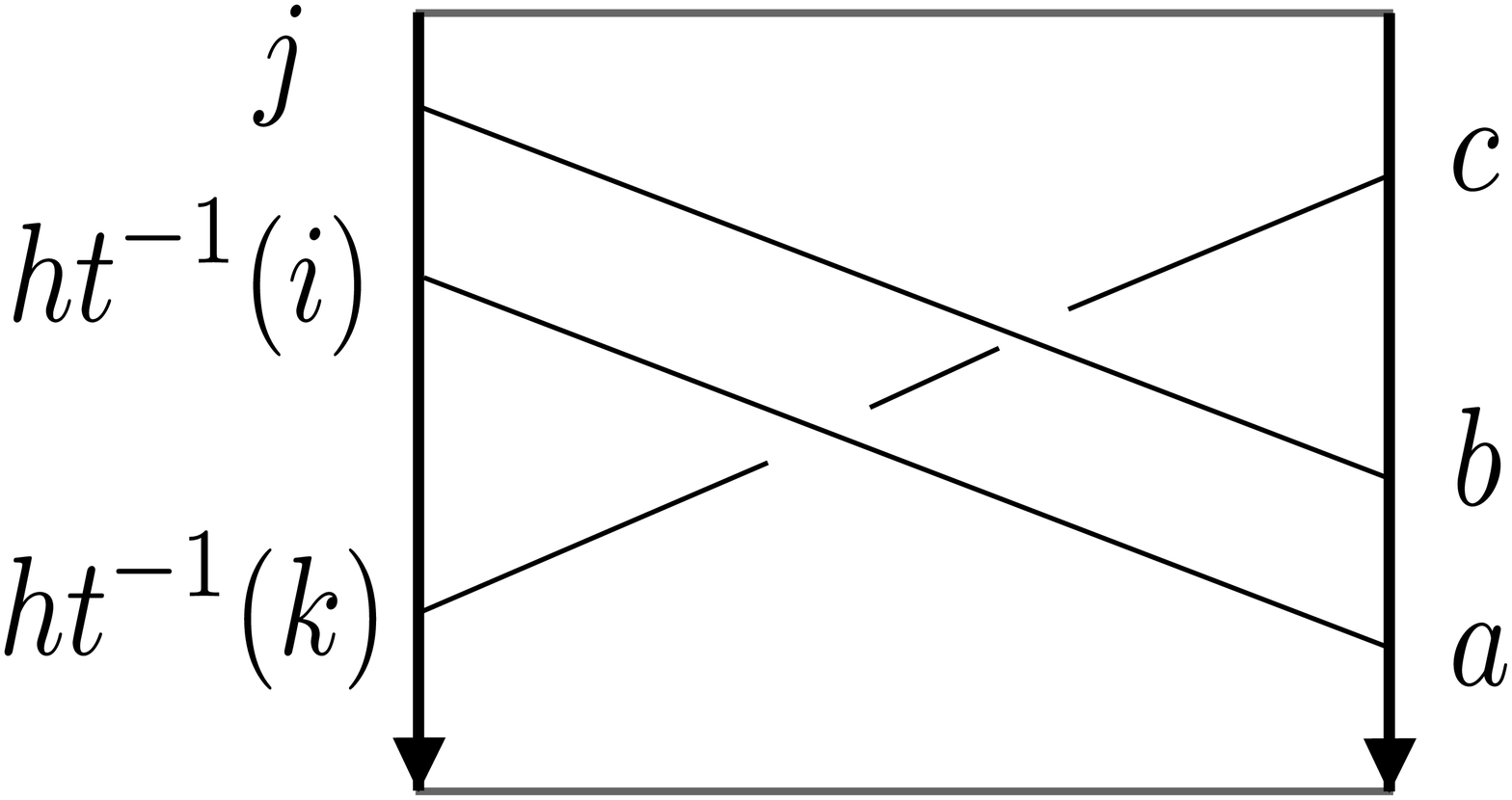}} \right).$$
Therefore: 
$$ f\circ \underline{\mu} \circ (f^{-1} \otimes f^{-1})(x\otimes y) = \sum_{i,j,k}  \adjustbox{valign=c}{\includegraphics[width=1.5cm]{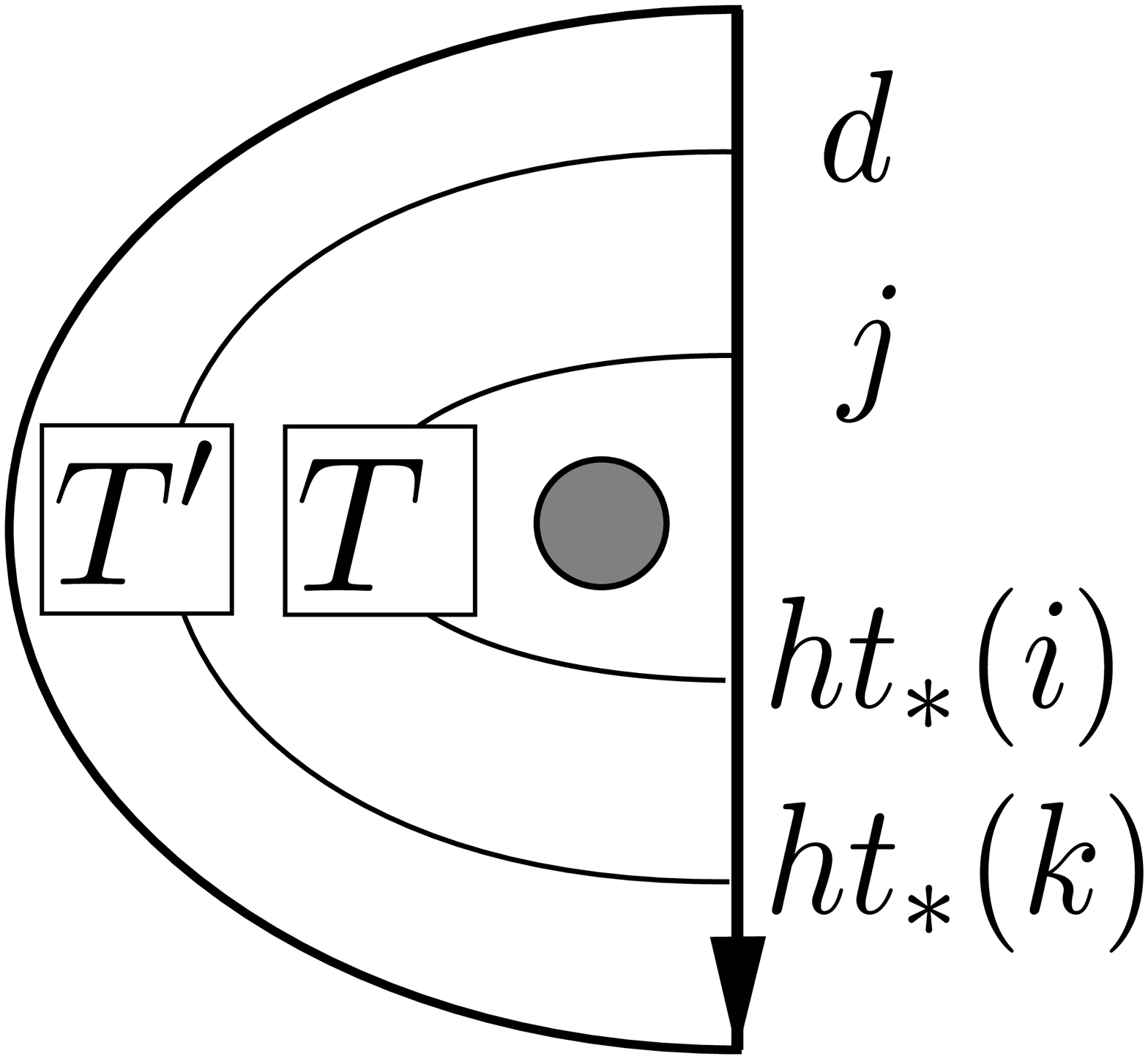}}\quad \epsilon \left(  \adjustbox{valign=c}{\includegraphics[width=3cm]{TangleTransmutedProd2.eps}} \right) = \adjustbox{valign=c}{\includegraphics[width=2cm]{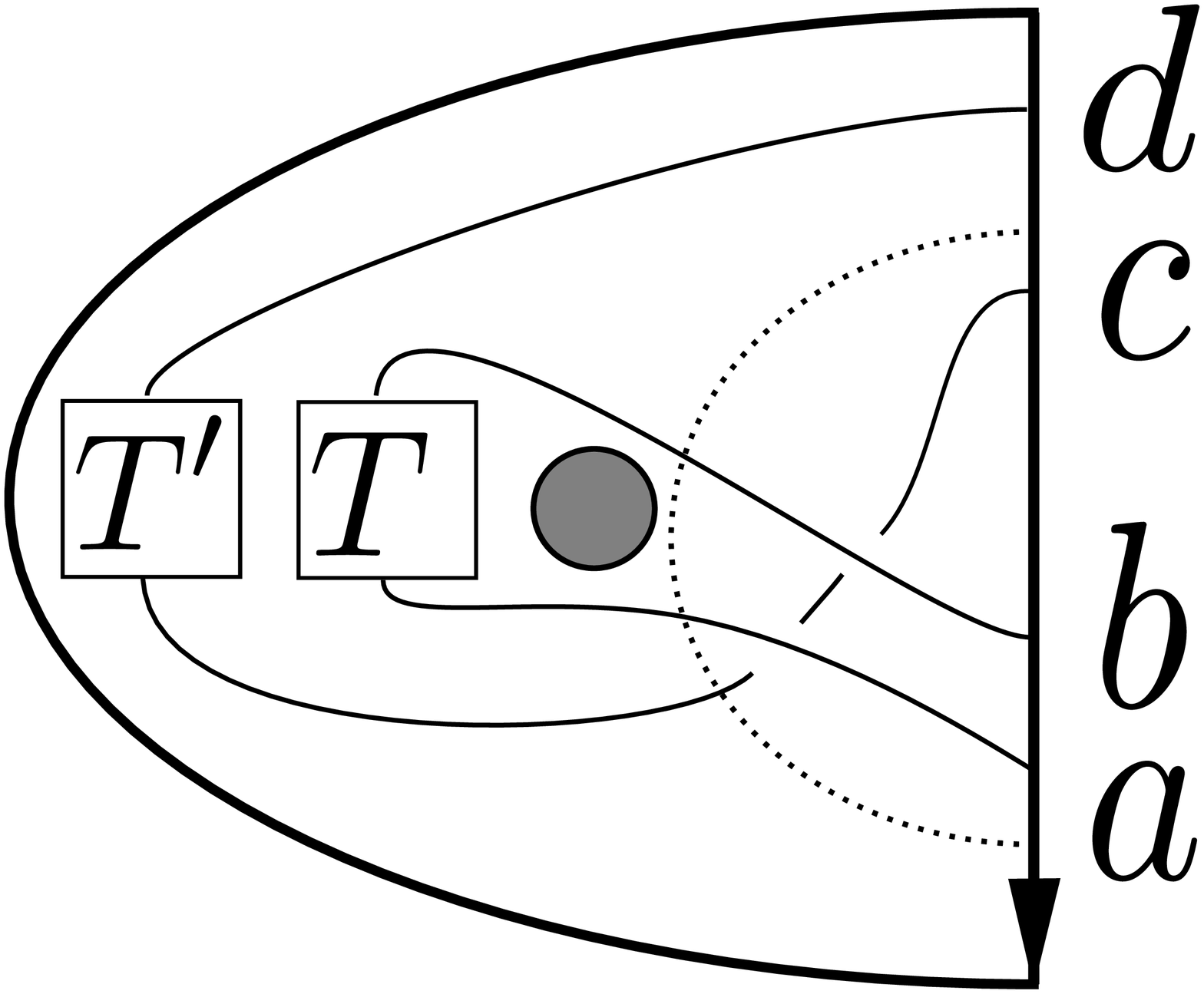}}, $$
so we recover the product in the skein algebra $\mathcal{S}_q(\mathbf{H}_1)$.
\vspace{2mm}
 \par $f$ \textit{is an isomorphism of co-algebras:}
\\ Let us prove that $(f\wedge f) \circ \Delta \circ f^{-1}$ coincides with the image by $\mathcal{S}_q$ of the coproduct in Figure \ref{fig_BTHopfAlg}. Let $x= \adjustbox{valign=c}{\includegraphics[width=0.5cm]{TangleMonogonSimpleT.eps}}{}^{b}_{a}$ so that $f^{-1}(x)=\hT_{*}^{-1}(a)\adjustbox{valign=c}{\includegraphics[width=1cm]{TangleSimple.eps}}b$. Then
$$\Delta ( f^{-1}(x)) = \sum_i \hT_{*}^{-1}(a)\adjustbox{valign=c}{\includegraphics[width=1cm]{TangleSimple.eps}}i \otimes i \adjustbox{valign=c}{\includegraphics[width=1cm]{TangleVSimple.eps}}b, $$
so 
$$ (f\wedge f ) (\Delta(f^{-1}(x))) = \sum_i  \adjustbox{valign=c}{\includegraphics[width=2cm]{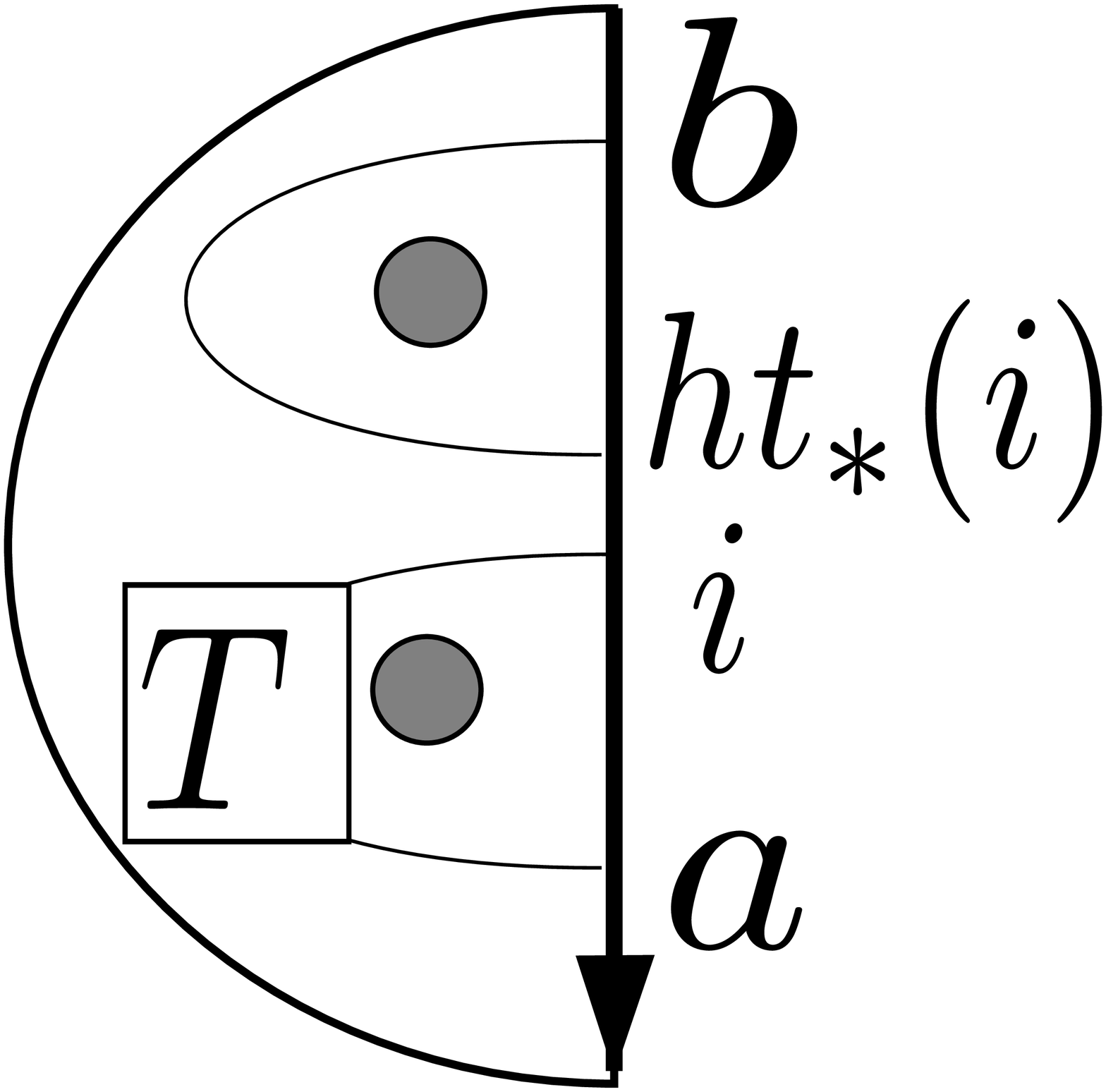}}= \adjustbox{valign=c}{\includegraphics[width=1.5cm]{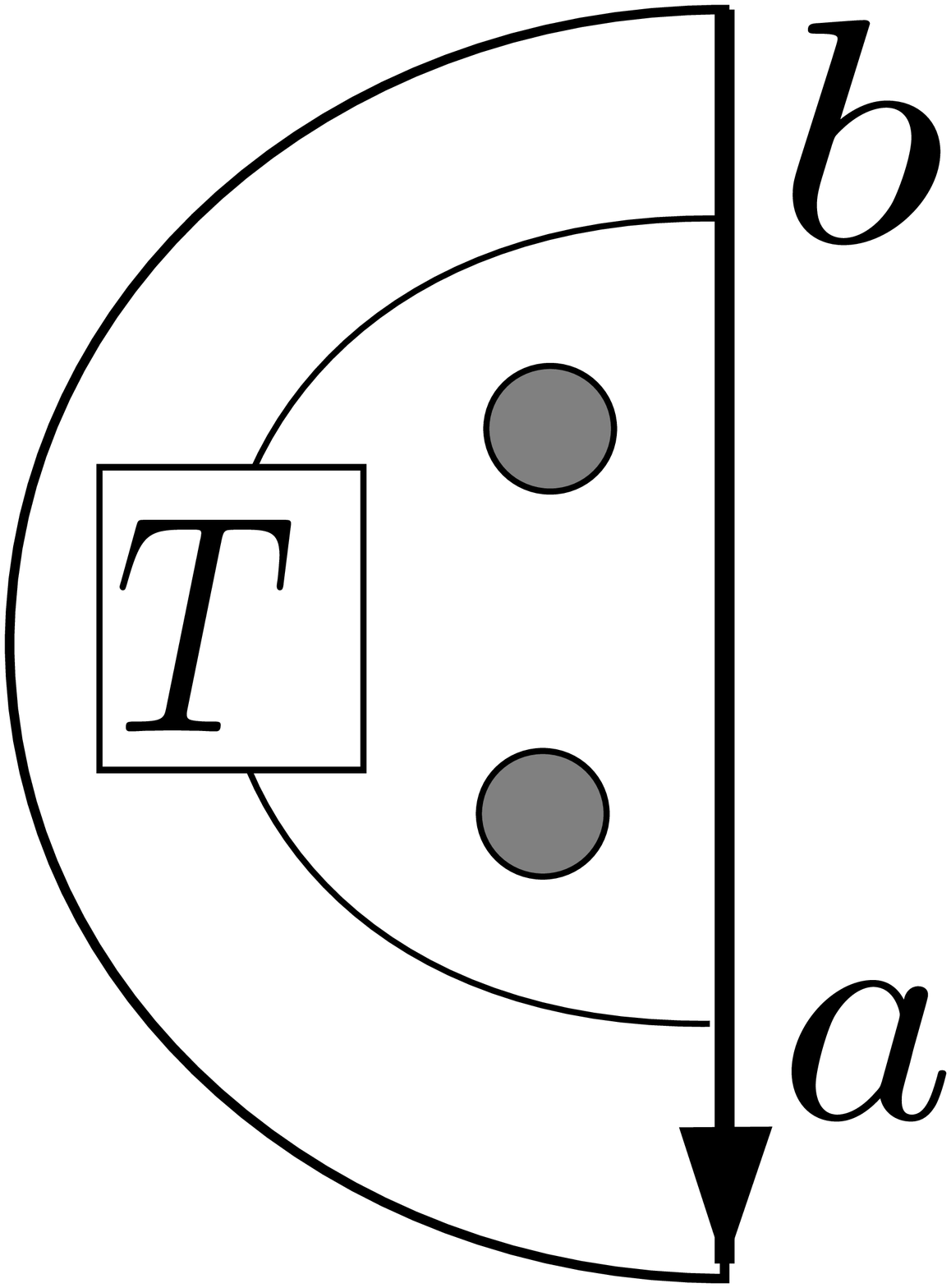}}.$$
\vspace{2mm}
 \par  \textit{Skein interpretation of the antipode:}
\\ At this stage we have proved that $f$ is an isomorphism of bialgebras so it preserves the antipode. However, it is instructive to verify directly that indeed the rather abstract Equation \eqref{eq_transmuted_antipod} in Definition \ref{def_transmutation} has a very natural topological interpretation given in Figure \ref{fig_BTHopfAlg}. Let $x= \adjustbox{valign=c}{\includegraphics[width=0.5cm]{TangleMonogonSimpleT.eps}}{}^{b}_{a}$ as before. 
Then Equation \eqref{eq_transmuted_antipod} reads: 
$$ f \circ \overline{S}( f^{-1}(x)) = \sum_{ijk}  \adjustbox{valign=c}{\includegraphics[width=0.5cm]{TangleMonogonSimpleT.eps}}{}^{\hT_*(i)}_{j} \epsilon \left(  \adjustbox{valign=c}{\includegraphics[width=3cm]{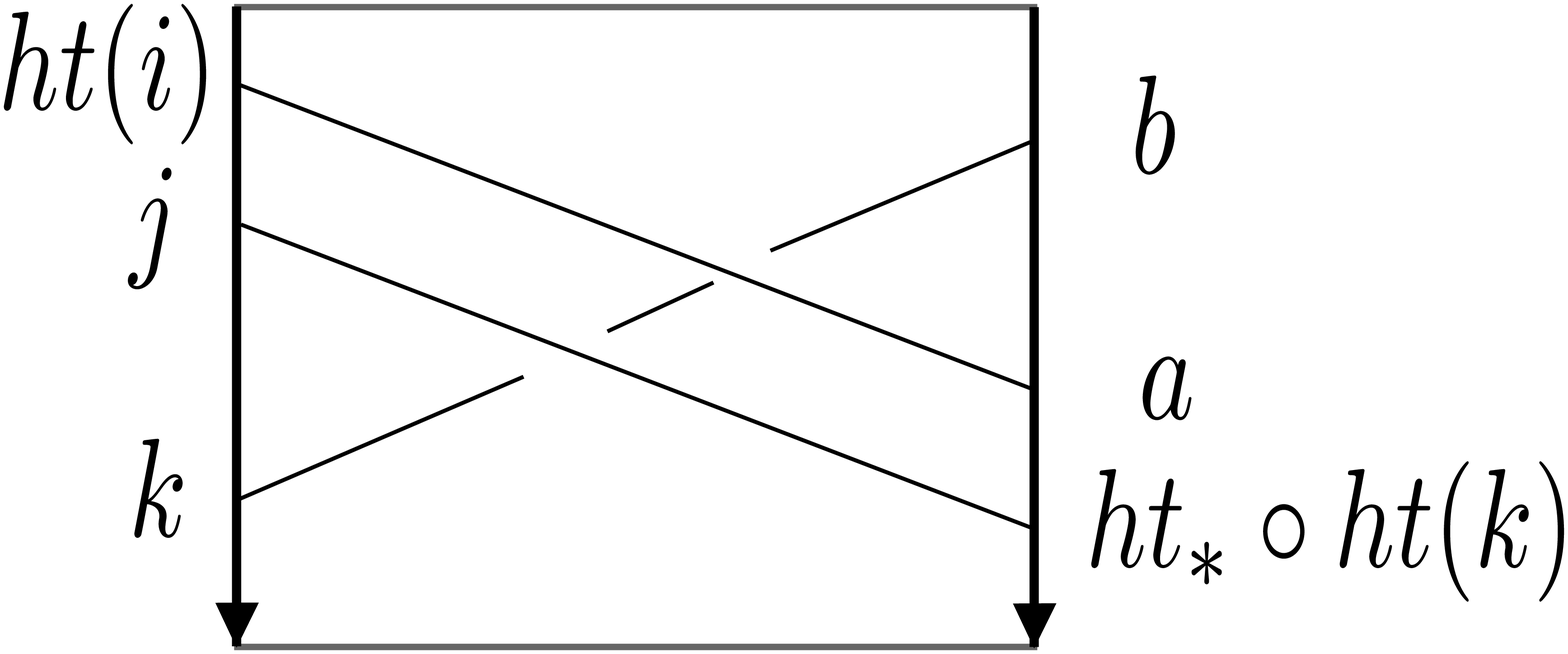}} \right).
$$
Now, for a basis vector $v_k$, then  $$\hT_* \circ \hT (v_k) = \sum_{l m } \left< v_l^* , \hT_* (v_m) \right> \left< v_m^*, \hT(v_k)\right> v_m = \sum_{l m } \epsilon \left( \adjustbox{valign=c}{\includegraphics[width=1cm]{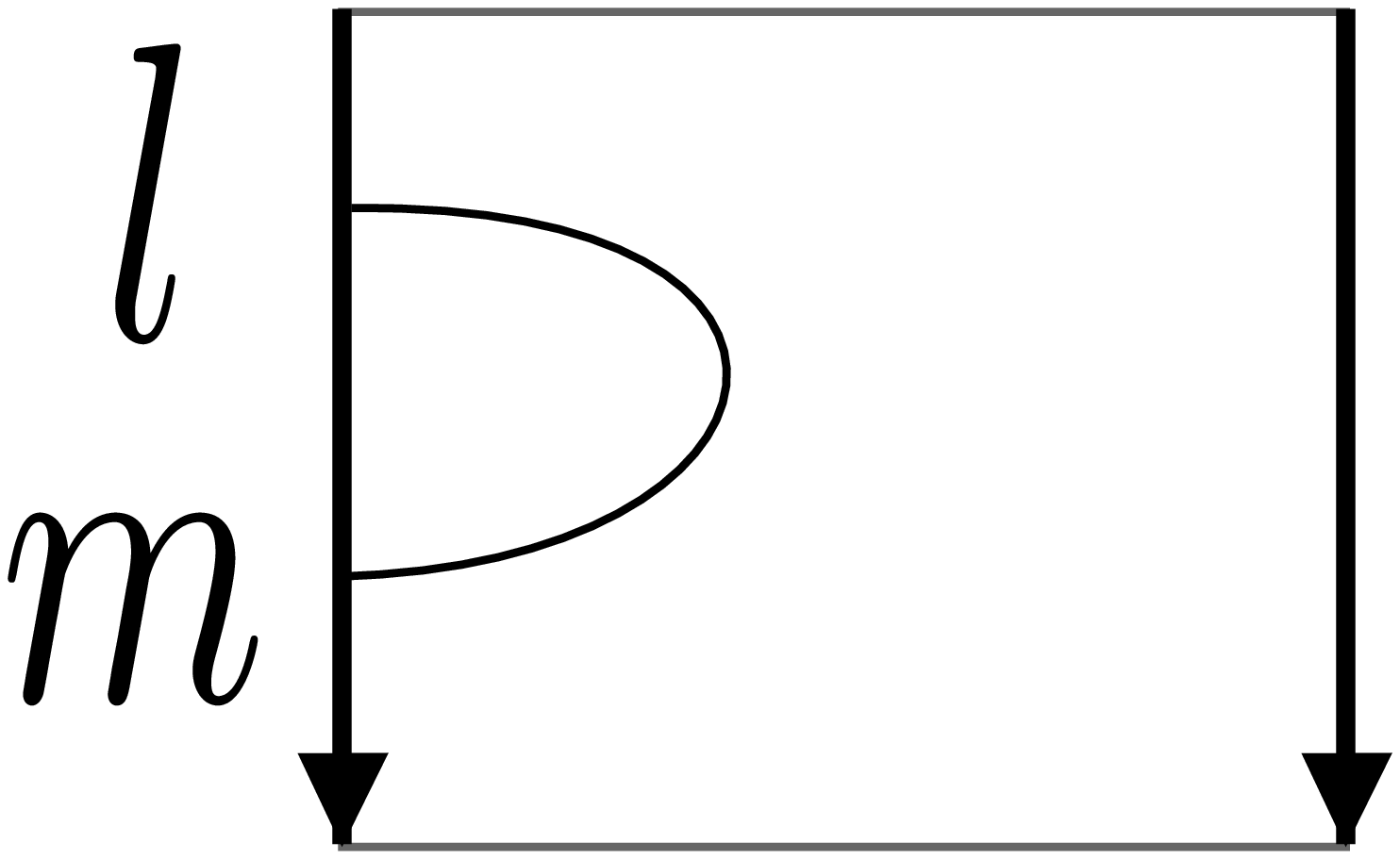}} \right) \epsilon \left( \adjustbox{valign=c}{\includegraphics[width=1cm]{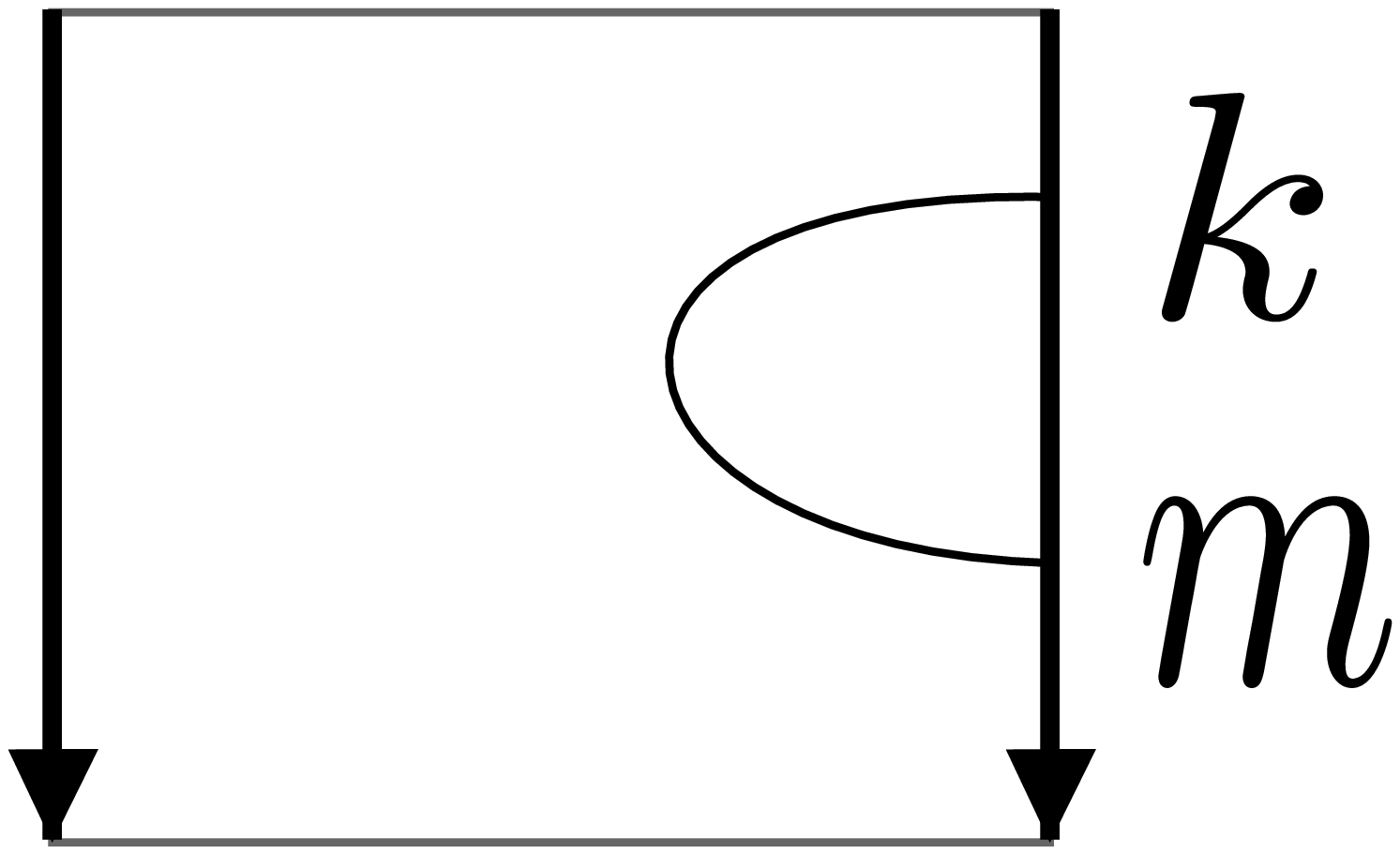}} \right) v_m.$$
Therefore: 
$$ f \circ \overline{S}( f^{-1}(x)) = \sum_{ijklm}  \adjustbox{valign=c}{\includegraphics[width=0.5cm]{TangleMonogonSimpleT.eps}}{}^{i}_{j} \epsilon \left( \adjustbox{valign=c}{\includegraphics[width=1cm]{TangleTransmutedAntipod4.eps}} \right) \epsilon \left( \adjustbox{valign=c}{\includegraphics[width=1cm]{TangleTransmutedAntipod5.eps}} \right) \epsilon \left(  \adjustbox{valign=c}{\includegraphics[width=2cm]{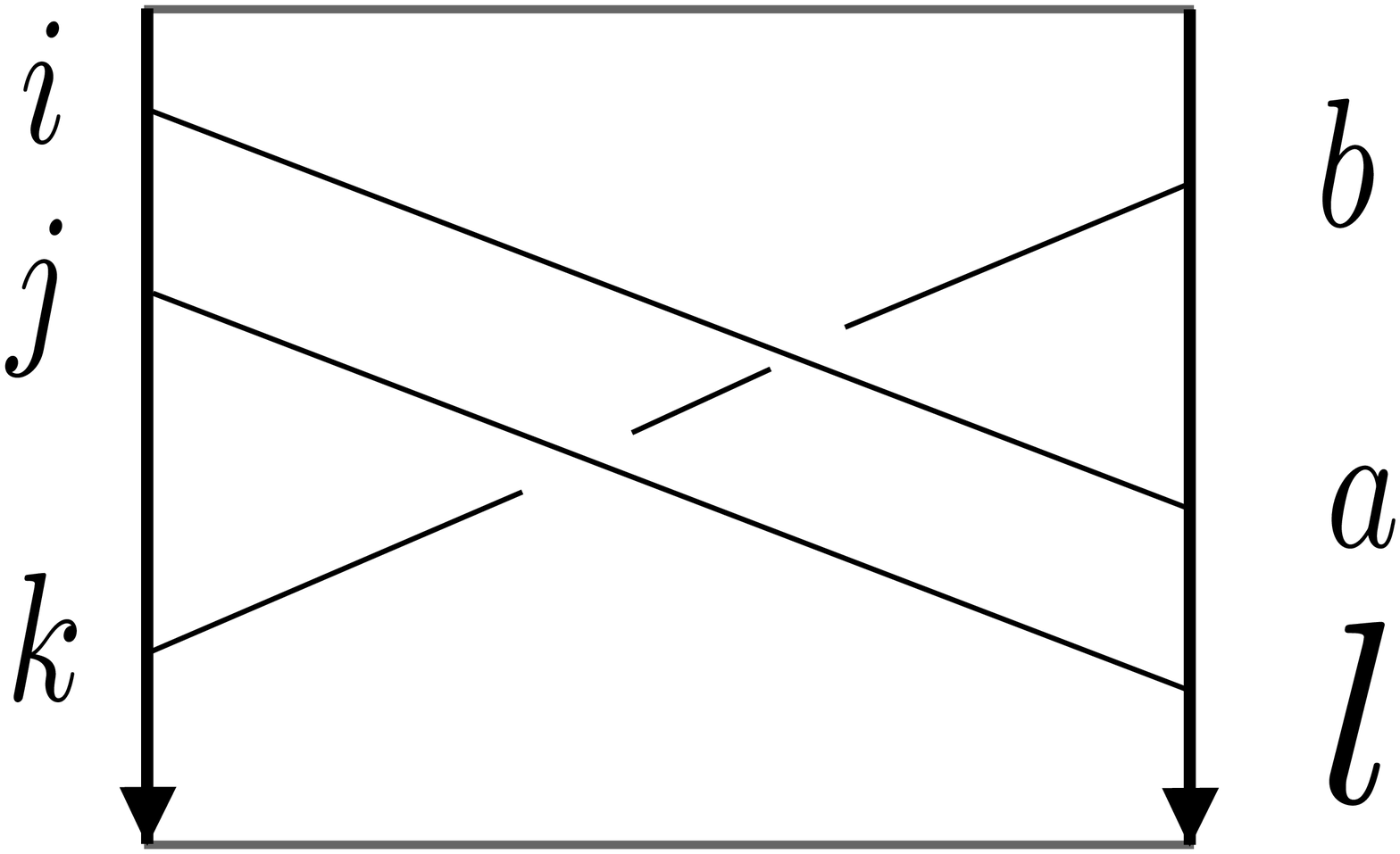}} \right) = \adjustbox{valign=c}{\includegraphics[width=3cm]{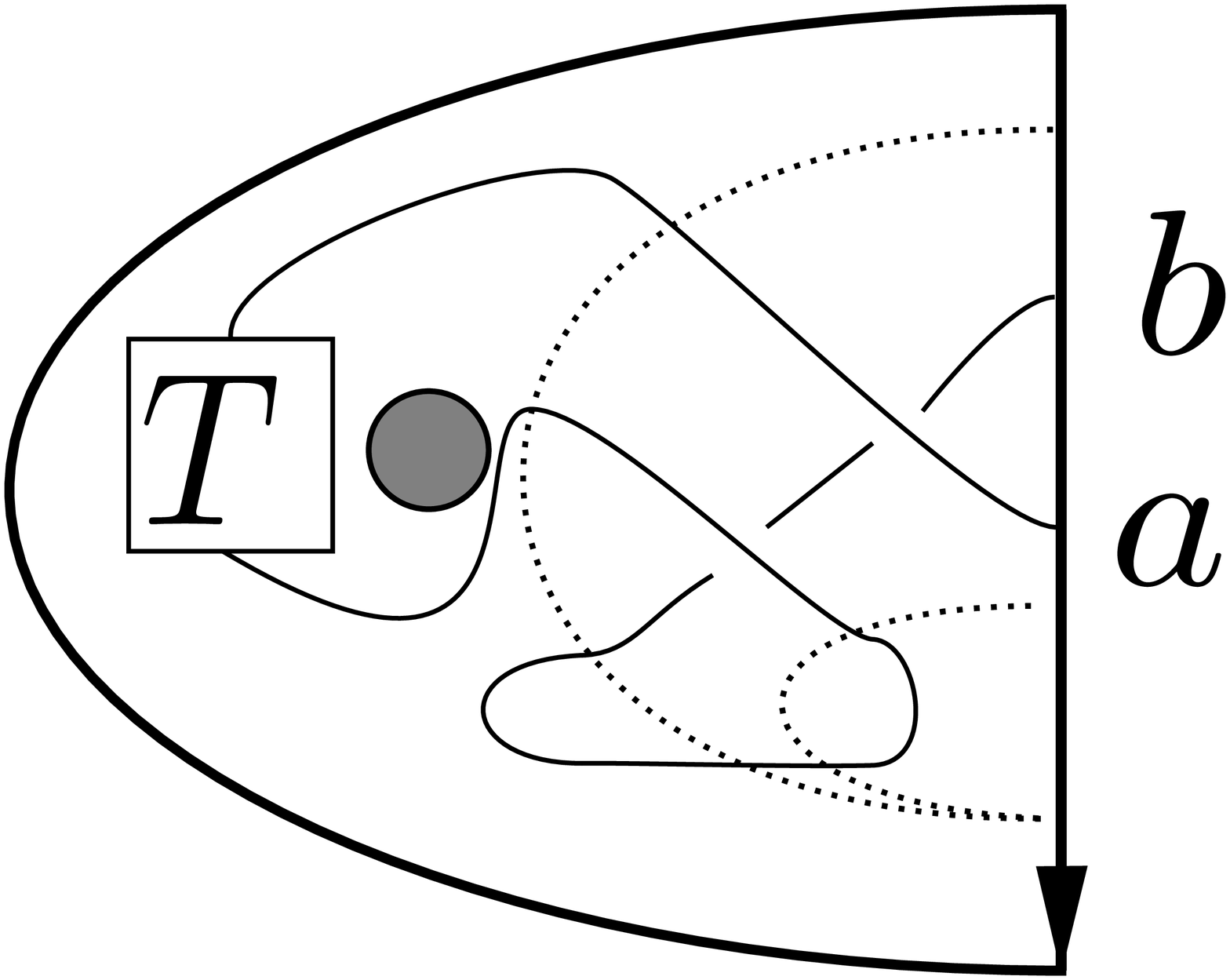}} .$$
\vspace{2mm}
 \par $f$ \textit{preserves the braiding:}
\\ Let $\Psi_0 : B\mathcal{S}_q(\mathbb{B})^{\otimes 2} \to B\mathcal{S}_q(\mathbb{B})^{\otimes 2} $ be the braiding from item $(5)$ of Definition \ref{def_HT} associated to $B\mathcal{S}_q(\mathbb{B})$ seen as a comodule over $\mathcal{S}_q(\mathbb{B})$ via $\Ad$. Said differently, 
$$ \Psi_0= (\tau \otimes r) (\id \otimes \tau \otimes \id)(\Ad \otimes \Ad).$$
Let us prove that $(f\wedge f)\circ \Psi_0 \circ (f^{-1}\otimes f^{-1})$ coincides with the image by $\mathcal{S}_q$ of the braiding $\Psi$ of Figure \ref{fig_BTHopfAlg}. Let $x= \adjustbox{valign=c}{\includegraphics[width=0.5cm]{TangleMonogonSimpleT.eps}}{}^{b}_{a}$ and $y=\adjustbox{valign=c}{\includegraphics[width=0.5cm]{TangleMonogonSimpleT2.eps}}{}^{d}_{c}$. Using the above expression for $\Psi_0$, we compute:
$$ (f\wedge f)\circ \Psi_0 \circ (f^{-1}\otimes f^{-1})(x\otimes y) = \sum_{ijkl} \adjustbox{valign=c}{\includegraphics[width=1.7cm]{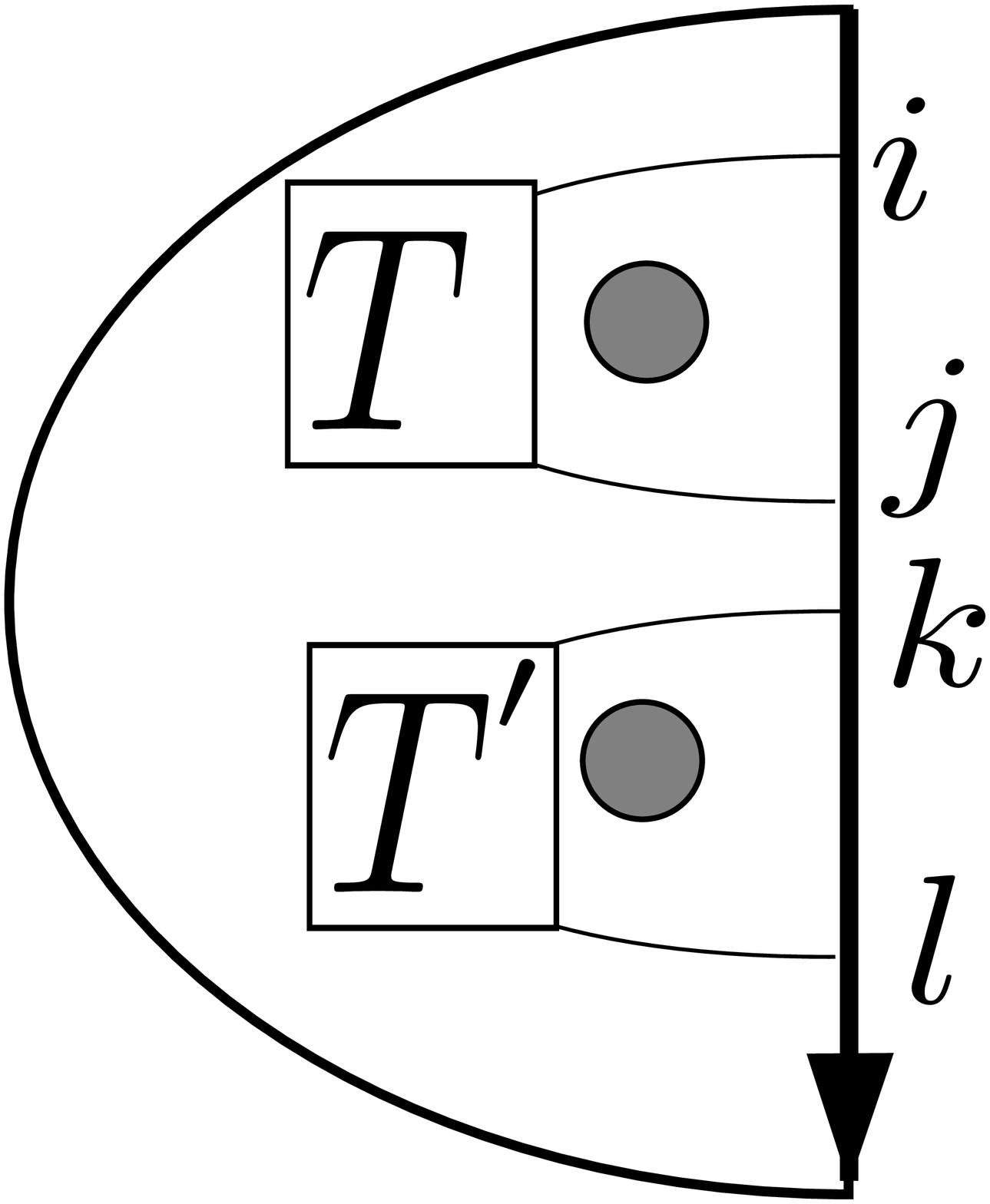}}  \epsilon \left( \adjustbox{valign=c}{\includegraphics[width=2cm]{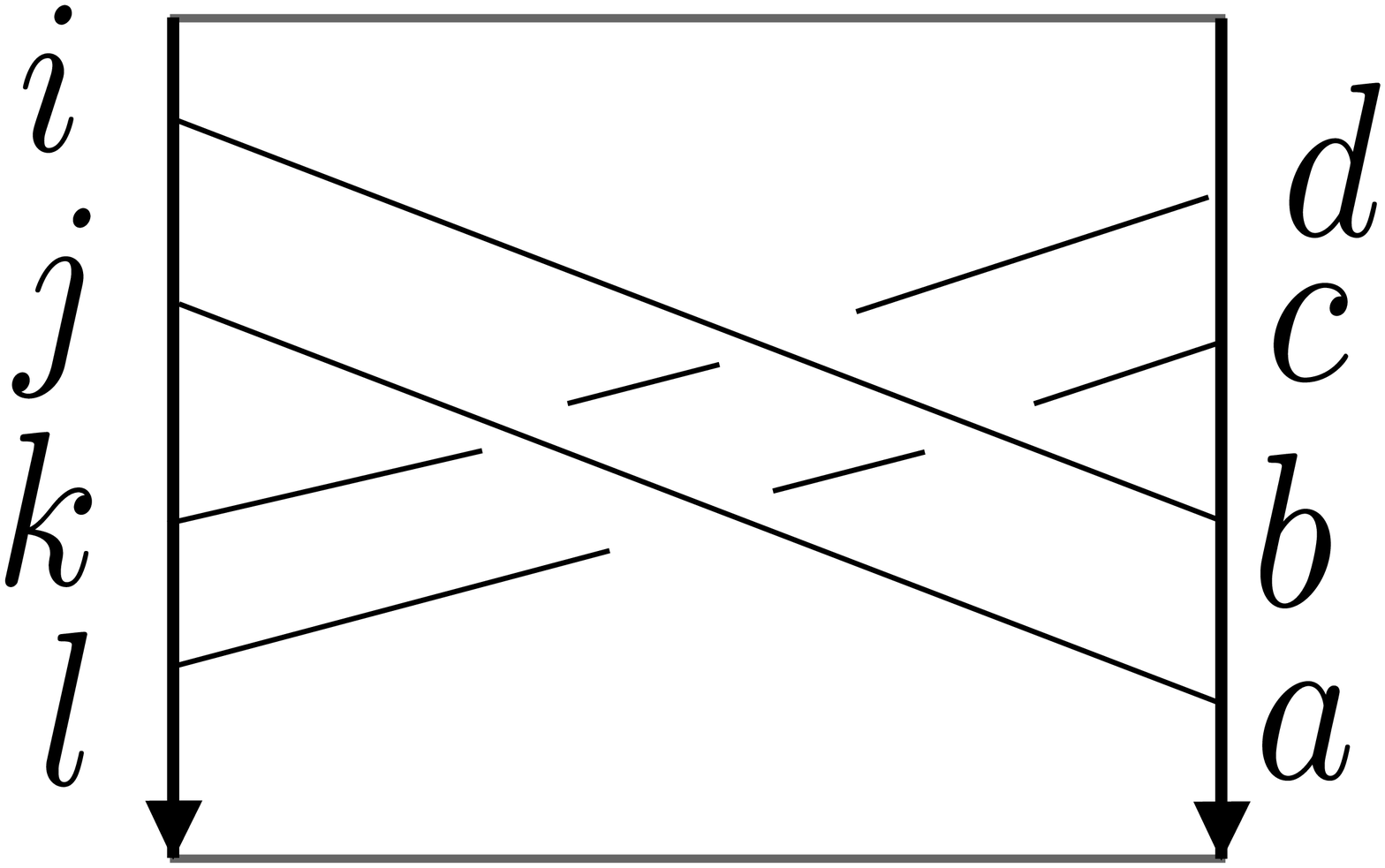}} \right) = \adjustbox{valign=c}{\includegraphics[width=2.5cm]{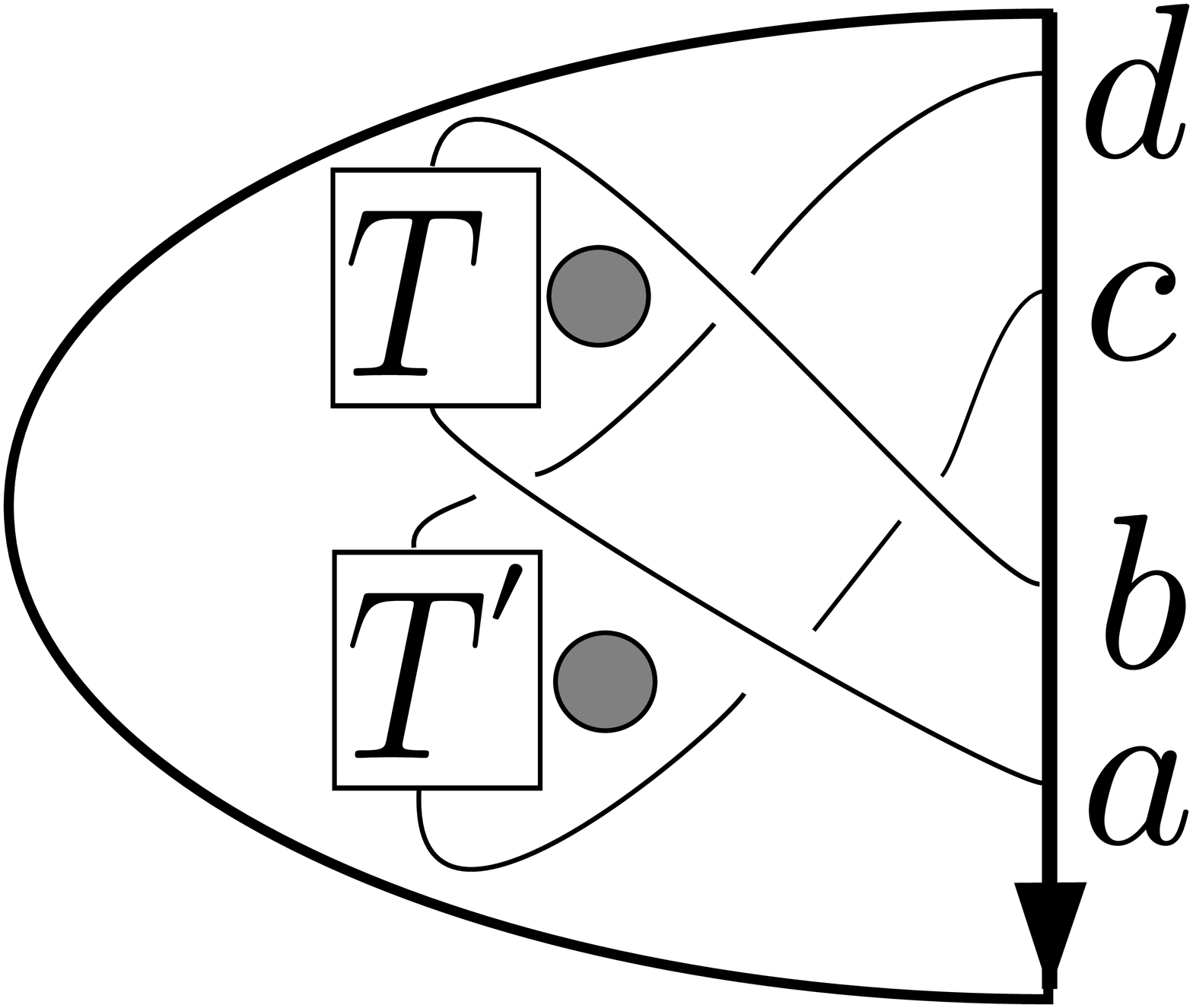}} .$$
Now, the fact that $f$ intertwines the comodule maps $\Ad^B$ and $\ad$ comes from that it intertwines the product, coproduct, antipode and braiding and from the fact that both comodules map are defined with the same formula, namely that
$$ \Ad^B=  \adjustbox{valign=c}{\includegraphics[width=1.5cm]{AdB.eps}}\quad \mbox{ and } \quad \ad = \adjustbox{valign=c}{\includegraphics[width=1.5cm]{Aad.eps}}.$$

\end{proof}

\subsection{Relating $\mathcal{S}_q$ to $\Rep_{q}^{\SL_2}$}

\begin{theorem}\label{theorem_SkeinQRep} 
Let $\restriction{\mathcal{S}_q}{\BT} : \BT \to \overline{\mathcal{C}_q^{\SL_2}}$ be the restriction of the functor $\mathcal{S}_q: \mathcal{M}^{(1)}_{\con}  \to\overline{\mathcal{C}_q^{\SL_2}}$ and $i: \BT\hookrightarrow \mathcal{M}^{(1)}_{\con}$ the inclusion. Then $\mathcal{S}_q = Lan_i\restriction{\mathcal{S}_q}{\BT}$, i.e. $\mathcal{S}_q$ is the left Kan extension lying in the diagram:
 $$
 \begin{tikzcd}
 {} & \mathcal{M}^{(1)}_{\con}\ar[rd,  "\mathcal{S}_q"] & {} \\
 \BT  \ar[ru, "i"] \ar[rr, "\restriction{\mathcal{S}_q}{\BT} "] &{}&\overline{\mathcal{C}_q^{\SL_2}}
 \end{tikzcd}
 $$ 

\end{theorem}

\begin{proof}
Let $\mathbf{M} \in \mathcal{M}^{(1)}_{\con}$ and consider the left Kan extension of Section \ref{sec_Cat}:
$$ L(\mathbf{M})=Lan_i\restriction{\mathcal{S}_q}{\BT} (\mathbf{M}):= \int^{n\geq 0} k[\Hom_{\mathcal{M}}( \mathbf{H}_n, \mathbf{M})]\otimes \mathcal{S}_q(\mathbf{H}_n).$$
Define a linear morphism $F^{(n)}: k[\Hom_{\mathcal{M}}( \mathbf{H}_n, \mathbf{M})]\otimes \mathcal{S}_q(\mathbf{H}_n) \to \mathcal{S}_q(\mathbf{M})$ by $F^{(n)} ( f \otimes x):= f_{*}(x)$. For $g : \mathbf{H}_n \to \mathbf{H}_m$ a morphism in $\BT$, $f : \mathbf{H}_m \to \mathbf{M}$ an embedding and $x\in \mathcal{S}_q(\mathbf{H}_n)$ then $F^{(n)}(f\circ g \otimes x) = f_{*}\circ g_{*} (x) = F^{(m)}(f \otimes g_* x)$ so the morphisms $F^{(n)}$ induce a linear map $F_{\mathbf{M}} : L(\mathbf{M}) \to \mathcal{S}_q(\mathbf{M})$. Clearly the  morphisms $F_{\mathbf{M}}$ are natural in $\mathbf{M}$ so define a natural morphism $F: L \to \mathcal{S}_q$. Let us define an inverse morphism $G$. Let $(T,s)$ be a stated tangle in $\mathbf{M}$. Let $\mathcal{E}\subset \bigsqcup_{n\geq 0}\Hom_{\mathbf{M}^{(1)}}(\mathbf{H}_n, \mathbf{M})$ be the subset of embeddings $f: \mathbf{H}_n \to \mathbf{M}$ such that $T$ is included in the image of $f$. For $f,g \in \mathcal{E}$, write $f\prec g$ if the image of $f$ is included in the image of $g$. First, since $M$ is connected, the set $\mathcal{E}$ is not empty. Next, the partially ordered set $(\mathcal{E}, \prec)$ is filtrant: for $f,g \in \mathcal{E}$, one can always find $h \in \mathcal{E}$ such that the image of $h$ contains both the images of $f$ and $g$. 
For $f: \mathbf{H}_n \to \mathbf{M}$ an embedding in $\mathcal{E}$, consider the stated tangle $(T_f, s_f)$ in $\mathbf{H}_n$ such that $(T,s)=(f(T_f), s_f\circ f^{-1})$ and define $G_f ([T,s]):= [f \otimes [T_f,s_f] ] \in L(\mathbf{M})$. For $f,  g$ two elements in $\mathcal{E}$  with $f: \mathbf{H}_n \to \mathbf{M}$ and $g: \mathbf{H}_m \to \mathbf{M}$ such that $f\prec g$, then one can find $h: \mathbf{H}_n \to \mathbf{H}_m$ such that $f=g\circ h$. Thus one has $[T_g, s_g]=h_*[T_f,s_f]$ so $G_g([T,s])= [g \otimes h_*[T_f,s_f ]] = [g \circ h , [T_f, s_f]] = G_f([T,s])$. Since $\prec$ is filtrant, this proves that $G_f([T,s])$ is independant of $f$, let us denote it by $G_{\mathbf{M}}(T,s)$ which is clearly invariant by isotopy. Extend $G_{\mathbf{M}}$ linearly to the free $k$ module generated by stated tangles in $\mathbf{M}$. Let us prove that $G_{\mathbf{M}}$ passes to the quotient to a linear morphism 
$G_{\mathbf{M}} : \mathcal{S}_q(\mathbf{M})\to L(\mathbf{M})$, i.e. that it sends the skein relations of Definition \ref{def_skein} to $0$. If $X = \sum_i \alpha_i (T^i, s^i)$ is a linear combination of stated tangles such that $\sum_i \alpha_i [T^i, s^i]=0$ in $\mathcal{S}_q(\mathbf{M})$ (i.e. $X$ is a skein relation), consider an embedding $f:\mathbf{H}_n \to \mathbf{M} $ such that every $T_i$ is included in the image of $f$ and denote by $(T_f^i, s^i_f) \subset H_n$ the stated tangle such that $(T^i, s^i) = (f(T_f^i), s_f^i \circ f^{-1})$ as before. By locality of the skein relations, $X_f: = \sum_i \alpha_i (T_f^i, s_f^i)$ is a skein relation in $\mathbf{H}_n$ so its class in $\mathcal{S}_q(\mathbf{H}_n)$ vanishes. Therefore $G_{\mathbf{M}}( X) = [ f \otimes [X_f]]= 0$ so one has a linear map $G_{\mathbf{M}} : \mathcal{S}_q(\mathbf{M})\to L(\mathbf{M})$.

 Clearly the maps $G_{\mathbf{M}}$ are natural in $\mathbf{M}$ so define a natural morphism $G : \mathcal{S}_q \to L$. It is a straightforward consequence of the definitions of $F$ and $G$ that they are inverse to each other. 

\end{proof}

The following implies Theorem \ref{theorem1}.

\begin{corollary}\label{coro_RepSkein}
The functors $\Rep_q^{\SL_2} : \mathcal{M}^{(1)}_{\con} \to \overline{\mathcal{C}_q^{\SL_2}}$ and $\mathcal{S}_q: \mathcal{M}^{(1)}_{\con} \to \overline{\mathcal{C}_q^{\SL_2}}$ are isomorphic.
\end{corollary}

\subsection{Coinvariant vectors}

Let $\mathbf{M}=(M, \iota_M) \in \mathcal{M}^{(1)}_{\con}$ and consider the unmarked $3$-manifold $M\in \mathcal{M}^{(0)}_{\con}$ without marked disc. The identity $\id_M: M\to M$ is a marked $3$-manifolds embedding $\iota: M\to \mathbf{M}$ so induces a morphism $\iota_*: \mathcal{S}_q(M) \to \mathcal{S}_q(\mathbf{M})$ between the usual skein module of $M$ to the stated skein module of $\mathbf{M}$. Clearly, the elements of the image of $\iota_*$ are coinvariant for the $\mathcal{O}_q[\SL_2]$ coaction. By Theorem \ref{theorem_SkeinQRep}, the submodule $\mathcal{S}_q(\mathbf{M})^{coinv}$ of coinvariant vectors is isomorphic to Habiro's quantum character variety. So in order to relate skein module with quantum character variety, we need to understand whether $\iota_*: \mathcal{S}_q(M) \to \mathcal{S}_q(\mathbf{M})^{coinv}$ is an isomorphism or not. 
Recall that $k_{\SL_2}= \mathbb{Z}[q^{\pm 1/4}]$ and consider the field of fractions $K_{\SL_2}:=\mathbb{Q}(q^{1/4})$ and the $K_{\SL_2}$ vector space $\mathcal{S}_q^{rat}(\mathbf{M}):= \mathcal{S}_q(\mathbf{M})\otimes_{k_{\SL_2}}K_{\SL_2}$. It is a comodule structure over $\mathcal{O}_q\SL_2^{rat}:=\mathcal{O}_q\SL_2\otimes_{k_{\SL_2}}K_{\SL_2}$.

The goal of this subsection is to prove the following:

\begin{theorem}\label{theorem_surjectivity}
The morphism $\iota_*: \mathcal{S}_q(M) \to \mathcal{S}_q(\mathbf{M})^{coinv}$ is surjective. Moreover, after tensoring by $K_{\SL_2}$ it becomes an isomorphism  $\iota_*: \mathcal{S}^{rat}_q(M) \to \mathcal{S}^{rat}_q(\mathbf{M})^{coinv}$.
\end{theorem}

In particular the kernel of $\iota_*: \mathcal{S}_q(M) \to \mathcal{S}_q(\mathbf{M})^{coinv}$  lies in the torsion submodule of $\mathcal{S}_q(M)$.

\begin{corollary} The Kauffman-bracket skein module $\mathcal{S}^{rat}_q(M)$ is isomorphic to the quantum character variety $\Char_q^{\SL_2, rat}(\mathbf{M}):=(\Rep_q^{\SL_2, rat}(\mathbf{M}))^{coinv}$ while working over the field $K_{\SL_2}$.
\end{corollary}

The reason we need to work over the field $K_{\SL_2}$ instead of the ring $k_{\SL_2}$ is that the category $\overline{\mathcal{C}^{\SL_2, rat}_q}$ is semi-simple, so its elements are flat.

\begin{definition} 
\begin{enumerate}
\item The \textit{quantum plane} $\mathcal{O}_q[\mathbb{A}^2]$ is the quotient of the (non commutative) $k$-algebra freely generated by two generators $x$ and $y$ by the relation $xy=q^{-1}yx$.
\item Define a left comodule map $\Delta^L: \mathcal{O}_q[\mathbb{A}^2] \to \mathcal{O}_q[\SL_2] \otimes \mathcal{O}_q[\mathbb{A}^2]$ by the formula:
$$ \Delta^L \begin{pmatrix} x \\ y \end{pmatrix} := \begin{pmatrix} a & b \\ c & d \end{pmatrix} \otimes \begin{pmatrix} x \\ y \end{pmatrix} = \begin{pmatrix} a\otimes x + b\otimes y \\  c\otimes x + d\otimes y \end{pmatrix}.$$
\item The quantum plane is graded $\mathcal{O}_q[\mathbb{A}^2]= \oplus_{n\geq 0} \mathcal{O}_q[\mathbb{A}^2]^{(n)}$ where $\mathcal{O}_q[\mathbb{A}^2]^{(n)}= \Span \left( x^i y^j, i+j=n \right)$ and clearly $\mathcal{O}_q[\mathbb{A}^2]^{(n)} \cdot \mathcal{O}_q[\mathbb{A}^2]^{(m)}\subset \mathcal{O}_q[\mathbb{A}^2]^{(n+m)}$. The comodule map $\Delta^L$ restricts to comodule maps $\Delta^L_{(n)}:  \mathcal{O}_q[\mathbb{A}^2]^{(n)} \to \mathcal{O}_q[\SL_2] \otimes \mathcal{O}_q[\mathbb{A}^2]^{(n)}$.
\end{enumerate}
\end{definition}

\begin{lemma}\label{lemma_coinv1}
The set of coinvariant vectors $\mathcal{O}_q[\mathbb{A}^2]^{\coinv}$ is the set $\mathcal{O}_q[\mathbb{A}^2]^{(0)} \cong k$ of scalars.
\end{lemma}

\begin{proof}
The algebras $\mathcal{O}_q[\mathbb{A}^2]$ and $ \mathcal{O}_q[\SL_2]$ are quadratic and quadratic inhomogeneous respectively and they both satisfy the Koszul condition. An easy application of the Diamond lemma for PBW bases implies that they have bases $\mathcal{B}_1:= \{x^iy^j, i,j\geq 0\}$ and $\mathcal{B}_2:= \{a^{n_a}b^{n_b} d^{n_d}, n_a, n_b, n_c \geq 0\} \cup \{ a^{n_a}c^{n_c}d^{n_d}, n_a, n_c, n_d\geq 0\}$ respectively (see \cite{Kassel} for details). Therefore  $\mathcal{O}_q[\mathbb{A}^2]\otimes  \mathcal{O}_q[\SL_2]$ has basis $\mathcal{B}:=\mathcal{B}_1 \otimes \mathcal{B}_2$. For $\mathbf{n}=(n_x, n_y, n_a, n_b, n_c, n_d)\in \mathbb{N}^{6}$ (here $\mathbb{N}=\mathbb{Z}^{\geq 0}$), let $z^{\mathbf{n}}:= x^{n_x} y^{n_y} \otimes a^{n_a}b^{n_b} c^{n_c}d^{n_d} \in \mathcal{O}_q[\mathbb{A}^2]\otimes  \mathcal{O}_q[\SL_2]$. Let $\mathcal{N} \subset \mathbb{N}^6$ be the subset of $6$-tuples such that $n_bn_c =0$ so that $\mathcal{B}=\{ z^{\mathbf{n}}, \mathbf{n}\in \mathcal{N} \}$. Equip $\mathcal{N}$ with the lexicographic order $\prec$. For $z=\sum_{\mathbf{n}\in \mathcal{N}} \alpha_{\mathbf{n}} z^{\mathbf{n}} \in  \mathcal{O}_q[\mathbb{A}^2]\otimes  \mathcal{O}_q[\SL_2]$, let $\mathbf{n}_0$ be the biggest index (for $\prec$) such that $\alpha_{\mathbf{n}_0}\neq 0$ and define the \textit{leading term} of $z$ to be 
$$ \lt (z) := \alpha_{\mathbf{n}_0} z^{\mathbf{n_0}}.$$
Let $X:= \sum_{i,j\geq 0} x_{i,j} x^iy^j \in \mathcal{O}_q[\mathbb{A}^2]$ be a coinvariant vector and let $(i_0,j_0) \in \mathbb{N}^2$ be the biggest index (for the lexicographic order) such that $x_{i_0, j_0}\neq 0$. By definition, $X$ is coinvariant means that 
$$ \sum_{i,j\geq 0} x_{i,j} 1 \otimes x^i y^j  = \sum_{i,j \geq 0} x_{i,j} (a\otimes x + b\otimes y)^i (c\otimes x + d \otimes y)^j.$$
Taking the leading terms of each side of this equality, and using the $q$-binomial formula (\cite[Proposition $IV.2.2$]{Kassel}) we obtain the equality:
$$ x_{i_0, j_0} 1\otimes x^{i_0}y^{j_0} = x_{i_0, j_0} a^{i_0}c^{j_0} \otimes x^{i_0}y^{j_0}.$$
Therefore $(i_0,j_0)=(0,0)$ and $X$ is scalar.
\end{proof}

\begin{definition}\label{def_skein_qplane}
\begin{enumerate}
\item Recall that $\mathbb{B}$ is a ball with two boundary discs $a_L$ and $a_R$. For $T$ a tangle in $\mathbb{B}$ a \textit{left state} is a map $s^L : T\cap a_L \to V$. 
 The \textit{skein quantum plane} $\mathcal{S}_q[\mathbb{A}^2]$ is the quotient of the $k$-module freely generated by isotopy classes of left stated tangles $(T,s^L)$ in $\mathbb{B}$ by the ideal generated by the skein relations of Definition \ref{def_skein} applied either in the interior of $\mathbb{B}$ or in a ball intersecting $a_L$ (so there are no skein relations along $a_R$). It has an algebra with product given by stacking tangles like in stated skein algebras.
 \item Define a left comodule map $\Delta^L : \mathcal{S}_q(\mathbb{A}^2) \to \mathcal{S}_q(\mathbb{B}) \otimes \mathcal{S}_q(\mathbb{A}^2)$ as follows. For $(T,s^L)$ a left stated tangle choose an arbitrary right stated tangle $s^R$ so that $[T,(s^L, s^R)]$ is an element of $\mathcal{S}_q(\mathbb{B})$. Split $T$ along a disc as $T=T^L\cup T^R$ such that, by definition of the splitting morphism $\Delta=\theta_{a_R\# a'_L}$, one has $\Delta([T,(s^L, s^R)])= \sum_s [T_L, (s^L, s)] \otimes [T_R, (s, s^R)]$. Then define $\Delta^L ([T, s^L]):= \sum_s [T_L, (s^L, s)] \otimes [T_R, s]$. This formula clearly does not depend on the choice of $s^R$ and the fact that $\Delta$ does not depend on the choice of the splitting $T= T_L\cup T_R$ implies that $\Delta^L$ does not depend on this choice either.
 The left comodule $\Delta^L$ is illustrated in Figure \ref{fig_coproduct_qplane}. 
 \item Define a filtration $\mathcal{S}_q[\mathbb{A}^2]=\cup_{n\geq 0} \mathcal{F}^{(n)}$, where $\mathcal{F}^{(n)}:= \Span( [T,s^L], |\partial T\cap a_L| \leq n)$. One has $\mathcal{F}^{(n)}\cdot \mathcal{F}^{(m)}\subset \mathcal{F}^{(n+m)}$ and 
 the comodule map $\Delta^L$ restricts to comodule maps $\Delta^L_{(n)}: \mathcal{F}^{(n)} \to \mathcal{S}_q(\mathbb{B}) \otimes \mathcal{F}^{(n)}$.
  Set $\Gr_0(\mathcal{S}_q[\mathbb{A}^2])= \mathcal{F}_0$ and for $n\geq 1$, write  $\Gr_n (\mathcal{S}_q[\mathbb{A}^2]):=  \quotient{ \mathcal{F}_n}{\mathcal{F}_{n-1}}$. The graded algebra $\Gr (\mathcal{S}_q[\mathbb{A}^2]):= \oplus_{n\geq 0} \Gr_n( \mathcal{S}_q[\mathbb{A}^2])$ receives, by passing to the quotient, a left $\mathcal{S}_q(\mathbb{B})$ comodule structure, so does its graded components $\Gr_n(\mathcal{S}_q[\mathbb{A}^2])$. 
  \end{enumerate}
 \end{definition}
 
 \begin{figure}[!h] 
\centerline{\includegraphics[width=8cm]{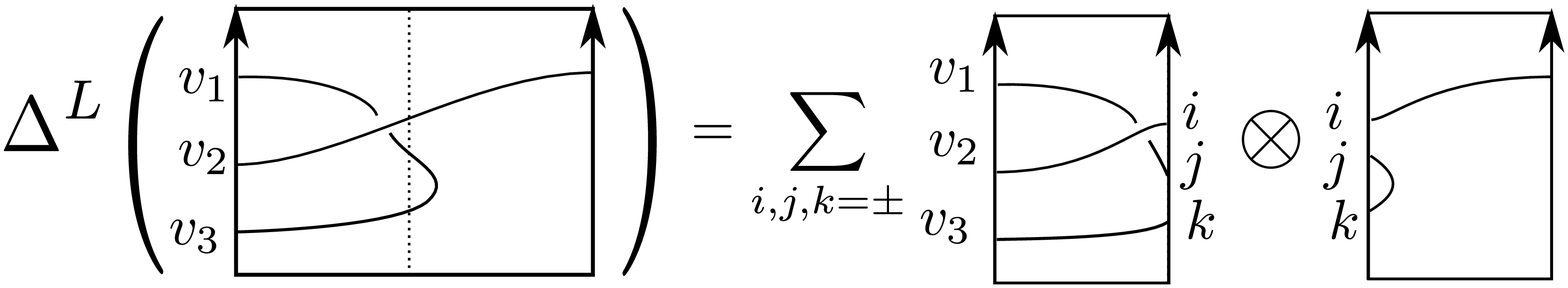} }
\caption{An illustration of the coproduct in the skein quantum plane.} 
\label{fig_coproduct_qplane} 
\end{figure}

 Recall that $\alpha$ is a connected tangle in $\mathbb{B}$ (an arc) connecting $a_L$ to $a_R$. Let $\alpha_+, \alpha_- \in \mathcal{S}_q[\mathbb{A}^2]$ be the classes of the right stated tangles $(\alpha, +)$ and $(\alpha, -)$.
 Recall from Theorem \ref{theorem_bigon} the Hopf algebra isomorphism $f_{\mathbb{B}} : \mathcal{O}_q[\SL_2] \xrightarrow{\cong} \mathcal{S}_q(\mathbb{B})$ sending $a,b,c,d$ to $\alpha_{++}, \alpha_{+-}, \alpha_{-+}, \alpha_{--}$ respectively.
 
  Let $k[\mathbb{N}]=k[X_1,X_2, \ldots, X_n, \ldots]$ be the algebra of polynomials with an infinite number of variables and consider
  $\mathcal{O}_q[\mathbb{A}^2] \star k[\mathbb{N}]$ the free product of $\mathcal{O}_q[\mathbb{A}^2]$ with $k[\mathbb{N}]$. Define a comodule map $\Delta^L : \mathcal{O}_q[\mathbb{A}^2] \star k[\mathbb{N}] \to \mathcal{O}_q[\SL_2] \otimes (\mathcal{O}_q[\mathbb{A}^2] \star k[\mathbb{N}])$ whose restriction to $\mathcal{O}_q[\mathbb{A}^2]$ is the standard coproduct and such that $\Delta^L(X_i)=1\otimes X_i$. Extend the grading to  $\mathcal{O}_q[\mathbb{A}^2] \star k[\mathbb{N}]$ by stating that $X_i$ has degree $0$.
 
 \begin{lemma}\label{lemma_coinv2}
 \begin{enumerate}
 \item There is an  isomorphism of graded algebras $f_{\mathbb{A}^2}: \mathcal{O}_q[\mathbb{A}^2]\star k[\mathbb{N}] \xrightarrow{\cong} \Gr \left( \mathcal{S}_q[\mathbb{A}^2]\right)$ characterized by the fact that $f_{\mathbb{A}^2}(x)= \alpha_+$, $f_{\mathbb{A}^2}(y)=\alpha_-$ and $f(X_i)= \adjustbox{valign=c}{\includegraphics[width=1cm]{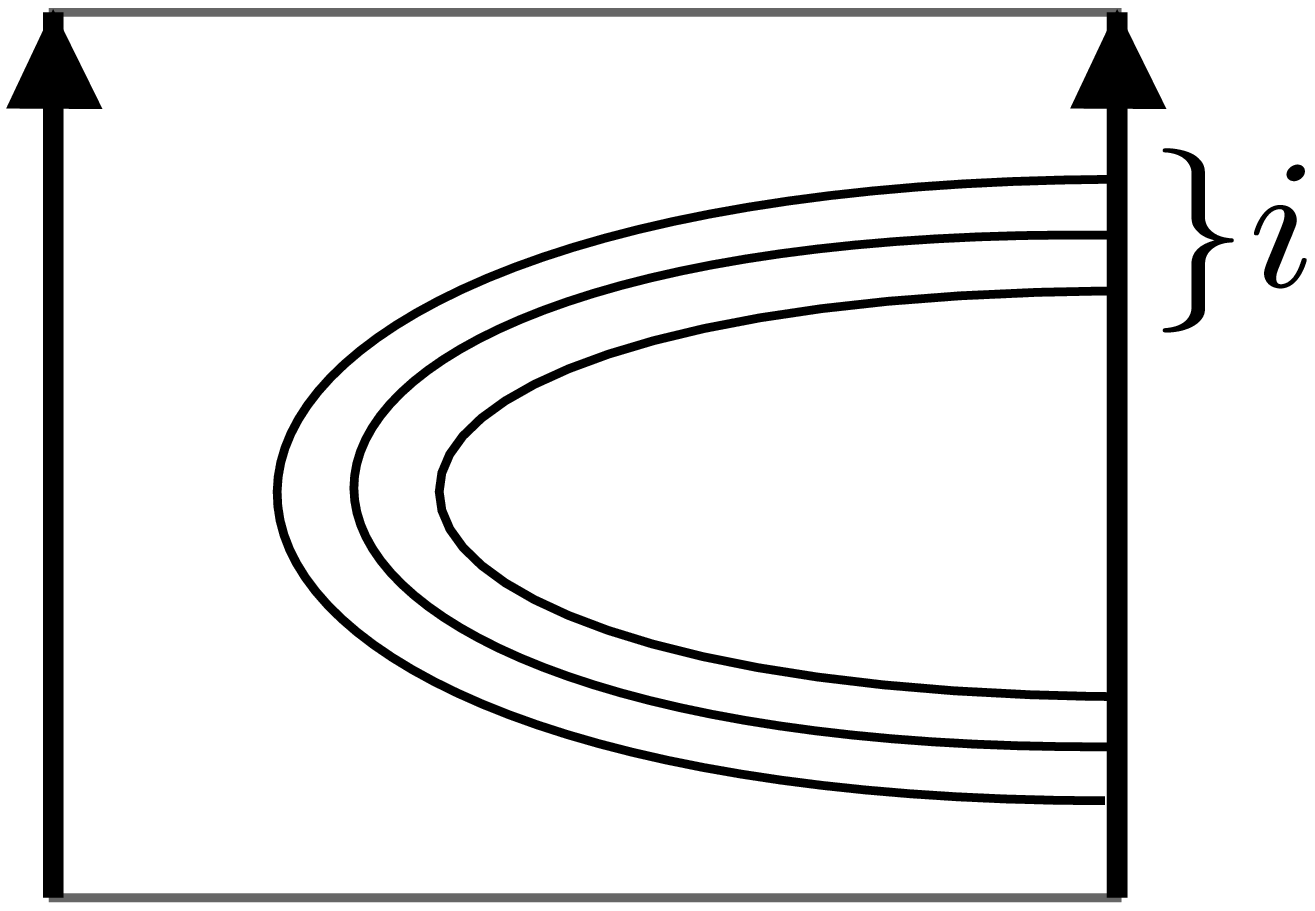}}$.
 \item The following diagram commutes: 
 $$ \begin{tikzcd}
 \mathcal{O}_q[\mathbb{A}^2]\star k[\mathbb{N}] \ar[d, "f_{\mathbb{A}^2}", "\cong"'] \ar[r, "\Delta^L"] &
 \mathcal{O}_q[\SL_2] \otimes (\mathcal{O}_q[\mathbb{A}^2]\star k[\mathbb{N}]) \ar[d, "f_{\mathbb{B}}\otimes f_{\mathbb{A}^2}", "\cong"'] \\
 \Gr(\mathcal{S}_q[\mathbb{A}^2]) \ar[r, "\Delta^L"] &
 \mathcal{S}_q(\mathbb{B}) \otimes \Gr \left( \mathcal{S}_q[\mathbb{A}^2] \right)
\end{tikzcd}
$$
\item 
The submodule $ \mathcal{F}^{(0)}$, spanned by left stated tangles $[T,s^L]$ such that $T\cap a_L=\emptyset$,  is equal to the submodule of  coinvariant vectors of $ \mathcal{S}_q[\mathbb{A}^2] $.
\end{enumerate}
\end{lemma}

\begin{proof}
$(1)$ The fact that $\alpha_+ \alpha_- = q^{-1} \alpha_- \alpha_+$ follows from the following skein relation: 
$$ \qplanedouble{+}{-}= q^{-1} \qplanedouble{-}{+} + q^{1/4} \heightcurverightdeux \equiv q^{-1} \qplanedouble{-}{+} \pmod{\mathcal{F}_1}. $$
Therefore the graded algebra morphism $f_{\mathbb{A}^2}: \mathcal{O}_q[\mathbb{A}^2]\star k[X] \xrightarrow{\cong} \Gr \left( \mathcal{S}_q[\mathbb{A}^2]\right)$ is well defined. 
The polynomial algebra $k[\mathbb{N}]$ has basis elements $X^{\mathbf{n}}= X_1^{n_1}\ldots X_k^{n_k}\ldots$ for $\mathbf{n}: \mathbb{N} \to \mathbb{N}$ a map with finite support.
Recall the basis $\mathcal{B}_1=\{x^iy^j, i,j\geq 0\}$ of $\mathcal{O}_q[\mathbb{A}^2]$ so that  $\widehat{\mathcal{B}_1}:= \{ b_1 X^{\mathbf{n}_1}\ldots b_kX^{\mathbf{n}_k}, k\geq 0, b_i \in \mathcal{B}_1\}$ is a basis of  $\mathcal{O}_q[\mathbb{A}^2]\star k[\mathbb{N}] $. We need to show that $\mathcal{B}':=f( \widehat{\mathcal{B}_1})$ is a basis of $\Gr\left( \mathcal{S}_q[\mathbb{A}^2]\right)$ in order to prove that $f_{\mathbb{A}^2}$ is an isomorphism. This is equivalent to proving that $\mathcal{B}'$ is a basis of $\mathcal{S}_q[\mathbb{A}^2]$. The proof is a straightforward adaption of L\^e's arguments in the proof of \cite[Theorem $2.8$]{LeStatedSkein}, based on the Diamond lemma,  that we now develop. Identify the ball $\mathbb{B}$ with a thickened disc $\mathbb{D}^2\times I$ and let $\pi: \mathbb{D}^2\times I \to \mathbb{D}^2$ the projection map on $\mathbb{D}^2\times \{0\}$. A tangle $T$ is in generic position if its framing at every point points in the height direction towards $1$ and if the projection $\pi: T\to \pi(T)$ only has transverse double points in the interior of $\mathbb{D}^2$. A \textit{diagram} $D$ of $T$ is then the data to the planar graph $\pi(T)$ together with for each of its double point, the over/under crossing information. A left stated diagram is $(D,s^L)$ with $s^L: D\cap a_L \to \{v_-, v_+\}$. Note that the states are elements of the basis $\{v_+, v_-\}$ instead of arbitrary vectors in $V$.
Fix the orientation $\mathfrak{o}$ of both boundary arcs $a_L, a_R$  corresponding to the arrows in the pictures $ \qplanedouble{+}{-}$. Then a stated diagram defines an element $[D,s^L] \in \mathcal{O}_q[\mathbb{A}^2]$ and $\mathcal{O}_q[\mathbb{A}^2]$ is the quotient of the free $k$-module generated by (planar) isotopy classes of left stated diagrams by the framed Reidemeister moves together with the skein relations. Let $\overline{\mathcal{B}}'$ be the set of left stated diagrams. Then $\mathcal{B}'$ is the set of classes of elements $(D,s^L)$ of  $\overline{\mathcal{B}}'$ such that $(1)$ $D$ does not have any crossing, $(2)$ $D$ does not have any connected component with both endpoints in $a_L$ and $(3)$ the state $s^L$ is $\mathfrak{o}$ increasing, i.e if $p_1, p_2 \in D\cap a_L$ are such that $p_1$ is on top $p_2$ then $s^L(p_1)\geq s^L(p_2)$ (here $v_+>v_-$). Define a binary operation $\to$ on $k[\mathcal{B}']$ as follows. If $[D,s] \in \mathcal{B}'$ and $E'\in k[\mathcal{B}']$ write $D\to E'$ if they are related by one of the following skein manipulation: 

  \begin{align*}
{}& \begin{tikzpicture}[baseline=-0.4ex,scale=0.5,>=stealth]	
\draw [fill=gray!45,gray!45] (-.6,-.6)  rectangle (.6,.6)   ;
\draw[line width=1.2,-] (-0.4,-0.52) -- (.4,.53);
\draw[line width=1.2,-] (0.4,-0.52) -- (0.1,-0.12);
\draw[line width=1.2,-] (-0.1,0.12) -- (-.4,.53);
\end{tikzpicture}
\to q^{1/2}
\begin{tikzpicture}[baseline=-0.4ex,scale=0.5,>=stealth] 
\draw [fill=gray!45,gray!45] (-.6,-.6)  rectangle (.6,.6)   ;
\draw[line width=1.2] (-0.4,-0.52) ..controls +(.3,.5).. (-.4,.53);
\draw[line width=1.2] (0.4,-0.52) ..controls +(-.3,.5).. (.4,.53);
\end{tikzpicture}
+q^{-1/2}
\begin{tikzpicture}[baseline=-0.4ex,scale=0.5,rotate=90]	
\draw [fill=gray!45,gray!45] (-.6,-.6)  rectangle (.6,.6)   ;
\draw[line width=1.2] (-0.4,-0.52) ..controls +(.3,.5).. (-.4,.53);
\draw[line width=1.2] (0.4,-0.52) ..controls +(-.3,.5).. (.4,.53);
\end{tikzpicture}
, \quad 
\begin{tikzpicture}[baseline=-0.4ex,scale=0.5,rotate=90] 
\draw [fill=gray!45,gray!45] (-.6,-.6)  rectangle (.6,.6)   ;
\draw[line width=1.2,black] (0,0)  circle (.4)   ;
\end{tikzpicture}
\to  -(q+q^{-1}) 
\begin{tikzpicture}[baseline=-0.4ex,scale=0.5,rotate=90] 
\draw [fill=gray!45,gray!45] (-.6,-.6)  rectangle (.6,.6)   ;
\end{tikzpicture}
\\
{}&
\begin{tikzpicture}[baseline=-0.4ex,scale=0.5,>=stealth]
\draw [fill=gray!45,gray!45] (-.7,-.75)  rectangle (.4,.75)   ;
\draw[->] (-0.7,-0.75) to (-.7,.75);
\draw[line width=1.2] (-0.7,-0.3) to (-0.3,-.3);
\draw[line width=1.2] (-0.7,0.3) to (-0.3,.3);
\draw[line width=1.15] (-.4,0) ++(-90:.3) arc (-90:90:.3);
\draw (-0.9,0.3) node {\scriptsize{$+$}}; 
\draw (-0.9,-0.3) node {\scriptsize{$+$}}; 
\end{tikzpicture}
\to 
0
 , \quad 
 \begin{tikzpicture}[baseline=-0.4ex,scale=0.5,>=stealth]
\draw [fill=gray!45,gray!45] (-.7,-.75)  rectangle (.4,.75)   ;
\draw[->] (-0.7,-0.75) to (-.7,.75);
\draw[line width=1.2] (-0.7,-0.3) to (-0.3,-.3);
\draw[line width=1.2] (-0.7,0.3) to (-0.3,.3);
\draw[line width=1.15] (-.4,0) ++(-90:.3) arc (-90:90:.3);
\draw (-0.9,0.3) node {\scriptsize{$-$}}; 
\draw (-0.9,-0.3) node {\scriptsize{$-$}}; 
\end{tikzpicture}
\to 
0
, \quad 
\begin{tikzpicture}[baseline=-0.4ex,scale=0.5,>=stealth]
\draw [fill=gray!45,gray!45] (-.7,-.75)  rectangle (.4,.75)   ;
\draw[->] (-0.7,-0.75) to (-.7,.75);
\draw[line width=1.2] (-0.7,-0.3) to (-0.3,-.3);
\draw[line width=1.2] (-0.7,0.3) to (-0.3,.3);
\draw[line width=1.15] (-.4,0) ++(-90:.3) arc (-90:90:.3);
\draw (-0.9,0.3) node {\scriptsize{$+$}}; 
\draw (-0.9,-0.3) node {\scriptsize{$-$}}; 
\end{tikzpicture}
\to 
q^{-1/4}
\begin{tikzpicture}[baseline=-0.4ex,scale=0.5,>=stealth]
\draw [fill=gray!45,gray!45] (-.7,-.75)  rectangle (.4,.75)   ;
\draw[->] (-0.7,-0.75) to (-.7,.75);
\end{tikzpicture}
\\
\mbox{or} \quad &
\begin{tikzpicture}[baseline=-0.4ex,scale=0.5,>=stealth]
\draw [fill=gray!45,gray!45] (-.7,-.75)  rectangle (.4,.75)   ;
\draw[->] (-0.7,-0.75) to (-.7,.75);
\draw[line width=1.2] (-0.7,-0.3) to (0.4,-.3);
\draw[line width=1.2] (-0.7,0.3) to (0.4,.3);
\draw (-0.9,0.3) node {\scriptsize{$-$}}; 
\draw (-0.9,-0.3) node {\scriptsize{$+$}}; 
\end{tikzpicture}
\to q^{-1} 
\begin{tikzpicture}[baseline=-0.4ex,scale=0.5,>=stealth]
\draw [fill=gray!45,gray!45] (-.7,-.75)  rectangle (.4,.75)   ;
\draw[->] (-0.7,-0.75) to (-.7,.75);
\draw[line width=1.2] (-0.7,-0.3) to (0.4,-.3);
\draw[line width=1.2] (-0.7,0.3) to (0.4,.3);
\draw (-0.9,0.3) node {\scriptsize{$+$}}; 
\draw (-0.9,-0.3) node {\scriptsize{$-$}}; 
\end{tikzpicture}
+ q^{1/4} \heightcurveright
.
\end{align*}

More generally, write $E\to E'$ if $E=\sum_i \alpha_i [D_i,s_i]$ and there exists $i_0$ with $\alpha_{i_0}\neq 0$, $[D_{i_0},s_{i_0}] \to E_{i_0}$ related by a skein relation as above and such that $E'=\alpha_{i_0}E_{i_0} + \sum_{i\neq i_0} \alpha_i E_i$. Let $\sim$ be the equivalence relation on $k[\overline{\mathcal{B}}']$ generated by $\to$. Note that two stated diagrams related by framed Reidemeister moves are equivalent for $\sim$ (this follows using the relations given by the Kauffman-bracket skein relations), therefore $\mathcal{S}_q[\mathcal{A}]= \quotient{k[\overline{\mathcal{B}}']}{\sim}$. The arguments in the proof of \cite[Lemma $2.10$]{LeStatedSkein} extend world-by-world and show that $\to$ is terminal and locally confluent, therefore the Diamond lemma implies that the set of initial objects for $\to $ is a basis of $\mathcal{S}_q[\mathcal{A}]$. This set is precisely $\mathcal{B}'$. Since this basis is made of graded elements, its image in $\Gr\left( \mathcal{S}_q[\mathcal{A}]\right)$ is a basis as well. Therefore $f_{\mathcal{A}}$ is an isomorphism. This concludes the proof of $(1)$.

\par $(2)$ It is sufficient to prove the commutativity of the diagram for the generators $x,y, X$ of $\mathcal{O}_q[\mathbb{A}^2]\star k[X]$. This follows from the following computations:
$$ \Delta^L\left( \traitalacondeux{+} \right)=  \left( \traitalacontrois{+}{+}\otimes \traitalacondeux{+}\right) + \left(\traitalacontrois{+}{-}\otimes \traitalacondeux{-} \right); $$
$$ \Delta^L\left( \traitalacondeux{-} \right)=  \left( \traitalacontrois{-}{+}\otimes \traitalacondeux{+}\right) + \left(\traitalacontrois{-}{-}\otimes \traitalacondeux{-}\right); $$
$$ \Delta^L \left( \heightcurverightdeux\right) = 1 \otimes \heightcurverightdeux.$$

\par $(3)$ Since $X$ is coinvariant by definition, by Lemma \ref{lemma_coinv1} the subset of coinvariant vectors of $\mathcal{O}_q[\mathbb{A}^2]\star k[\mathbb{N}]$ is its graded $0$ part, so the same is true for  $\Gr(\mathcal{S}_q[\mathbb{A}^2])$ by $(2)$ and since projection $ \mathcal{S}_q[\mathbb{A}^2]\to  \Gr(\mathcal{S}_q[\mathbb{A}^2])$ sends coinvariant vectors to coinvariant vectors and preserves the grading, the results follows.
\end{proof}

\begin{definition}
\begin{enumerate}
\item 
For  $n\geq 0$ let $[n]$ be the $n$-tuple of framed points $(p_1,\ldots, p_n)$ where $p_i:= (0 , \frac{i}{n}) \in \mathbb{D}^2\subset  \mathbb{R}^2$ with framing pointing towards the height direction. For $n, m\geq 0$ a $[n]-[m]$ tangle is a tangle $T$ in $\mathbb{B}$ such that $\iota^{-1}_{a_R} (\partial T \cap a_R ) = [n]$ and $\iota^{-1}_{a_L}(\partial T \cap a_L) = [m]$. The \textit{Temperley-Lieb category} $\TL$ has objects the non negative integers $n\geq 0$ and the set of morphisms $\TL(n,m)$ is the quotient of $k$-module freely generated by isotopy classes of  $[n]-[m]$ tangles by the ideal generated by the Kauffman-bracket relations: 
  $$  \begin{tikzpicture}[baseline=-0.4ex,scale=0.5,>=stealth]	
\draw [fill=gray!45,gray!45] (-.6,-.6)  rectangle (.6,.6)   ;
\draw[line width=1.2,-] (-0.4,-0.52) -- (.4,.53);
\draw[line width=1.2,-] (0.4,-0.52) -- (0.1,-0.12);
\draw[line width=1.2,-] (-0.1,0.12) -- (-.4,.53);
\end{tikzpicture}
= q^{1/2}
\begin{tikzpicture}[baseline=-0.4ex,scale=0.5,>=stealth] 
\draw [fill=gray!45,gray!45] (-.6,-.6)  rectangle (.6,.6)   ;
\draw[line width=1.2] (-0.4,-0.52) ..controls +(.3,.5).. (-.4,.53);
\draw[line width=1.2] (0.4,-0.52) ..controls +(-.3,.5).. (.4,.53);
\end{tikzpicture}
+q^{-1/2}
\begin{tikzpicture}[baseline=-0.4ex,scale=0.5,rotate=90]	
\draw [fill=gray!45,gray!45] (-.6,-.6)  rectangle (.6,.6)   ;
\draw[line width=1.2] (-0.4,-0.52) ..controls +(.3,.5).. (-.4,.53);
\draw[line width=1.2] (0.4,-0.52) ..controls +(-.3,.5).. (.4,.53);
\end{tikzpicture}
, \quad 
\begin{tikzpicture}[baseline=-0.4ex,scale=0.5,rotate=90] 
\draw [fill=gray!45,gray!45] (-.6,-.6)  rectangle (.6,.6)   ;
\draw[line width=1.2,black] (0,0)  circle (.4)   ;
\end{tikzpicture}
=  -(q+q^{-1}) 
\begin{tikzpicture}[baseline=-0.4ex,scale=0.5,rotate=90] 
\draw [fill=gray!45,gray!45] (-.6,-.6)  rectangle (.6,.6)   ;
\end{tikzpicture}
.$$ 
The composition is obtained by gluing the tangles together.
\item Define a right module $F_M: \TL\to \Mod_k$ by letting  $F_M(n)$ be the quotient of the $k$-module freely generated by isotopy classes of tangles $T$ in $\mathbf{M}$ such that $\iota_M^{-1}(T)=[n]$, by the ideal generated by the Kauffman-bracket relations. When $[T'] \in \TL(n,m)$ is the class of   a $[n]-[m]$ tangle we define $F_M([T]) : F_M(n) \to F_M(m)$ to be the linear map sending the class $[T]\in F_M(n)$  of a tangle in $\mathbf{M}$ to $[T'\cup T]$ obtained by gluing the bigon $\mathbb{B}$ to $\mathbf{M}$ while gluing $\mathbb{D}_M$ to $a_R$. The functor $F_M$ was called the \textit{internal skein module} in \cite{GunninghamJordanSafranov_FinitenessConjecture} and its restriction to marked surfaces  appeared earlier in \cite{BenzviBrochierJordan_FactAlg1, Cooke_FactorisationHomSkein}. 
\item In addition to the filtration of Definition \ref{def_skein_qplane}, the skein quantum plane $\mathcal{S}_q[\mathbb{A}]$ has a graduation $\mathcal{S}_q[\mathbb{A}]= \oplus_{n\geq 0} \mathcal{S}_q[\mathbb{A}]^{<n>}$, where 
$$\mathcal{S}_q[\mathbb{A}]^{<n>}:= \Span\left( [T,s^L], |\partial T \cap a_R| = n \right).$$
So whereas the filtration of Definition \ref{def_skein_qplane} is based on the number of intersection points of $T$ with $a_L$, this graduation counts the number of intersection points of $T$ with $a_R$. 
 Clearly $\mathcal{S}_q[\mathbb{A}]^{<n>} \cdot \mathcal{S}_q[\mathbb{A}]^{<m>}= \mathcal{S}_q[\mathbb{A}]^{<n+m>}$, so we have an algebra graduation. The comodule structure of $\mathcal{S}_q[\mathbb{A}]$ preserves each graded component so induces comodule maps 
$\Delta^L_{<n>} : \mathcal{S}_q[\mathbb{A}]^{<n>} \to \mathcal{S}_q(\mathbb{B}) \otimes \mathcal{S}_q[\mathbb{A}]^{<n>}$. 
\item 
 Define a left module $F_{\mathbb{A}} : \TL^{op}\to \Mod_k$  by sending $n$ to $\mathcal{S}_q[\mathbb{A}]^{<n>}$. The action on morphisms is defined as follows. For  $T' \in \TL(m,n)$ then $F_{\mathbb{A}} (T')$ sends $[T, s^L] \in \mathcal{S}_q[\mathbb{A}]^{<n>}$ to $[T\cup T', s^L]$ where $T\cup T'$ is obtained by gluing the bigon $\mathbb{B}$ where lives $T$ to the bigon $\mathbb{B}'$ where lives $T'$ by gluing $a_R$ to $a'_L$. The comodules maps $\Delta^L_{<n>}$ define a natural morphism $\Delta^L : F_{\mathbb{A}} \to \mathcal{S}_q(\mathbb{B})\otimes F_{\mathbb{A}}$. 
\end{enumerate}
\end{definition}

\begin{lemma}\label{lemma_exact}
The functor $\bullet \otimes_{\TL} F_M : \widehat{\TL} \to \Mod_k$, sending a functor $G: \TL^{op} \to \Mod_k$ to the $k$-module $G\otimes_{\TL}F_M$, is right exact while working over the ring $k_{\SL_2}$ and becomes exact when working over the field $K_{\SL_2}$.
\end{lemma}

\begin{proof}
Let $\iota: \TL\to \mathcal{C}_q^{\SL_2}$ be the ribbon functor sending $[n]$ to $V^{\otimes n}$ (see \cite{Tingley_MinusSign} for details), which identifies $\mathcal{C}_q^{\SL_2}$ with the Cauchy closure of $\TL$. Then the functor $\iota^*: \widehat{\mathcal{C}_q^{\SL_2}} \to \widehat{\TL}$, defined by $\iota^*(G):= G\circ \iota$, is an equivalence of categories. Still denote by $F_M\in \widehat{\mathcal{C}_q^{\SL_2}} $ a fixed lift of $F_M$ by $\iota^*$ (unique up to unique isomorphism). It is sufficient to prove that the functor $\bullet \otimes_{\TL} F_M : \widehat{\mathcal{C}_q^{\SL_2}} \to \Mod_k$ is exact.
Recall from Lemma \ref{lemma_cocompletion} the $\dot{U}_q\SL_2-\mathcal{C}_q^{\SL_2}$ bimodule $E'$ such that $E'\otimes_{\mathcal{C}_q^{\SL_2}} \bullet$ is an equivalence of categories. Write $G_M:= E'\otimes_{\mathcal{C}_q^G} F_M$ so that the following diagram commutes:
\begin{equation}\label{diagram_KiTu}
\begin{tikzcd}
\widehat{\mathcal{C}_q^{\SL_2}} \ar[rrd, "\bullet \otimes_{\TL}F_M"] \ar[dd, "E'\otimes_{\mathcal{C}_q^{\SL_2}} \bullet", "\cong"'] &{}&{}\\
{}&{}& \Mod_k \\
\overline{\mathcal{C}_q^{\SL_2}} \ar[rru, "\bullet {\otimes}_{\dot{U}_q\SL_2} G_M"'] &{}&{}
\end{tikzcd}
\end{equation}
\par Since $\bullet {\otimes}_{\dot{U}_q[\SL_2]} G_M$ is right exact,  the commutativity of Diagram \ref{diagram_KiTu} implies that $\bullet \otimes_{\TL} F_M$ is right exact as well.
By Lemma \ref{lemma_cocompletion}, $\overline{\mathcal{C}_q^{\SL_2, rat}}$ is semi-simple, so all its modules are projective and thus flat. In particular $G_M$ is flat when we work over $K_{\SL_2}$ so the functor $\bullet {\otimes}_{\dot{U}_q[\SL_2]} G_M$ is exact and the commutativity of Diagram \ref{diagram_KiTu} implies that $\bullet \otimes_{\TL} F_M$ is exact as well.
\end{proof}

\begin{remark} It was conjectured in \cite[Remark $2.21$]{GunninghamJordanSafranov_FinitenessConjecture} and proved in \cite[Theorem $1.1$]{Haioun_Sskein_FactAlg} that $G_M$ is isomorphic to the stated skein module $\mathcal{S}^{rat}_q(\mathbf{M})$.
\end{remark}

\begin{proof}[Proof of Theorem \ref{theorem_surjectivity}]
By Lemma \ref{lemma_coinv2}, the submodule of coinvariant vectors of $F_{\mathbb{A}}(n)= \mathcal{S}_q[\mathbb{A}]^{<n>}$ is the submodule $F^{(0)}_{\mathbb{A}}(n) \subset F_{\mathbb{A}}(n)$ spanned by $[T,s^L]$ where $T\cap a_L = \emptyset$. Let $F^{(0)}_{\mathbb{A}}: \TL^{op} \to \Mod_k$ the associated functor and $i: F^{(0)}_{\mathbb{A}}\hookrightarrow F_{\mathbb{A}}$ the inclusion morphism. 
 We have a left exact sequence in $\widehat{\TL}$: 
$$0\to  F_{\mathbb{A}}^{(0)} \xrightarrow{i} F_{\mathbb{A}} \xrightarrow{ \Delta^L - \eta\otimes \id} \mathcal{S}_q(\mathbb{B})\otimes F_{\mathbb{A}}.$$
By tensoring with $F_{\mathbf{M}}$ we get a sequence:
$$0\to  F_{\mathbb{A}}^{(0)}\otimes_{\TL} F_{\mathbf{M}}  \xrightarrow{i\otimes \id} F_{\mathbb{A}}\otimes_{\TL} F_{\mathbf{M}}  \xrightarrow{ (\Delta^L - \eta\otimes \id)\otimes \id} \mathcal{S}_q(\mathbb{B})\otimes F_{\mathbb{A}}\otimes_{\TL} F_{\mathbf{M}}.$$
 which is exact while working over $K_{\SL_2}$ and only right exact while working over $k_{\SL_2}$ by  Lemma \ref{lemma_exact}.
Define an  isomorphism $f: F_{\mathbb{A}}\otimes_{\TL} F_{\mathbf{M}} \xrightarrow{\cong} \mathcal{S}_q(\mathbf{M}) $ as follows. For $[T, s^L] \in F_{\mathbb{A}}(n)$ and $T' \in F_{\mathbf{M}}(n)$, set $f\left( [ [T,s^L] \otimes [T']]\right):= [T\cup T', s^L]$ where $T\cup T'$ is the tangle obtained by gluing $T$ to $T'$ while gluing $a_R$ to $\mathbb{D}_M$. The inverse morphism $f^{-1}$ is defined by splitting stated tangles in the same manner than in the definition of $\theta_{a_R\# a_L}$. Define an isomorphism $f_0: F_{\mathbb{A}}^{(0)}\otimes_{\TL} F_{\mathbf{M}} \xrightarrow{\cong} \mathcal{S}_q(M)$ in the same manner. The commutativity of the following diagram is a straightforward consequence of the definitions:
$$
\begin{tikzcd}
 0 \ar[r] & F_{\mathbb{A}}^{(0)}\otimes_{\TL} F_{\mathbf{M}}
  \ar[d, "f_0", "\cong"'] \ar[r, "i \otimes \id"] &
   F_{\mathbb{A}}\otimes_{\TL} F_{\mathbf{M}} 
   \ar[d, "f", "\cong"'] \ar[rr, " (\Delta^L - \eta\otimes \id)\otimes \id" ]&{}&
    \mathcal{S}_q(\mathbb{B})\otimes F_{\mathbb{A}}\otimes_{\TL} F_{\mathbf{M}} 
    \ar[d, "\id \otimes f", "\cong"']     \\
    0 \ar[r] &
\mathcal{S}^{(rat)}_q(M) \ar[r, "\iota_*"] &
\mathcal{S}^{(rat)}_q(\mathbf{M}) \ar[rr, "\Delta^L - \eta \otimes \id"] &{}&
\mathcal{S}^{(rat)}_q(\mathbb{B})\otimes \mathcal{S}_q(\mathbf{M})
\end{tikzcd}
$$
So the exactness of the first line when working over $K_{\SL_2}$ (resp. the right exactness of the first line when working over $k_{\SL_2}$)  implies the exactness (resp. right exactness) of the second which concludes the proof.

\end{proof}

\begin{remark}
Identifying locally finite $U_q\mathfrak{sl}_2$ modules with $\mathcal{O}_q[\SL_2]$-comodules, the $n+1$ dimensional irreducible representation $V_n$ of $U_q\mathfrak{sl}_2$ (so $V_1=V$ in our notations) corresponds to the comodule $\mathcal{O}_q[\mathbb{A}]^{(n)}$. It is easy to see that the quotient $\widetilde{\mathcal{S}_q[\mathbb{A}]}$ of 
${\mathcal{S}_q[\mathbb{A}]}$ by the ideal generated by the $\adjustbox{valign=c}{\includegraphics[width=1cm]{TangleXi.eps}}$ is isomorphic to the limit $\lim_{n\geq 0} (\mathcal{O}_q[\mathbb{A}]^{(1)})^{\otimes n}$ and the equality $\Gr( \widetilde{\mathcal{S}_q[\mathbb{A}]})= \oplus_{n\geq 0}\mathcal{O}_q[\mathbb{A}]^{(n)}$ in Lemma \ref{lemma_coinv2} can be reinterpreted dually in terms of $U_q\mathfrak{sl}_2$-modules,   by the equality in $K^0(\mathcal{C}_q^{\SL_2})$:
$$ [(V_{1}) ^{\otimes n}] = [V_{n}] + \mbox{lower terms}, $$
 where "lower terms" is a linear combinations of $[V_i]$ with $i< n$. The latter equality comes from the fact that $[V_n]= S_n([V_1])$ in $K^0(\mathcal{C}_q^{\SL_2})$ (by the quantum Clebsch-Gordan formula) and that the $n$-th Chebyshev polynomial of second species satisfies $S_n(X) = X^n + \mbox{lower terms}$.
\end{remark}

\subsection{Spherical boundary component}

Recall that $k_G= \mathbb{Z}[q^{\pm 1/n}]$ for $n=n_G$ and consider the field of fractions $K_G:=\mathbb{Q}(q^{1/n})$ and the $K_G$ vector spaces $\mathcal{S}_q^{rat}(\mathbf{M}):= \mathcal{S}_q(\mathbf{M})\otimes_{k_{\SL_2}}K_{\SL_2}$ and $\Rep_q^{G,rat}(\mathbf{M}):=\Rep_q^G(\mathbf{M})\otimes_{k_G}K_G$. They both admit a comodule structure over $\mathcal{O}_qG^{rat}:=\mathcal{O}_qG\otimes_{k_G}K_G$. 
The goal of this subsection is to prove the 

\begin{theorem}\label{theorem_spherical}
Let $\mathbf{M} \in \mathcal{M}_{\con}^{(1)}$ be such that the connected component of $\partial M$ containing the boundary disc $\mathbb{D}_M$ is a sphere. Then every element of $\Rep_q^{G, rat}(\mathbf{M})$ are coinvariant, i.e. $\Rep_q^{G, rat}(\mathbf{M})=\Char_q^{G, rat}(\mathbf{M})$. Similarly, every element of $\mathcal{S}^{rat}_q(\mathbf{M})$ are coinvariant, i.e. $\mathcal{S}^{rat}_q(\mathbf{M})=\mathcal{S}^{rat}_q(\mathbf{M})^{coinv}$.
\end{theorem}

Together with Theorem \ref{theorem_surjectivity} this implies the 

\begin{corollary}\label{coro_spherical} 
Let $\mathbf{M} \in \mathcal{M}_{\con}^{(1)}$ be such that the connected component of $\partial M$ containing the boundary disc $\mathbb{D}_M$ is a sphere. Then the usual (non stated) skein module of $M$ over $K_{\SL_2}$ is equal to the stated skein module of $\mathbf{M}$ over $K_{\SL_2}$, i.e. the inclusion $i_*: \mathcal{S}^{rat}_q(M)\to \mathcal{S}^{rat}_q(\mathbf{M})$ is an isomorphism.
\end{corollary}

Theorem \ref{theorem_spherical} and its proof are very similar and largely inspired (though different) from the work of Gunningham-Jordan-Safronov in \cite[Corollary $4.2$]{GunninghamJordanSafranov_FinitenessConjecture} however our argument, illustrated in Figure \ref{fig_spherical}, is more topological and arguably more enlightening (in \cite{GunninghamJordanSafranov_FinitenessConjecture} no marked $3$-manifold with spherical boundary component is mentioned). 
As explained in Remark \ref{remark_spherical}, Theorem \ref{theorem_spherical} would not hold if we were working over the ring $k_G$ instead of $K_G$. Let $\mathcal{C}_q^{G,rat}=\mathcal{C}_q^G\otimes_{k_G}K_G$ the category of finite dimensional $\mathcal{O}_qG^{rat}$ comodules and $\overline{\mathcal{C}_q^{G,rat}} := \overline{\mathcal{C}_q^G}\otimes_{k_G}K_G$ the category of possibly infinite dimensional comodules. The heart of the proof and the reason why we need to work over the field $K_G$ is the fact that the M\"uger center of $\overline{\mathcal{C}_q^{G,rat}}$ is trivial whereas the M\"uger center of $\overline{\mathcal{C}_q^{G}}$ is not.
\par 
Throughout this subsection, we fix $\mathbf{M}$ satisfying the hypothesis of Theorem \ref{theorem_spherical}. Recall that in the braided quantum group $B_qG:= \int^{W\in \mathcal{C}_q^G} W^* \otimes W \in  \overline{\mathcal{C}_q^G}$, the counit $\epsilon: B_qG \to k$ is defined by $\epsilon([w^*\otimes w]):= \overrightarrow{ev}_W(w^*\otimes w)$. For $V, W \in  \overline{\mathcal{C}_q^G}$, 
consider the map
$$ \tau_{V, W} := V\otimes W^* \otimes W \to W^* \otimes W \otimes V, \quad \tau_{V,W}:= (\id_{W^*}\otimes c_{V, W})(c_{W^*, V}^{-1} \otimes \id_W) =   \adjustbox{valign=c}{\includegraphics[width=1.5cm]{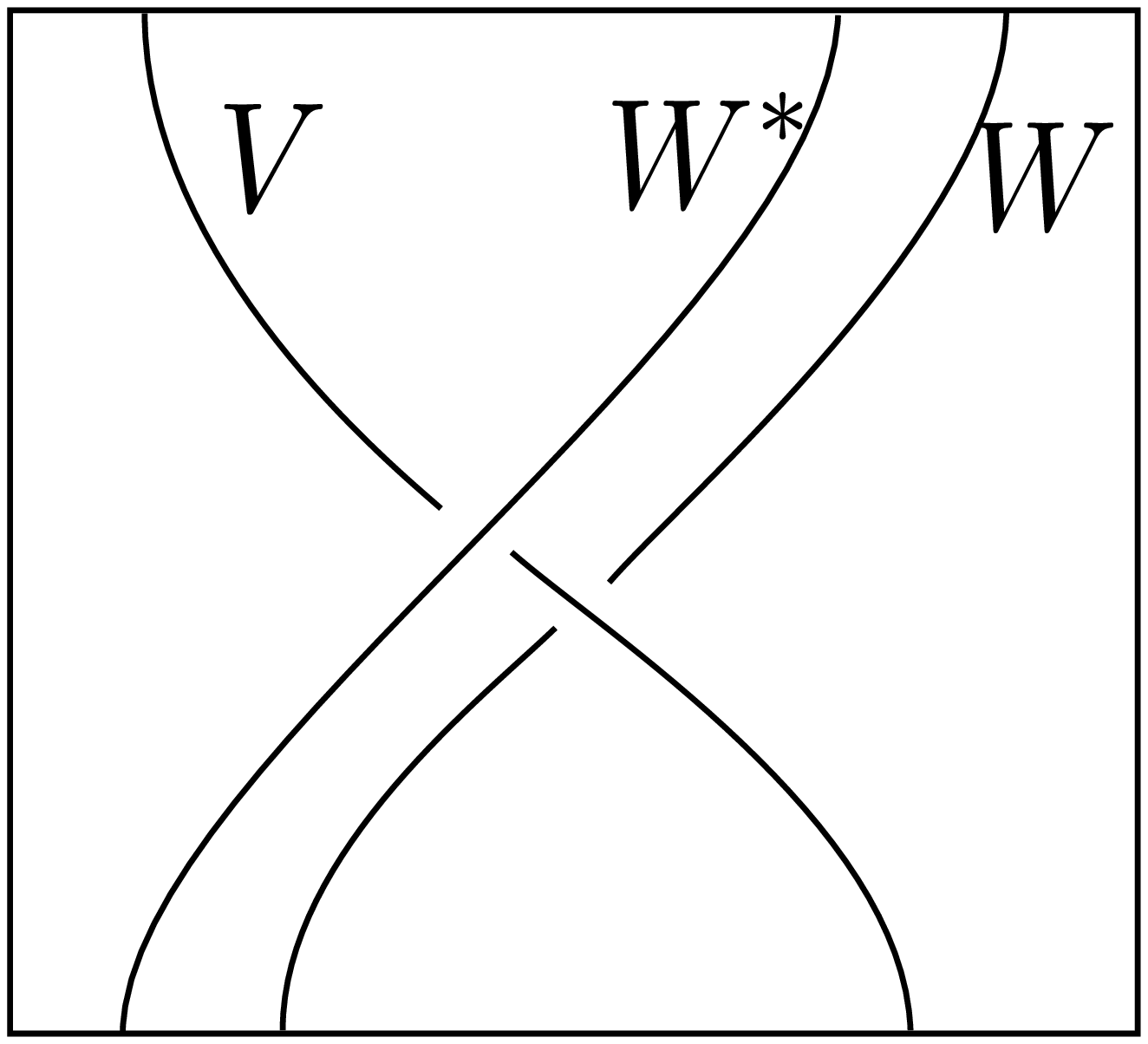}}.$$
Then taking the colimit over $W\in \mathcal{C}_q^G$, one gets a morphism $\tau_V : V\otimes B_qG \to B_qG \otimes V$ named the \textit{field goal transformed} by Lyubashenko (see \cite{Lyubashenko_ModularTransoTensorCat}). The following lemma is essentially a reformulation of \cite[Corollary $1.28$, Proposition $1.37$]{GunninghamJordanSafranov_FinitenessConjecture}. We add the (very simple) proof for the reader convenience.

\begin{lemma}[\cite{GunninghamJordanSafranov_FinitenessConjecture}]\label{lemma_spherical}  Suppose that $V\in  \overline{\mathcal{C}_q^{G, rat}}$ satisfies 
$$ \epsilon(x) v = (\epsilon \otimes \id_V)\circ \tau_V (v\otimes x) , \quad \mbox{ for all }x\in B_qG, v\in V, $$
then $V$ is a direct sum of some copies of the trivial comodule $\mathds{1}\in \overline{\mathcal{C}_q^{G, rat}}$.
\end{lemma}

\begin{proof}
Let us show that $V$ belongs to the  M\"uger center $\mathcal{Z}^{M\ddot{u}ger} $ of $\overline{\mathcal{C}_q^{G, rat}}$:
\begin{align*}
 V\in \mathcal{Z}^{M\ddot{u}ger} & \Leftrightarrow c_{W, V} \circ c_{V,W} = \id_{V \otimes W} , \quad \mbox{ for all } W\in  \mathcal{C}_q^G, \\
{} & \Leftrightarrow (\id_V \otimes \overrightarrow{ev}_W) \circ (c_{W, V} c_{V,W} \otimes \id_{W^*}) =  (\id_V \otimes \overrightarrow{ev}_W) , \quad \mbox{ for all } W\in  \mathcal{C}_q^G, \\
{} & \Leftrightarrow (\overrightarrow{ev}_W \otimes \id_V) \tau_{V,W} = \id_V \otimes \overrightarrow{ev}_W , \quad \mbox{ for all } W\in  \mathcal{C}_q^G, \\
{} & \Leftrightarrow (\id_V \otimes \epsilon) = (\epsilon \otimes \id_V) \circ \tau_V \\
{} & \Leftrightarrow \epsilon(x) v = (\epsilon \otimes \id_V)\circ \tau_V (v\otimes x) , \quad \mbox{ for all }x\in B_qG, v\in V.
\end{align*}
We conclude using the  fact that any element in $\mathcal{Z}^{M\ddot{u}ger}$ is a trivial comodule. 

\end{proof}

\begin{proof}[Proof of Theorem \ref{theorem_spherical}]
Write $V= \Rep_q^{G, rat}(\mathbf{M})$ and let us prove that it satisfies the hypothesis of Lemma \ref{lemma_spherical}. The proof is illustrated in Figure \ref{fig_spherical}. Let $x\in B_qG$ and  $v\in \Rep_q^{G,rat}(\mathbf{M})$ be an element of the form $v:= [y\otimes \alpha] $ with $\alpha\in P_n(M)$ a $n$-bottom tangle and $y\in (B_qG)^{\otimes n}$. Let us prove that $ \epsilon(x) v = (\epsilon \otimes \id_V)\circ \tau_V (v\otimes x) $. Let $\beta_0 \in P_1(M)$ be the trivial $1$-bottom tangle made of an arc lying in a small neighborhood of $\mathbb{D}_M$. More precisely, if $\boldsymbol{\eta}\in BT(1,0)$ is the counit of Figure \ref{fig_BTHopfAlg} and $\eta_M$ is the only element of $P_0(M)$ then $\beta_0:= \eta_M\circ \boldsymbol{\eta}$. Let $\beta_0 \cup \alpha \in P_{n+1}(M)$ be the $n+1$ bottom tangle obtained by stacking $\beta_0$ on top (in the height direction) of $\alpha$. Similarly, let $\alpha\cup \beta_0 \in P_{n+1}(M)$ be the bottom tangle obtained by stacking $\beta_0$ on the bottom of $\alpha$. In more precise terms, we set 
$$ \beta_0 \cup \alpha:= \alpha \circ (\boldsymbol{\eta} \wedge \mathds{1} \wedge \ldots \wedge \mathds{1}), \quad \alpha\cup \beta_0 := \alpha \circ ( \mathds{1} \wedge \ldots \wedge \mathds{1} \wedge \boldsymbol{\eta}).$$
So by definition of the counit $\epsilon$, one has 
$$ \epsilon(x) v = [(x\otimes y) \otimes (\beta_0\cup \alpha)] = [(y\otimes x) \otimes (\alpha\cup \beta_0)] \in \Rep_q^G(\mathbf{M}).$$
Let 
$$T:=     \adjustbox{valign=c}{\includegraphics[width=2cm]{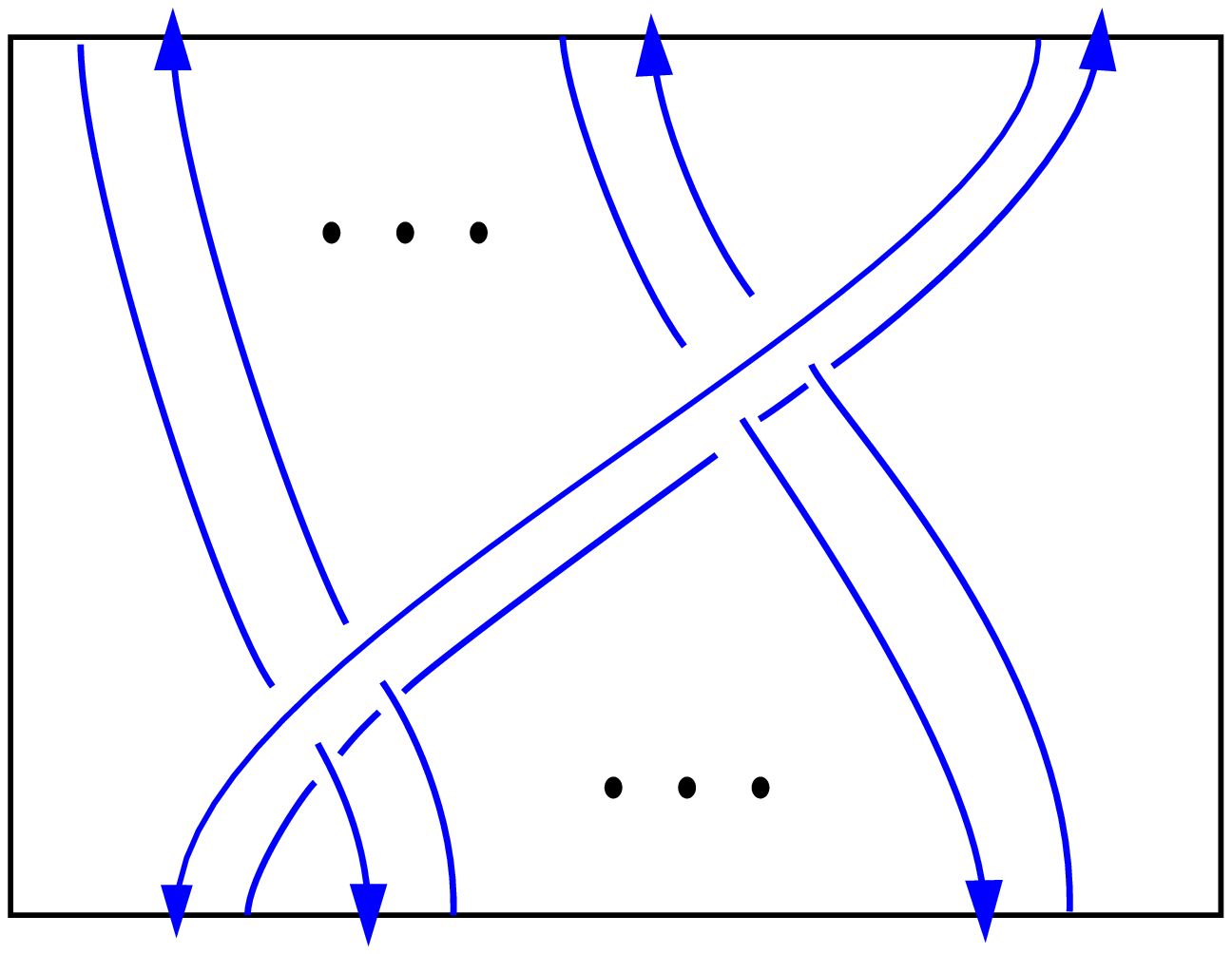}}   \in \mathfrak{bt}(n+1, n+1).$$
Then by definition of the field goal transform, one has 
$$ (\epsilon \otimes \id_V) \circ \tau_V(v\otimes x) = [ (x\otimes y) \otimes \left((\alpha\cup \beta_0)\cdot T\right)] \in \Rep_q^{G, rat}(\mathbf{M}).$$
The equality $ \epsilon(x) v = (\epsilon \otimes \id_V)\circ \tau_V (v\otimes x) $ follows from the fact, illustrated in Figure \ref{fig_spherical}, that the $n+1$ bottom tangles $\beta_0\cup \alpha$ and $(\alpha \cup \beta_0)\cdot T$ are isotopic in $\mathbf{M}$ because the boundary component containing $\mathbb{D}_M$ is a sphere. We conclude using Lemma \ref{lemma_spherical}.

 \begin{figure}[!h] 
\centerline{\includegraphics[width=16cm]{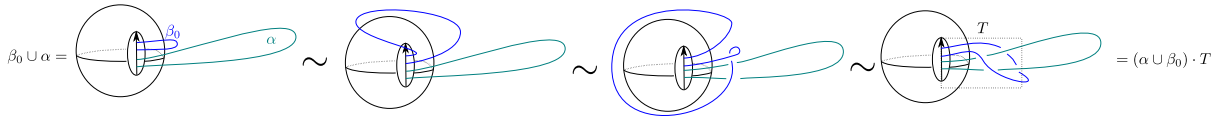} }
\caption{An illustration of the isotopy $\beta_0\cup \alpha \sim (\alpha \cup \beta_0)\cdot T$.} 
\label{fig_spherical} 
\end{figure}

\end{proof}

\begin{remark}\label{remark_spherical}
Note that in Theorem \ref{theorem_spherical}, it is important to work over the field $K_G$ instead of the ring $k_G$. Indeed, as we shall prove in Section \ref{sec_classical}, the module $\mathcal{S}_q(\mathbf{M}) \otimes_{q^{1/4}=1}\mathbb{Z}$, obtained by setting $q^{1/4}=1$, is isomorphic as an $\mathcal{O}[\SL_2]$-comodule to the ring of regular functions of the representation scheme with coaction given by conjugacy. Since the action of $\SL_2$ on the space $\Hom(\pi_1(M), \SL_2)$ by conjugacy is not trivial, this implies that $\mathcal{S}_q(\mathbf{M})$ does not have a trivial coaction, i.e. that the analogue of Theorem \ref{theorem_spherical} over $k_{\SL_2}$ does not hold. Let us illustrate this phenomenon on a concrete example. Let $\alpha$ be a non contractible $1$-bottom tangle in $\mathbf{M}$ and for $i,j=\pm$ let $\alpha_{ij}\in \mathcal{S}_q(\mathbf{M})$ denote the class of the arc $\alpha$ with state $v_i$ at its higher endpoint and $v_j$ at its lower endpoint. Let $\alpha^0\in \mathcal{S}_q(\mathbf{M})$ be the class of the closed arc obtained from $\alpha$ by gluing its two endpoints together in the interior of $M$. 
Write $\begin{pmatrix} D_{++} & D_{+-} \\ D_{-+} & D_{--} \end{pmatrix} := \begin{pmatrix} 0& -A^{5/2} \\ A^{1/2} & 0 \end{pmatrix}$ the matrix coefficients of the half-twist.
A simple skein computation illustrated in Figure \ref{fig_spherical_cex} shows that in $\mathcal{S}_q(\mathbf{M})$ we have the equality in $\mathcal{S}_q(\mathbf{M})$:
\begin{equation}\label{eq_spherical_cool}
(q^4-1) \alpha_{ij} = q(1-q^2) D_{ij} \alpha^0.
\end{equation}
In particular $\alpha_{++}\in \mathcal{S}_q(\mathbf{M})$ is not a coinvariant vector, but since $(q^4-1) \alpha_{++} = 0$, its class in $\mathcal{S}_q^{rat}(\mathbf{M})$ vanishes. In general, Equation \eqref{eq_spherical_cool} shows that the class of any stated arc $\alpha_{ij}$ in $\mathcal{S}_q^{rat}(\mathbf{M})$ is proportional to the coinvariant vector $\alpha^0$, so is coinvariant. The proof of Theorem \ref{theorem_spherical} shows that $\mathcal{S}_q(\mathbf{M})$ belongs to the M\"uger center of $\overline{\mathcal{C}_q^G}$ even though it is not a trivial comodule because of the existence of torsion elements which are not coinvariant.

 \begin{figure}[!h] 
\centerline{\includegraphics[width=16cm]{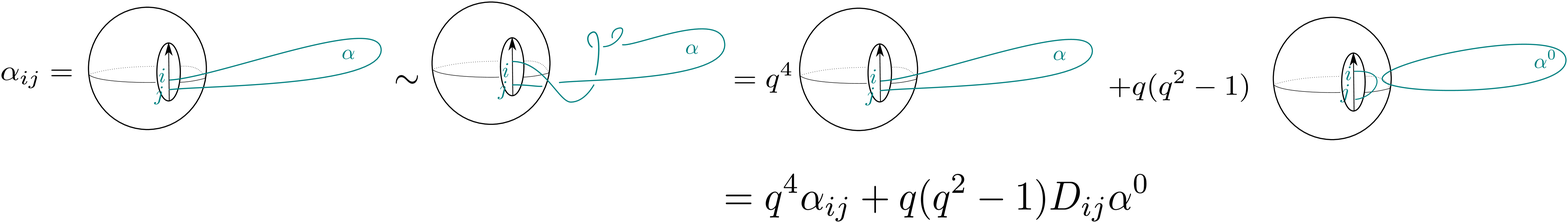} }
\caption{An illustration of Equation \eqref{eq_spherical_cool}. The isotopy $\sim$ is similar to the isotopy of Figure \ref{fig_spherical}. The second equality is obtained by resolving the crossings using the Kauffman-bracket skein relations.} 
\label{fig_spherical_cex} 
\end{figure} 

\end{remark}

\section{Classical limit}\label{sec_classical}
Let $\mathcal{S}_{+1}(\mathbf{M}):= \mathcal{S}_q(\mathbf{M}) \otimes_{q^{1/4}=1}\mathbb{Z}$. The goal of this section is to endow $\mathcal{S}_{+1}(\mathbf{M})$ with a ring structure and to prove that it is isomorphic to the ring of regular functions of the representation scheme $\mathcal{R}_{\SL_2}(\mathbf{M})$.

\subsection{Ring structure on $\mathcal{S}_{+1}(\mathbf{M})$}

Let $\mathcal{O}[\SL_2]:= \mathcal{O}_q[\SL_2]\otimes_{q^{1/4}=1}\mathbb{Z}$: it is the classical integral form of $\SL_2$. Let $\mathcal{S}_{+1}: \mathcal{M}_{\con}^{(1)} \to \RComod_{\mathcal{O}[\SL_2]}$ be the composition of $\mathcal{S}_q$ with the change of scalars $\bullet\otimes_{q^{1/4}=1}\mathbb{Z}: \Mod_k \to \Mod_{\mathbb{Z}}$. 

\begin{lemma}\label{lemma_classic1} $\mathcal{S}_{+1}$ is the left Kan extension of $\restriction{\mathcal{S}_{+1}}{\BT}$ along $\iota: \BT\to \mathcal{M}_{\con}^{(1)}$. \end{lemma}

\begin{proof}
The lemma follows from Lemma \ref{lemma_KAN} together with the fact that the functor $\bullet \otimes_{q^{1/4}=1}  \mathbb{Z}$ has right adjoint $\Hom_{k}(\cdot, \mathbb{Z})$.
\end{proof}

Note that, thanks to the skein relation 
$\begin{tikzpicture}[baseline=-0.4ex,scale=0.5,>=stealth]	
\draw [fill=gray!45,gray!45] (-.6,-.6)  rectangle (.6,.6)   ;
\draw[line width=1.2,-] (-0.4,-0.52) -- (.4,.53);
\draw[line width=1.2,-] (0.4,-0.52) -- (0.1,-0.12);
\draw[line width=1.2,-] (-0.1,0.12) -- (-.4,.53);
\end{tikzpicture}
= 
\begin{tikzpicture}[baseline=-0.4ex,scale=0.5,>=stealth] 
\draw [fill=gray!45,gray!45] (-.6,-.6)  rectangle (.6,.6)   ;
\draw[line width=1.2] (-0.4,-0.52) ..controls +(.3,.5).. (-.4,.53);
\draw[line width=1.2] (0.4,-0.52) ..controls +(-.3,.5).. (.4,.53);
\end{tikzpicture}
+
\begin{tikzpicture}[baseline=-0.4ex,scale=0.5,rotate=90]	
\draw [fill=gray!45,gray!45] (-.6,-.6)  rectangle (.6,.6)   ;
\draw[line width=1.2] (-0.4,-0.52) ..controls +(.3,.5).. (-.4,.53);
\draw[line width=1.2] (0.4,-0.52) ..controls +(-.3,.5).. (.4,.53);
\end{tikzpicture}
$, the elements of $\mathcal{S}_{+1}(\mathbf{M})$ are transparent, i.e. the class of a stated tangle in $\mathcal{S}_{+1}(\mathbf{M})$ does not change when we change a crossing $\Crosspos$ to $\Crossneg$.

\begin{definition}
Equip $\mathcal{S}_{+1}(\mathbf{M})$ with the product defined on two classes of stated tangles  $[T_1,s_1], [T_2, s_2] \in \mathcal{S}_{+1}(\mathbf{M})$ by first isotoping $T_1$ and $T_2$ such that they do not intersect and such that the heights of $\partial T_1$ are bigger than the heights of $\partial T_2$ and then defining
$[T_1,s_1] \cdot [T_2, s_2] := [T_1 \cup T_2, (s_1, s_2)]$.
\end{definition}

The transparence skein relation $\Crosspos=\Crossneg$ implies that the class $[T_1 \cup T_2, (s_1, s_2)]$ does not depend on the choice of the representatives for $T_1\cup T_2$. Moreover, thanks to the height exchange relations 
$\heightexch{->}{i}{j}=  \heightexch{<-}{i}{j}$ (proved in \cite[Lemma $2.4.b$]{LeStatedSkein}), the class $[T,s]$ of a stated tangle does not change if we change the height order of the elements of $\partial T$. In particular $\mathcal{S}_{+1}(\mathbf{M})$  is a commutative ring.

\begin{remark}
 If we had chosen $q^{1/4}=-1$, we still would have obtained a commutative ring; it is not difficult to prove that this ring is (non canonically) isomorphic to $\mathcal{S}_{+1}(\mathbf{M})$. However for the more standard choice $q^{1/4}=\sqrt{-1}$, for which $A=-1$, the product is still defined but the stated skein module is no longer commutative (unlike the usual skein module) because the height exchange relation $\heightexch{->}{i}{j}=  \heightexch{<-}{i}{j}$  does not hold anymore. The inconvenient of having chosen $A=+1$ is that, unlike for $A=-1$, the class of a stated tangle depends on the framing (up to a sign) because of the skein relation 
 $\adjustbox{valign=c}{\includegraphics[width=2cm]{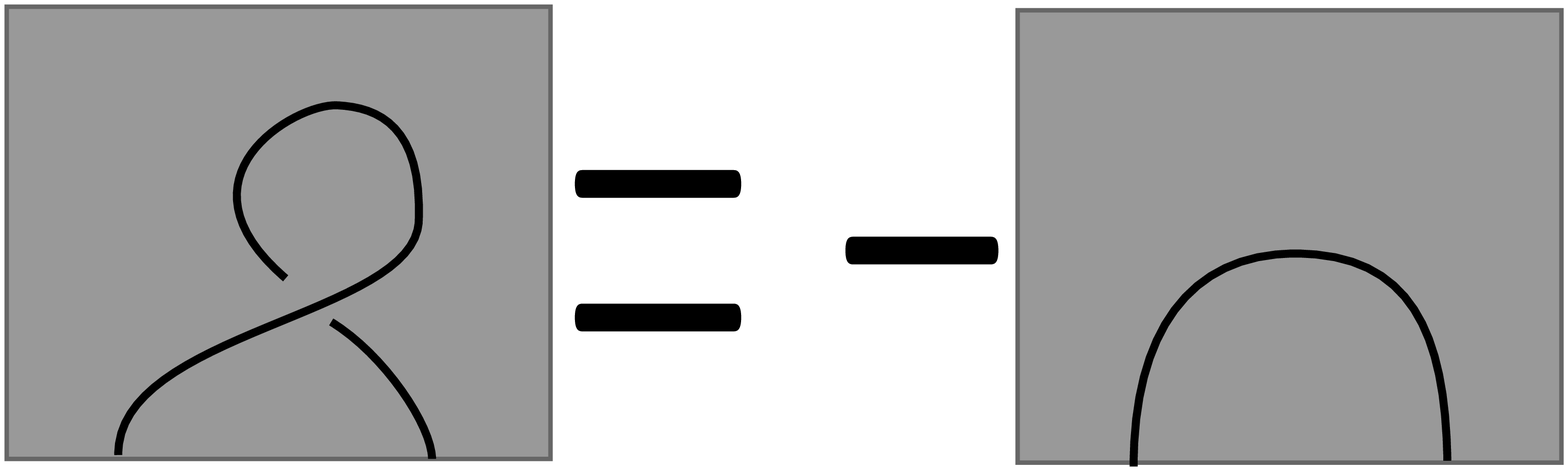}}$. This is the main reason why we will need to consider spin functions in the next subsection.
\end{remark}

\subsection{Spin functions}

Let us define a functor $\tau \in \widehat{\BT}_{\mathbb{Z}}=[\BT^{op}, \Mod_{\mathbb{Z}}]$. For $n\geq 0$  set $\tau(\mathbf{H}_n):=\mathrm{H}^1(\mathbf{H}_n; \mathbb{Z}/2\mathbb{Z})$ and for $\mu: \mathrm{H}_a \to \mathrm{H}_b$ set $\tau(\mu):=\mu^* : \mathrm{H}^1(\mathbf{H}_b; \mathbb{Z}/2\mathbb{Z}) \to \mathrm{H}^1(\mathbf{H}_a; \mathbb{Z}/2\mathbb{Z})$. Let us identify $\mathrm{H}^1(\mathbf{H}_1; \mathbb{Z}/2\mathbb{Z})$ with $\mathbb{Z}/2\mathbb{Z}$.


\begin{definition}\label{def_SpinFunction} For $\mathbf{M} \in \mathcal{M}_{\con}^{(1)}$, a \textit{spin function} is a natural morphism $w: P_M \to \tau$ such that   $w_1\left(\adjustbox{valign=c}{\includegraphics[width=1cm]{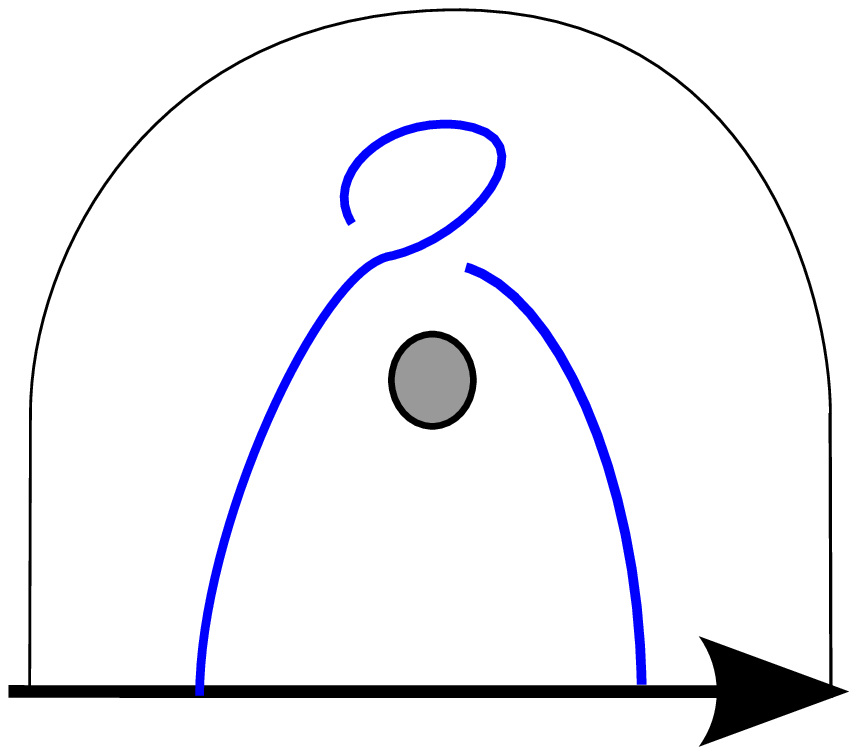}}  \right) = 1 \in \mathbb{Z}/2\mathbb{Z} \cong \mathrm{H}^1(\mathbf{H}_1; \mathbb{Z}/2\mathbb{Z})$.
\end{definition}

So a spin function is a collection of maps $w_n: P_n(\mathbf{M}) \to \mathrm{H}^1(\mathbf{H}_n; \mathbb{Z}/2\mathbb{Z})$ such that $w_1\left(\adjustbox{valign=c}{\includegraphics[width=1cm]{BT_Theta.eps}}  \right)  =1$ and such that for every morphism $\mu : \mathbf{H}_a \to \mathbf{H}_b$, the following diagram commutes:
$$
\begin{tikzcd}
P_b(\mathbf{M}) \ar[r, "\mu^*"] \ar[d, "w_b"] & P_a(M) \ar[d, "w_a"] \\
\mathrm{H}^1(\mathbf{H}_b; \mathbb{Z}/2\mathbb{Z})\ar[r, "\mu^*"] & \mathrm{H}^1(\mathbf{H}_a; \mathbb{Z}/2\mathbb{Z})
\end{tikzcd}
$$

\begin{remark}\label{remark_spin}
\begin{enumerate}
\item A spin function is completely determined by the map $w_1 : P_1(\mathbf{M}) \to \mathbb{Z}/2\mathbb{Z}$. Indeed, let $[\beta_i] \in \mathrm{H}_1(\mathbf{H}_n; \mathbb{Z}/2\mathbb{Z})$ be the class of a simple closed curve encircling the $i$-th hole of $\mathbb{D}_n\times \{1/2\} \subset \mathbf{H}_n$ once, and consider the isomorphism $\mathrm{H}^1(\mathbf{H}_n; \mathbb{Z}/2\mathbb{Z}) \cong (\mathbb{Z}/2\mathbb{Z})^n$ sending $\varphi$ to $(\varphi([\beta_1]), \ldots, \varphi([\beta_n]))$.
 If $\alpha= \alpha^{(1)}\cup \ldots \cup \alpha^{(n)}$ is a $n$-bottom tangle with connected components $\alpha^{(i)}$, by naturality of $w$ and under the above isomorphism one has $w_n(\alpha) = (w_1(\alpha^{(1)}), \ldots, w_1(\alpha^{(n)}))$ so the maps $w_n$ are determined by $w_1$.
\item For $\alpha$ a $1$-bottom tangle and $s$ a state on $\alpha$, the class $(-1)^{w_1(\alpha)}[\alpha, s] \in \mathcal{S}_{+1}(\mathbf{M})$ does not depend on the framing of $\alpha$ thanks to the condition $w_1\left(\adjustbox{valign=c}{\includegraphics[width=1cm]{BT_Theta.eps}}  \right) = 1$ and to the skein relation $\adjustbox{valign=c}{\includegraphics[width=2cm]{TwistRel.eps}}$.
\item When $\mathbf{M}= \mathbf{\Sigma}\times I$ is a thickened marked surface, then the set of spin functions is in one-to-one correspondence with the set of isomorphism classes of spin structures on $\Sigma$ (hence the name "spin functions"). Indeed, fix a Riemannian metric on $\Sigma$ and let $\pi: UM\to M$ be the unitary tangent bundle. Let $[\theta] \in \mathrm{H}_1(UM; \mathbb{Z}/2\mathbb{Z})$ be the homology class of a loop in $UM$ whose projection in $M$ is the constant point $p_M$  and such that $\theta$ makes a full positive twist in the fiber $\pi^{-1}(p_M)$. For $w$ a spin function, one associates $\mathbf{w} :  \mathrm{H}_1(UM; \mathbb{Z}/2\mathbb{Z}) \to \mathbb{Z}/2\mathbb{Z}$ a group morphism such that $\mathbf{w}([\theta])=1$ as follows. Every class $[\alpha]\in  \mathrm{H}_1(UM; \mathbb{Z}/2\mathbb{Z})$ can be represented by a $1$-bottom tangle $\alpha$ (where we isotope both endpoints to $p_M$) and we set $\mathbf{w}([\alpha]):= w_1(\alpha)$. That $\mathbf{w}$ is a group morphism follows from the naturality with respect to the morphism $\mu: \mathbf{H}_2\to \mathbf{H}_1$ of Figure \ref{fig_BTHopfAlg}  and that $\mathbf{w}([\theta])=1$ follows from the fact that  $w_1\left(\adjustbox{valign=c}{\includegraphics[width=1cm]{BT_Theta.eps}}  \right) =1$. The morphism $\mathbf{w}$ defines a (regular) double covering $p:\widetilde{U\Sigma}\to U\Sigma$ and the condition $w([\theta])\neq 0$ ensures that each fiber of $\pi$ lifts to a non trivial double covering by $p$. Since $\Spin(2)$ is the only non trivial double covering of $\SO(2)$, the composition $\pi\circ p : \widetilde{U\Sigma} \to \Sigma$ is a principal $\Spin(2)$ bundle over $U\Sigma$, i.e. a spin structure (see Milnor \cite{MilnorSpinStructure} for details).
\end{enumerate}
\end{remark}

\subsection{Representation and character schemes}

Let $\Gamma$ be a finitely presented group and consider the commutative ring 
$$\mathcal{O}[\mathcal{R}_{\SL_2}(\Gamma)]:= \quotient{\mathbb{Z}[X^{\gamma}_{ij}, i,j\in \{+, -\}, \gamma \in \Gamma]}{(M_{\alpha}M_{\beta}=M_{\alpha\beta}, \det(M_{\alpha})=1, \alpha, \beta\in \Gamma)}, $$
where $M_{\gamma}=\begin{pmatrix} X^{\gamma}_{++} & X^{\gamma}_{+-} \\ X^{\gamma}_{-+} & X^{\gamma}_{--} \end{pmatrix}$ is a $2\times 2$ matrix with coefficients in the polynomial ring $\mathbb{Z}[X^{\gamma}_{ij}]$ (so the relation $M_{\alpha}M_{\beta}=M_{\alpha\beta}$ represents in fact four relations). Set $\mathcal{O}[\SL_2]:= \quotient{\mathbb{Z}[x_{ij}, i,j\in \{-, +\}]}{(x_{++}x_{--}-x_{+-}x_{-+}-1)}$.
The set of characters $\Hom_{\Alg}\left(\mathcal{O}[\mathcal{R}_{\SL_2}(\Gamma)], \mathbb{C}\right)$ is in bijection with the set of representations $\rho: \Gamma \to \SL_2(\mathbb{C})$. The group $\SL_2(\mathbb{C})$ acts (on the left) by conjugacy on the set of representations by the formula: 
$$ (g\cdot \rho) (\gamma):= g \rho(\gamma)g^{-1} , \quad \mbox{ for all }g\in \SL_2, \gamma \in \Gamma.$$
The above action is algebraic, i.e. induced by a comodule map $\Delta^R: \mathcal{O}[\mathcal{R}_{\SL_2}(\Gamma)] \to \mathcal{O}[\mathcal{R}_{\SL_2}(\Gamma)] \otimes_{\mathbb{Z}} \mathcal{O}[\SL_2]$ defined by the formula:
$$ \Delta^R (X_{ij}^{\gamma}):= \sum_{a,b= \pm}  X_{ab}^{\gamma} \otimes x_{i a}S(x_{bj}) .$$
For instance 
$$\Delta^R (X_{++}^{\gamma}) =  X_{++}^{\gamma} \otimes x_{++}x_{--} +  X_{-+}^{\gamma}\otimes x_{+-}x_{--} -  X_{+-}^{\gamma}\otimes x_{-+}x_{++} -  X_{--}^{\gamma}\otimes x_{-+}x_{+-}.$$

Let $\mathcal{O}[\mathcal{X}_{\SL_2}(\mathbf{M})] \subset \mathcal{O}[\mathcal{R}_{\SL_2}(\Gamma)]$ be the subalgebra of coinvariant vectors.
\par Both $\mathcal{O}[\mathcal{R}_{\SL_2}(\Gamma)]$ and $\mathcal{O}[\mathcal{X}_{\SL_2}(\Gamma)]$ are finitely generated, however they might be non reduced, i.e. their nilradical $\sqrt{0}$ might be non trivial. 

\begin{definition} The $\SL_2$-\textit{representation scheme} is $\mathcal{R}_{\SL_2}(\Gamma):= \Spec(\mathcal{O}[\mathcal{R}_{\SL_2}(\Gamma)])$. The $\SL_2$-\textit{character scheme} is $\mathcal{X}_{\SL_2}(\Gamma):=\Spec(\mathcal{O}[\mathcal{X}_{\SL_2}(\Gamma)] )$. 
\par The $\SL_2$-\textit{representation variety} is $\mathcal{R}_{\SL_2}^{red}(\Gamma):= \Spec(\quotient{\mathcal{O}[\mathcal{R}_{\SL_2}(\Gamma)])\otimes_{\mathbb{Z}}\mathbb{C}}{\sqrt{0}})$. The $\SL_2$-\textit{character variety} is $\mathcal{X}_{\SL_2}^{red}(\Gamma):= \Spec(\quotient{\mathcal{O}[\mathcal{X}_{\SL_2}(\Gamma)])\otimes_{\mathbb{Z}}\mathbb{C}}{\sqrt{0}})$.
\end{definition}
 Recall that to  $\mathbf{M}\in \mathcal{M}^{(1)}_{\con}$ one associates a canonical based point $p_M:=\iota_M(0)\in \mathbb{D}_M$. Since every compact $3$-manifold can be triangulated (see e.g. \cite{Hempel_3mfd}), we easily deduced that the fundamental group $\pi_1(\mathbf{M}):=\pi_1(M, p_M)$ is finitely presented.
Write $\mathcal{R}_{\SL_2}(\mathbf{M}):= \mathcal{R}_{\SL_2}(\pi_1(\mathbf{M}))$ and $\mathcal{X}_{\SL_2}(\mathbf{M}):= \mathcal{X}_{\SL_2}(\pi_1(\mathbf{M}))$. In \cite[Corollary $1.3$]{KapovitchMillson}, Kapovitch-Millson proved the existence of large families of $3$-manifolds for which $\mathcal{X}_{\SL_2}(\mathbf{M})$ (and thus $\mathcal{R}_{\SL_2}(\mathbf{M})$) is non reduced, so it is important to distinguish between representation schemes and varieties. However, when $M$ is a thickened surface, $\mathcal{R}_{\SL_2}(\mathbf{\Sigma}\times I)$ is reduced (see \cite{PS00, ChaMa}). An embedding $f: \mathbf{M} \to \mathbf{M}'$ between two marked $3$-manifolds induces a group morphism $f_{*}: \pi_1(\mathbf{M})\to \pi_1(\mathbf{M}')$ and thus a  morphism $f_* : \mathcal{O}[\mathcal{R}_{\SL_2}(\mathbf{M})] \to \mathcal{O}[\mathcal{R}_{\SL_2}(\mathbf{M}')]$ of $\mathcal{O}[\SL_2]$ comodules. We denote by $\mathcal{O}[\mathcal{R}] : \mathcal{M}_{\con}^{(1)} \to {\mathcal{O}[\SL_2]}-\RComod$ the functor sending $\mathbf{M}$ to $\mathcal{O}[\mathcal{R}_{\SL_2}(\mathbf{M})]$. 

\begin{lemma}\label{lemma_LKan_Classical}
The following diagram is a left Kan extension:
 $$
 \begin{tikzcd}
 {} & \mathcal{M}^{(1)}_{\con}
 \ar[rd, "{\mathcal{O}[\mathcal{R}]}"] 
 & {} \\
 \BT  \ar[ru, "i"] 
 \ar[rr, "\restriction{\mathcal{O}[\mathcal{R}]}{\BT}"]
  &{}&  {\mathcal{O}[\SL_2]}-\RComod
 \end{tikzcd}
 $$ 

In particular, one has:
$$ \mathcal{O}[\mathcal{R}_{\SL_2}(\mathbf{M})] \cong \restriction{\mathcal{O}[\mathcal{R}]}{\BT}\otimes_{\BT} \mathbb{Z}[P_M].$$
\end{lemma}

\begin{proof} For each $n\geq 0$, we fix a basis of the free group $\pi_1(\mathbf{H}_n)$ inducing an isomorphism $ \mathcal{O}[\mathcal{R}_{\SL_2}(\mathbf{H}_n)]\cong \mathcal{O}[\SL_2]^{\overline{\otimes}n}$.
Let $L:= \Lan_i \restriction{\mathcal{O}[\mathcal{R}]}{\BT}$ be the left Kan extension of $\restriction{\mathcal{O}[\mathcal{R}]}{\BT}$ along $i$ given by the formula:
$$ L(\mathbf{M}):= \restriction{\mathcal{O}[\mathcal{R}]}{\BT}\otimes_{\BT} \mathbb{Z}[P_M] =\int^{n\geq 0} \mathcal{O}[\SL_2]^{\overline{\otimes}n} \otimes \mathbb{Z}[P_n(\mathbf{M})].$$
By definition of the coend, one has 
$$ L(\mathbf{M})= \quotient{ \left( \oplus_{n\geq 0}\mathcal{O}[\SL_2]^{\overline{\otimes}n} \otimes \mathbb{Z}[P_n(\mathbf{M})]\right)}{\mathcal{J}}, $$
where $\mathcal{J}$ is the ideal generated by elements $(\mu_* X \otimes \gamma) - (X\otimes \mu^* \gamma)$ with $\mu: \mathbf{H}_n \to \mathbf{H}_m$, $\gamma\in P_m(\mathbf{M})$ and $X\in \mathcal{O}[\SL_2]^{\overline{\otimes}n}$.
\par Let $\pi_{M} \in \widehat{\BT}$ be the functor sending $\mathbf{H}_n$ to $\Hom_{\Gp}(\pi_1(\mathbf{H}_n), \pi_1(\mathbf{M}))\cong \pi_1(\mathbf{M})^n$ (the latter isomorphism is given by the fixed basis of $\pi_1(\mathbf{H}_n)$). A $1$-bottom tangle $\alpha\subset M$ defines an element $[\alpha] \in \pi_1(\mathbf{M})$ by forgetting the framing and isotoping locally both endpoints of $\alpha$ to $p_M$. For a $n$-bottom tangle $\alpha \in P_n(\mathbf{M})$ with components $\alpha= \alpha^{(1)}\cup \ldots \cup \alpha^{(n)}$,  ordered such that the heights of the endpoints $\partial \alpha^{(i)}$ is bigger than the heights of the endpoints $\partial \alpha^{(i+1)}$, we set $[\alpha]:= ([\alpha^{(1)}], \ldots, [\alpha^{(n)}]) \in \pi_1(\mathbf{M})^n$. 

Define a natural morphism $\eta^M: P_M \Rightarrow \pi_M$ by sending $\alpha\in P_n(\mathbf{M})\cong P_M(\mathbf{H}_n)$ to $\eta^M_n(\alpha):= [\alpha]$. The map $\eta^M_n: P_n(\mathbf{M}) \to \pi_1(\mathbf{M})^n$ is clearly surjective and two bottom tangles $\alpha$, $\alpha'$ have the same image if and only if one can pass from $\alpha$ to $\alpha'$ by a finite sequence of crossings changes $\Crosspos \leftrightarrow \Crossneg$ and twist moves $\adjustbox{valign=c}{\includegraphics[width=2cm]{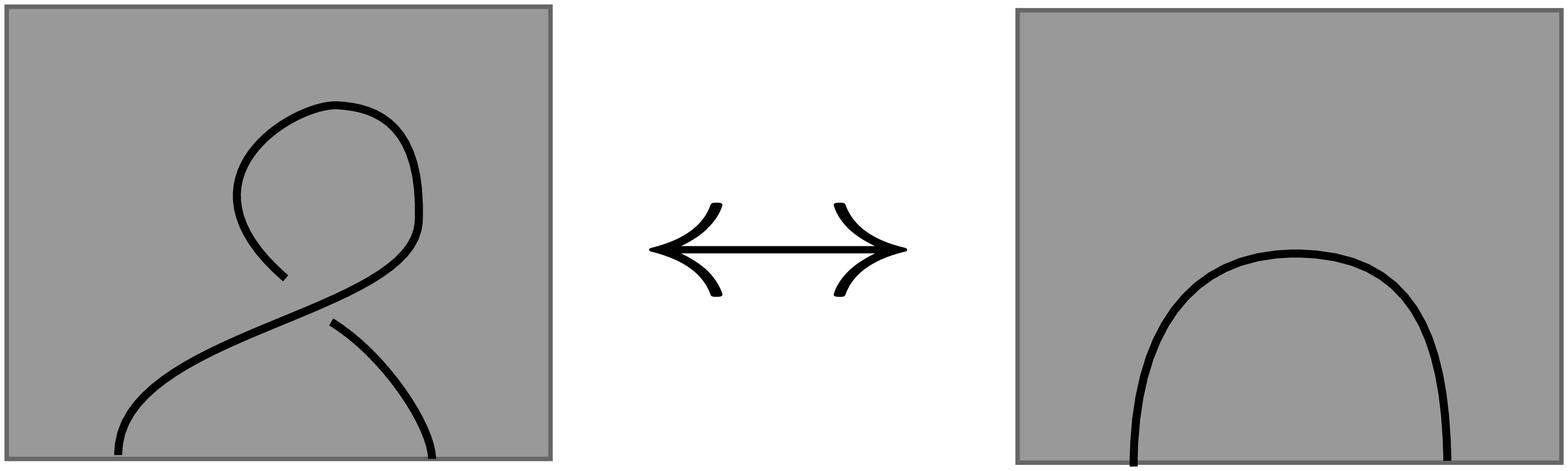}}$. Write $\alpha\sim \alpha'$ if $\alpha$ and $\alpha'$ are related by such a finite sequence, so that $\eta^M_n$ induces a bijection $\quotient{P_n(\mathbf{M})}{\sim} \xrightarrow{\cong} \pi_1(\mathbf{M})^n$. We claim that 
$$ L(\mathbf{M}):= \restriction{\mathcal{O}[\mathcal{R}]}{\BT}\otimes_{\BT} \mathbb{Z}[P_M] \xrightarrow[\id \otimes \eta^M]{\cong}  \restriction{\mathcal{O}[\mathcal{R}]}{\BT}\otimes_{\BT} \mathbb{Z}[\pi_M],$$
i.e. that for $X\in \mathcal{O}[\SL_2]^{\overline{\otimes}n}$ and $\alpha, \alpha' \in P_n(\mathbf{M})$ which differ by some change of crossings $\Crosspos \leftrightarrow \Crossneg$, then $[X\otimes \alpha]=[X \otimes \alpha'] \in L(\mathbf{M})$. Indeed, in such case, one has an endomorphism $\mu : \mathbf{H}_n \to \mathbf{H}_n$ such that $\mu(\alpha)= \alpha'$ and such that $\eta^M(\mu)=\id$. For such an endomorphism $\mu$ whose action on $\pi_1(\mathbf{M})$ is trivial, the induced morphism $\mu_*\in \End(\mathcal{O}[\mathcal{R}_{\SL_2}(\mathbf{H}_n)])$ is the identity, therefore $\alpha \otimes X - \alpha' \otimes X \in \mathcal{J}$, which proves the claim.

\par We can thus write
$$L(\mathbf{M})= \quotient{ \left( \oplus_{n\geq 0}\mathcal{O}[\SL_2]^{\overline{\otimes}n} \otimes \mathbb{Z}[\pi_1(M)^n]\right)}{\mathcal{J}'}, $$
with $\mathcal{J}'$ the image of $\mathcal{J}$ though the quotient map.
On the other hand, write 
$$ \mathcal{O}[\mathcal{R}_{\SL_2}(\mathbf{M})] = \quotient{A}{\mathcal{I}}, $$
where 
$$ A:= \quotient{\mathbb{Z}[X^{\gamma}_{ij}, i,j\in \{+, -\}, \gamma \in \pi_1(\mathbf{M})]}{(\det(M_{\alpha})=1, \alpha\in \pi_1(\mathbf{M}))},$$
and $\mathcal{I}$ is the ideal generated by the relations given by the matrix coefficients of the equalities $M_{\alpha}M_{\beta}=M_{\alpha \beta}$, $\alpha, \beta \in \pi_1(\mathbf{M})$. Let
$$\widetilde{\kappa}_{\mathbf{M}}:  A \cong \oplus_{n\geq 0}\mathcal{O}[\SL_2]^{\overline{\otimes}n} \otimes \mathbb{Z}[\pi_1(M)^n]$$
 be the isomorphism sending a monomial $X_{a_1 b_1}^{\gamma_1} \ldots X_{a_n b_n}^{\gamma_n}$ to $[(x_{a_1 b_1}\otimes \ldots \otimes x_{a_n b_n}) \otimes (\gamma_1, \ldots, \gamma_n)]$. By definition of $\mathcal{J}'$, we have $\widetilde{\kappa}_{\mathbf{M}}(\mathcal{I}) = \mathcal{J}'$, so $\widetilde{\kappa}$ induces an isomorphism 
$$ \kappa_{\mathbf{M}} :  \mathcal{O}[\mathcal{R}_{\SL_2}(\mathbf{M})]  \cong L(\mathbf{M}).$$
That $\kappa_{\mathbf{M}}$ is an isomorphism of $\mathcal{O}[\SL_2]$ comodules and that it is natural in $\mathbf{M}$  are immediate consequences of the definitions. We thus have defined a natural isomorphism $\kappa: \mathcal{O}[\mathcal{R}]\cong L$. This concludes the proof.

\end{proof}

\subsection{Classical limit of stated skein modules}

 Let $\mathbf{M}\in \mathcal{M}^{(1)}_{\con}$ and $\gamma \in \pi_1(\mathbf{M})$. Let $T(\gamma)$ be a tangle isotopic to $\gamma$ with arbitrary framing with distinct endpoints $\partial{T}=\{s,t\} \in \mathbb{D}_M$ such that $\gamma$ is oriented from $s$ (starting point) to $t$ (target point). For $i,j \in \{+, -\}$, let $\gamma_{ij}:= [T(\gamma), s_{ij}] \in \mathcal{S}_{+1}(\mathbf{M})$ be the class of the tangle $T(\gamma)$ with state defined by $s_{ij}(s):= v_i$ and $s_{ij}(t):= v_j$. The main result of this section is the following: 

\begin{theorem}\label{theorem_classical} Let $w$ be a spin function.
We have a ring isomorphism $\eta_{w}: \mathcal{O}[\mathcal{R}_{\SL_2}(\mathbf{M})] \xrightarrow{\cong} \mathcal{S}_{+1}(\mathbf{M})$ whose image on the generators $X_{ij}^{\gamma}$ is given by 
$$ \eta_{w} \begin{pmatrix} X^{\gamma}_{++} & X^{\gamma}_{+-} \\ X^{\gamma}_{-+} & X^{\gamma}_{--} \end{pmatrix} := (-1)^{w(\gamma)}\begin{pmatrix} 0 & -1 \\ 1 & 0 \end{pmatrix} \begin{pmatrix} \gamma_{++} & \gamma_{+-} \\ \gamma_{-+} & \gamma_{--} \end{pmatrix}.$$
Moreover $\eta_w$ is a morphism of $\mathcal{O}[\SL_2]$ comodules.
\end{theorem}

In particular, we obtain a variant of a  classical theorem of Bullock: 
\begin{corollary}\label{coro_classical} One has an isomorphism $\mathcal{O}[\mathcal{X}_{\SL_2}({M})] \cong \mathcal{S}_{+1}(M)$ sending a trace function $\tau_{\gamma}$ to $(-1)^{w(\gamma)}[\gamma]$.
\end{corollary}

\begin{remark}
\begin{enumerate}
\item In the particular case where $\mathbf{M}=\mathbf{\Sigma}\times I$ is a thickened marked surface, Theorem \ref{theorem_classical} was proved independently in \cite{KojuQuesneyClassicalShadows} and \cite{CostantinoLe19} using triangulations. An alternative proof using explicit finite presentations of stated skein algebras was also presented in \cite{KojuPresentationSSkein}.
\item Corollary \ref{coro_classical} is closely related to a theorem of Bullock in \cite{Bullock} who constructed an isomorphism between $\mathcal{O}[\mathcal{X}_{\SL_2}(\mathbf{M})] $ and the skein algebra  $\mathcal{S}_{A=-1}(M)$ evaluated in $A=-1$ by sending a trace function $\tau_{\gamma}$ to $-[\gamma]$. Putting these two results together, one obtains an isomorphism (depending on $w$) between the skein algebra in $A=-1$ and the skein algebra in $A=+1$ sending $[\gamma]$ to $(-1)^{w(\gamma)+1}[\gamma]$. In the particular case where  $\mathbf{M}=\mathbf{\Sigma}\times I$ is a thickened surface (so $w$ is given by a spin structure on $\Sigma$  as explained in Remark \ref{remark_spin}), the existence of such an isomorphism $\mathcal{S}_{+1}(\Sigma)\cong \mathcal{S}_{-1}(\Sigma)$ was proved by Barett in \cite{Barett}. 
\end{enumerate}
\end{remark}

\begin{notations}\label{notations_extension}
Let us define two quotients $\BT_0$ and $\BT_1$ of $\BT$. The objects of $\BT_0$ and of $\BT_1$ are the same than the objects of $\BT$, so are the handlebodies $\mathbf{H}_n$ for $n\geq 0$. Identify $\BT(a,b)=\BT(\mathbf{H}_a, \mathbf{H}_b)$ with the set $P_a(\mathbf{H}_b)$ of isotopy classes of $a$-bottom tangles in $\mathbf{H}_b$. 
\par The set  $\BT_0(a,b)$ is defined as the quotient of $P_a(\mathbf{H}_b)$ by the skein relations $\Crosspos \leftrightarrow \Crossneg$ and $\adjustbox{valign=c}{\includegraphics[width=2cm]{TwistRel+1.eps}}$. 
\par The set $\BT_1(a,b)$ is defined as the quotient of $P_a(\mathbf{H}_b)$ by the skein relations $\Crosspos \leftrightarrow \Crossneg$ and $\adjustbox{valign=c}{\includegraphics[width=2cm]{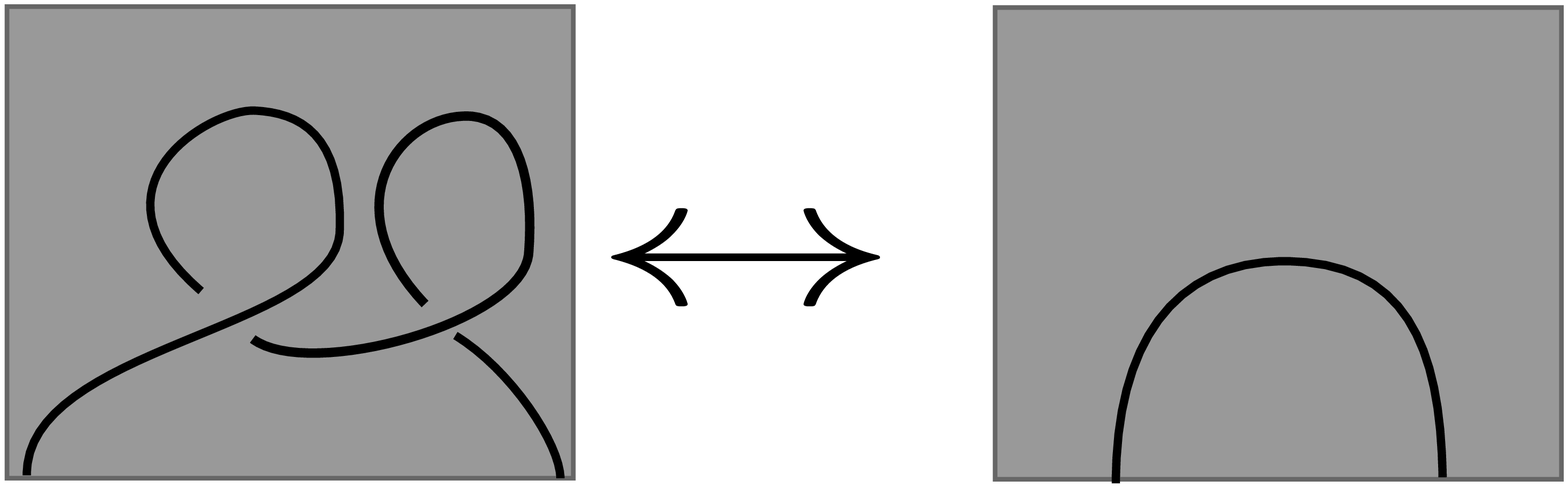}}$.
The compositions in $\BT_0$ and $\BT_1$ are induced by the compositions in $\BT$ after passing to the quotient. We thus have projection functors
$$ \BT \xrightarrow{p'} \BT_1 \xrightarrow{p} \BT_0$$
which are the identity on objects. For $1\leq i \leq q$, let $\theta_i^{a} \in \BT_1(a,a)$ be the bottom tangle 
$$\theta_i^a:= \adjustbox{valign=c}{\includegraphics[width=3cm]{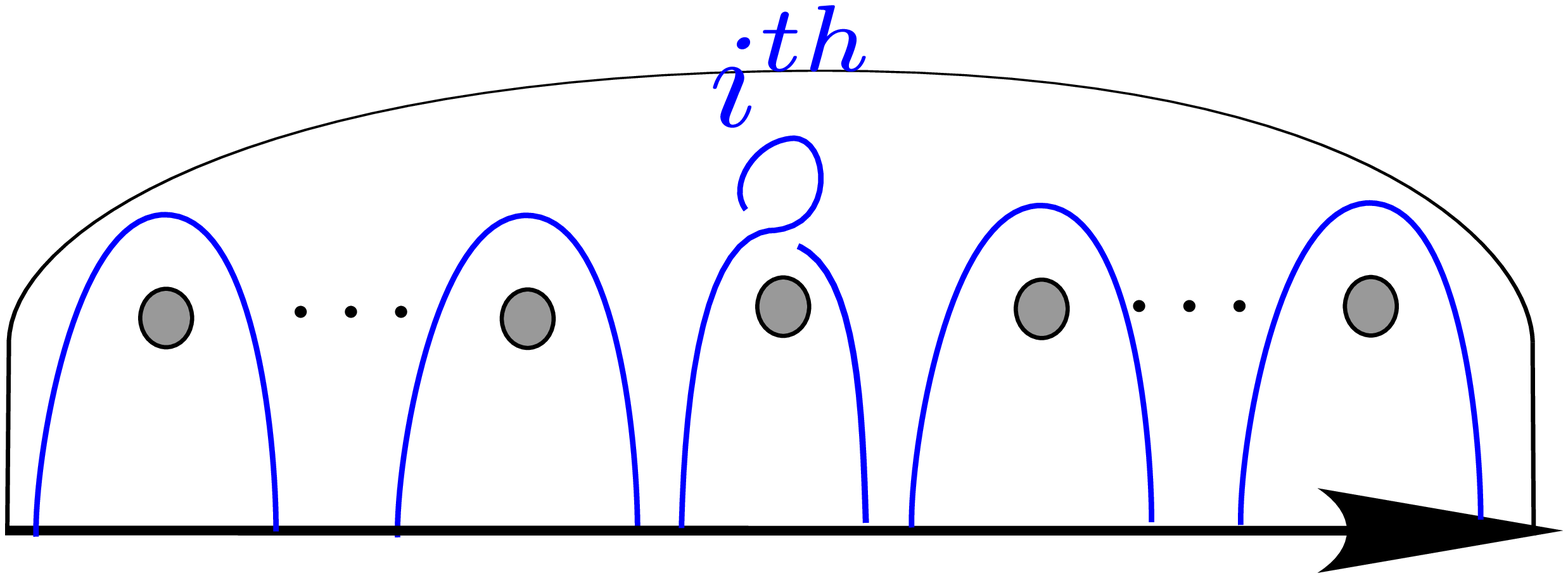}}. $$ Then $(\theta_i^a)^2=\id$ and the $\theta_i^a$ pairwise commute so they generate a group $G_a \cong (\mathbb{Z}/2\mathbb{Z})^a \subset \BT_1(a,a)$. This group naturally acts freely on the right of $\BT_1(a,b)$ and the quotient map $p: \BT_1(a,b) \to \BT_0(a,b)$ induces an isomorphism 
$$ \widetilde{p} : \quotient{\BT_1(a,b)}{G_a} \xrightarrow{\cong} \BT_0(a,b).$$
Define also a section $s: \BT_0(a,b) \hookrightarrow \BT_1(a,b)$ such that $s\circ p = \id$ by lifting a class $[\alpha] \in \BT_0(a,b)$ to the class $[\alpha_0] \in \BT_1(a,b)$ where the bottom tangle $\alpha_0$ is chosen such that each of its component has self-linking number $0$. The section $s$ defines an isomorphism $\BT_1(a,b)\cong \BT_0(a,b)\times G_a$. We will now denote a morphism $\mu \in \BT_1(a,b)$ by $\mu=(\mu_0, g)$ where $\mu_0=s(\mu)$ and $g\in G_a$ is the unique element such that $\mu_0 \cdot g = \mu$. Note that $s: \BT_0 \to \BT_1$ is a functor.

\end{notations}

Note that the right $\BT$-module $\restriction{\mathcal{O}[\mathcal{R}]}{\BT}$ passes to the quotient to a right $\BT_0$ module:  this follows from the facts that for $\mu : \mathbf{H}_a \to \mathbf{H}_b$ an embedding, the morphism $\mu_* : \mathcal{O}[\mathcal{R}_{\SL_2}(\mathbf{H}_a)] \to \mathcal{O}[\mathcal{R}_{\SL_2}(\mathbf{H}_b)]$ only depends on the morphism $\mu_* : \pi_1(\mathbf{H}_a)\to \pi_1(\mathbf{H}_b)$ i.e. only depends on the class of $\mu$ in $\BT_0(a,b)$. 
We defined during the proof of Lemma \ref{lemma_LKan_Classical} an explicit isomorphism $\mathcal{O}[\mathcal{R}_{\SL_2}(\mathbf{M})] \cong \restriction{\mathcal{O}[\mathcal{R}]}{\BT} \otimes_{\BT} \mathbb{Z}[\pi_M]$ where $\pi_M$ is a right $\BT_0$ module obtained as a quotient of $P_M$. Therefore one has  an isomorphism
$$ \kappa_{\mathbf{M}}: \mathcal{O}[\mathcal{R}_{\SL_2}(\mathbf{M})] \cong \restriction{\mathcal{O}[\mathcal{R}]}{\BT_0} \otimes_{\BT_0} \mathbb{Z}[\pi_M].$$
On the other hand, the right $\BT$-module $\restriction{\mathcal{S}_{+1}}{\BT}$ passes to the quotient to a right $\BT_1$ module: for $\mu : \mathbf{H}_a \to \mathbf{H}_b$ an embedding, the fact that the morphism $\mu_*: \mathcal{S}_{+1}(\mathbf{H}_a) \to \mathcal{S}_{+1}(\mathbf{H}_b)$ is invariant under the skein relations $\Crosspos \leftrightarrow \Crossneg$ and $\adjustbox{valign=c}{\includegraphics[width=2cm]{TwistReldouble.eps}}$ follows from the fact that the same skein relations hold in $\mathcal{S}_{+1}(\mathbf{M})$. By Lemma \ref{lemma_classic1}, one has an isomorphism
$$ \mathcal{S}_{+1}(\mathbf{M}) \cong \restriction{\mathcal{S}_{+1}}{\BT} \otimes_{\BT}  \mathbb{Z}[P_M] $$
which is explicited  in the proof of Theorem \ref{theorem_SkeinQRep} where it is denoted by $G^{-1}$. Let $\pi_M^{fr}$ be the left $\BT_1$ module sending $\mathbf{H}_n$ to the quotient of $P_n(\mathbf{M})$ by the relations $\Crosspos \leftrightarrow \Crossneg$ and $\adjustbox{valign=c}{\includegraphics[width=2cm]{TwistReldouble.eps}}$: so $\pi_M^{fr}$ is a quotient of $P_M$ and we have quotient maps (of left $\BT$-modules) $P_M \to \pi_M^{fr} \to \pi_M$. The isomorphism of Lemma \ref{lemma_classic1} can rewritten as 
$$ g_M : \mathcal{S}_{+1}(\mathbf{M}) \cong \restriction{\mathcal{S}_{+1}}{\BT_1} \otimes_{\BT_1}  \mathbb{Z}[\pi^{fr}_M].$$
In order to prove Theorem \ref{theorem_classical}, we are thus reduced to construct an isomorphism $$\restriction{\mathcal{S}_{+1}}{\BT_1} \otimes_{\BT_1}  \mathbb{Z}[\pi^{fr}_M] \cong \restriction{\mathcal{O}[\mathcal{R}]}{\BT_0} \otimes_{\BT_0} \mathbb{Z}[\pi_M].$$
By precomposing the functor $\restriction{\mathcal{S}_{+1}}{\BT_1} : \BT_1 \to \Mod_{\mathbb{Z}}$ with $s:\BT_0 \to \BT_1$, we can see $\restriction{\mathcal{S}_{+1}}{\BT_1} $ as a right $\BT_0$ module. We now define an explicit isomorphism between this right $\BT_0$-module and $ \restriction{\mathcal{O}[\mathcal{R}]}{\BT_0}$. 

Let $\beta_i \in P_1(\mathbf{H}_n)$ be the oriented bottom tangle depicted by $\beta_i:= \adjustbox{valign=c}{\includegraphics[width=2cm]{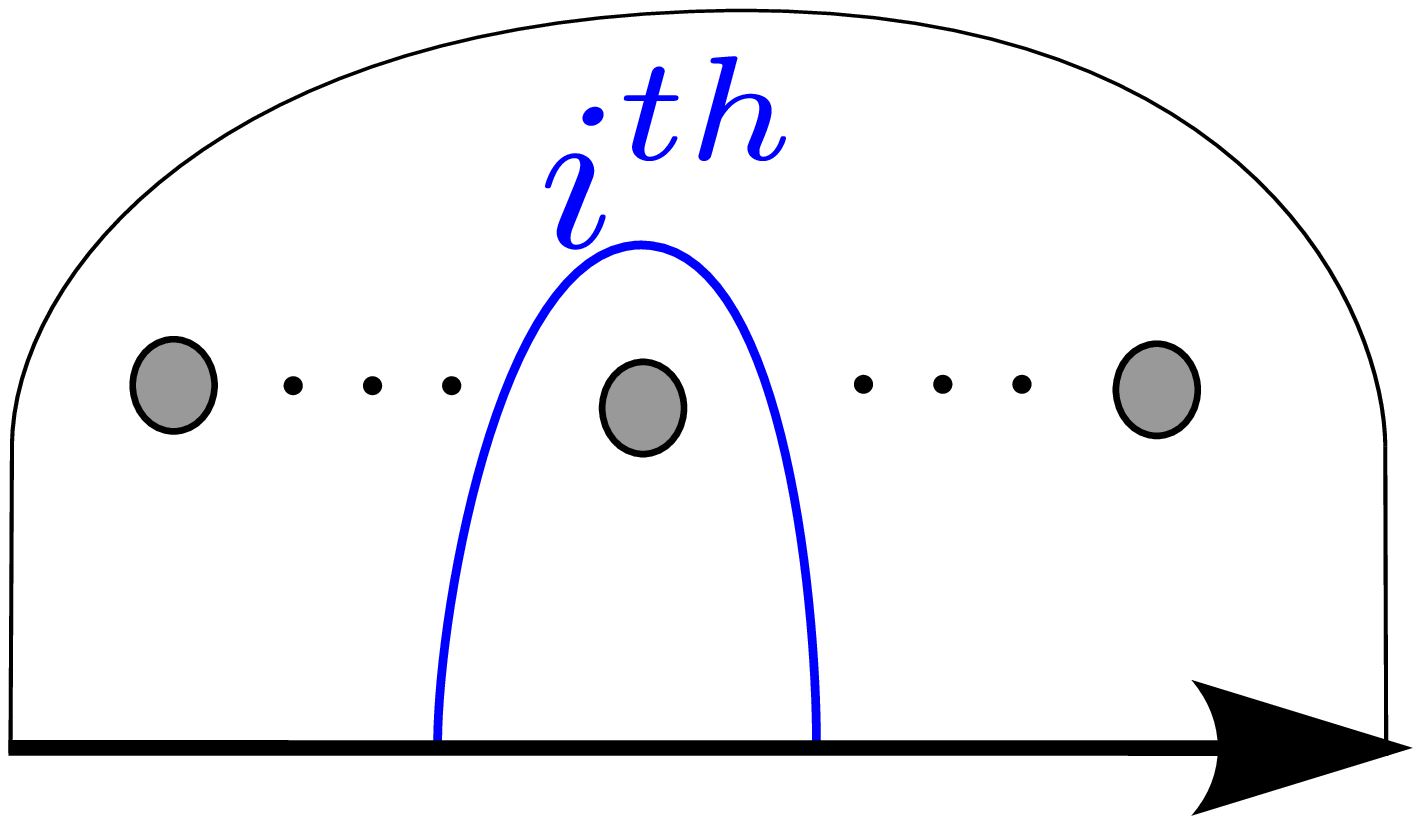}}$. 

\begin{lemma}\label{lemma_classical1}
For each $n\geq 0$, one has  isomorphisms $\omega_n: \mathcal{O}[\mathcal{R}_{\SL_2}(\mathbf{H}_n)] \xrightarrow{\cong} \mathcal{S}_{+1}(\mathbf{H}_n)$ of $\mathcal{O}[\SL_2]$ comodules algebras  characterized by the formula
$$ \omega_n \begin{pmatrix} X^{\beta_i}_{++} & X^{\beta_i}_{+-} \\ X^{\beta_i}_{-+} & X^{\beta_i}_{--} \end{pmatrix} := \begin{pmatrix} 0 & -1 \\ 1 & 0 \end{pmatrix} \begin{pmatrix} {\beta_i}_{++} & {\beta_i}_{+-} \\ {\beta_i}_{-+} & {\beta_i}_{--} \end{pmatrix}.$$
Moreover, for $\mu_0 \in \BT_0(a,b)$ then the following diagram commutes
$$ \begin{tikzcd}
\mathcal{O}[\mathcal{R}_{\SL_2}(\mathbf{H}_a)]
\ar[r, "(\mu_0)_*"] \ar[d, "\omega_a"] &
\mathcal{O}[\mathcal{R}_{\SL_2}(\mathbf{H}_b)] 
\ar[d, "\omega_b"] \\
\mathcal{S}_{+1}(\mathbf{H}_a) 
\ar[r, "s(\mu_0)_*"] &
\mathcal{S}_{+1}(\mathbf{H}_b) 
\end{tikzcd}
$$
\end{lemma}

\begin{proof}
The fact that the $\omega_n$ are isomorphisms of rings and of $\mathcal{O}[\SL_2]$-comodules is a particular (easy) case of \cite[Theorem $3.17$]{KojuQuesneyClassicalShadows} (see also \cite[Theorem $4.7$]{KojuPresentationSSkein}). By Theorem \ref{theorem_presBT},  every morphism in $\BT_0$ is obtained by composition and tensoring the generating morphisms $(\mu, \eta, \Delta, \epsilon, S^{\pm 1},  \theta^{\pm 1}) $ of Figure \ref{fig_BTHopfAlg}. So to prove the naturality of $\omega$, we need to prove the commutativity of the diagram in the particular cases where $s(\mu_0)$ is one of these morphisms. For $(\mu, \eta)$ this follows from the fact that $\omega_n$ are algebras morphisms. For $\Delta, \epsilon, S^{\pm 1}$ this follows from the computations made in the proof of Theorem \ref{theorem_skein_transmutation} to prove that $f$ is a morphism of braided Hopf algebra object. For $\theta^{\pm 1}$ this is obvious. This completes the proof.

\end{proof}

Recall from Notations \ref{notations_extension} the group $G_a\subset \BT_1(a,a)$ generated by the involutive elements $\theta_i^a$. Let us describe the right action of $G_a$ on $\mathcal{S}_{+1}(\mathbf{H}_a)$. First define $\Theta \in \Aut( \mathcal{O}[\SL_2])$ the automorphism sending the generator $x_{ij}$ for $i,j =\pm $ to $-x_{ij}$. Recall that as a $\mathbb{Z}$-module one has $\mathcal{S}_{+1}(\mathbf{H}_a) \cong \mathcal{S}_{+1}(\mathbf{H}_1)^{\otimes a} \cong (\mathcal{O}[\SL_2])^{\otimes a}$ and define $\Theta_i^a \in \Aut (\mathcal{S}_{+1}(\mathbf{H}_a))$ by $\Theta_i^a = \id^{\otimes i-1}\otimes \Theta \otimes \id^{\otimes a-i}$. Said differently, the automorphism $\Theta_i^a$ sends each generator $(\beta_i)_{\varepsilon, \varepsilon'}$ to $- (\beta_i)_{\varepsilon, \varepsilon'}$ and each $(\beta_j)_{\varepsilon, \varepsilon'}$ with $j\neq i$ to $+(\beta_j)_{\varepsilon, \varepsilon'}$. In general, for $g=\theta_1^{n_1}\ldots \theta_a^{n_a} \in G_a$, we write $\Theta_g:= (\Theta_1^a)^{n_1}\ldots (\Theta_a^a)^{n_a}$.

\begin{lemma}\label{lemma_classical2} For $g\in G_a$ and $x\in \mathcal{S}_{+1}(\mathbf{H}_a)$, one has $x\cdot g = \Theta_g(x)$, where $x \cdot g$ denotes the right action coming from the functoriality of $\mathcal{S}_{+1}$.
\end{lemma}

\begin{proof}
We need to prove that for each generator $(\beta_i)_{\varepsilon, \varepsilon'}$ of $\mathcal{S}_{+1}(\mathbf{H}_a)$ and for each generator $\theta_j^a$ of $G_a$, one has $(\beta_i)_{\varepsilon, \varepsilon'} \cdot \theta_i^a = (-1)^{\delta_{i,j}} (\beta_i)_{\varepsilon, \varepsilon'}$. This follows from the skein relation $\adjustbox{valign=c}{\includegraphics[width=2cm]{TwistRel.eps}}$ in $\mathcal{S}_{+1}(\mathbf{H}_a)$.

\end{proof}

Note that the left $\BT$ module $\eta$ of Definition \ref{def_SpinFunction} passes to the quotient to a left $\BT_1$ module and that 
a spin function is, by definition, a morphism $w: \pi_M^{fr} \to \eta$ of left $\BT_1$ modules. 
\par Let $[\beta_i] \in \mathrm{H}_1(\mathbf{H}_a; \mathbb{Z}/2\mathbb{Z})$ be the homology class of the closed curve obtained from $\beta_i$ by joining its endpoints, i.e. the simple closed curve encircling the $i$-th hole of the punctured disc $\mathbb{D}_a$. Let $\varphi_i \in  \mathrm{H}^1(\mathbf{H}_a; \mathbb{Z}/2\mathbb{Z})$ be the dual element sending $[\beta_j]$ to $\delta_{ij}$. 
Consider the isomorphism $G_a \cong \mathrm{H}^1(\mathbf{H}_a; \mathbb{Z}/2\mathbb{Z})$ sending $\theta_i^a$ to $\varphi_i$. Using this isomorphism, for $\alpha \in \pi_M^{fr}(\mathbf{H}_a)$, we can consider the automorphism $\Theta_{w_1(\alpha)} \in \Aut(\mathcal{S}_{+1}(\mathbf{H}_a))$ and the group element $\theta_{w_1(\alpha)}\in G_a$.

\begin{proof}[Proof of Theorem \ref{theorem_classical}]
Let us define an isomorphism
$$ f : \restriction{\mathcal{S}_{+1}}{\BT_1} \otimes_{\BT_1}  \mathbb{Z}[\pi^{fr}_M] \cong \restriction{\mathcal{O}[\mathcal{R}]}{\BT_0} \otimes_{\BT_0} \mathbb{Z}[\pi_M].$$
Identify the tensor products as 
\begin{align*}
& \restriction{\mathcal{S}_{+1}}{\BT_1} \otimes_{\BT_1}  \mathbb{Z}[\pi^{fr}_M] = \quotient{\left( \oplus_{n \geq 0} \mathcal{S}_{+1}(\mathbf{H}_n) \otimes_{\mathbb{Z}} \mathbb{Z}[\pi^{fr}_M(\mathbf{H}_n)] \right)}{\mathcal{I}_1}, \\
& \restriction{\mathcal{O}[\mathcal{R}]}{\BT_0} \otimes_{\BT_0} \mathbb{Z}[\pi_M] = \quotient{\left( \oplus_{n \geq 0} \mathcal{O}[\mathcal{R}_{\SL_2}(\mathbf{H}_n)] \otimes_{\mathbb{Z}} \mathbb{Z}[\pi_M(\mathbf{H}_n)] \right)}{\mathcal{I}_0},
\end{align*}
where $\mathcal{I}_i$ is spanned by elements of the form $u \cdot \mu \otimes v - u\otimes \mu\cdot v$ for $\mu$ a morphism in $\BT_i$. Let us slightly re-write the first quotient as
$$ \restriction{\mathcal{S}_{+1}}{\BT_1} \otimes_{\BT_1}  \mathbb{Z}[\pi^{fr}_M] = \quotient{\left( \oplus_{n \geq 0} \mathcal{S}_{+1}(\mathbf{H}_n) \otimes_{\mathbb{Z}[G_n]} \mathbb{Z}[\pi^{fr}_M(\mathbf{H}_n)] \right)}{\mathcal{I}'_1}$$
where $\mathcal{I}_1'$ is spanned by elements $u \cdot \mu \otimes v - u\otimes \mu\cdot v$ where $\mu$ is a morphism in $\BT_1$ such that $s(\mu)=\mu$.
\par
For $n\geq 0$, we define an isomorphism 
$$f_n : \mathcal{S}_{+1}(\mathbf{H}_n) \otimes_{\mathbb{Z}[G_n]} \mathbb{Z}[\pi^{fr}_M(\mathbf{H}_n)] \xrightarrow{\cong}  \mathcal{O}[\mathcal{R}_{\SL_2}(\mathbf{H}_n)] \otimes_{\mathbb{Z}} \mathbb{Z}[\pi_M(\mathbf{H}_n)] $$
by $$f_n ( x\otimes y):=  \omega_n^{-1}(x\cdot w_n(y)) \otimes p(y). $$ Here $p: \pi^{fr}_M(\mathbf{H}_n) \to \pi_M(\mathbf{H}_n)$ is the quotient map and $w_n : \pi_M^{fr}(\mathbf{H}_n) \to \mathrm{H}^1(\mathbf{H}_a; \mathbb{Z}/2\mathbb{Z})\cong G_n$ is the spin function. 
The fact that $f_n$ is well-defined, i.e. that $f_n(x\cdot g \otimes y)=f_n(x\otimes g\cdot y)$ for all $g\in G_n$, comes from the naturality of $w_n$ with respect to the morphisms in $G_n$ together with the fact that $w_1\left(\adjustbox{valign=c}{\includegraphics[width=1cm]{BT_Theta.eps}}  \right)  =1$.
 The fact that $f_n$ is an isomorphism comes from the fact that $\omega_n$ is an isomorphism together with the fact that $\pi_M(\mathbf{H}_n)=\quotient{\pi_M^{fr}(\mathbf{H}_n)}{G_n}$. The fact that $f_n$ is $\mathcal{O}[\SL_2]$ equivariant follows from the fact that $\omega_n$ is equivariant.
\par 
Let $\widetilde{f}:= \oplus_{n\geq 0} f_n$ and let us prove that $\widetilde{f}(\mathcal{I}'_1)= \mathcal{I}_0$. Consider an element $X= u \cdot \mu \otimes v - u\otimes \mu\cdot v \in \mathcal{I}'_1$ with $\mu\in  \BT_1(a,b)$ and $u\in \mathcal{S}_{+1}(\mathbf{H}_b)$, $v\in \pi_M^{fr}(\mathbf{H}_a)$. Then
\begin{align*}
f_a (u \cdot \mu \otimes v) &:= \omega_a^{-1}( u \cdot \mu \cdot \theta_{w_a(v)}) \otimes p(v) & {} \\
 {} & = \omega_a^{-1}(u \cdot \theta_{w_b(\mu\cdot v)} \cdot \mu) \otimes p(v) & \mbox{ by naturality of }w \\
 {} &= \omega_a^{-1}(u \cdot \theta_{w_b(\mu \cdot v)})\cdot \mu \otimes p(v) & \mbox{ by  Lemma \ref{lemma_classical1} and the fact that }s(\mu)=\mu.
 \end{align*}
 Thus,  writing $x:= u\cdot \theta_{w_b(\mu \cdot v)}$ and $y:=p(v)$ one has 
 $$ f_a( u\cdot \mu \otimes v) - f_b(u\otimes \mu \cdot v) = x \cdot \mu \otimes y - x \otimes \mu \cdot y \in \mathcal{I}_0.$$
 Therefore, we have proved the inclusion  $\widetilde{f}(\mathcal{I}'_1)\subset \mathcal{I}_0$. To prove the reverse inclusion $\mathcal{I}_0 \subset \widetilde{f}(\mathcal{I}'_1)$, consider $x \in \mathcal{O}[\mathcal{R}_{\SL_2}(\mathbf{H}_b)]$, $y\in \pi_M(\mathbf{H}_a)$ and $\mu_0 \in \BT_0(a,b)$ and consider the generator $X:= x\cdot \mu_0 \otimes y - x\otimes \mu_0\cdot y \in \mathcal{I}_0$. Let $\mu\in \BT_1(a,b)$ be a lift of $\mu_0$ such that $s(\mu)=\mu$ and choose $v \in \pi_M^{fr}(\mathbf{H}_a)$ a lift of $y$ such that $w_a(v)=0$. Set $u:= \omega_b(x)$ and $Y:= u \cdot \mu \otimes v - u \otimes \mu \cdot v \in \mathcal{I}_1'$. Then by definition one has $\widetilde{f}(Y)=X$ so $\mathcal{I}_0 \subset \widetilde{f}(\mathcal{I}'_1)$. We thus have proved that $\widetilde{f}(\mathcal{I}'_1)= \mathcal{I}_0$ so the isomorphism $\widetilde{f}$ induces an isomorphism 
 $$ f : \restriction{\mathcal{S}_{+1}}{\BT_1} \otimes_{\BT_1}  \mathbb{Z}[\pi^{fr}_M] \cong \restriction{\mathcal{O}[\mathcal{R}]}{\BT_0} \otimes_{\BT_0} \mathbb{Z}[\pi_M]$$
 which is $\mathcal{O}[\SL_2]$-equivariant since $\widetilde{f}$ is equivariant as well.
 Define the isomorphism $\eta_w^{-1} :  \mathcal{S}_{+1}(\mathbf{M}) \xrightarrow{\cong} \mathcal{O}[\mathcal{R}_{\SL_2}(\mathbf{M})]$ as the composition
 $$ \eta_w^{-1} :  \mathcal{S}_{+1}(\mathbf{M}) \xrightarrow{g_M} \restriction{\mathcal{S}_{+1}}{\BT_1} \otimes_{\BT_1}  \mathbb{Z}[\pi^{fr}_M]  \xrightarrow{f} \mathcal{O}[\mathcal{R}_{\SL_2}(\mathbf{M})] \xrightarrow{\kappa_M^{-1}}  \mathcal{O}[\mathcal{R}_{\SL_2}(\mathbf{M})].$$
 That $\eta_w$ is equivariant follows from the fact that each above map is equivariant. It remains to prove the  explicit formula for a stated arc given in Theorem \ref{theorem_classical}. Write $C_{++}=C_{--}:=0$, $C_{-+}:=-C_{+-}=1$ so that the formula we need to prove writes $\eta_w(X_{ij}^{\gamma})= (-1)^{w_1(\gamma)} \sum_{k=\pm} C_{ik} \gamma_{kj}$. The arc $\gamma$ defines a $1$-bottom tangles $\gamma\in P_1(M)$ and we denote by $\gamma_0, \gamma_1$ its images in the quotients $\pi_M(\mathbf{H}_1)$ and $\pi_M^{fr}(\mathbf{H}_1)$ respectively. Then $\kappa_M$ (defined in the proof of  Lemma \ref{lemma_LKan_Classical}) sends $X_{ij}^{\gamma}$ to the class $[X^{\beta_1}_{ij}\otimes \gamma_0]$. The isomorphism $g_M$ (defined from the isomorphism $G^{-1}$ in the proof of Theorem \ref{theorem_SkeinQRep} by tensoring by $\mathbb{Z}$) sends a stated arc $\gamma_{ab}$ to the class $[(\beta_1)_{ab} \otimes \gamma_1]$. Now the isomorphism $\omega_1$ in Lemma \ref{lemma_classical1} reads $\omega_1(X^{\beta_1}_{ij})=\sum_k C_{ik} (\beta_1)_{kj}$ so the desired formula follows from the equalities
 $$ \eta_M(X_{ij}^{\gamma})= g_M^{-1}\circ f \circ \kappa_M(X_{ij}^{\gamma})= g_M^{-1}\circ f ([X^{\beta_1}_{ij} \otimes \gamma_0]) = (-1)^{w_1(\gamma)} g_M^{-1}(\sum_{k= \pm} C_{ik} [(\beta_1)_{kj} \otimes \gamma_1]) = (-1)^{w_1(\gamma)} \sum_{k=\pm} C_{ik} \gamma_{kj}, $$
 where we used that $(\beta_1)_{kj} \cdot \theta_{w_1(\gamma)} = \Theta_{w_1(\gamma)}((\beta_1)_{kj})= (-1)^{w_1(\gamma)}(\beta_1)_{kj}$ as proved in Lemma \ref{lemma_classical2}. This concludes the proof.
 
\end{proof}

\section{Quantum Van Kampen theorems}\label{sec4}

\subsection{Quantum Van Kampen for quantum fundamental groups}\label{sec_VK_FG}

Recall that $\mathbb{D}^2$ is the unit disc of $\mathbb{R}^2$, $h: \mathbb{D}^2 \to [-1,1]$ is the projection on the $y$ axis  and let $\partial_{+}\mathbb{D}, \partial_{-}\mathbb{D} \subset \partial \mathbb{D}^2$ and $\mathbb{D}^{+}, \mathbb{D}^- \subset \mathbb{D}^2$ be the subsets of points $p$ for which $h(p)\geq 0$ and $h(p)\leq 0$ respectively. For $\mathbf{M}=(M, \iota_M) \in \mathcal{M}_{\con}^{(1)}$, we write $\partial_{\pm}\mathbb{D}_M:= \iota_M(\partial_{\pm} \mathbb{D})$ and $\mathbb{D}_M^{\pm}:= \iota_M(\mathbb{D}^{\pm})$.   
\par
Let $\mathbf{M}_1, \mathbf{M}_2 \in \mathcal{M}_{\con}^{(1)}$, $\mathbf{\Sigma}=(\Sigma, a) \in \MS$ a connected marked surface with a single boundary arc $a$. Consider oriented embeddings $\phi_1: \overline{\Sigma} \hookrightarrow \partial M_1$ and $\phi_2: \Sigma \hookrightarrow \partial M_2$ sending $a$ to $\partial_- \mathbb{D}_{M_1}$ and $\partial_+\mathbb{D}_{M_2}$ respectively. Here $\overline{\Sigma}$ is $\Sigma$ with the opposite orientation. 
\begin{definition}
Let $\mathbf{M}_1\cup_{\Sigma} \mathbf{M}_2 \in \mathcal{M}_{\con}^{(1)}$ be the marked $3$-manifold where 
$$M_1\cup_{\Sigma} M_2= \quotient{ M_1 \bigsqcup M_2}{(\phi_1(p)\sim \phi_2(p), p\in \Sigma)}$$ and the boundary disc $\mathbb{D}_{M_1\cup_{\Sigma} M_2}$ is obtained by gluing $\mathbb{D}_{M_1}$ with $\mathbb{D}_{M_2}$ by identifying $\phi_1(a)=\partial_-\mathbb{D}_{M_1}$ with $\phi_2(a)=\partial_+ \mathbb{D}_{M_2}$.
\end{definition}
By construction, there is a natural projection map $\pi^0 : \mathbf{M}_1\bigsqcup \mathbf{M}_2 \to \mathbf{M}_1\cup_{\Sigma} \mathbf{M}_2$. Since we prefer to work in the monoidal category $(\mathcal{M}^{(1)}_{\con}, \wedge)$, let us define a morphism $\pi : \mathbf{M}_1 \wedge \mathbf{M}_2 \to  \mathbf{M}_1\cup_{\Sigma} \mathbf{M}_2$ such that the diagram 
$$ \begin{tikzcd} 
M_1 \bigsqcup M_2 \ar[rd, "\pi^0"] \ar[dd, "\iota"] & {} \\
{} & M_1\cup_{\Sigma} M_2 \\
M_1\wedge M_2 \ar[ru, "\pi"] & {}
\end{tikzcd}$$
commutes up to isotopy (i.e. such that $\pi^0$ is isotopic to $\pi\circ \iota$), where $\iota$ is the natural inclusion. Recall that $\mathbb{T}$ is a ball $\mathbb{B}^3$ with three boundary discs and that  $M_1\wedge M_2$ is obtained from $M_1\bigsqcup M_2\bigsqcup \mathbb{T}$ by gluing $\mathbb{D}_{M_1}$ to one disc, say $e_1$ of $\mathbb{T}$ and by gluing $\mathbb{D}_2$ to another disc, say $e_2$ of $\mathbb{T}$. The closed intervals $\partial_- \mathbb{D}_{M_1}$ and $\partial_+ \mathbb{D}_{M_2}$ are glued along closed intervals $\partial_- e_1$ and $\partial_+e_2$ respectively. Isotope $e_2$ to a disc $e_2'$  such that $\partial_- e_1$ coincides with  $\partial_+e_2'$ and let $M'$ be the $3$-manifold obtained by gluing $M_1$ to $e_1$ and $M_2$ to $e_2'$ as in Figure \ref{fig_QuotientMap}, so that $M'$ is isotopic to $M_1\wedge M_2$. The manifold $M:= \quotient{M'}{(\phi_1(p)\sim \phi_2(p), p\in \Sigma)}$ is isotopic to $M_1\cup_{\Sigma} M_2$ so the quotient map 
$M'\to M$ defines a surjective morphism $\pi : \mathbf{M}_1 \wedge \mathbf{M}_2 \to  \mathbf{M}_1\cup_{\Sigma} \mathbf{M}_2$ and clearly  $\pi^0$ is isotopic to $\pi\circ \iota$.

\begin{definition}\label{def_gluing_morphism} The morphism 
$\pi : \mathbf{M}_1 \wedge \mathbf{M}_2 \to  \mathbf{M}_1\cup_{\Sigma} \mathbf{M}_2$ is called the \textit{gluing morphism}.
\end{definition}

 \begin{figure}[!h] 
\centerline{\includegraphics[width=16cm]{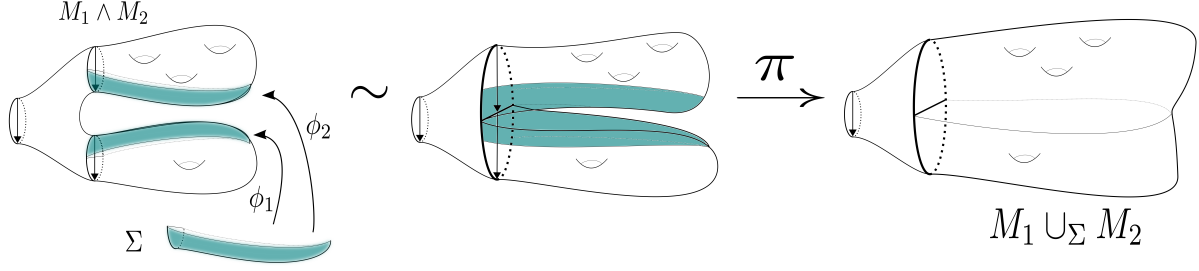} }
\caption{An illustration of the gluing morphism $\pi : \mathbf{M}_1 \wedge \mathbf{M}_2 \to  \mathbf{M}_1\cup_{\Sigma} \mathbf{M}_2$ .} 
\label{fig_QuotientMap} 
\end{figure} 

\begin{lemma}\label{lemma_associativity} The gluing operation is associative up to isotopy, i.e. one has $(\mathbf{M}_1\cup_{\Sigma}\mathbf{M}_2)\cup_{\Sigma'} \mathbf{M}_3 = \mathbf{M}_1\cup_{\Sigma}(\mathbf{M}_2\cup_{\Sigma'} \mathbf{M}_3)$ and if $f_{(12)}:\mathbf{M}_1\wedge \mathbf{M}_2 \to \mathbf{M}_1 \cup_{\Sigma}\mathbf{M}_2$, $f_{(23)}: \mathbf{M}_2\wedge \mathbf{M}_3 \to \mathbf{M}_2\cup_{\Sigma'} \mathbf{M}_3$, 
$f_{(12,3)}: (\mathbf{M}_1\cup_{\Sigma}\mathbf{M}_2)\wedge \mathbf{M}_3 \to \mathbf{M}_1\cup_{\Sigma}\mathbf{M}_2\cup_{\Sigma'} \mathbf{M}_3$ 
and $f_{(1,23)}: \mathbf{M}_1\wedge (\mathbf{M}_2\cup_{\Sigma'} \mathbf{M}_3) \to  \mathbf{M}_1\cup_{\Sigma}\mathbf{M}_2\cup_{\Sigma'} \mathbf{M}_3$
 are the four gluing morphisms, then $f_{(12,3)}\circ (f_{(12)}\wedge \id_{M_3})$ is isotopic to $f_{(1,23)}\circ (\id_{M_1}\wedge f_{(23)})$.
 \end{lemma}
 
 \begin{proof} The lemma is a straightforward consequence of the definitions.
 \end{proof}

\par By replacing $\mathbf{M}_1$ and $\mathbf{M}_2$ by $\mathbf{\Sigma}\times I$ in the preceding discussion, while gluing $\Sigma \times \{-1\}$ of the first copy to $\Sigma \times \{+1\}$ in the second copy, we get a gluing morphism 
$\mu : (\mathbf{\Sigma}\times I)\wedge (\mathbf{\Sigma}\times I) \to \mathbf{\Sigma}\times I$ which endows $\mathbf{\Sigma}\times I$ with a structure of algebra in $(\mathcal{M}^{(1)}_{\con}, \wedge)$. Similarly, by gluing $M_1$ to $\Sigma\times I$ while identifying $\phi_1(p)\sim (p, +1)$ for $p\in \Sigma$, we obtain a gluing morphism $\nabla_1: \mathbf{M}_1\wedge \mathbf{\Sigma}\times I \to \mathbf{M}_1$ turning $\mathbf{M}_1$ into a right $\mathbf{\Sigma}\times I$ module in 
$(\mathcal{M}^{(1)}_{\con}, \wedge)$. Similarly, by gluing $\mathbf{\Sigma}\times I$ to $\mathbf{M}_2$ while identifying $\phi_2(p)\sim (p, -1)$, we get a gluing morphism $\nabla_2 : \mathbf{\Sigma}\times I \wedge \mathbf{M}_2 \to \mathbf{M}_2$ turning $\mathbf{M}_2$ into a left $\mathbf{\Sigma}\times I$ module. Note that the diagram 
$$ \begin{tikzcd}
 \mathbf{M}_1\wedge (\mathbf{\Sigma} \times I) \wedge \mathbf{M}_2
 \ar[r,shift left=.75ex,"\nabla_1\wedge \id"]
  \ar[r,shift right=.75ex,swap,"\id \wedge \nabla_2"] 
&   \mathbf{M}_1\wedge \mathbf{M}_2 \ar[r, "\pi"] & \mathbf{M}_1\cup_{\Sigma} \mathbf{M}_2
\end{tikzcd}$$
commutes up to isotopy by Lemma \ref{lemma_associativity}. By abuse of notations, for $i=1,2$, still denote by $\nabla_i$ the image of $\nabla_i$ by the quantum fundamental group functor $P$ and let us identify the elements $P_{M_1\wedge M_2}\cong P_{M_1}\otimes_D P_{M_2}$ and $P_{M_1\wedge \Sigma \wedge M_2} \cong P_{M_1} \otimes_D P_{\Sigma}\otimes_D P_{M_2}$ in $\widehat{\BT}$ using the isomorphism of (the proof of) Lemma \ref{lemma_monoidal}.

\begin{definition}
Let 
 $P_{M_1} \otimes_{P_{\Sigma}} P_{M_2}\in \widehat{\BT}$ be the coequalizer
$$ \begin{tikzcd}
 P_{M_1} \otimes_D P_{\Sigma} \otimes_D P_{M_2} 
 \ar[r,shift left=.75ex,"\nabla_1\otimes \id"]
  \ar[r,shift right=.75ex,swap,"\id \otimes \nabla_2"] 
&  P_{M_1} \otimes_D P_{M_2} \ar[r] & P_{M_1}\otimes_{P_{\Sigma}} P_{M_2}.
\end{tikzcd}$$
\end{definition}

Since the diagram 
$$ \begin{tikzcd}
 P_{M_1} \otimes_D P_{\Sigma} \otimes_D P_{M_2} 
 \ar[r,shift left=.75ex,"\nabla_1\otimes \id"]
  \ar[r,shift right=.75ex,swap,"\id \otimes \nabla_2"] 
&  P_{M_1} \otimes_D P_{M_2} \ar[r, "P(\pi)"] & P_{M_1\cup_{\Sigma}M_2}
\end{tikzcd}$$
commutes, there exists a unique morphism $\kappa : P_{M_1} \otimes_{P_{\Sigma}} P_{M_2} \to P_{M_1\cup_{\Sigma}M_2}$ making commuting the obvious diagram.

The following was proposed without proof by Habiro in \cite{Habiro_QCharVar}:
\begin{theorem}\label{theorem_QVK_FG}(Quantum Van Kampen theorem for quantum fundamental groups)
\\ The morphism $\kappa : P_{M_1} \otimes_{P_{\Sigma}} P_{M_2} \to P_{M_1\cup_{\Sigma}M_2}$ is an isomorphism.
\end{theorem}

The proof of Theorem \ref{theorem_QVK_FG} is the most technical part of the paper and will be cut into several lemmas.
Let us first introduce some terminology. Write $M:= M_1\cup_{\Sigma}M_2$ for simplicity. Consider the embedding $\iota_{M_1\wedge M_2}: M_1\wedge M_2 \hookrightarrow M$ defined as the composition
$$ \iota_{M_1\wedge M_2}: M_1\wedge M_2 \subset M_1\wedge (\Sigma\times I) \wedge M_2 \xrightarrow{\pi \circ (\nabla_1\wedge \id)} M.$$
Here $M_1\wedge (\Sigma\times I) \wedge M_2 $ is, by definition, a ball $\mathbb{B}^3$ with three discs on the boundary to which we have attached $M_1, M_2, \Sigma \times I$ and we see $M_1\wedge M_2 \subset M_1\wedge (\Sigma\times I) \wedge M_2$ as the union of $\mathbb{B}^3$ with $M_1, M_2$ only. Define also an embedding $\iota_{\Sigma}$  as the composition:
$$ \iota_{\Sigma} : \Sigma\times I \subset M_1\wedge (\Sigma\times I) \wedge M_2 \xrightarrow{\pi \circ (\nabla_1\wedge \id)} M.$$
So $M$ is the union $M= \iota_{M_1\wedge M_2}(M_1\wedge M_2) \cup \iota_{\Sigma}(\Sigma\times I)$.
Write $\mathbb{B}_M := \pi \circ (\nabla_1\wedge \id)(\mathbb{B}^3)$, so $\mathbb{B}_M \subset M$ is a ball containing the boundary disc $\mathbb{D}_M$ and contained in the image of $\iota_{M_1\wedge M_2}$. We can further decompose $M$ as the union $M= \iota_{M_1\wedge M_2}(M_1) \cup \iota_{M_1\wedge M_2}(M_2) \cup  \iota_{\Sigma}(\Sigma\times I) \cup \mathbb{B}_M$ as illustrated in Figure \ref{fig_decomp_cheloud}.
 \begin{figure}[!h] 
\centerline{\includegraphics[width=10cm]{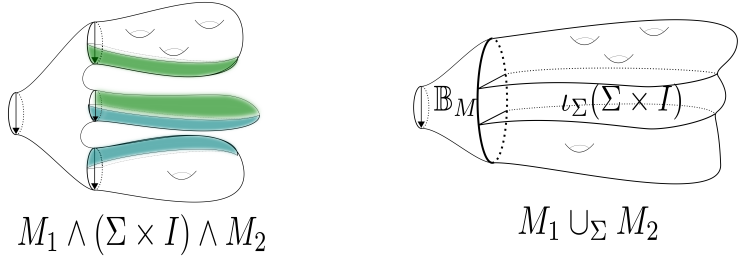} }
\caption{An illustration of the decomposition $M= \iota_{M_1\wedge M_2}(M_1) \cup \iota_{M_1\wedge M_2}(M_2) \cup  \iota_{\Sigma}(\Sigma\times I) \cup \mathbb{B}_M$ .} 
\label{fig_decomp_cheloud} 
\end{figure} 
 Let $\overline{M}_1:= \iota_{M_1\wedge M_2}(M_1) \cup  \iota_{\Sigma}(\Sigma\times I) \cup \mathbb{B}_M \subset M$ and $\overline{M}_2:= \iota_{M_1\wedge M_2}(M_2) \cup  \iota_{\Sigma}(\Sigma\times I) \cup \mathbb{B}_M \subset M$. So $\overline{M}_1\cup \overline{M}_2 = M$ and $\overline{M}_1\cap \overline{M}_2= \mathbb{B}_M \cup  \iota_{\Sigma}(\Sigma\times I) $. For $n\geq 0$, by abuse of notations, we denote by $P_n (M_1\wedge M_2)$ the set of isotopy classes of $n$-bottom tangles in $\iota_{M_1\wedge M_2}(M_1\wedge M_2)$, i.e. we now see $M_1\wedge M_2$ as a submanifold of $M$ via $\iota_{M_1\wedge M_2}$.

\begin{definition} For $n\geq 1$, and $\alpha, \beta \in P_n(M_1\wedge M_2)$, write: 
\begin{enumerate}
\item $\alpha \sim_{\phi} \beta$ if $\alpha$ and $\beta$ are related by a composition of isotopies, each having support either in $\overline{M}_1$ or $\overline{M}_2$; 
\item $\alpha \sim_{\psi} \beta$ if $\alpha$ and $\beta$ are isotopic in $M$.
\end{enumerate}
\end{definition}

Identify the spaces $P_M(\mathbf{H}_n)$ and $P_n(M)$ using Lemma \ref{lemma_BT} so that $P_{M_1}\otimes_{P_{\Sigma}}P_{M_2} (\mathbf{H}_n)$ becomes identified with the coequalizer:
$$ \begin{tikzcd}
 P_n(\mathbf{M}_1\wedge (\mathbf{\Sigma} \times I) \wedge \mathbf{M}_2)
 \ar[r,shift left=.75ex,"\nabla_1\wedge \id"]
  \ar[r,shift right=.75ex,swap,"\id \wedge \nabla_2"] 
&   P_n(\mathbf{M}_1\wedge \mathbf{M}_2) \ar[r, "P(\pi)"] & P_{M_1}\otimes_{P_{\Sigma}}P_{M_2} (\mathbf{H}_n).
\end{tikzcd}$$
Here $\nabla_i$ has image $\overline{M}_i$, so we have an isomorphism
\begin{equation}\label{eq_id1}
 \varphi: P_{M_1}\otimes_{P_{\Sigma}}P_{M_2} (\mathbf{H}_n) \cong \quotient{P_n(\mathbf{M}_1\wedge \mathbf{M}_2)}{\sim_{\phi}}.
 \end{equation}
The embedding $\iota_{M_1\wedge M_2}$ induces a map $\Psi:= (\iota_{M_1\wedge M_2})^*: P_n(M_1\wedge M_2) \to P_n(M)$ such that $\Psi(\alpha) = \Psi(\beta) \Leftrightarrow \alpha \sim_{\psi} \beta$.
Let $\Sigma^0:= \iota_{\Sigma}(\Sigma\times \{0\}) \subset M$ and $\Sigma^0\times I:= \iota_{\Sigma}(\Sigma \times I)$.
 To prove the surjectivity of $\Psi$, we introduce the: 

\begin{definition}\label{def_push}
Let $\alpha$ be an $n$ bottom tangle in $M_1\wedge M_2 \subset M$ which is transverse to $\Sigma^0$.  Let $v\in \alpha \cap \Sigma^0$, so $v=\iota_{\Sigma}(v_0)$ for $v_0 \in \Sigma$,  and suppose that the connected component $\alpha(v)$ of $\alpha \cap (\Sigma^0 \times I)$ is the straight line $\iota_{\Sigma}(v_0\times I)$ with a fixed framing of the form $\iota_{\Sigma}(\overrightarrow{v}_0\times I)$ for $\overrightarrow{v}_0$ a fixed vector in the unitary tangent bundle $U_{v_0}\Sigma$ of $v_0$. Let $w\in \mathbb{B}_M$ and $\gamma: v \to w$ a smooth path in $\Sigma_0\cup \mathbb{B}_M$ which only intersects $\alpha$ in $v$. 
\par The \textit{push} of $\alpha$ along $\gamma$ is the bottom tangle $\alpha'$ obtained from $\alpha$ by replacing the framed arc $\alpha(v)$ by an arc of the form $\alpha'(v)=\gamma^{-1}\gamma$, i.e. $\alpha'(v)$ goes from $\iota(v_0, -1)$ to $w$ along $\gamma$ and comes back to $\iota(v_0, +1)$ along $\gamma^{-1}$ so that $\alpha'(v)$ does not intersect $\Sigma_0$ in Figure \ref{fig_push}. The framing is chosen such that $\alpha(v)$ and $\alpha'(v)$ are isotopic, so are $\alpha$ and $\alpha'$.
\end{definition}

 \begin{figure}[!h] 
\centerline{\includegraphics[width=10cm]{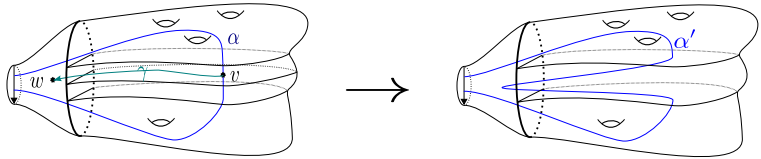} }
\caption{An illustration of the pushing operation.} 
\label{fig_push} 
\end{figure} 

\begin{lemma} The map  $\Psi: P_n(M_1\wedge M_2) \to P_n(M)$ is surjective. Therefore, $\Psi$ induces a bijection:
\begin{equation}\label{eq_id2}
P_M(\mathbf{H}_n) \cong \quotient{P_n(\mathbf{M}_1\wedge \mathbf{M}_2)}{\sim_{\psi}}.
\end{equation}
\end{lemma}

\begin{proof}
Starting with a $n$-bottom tangle $\alpha$ in $M$, we first isotope it such that it intersects $\Sigma^0$ transversally as in Definition \ref{def_push} and then push one-by-one each point of $\alpha \cap \Sigma^0$ inside $\mathbb{B}_M$ along arbitrary paths in order to get a bottom tangle $\alpha'$ such that: $(1)$ $\alpha'$ is isotopic to $\alpha$ and $(2)$ $\alpha'$ is in the image of $\iota_{M_1\wedge M_2}$. This proves the surjectivity.
\end{proof}

Obviously, $\alpha \sim_{\phi} \beta$ implies $\alpha \sim_{\psi} \beta$ so we get a surjective map $\quotient{P_n(\mathbf{M}_1\wedge \mathbf{M}_2)}{\sim_{\phi}} \to \quotient{P_n(\mathbf{M}_1\wedge \mathbf{M}_2)}{\sim_{\psi}}$ which, using the isomorphisms of Equations \eqref{eq_id1} and \eqref{eq_id2}, corresponds the morphism $\kappa : P_{M_1} \otimes_{P_{\Sigma}} P_{M_2} \to P_{M_1\cup_{\Sigma}M_2}$ of Theorem \ref{theorem_QVK_FG}. To prove the injectivity, we need to prove the

\begin{proposition}\label{prop_QVK}
For $n\geq 1$ and $\alpha, \beta \in P_n(M_1\wedge M_2)$, $\alpha \sim_{\psi} \beta$ implies $\alpha \sim_{\phi} \beta$.
\end{proposition}

\begin{lemma}\label{lem1}
If Proposition \ref{prop_QVK} is true for $n=1$, then it is true for every $n\geq 1$.
\end{lemma}

\begin{proof}
Let $n\geq 2$ and $\alpha, \beta \in \boldsymbol{P}_n(\mathbf{M}_1\wedge \mathbf{M}_2)$ and suppose that $\alpha$ and $\beta$ are related by an isotopy $H$ in $M$.  Let $\alpha= \alpha^{(1)} \cup \ldots \cup \alpha^{(n)}$ be the $n$ components of $\alpha$. Up to changing $H$ by an isotopic isotopy, one can suppose that $H$ is a composition of isotopies each of which moves one strand while leaving the other $n-1$ strands fixed. Therefore, by induction on the number of such isotopies in the decomposition, we are reduced to prove the lemma in the case where $H$ leaves $\alpha^{(2)}, \ldots, \alpha^{(n)}$ fixed whereas it only isotopes the component $\alpha^{(1)}$. Choosing an open tubular neighborhood $N$ of $\alpha^{(2)}\cup \ldots \cup \alpha^{(n)}$ and replacing $M$ by $M'= M\setminus N$, the isotopy $H$ induces an isotopy $H'$ between $\alpha^{(1)}$ and $\beta^{(1)}$ in $\boldsymbol{P}_1(M')$. Since we assume that Proposition \ref{prop_QVK} holds for $n=1$, we can decompose $H'$ as a composition of isotopies with support in $\overline{M_1} \setminus N$ and $\overline{M_2}\setminus N$. Therefore $\alpha \sim_{\phi} \beta$ and this concludes the proof.
\end{proof}

We are reduced to prove Proposition \ref{prop_QVK} in the case $n=1$. We will imitate the proof of the classical Van Kampen theorem as it appears in Hatcher's book \cite{Hatcher_Book}.  Fix a Riemannian metric on $M$ and let $\pi: UM\to M$ be the unitary tangent bundle. A $1$-ribbon tangle $\alpha$ can be seen as a smooth map $c_{\alpha} : [0,1] \to UM$ such that $(1)$ $c_{\alpha}$ is an embedding, $(2)$ $c_{\alpha}(0), c_{\alpha}(1) \in \mathbb{D}_M$ have distinct heights and framings towards the height direction and $(3)$ the framing is normal to the direction, i.e. for all $t\in [0,1]$ if $v:=\pi \circ c_{\alpha}(t)$ then $c_{\alpha}(t)$ is orthogonal to $\restriction{\frac{d}{dt}\pi\circ c_{\alpha}}{t}$ in $U_vM$. 
Conditions $(1)$, $(2)$ and $(3)$ are the reasons why the proof of the quantum Van Kampen theorem is slightly more involved than the proof of its classical version in \cite{Hatcher_Book}.
\par An isotopy between $\alpha$ and $\beta$ is a smooth map $H: [0,1]^2 \to UM$ such that $H(0, \cdot) = c_{\alpha}(\cdot)$, $H(1, \cdot)=c_{\beta}(\cdot)$ and $H(s, \cdot)$ satisfies $(1)$, $(2)$ and $(3)$ for all $s\in [0,1]$. 
Since $\overline{M}_1 \cup \overline{M}_2= M$, we can always find integers $a,b \geq 1$ and sequences $0=t_0 < \ldots < t_a=1$ and $0=s_0< \ldots < s_b=1$ such that $H([s_i, s_{i+1}] \times [t_j, t_{j+1}])$ is either included in $U\overline{M}_1$ or in $U\overline{M}_2$ for all $0\leq j \leq a-1$, $0\leq i \leq b-1$. The \textit{size} of $H$ is the minimal such pair $(a,b)$. We will prove Lemma \ref{lem1} by decreasing induction on $a$ and $b$. More precisely, let $\mathcal{P}(a,b)$ be the property: if $\alpha, \beta\in P_1(\mathbf{M}_1\wedge \mathbf{M}_2)$ are related by an isotopy in $M$ of size $(n,m)$ with $n\leq a$, $m\leq b$ then $\alpha\sim_{\phi} \beta$.

\begin{lemma}\label{lem2} If $\mathcal{P}(a,b)$ holds, so does $\mathcal{P}(a,b+1)$. \end{lemma}

\begin{proof}
Suppose that $\alpha, \beta$ are related by an isotopy $H$ of size $(a,b+1)$ with parameters $0=t_0 < \ldots < t_a=1$ and $0=s_0< \ldots < s_b=1$. Let $\eta \in \boldsymbol{P}_1(\mathbf{M})$ be the bottom tangle parametrized by $c_{\eta}(t):= H(s_{a-1}, t)$. Let $H_1$ be the isotopy between $\alpha$ and $\eta$ defined by $H_1(s,t)=H(\frac{s}{s_{a-1}}, t)$ and $H_2$ the isotopy between $\eta$ and $\beta$ be $H_2(s,t)=H\left( (1-s_{a-1})s+s_{a-1}, t \right)$. One can suppose that  $\eta$ intersects $\Sigma_0$ transversally  in at most $a-1$ points. For each such point $v$, choose a path $\gamma_v$ between $v$ and a point in $\mathbb{B}_M$ not intersecting $\eta$ and let $\eta'$ be the bottom tangle obtained from $\eta$ by pushing each $v$ along $\gamma_v$. So $\eta' \subset M_1\wedge M_2$ and we have an isotopy $H_3$ between $\eta$ and $\eta'$ with support in $\overline{M}_1\cap \overline{M}_2$. By composing $H_1$ with $H_3$ we get an isotopy between $\alpha$ and $\eta'$ of size $(a,b)$ so $\alpha\sim_{\phi} \eta'$ by hypothesis. By composing the inverse of $H_3$ with $H_2$ we get an isotopy between $\eta'$ and $\beta$ of size $(a,1)$ so $\eta' \sim_{\phi} \beta$ by hypothesis. Therefore $\alpha\sim_{\phi} \beta$.
\end{proof}

\begin{lemma}\label{lem3} If $\mathcal{P}(a,1)$ holds, so does $\mathcal{P}(a+1,1)$. \end{lemma}

\begin{proof} 
Suppose that $\mathcal{P}(a,1)$ holds.
Let $\alpha$ and $\beta$ be related by an isotopy $H$ of size $(a+1,1)$ and parameters $0=t_0<\ldots< t_{a+1}$.
Consider the two squares $R_1=[0, t_1]\times [0,1]$ and $R_2=[t_1, t_2]\times [0,1]$. By minimality of $a$, $H(R_1)$ and $H(R_2)$ lies in distinct $\overline{M}_i$. Without loss of generality, one can suppose that $H(R_1)\subset \overline{M}_1$ and $H(R_2)\subset \overline{M}_2$. Let $e:= \{t_1\} \times [0,1]=R_1\cap R_2$ so that $H(e) \subset \overline{M}_1\cap \overline{M}_2 = \mathbb{B}_M \cup \iota_{\Sigma}(\Sigma\times I)$. Let $\alpha_s$ be the bottom tangle parametrized by $H(\cdot, s)$ so that $\alpha_0=\alpha$ and $\alpha_1=\beta$.

\par We want to "cut" the homotopy $H$ in two parts: the restriction to $R_1$, seen as an isotopy of size $(1,1)$  and the restriction to $[t_1,1]\times [0,1]$ seen an isotopy of size $(a,1)$. The equivalence $\alpha\sim_{\phi}\beta$ will then result from $\mathcal{P}(a,1)$.  We will proceed in three steps.
\par \textit{Step 1: Reduction to the case where $H(t_1, 1)\in \mathbb{B}_M$}
 \\ Let $v_1:= \pi (H(t_1,1)) \in  \mathbb{B}_M \cup \iota_{\Sigma}(\Sigma\times I)$. By slightly isotoping $\beta$, one can suppose that $v_1\in \mathbb{B}_M\cup \Sigma^0$. If $v_1\in \Sigma^0$, choose $w_1\in \mathbb{B}_M$ and $\gamma$ path between $v_1$ and $w_1$ not intersecting $\beta$ and let $\beta'$ be the push of $\beta$ along $\gamma$. By post-composing $H$ with an isotopy between $\beta$ and $\beta'$ with support in $\mathbb{B}_M \cup \iota_{\Sigma}(\Sigma\times I)$, one gets an isotopy $H'$ between $\alpha$ and $\beta'$ of the same size $(a+1,1)$. Since $\beta\sim_{\phi} \beta'$, up to replacing $\beta$ and $H$ by $\beta'$ and $H'$, we can, and do,  suppose that $H(t_1, 1)\in \mathbb{B}_M$.

\par \textit{Step 2: Reduction to the case where $H( e)$ is a constant vector  in $U\mathbb{B}_M$}
\\ For $s\in [0,1]$, write $\overrightarrow{v}_s:= H(t_1,s) \in  U (\mathbb{B}_M\cup \iota_{\Sigma}(\Sigma \times I))$ and $v_s= \pi(\overrightarrow{v}_s)$. So $s\mapsto v_s$ parametrizes a smooth path $\delta$ in $\mathbb{B}_M\cup \iota_{\Sigma}(\Sigma \times I)$ between $v_0$ and $v_1 \in \mathbb{B}_M$. Up to slightly isotoping $\alpha$ and $H$, we can suppose that $\delta$ does not intersect $\alpha$. Let $\alpha'$ be the push of $\alpha$ along $\delta$ and let $H^0$ be an isotopy between $\alpha'$ and $\alpha$ with support in $\mathbb{B}_M \cup \iota_{\Sigma}(\Sigma\times I)$. Let $H'$ be the isotopy between $\alpha'$ and $\beta$ obtained by composing $H^0$ with $H$. By replacing $\alpha$ by $\alpha'$ and $H$ by $H'$, we are reduced to the case where the path $s\to \overrightarrow{v}_s$ is a contractible loop in $U\mathbb{B}_M$ (with base point $\overrightarrow{v}_1$). So, up to isotoping $H$, one can (and do) suppose that $H(e)$ is the constant vector 
$\overrightarrow{v}_1$.

\par \textit{Step 3: Cutting the isotopy in two}
Let $\eta$ be the $1$-bottom tangle parametrized by $c_{\eta}(t) = H(1, t)$ for $0\leq t \leq t_1$ and $c_{\eta}(t)=H(0,t)$ for $t_1\leq t \leq 1$. Define an isotopy $H^L$ between $\alpha$ and $\eta$ by $H^L(s,t)=H(s,t)$ for $0\leq t \leq t_1$ and $H^L(s,t)=H(0,t)$ for $t_1\leq t \leq 1$. Since $H^L$ has support in $H(R_1)\subset \overline{M}_1$, one has $\alpha \sim_{\phi} \eta$. Define an isotopy $H^R$ between $\eta$ and $\beta$ by $H^R(s,t)= H(1,t)$ for $0\leq t \leq t_1$ and $H^R(s,t)=H(s,t)$ for $t_1\leq t \leq 1$. Then $H^R$ has size $(1,a)$ so by hypothesis, we have $\eta\sim_{\phi} \beta$. Therefore $\alpha \sim_{\phi} \beta$. 

\end{proof}

\begin{proof}[Proof of Proposition \ref{prop_QVK}]
Since the property $\mathcal{P}(1,1)$ is trivially satisfied, by induction on $a$ Lemma \ref{lem3} implies that $\mathcal{P}(a, 1)$ holds for every $a\geq 1$ and by induction on $b$, Lemma \ref{lem2} implies that $\mathcal{P}(a,b)$ holds for every $a,b\geq 1$. Therefore Proposition \ref{prop_QVK} holds for $n=1$ and so, by Lemma \ref{lem1}, it holds for every $n\geq 1$. This completes the proof.
\end{proof}

\begin{proof}[Proof of Theorem \ref{theorem_QVK_FG}]
The theorem follows from the commutativity of the square: 
$$ \begin{tikzcd}
P_{M_1}\otimes_{P_{\Sigma}}P_{M_2} (\mathbf{H}_n)
\ar[r, "\kappa"] \ar[d, "\cong", "\varphi"'] &
P_M(\mathbf{H}_n) 
\ar[d, "\cong", "\Psi"'] \\
\quotient{P_n(\mathbf{M}_1\wedge \mathbf{M}_2)}{\sim_{\phi}} 
\ar[r, "="] &
\quotient{P_n(\mathbf{M}_1\wedge \mathbf{M}_2)}{\sim_{\psi}} 
\end{tikzcd}$$
where the equality on the bottom follows from Proposition \ref{prop_QVK}.

\end{proof}

\subsection{Quantum Van Kampen for quantum representation spaces and stated skein modules}\label{sec_QVK_Skein}

Let $\mathbf{M}_1$, $\mathbf{M}_2$ and $\mathbf{\Sigma}$ be as in the last subsection. Since $\mathbf{M}_1$ and $\mathbf{M}_2$ are right and left modules over $\mathbf{\Sigma}\times I$ in $(\mathcal{M}^{(1)}_{\con}, \wedge)$, then $\Rep^G_q(\mathbf{M}_1)$ and $\Rep^G_q(\mathbf{M}_2)$ are left and right modules over $\Rep^G_q(\mathbf{\Sigma})$ by monoidality of $\Rep^G_q$. The following was suggested by Habiro in \cite{Habiro_QCharVar}.

\begin{theorem}\label{theorem_QVK_Rep} 
One has an isomorphism $\Rep^G_q(\mathbf{M}_1\cup_{\Sigma} \mathbf{M}_2) \cong \Rep^G_q(\mathbf{M}_1) \otimes_{\Rep^G_q(\mathbf{\Sigma})}\Rep^G_q(\mathbf{M}_2)$.
\end{theorem}

\begin{proof}
By Theorem \ref{theorem_QVK_FG}, the following is a coequalizer

$$ \begin{tikzcd}
 P_{M_1} \otimes_D P_{\Sigma} \otimes_D P_{M_2} 
 \ar[r,shift left=.75ex,"\nabla_1\otimes \id"]
  \ar[r,shift right=.75ex,swap,"\id \otimes \nabla_2"] 
&  P_{M_1} \otimes_D P_{M_2} \ar[r, "P(\pi)"] & P_{M_1 \cup_{\Sigma} M_2}
\end{tikzcd}$$
thus the following sequence is exact:

$$ \begin{tikzcd}
 k[P_{M_1}] \otimes_D k[P_{\Sigma}] \otimes_D k[P_{M_2}] 
 \ar[rr,"\nabla_1\otimes \id- \id \otimes \nabla_2"] 
&{}&  k[P_{M_1}] \otimes_D k[P_{M_2}] \ar[r, "P(\pi)"] & k[P_{M_1 \cup_{\Sigma} M_2}] 
\ar[r] & 0.
\end{tikzcd}$$

Since tensoring on the right preserves right exact sequences, the following sequence is exact as well:

\begin{multline*}
 (k[P_{M_1}] \otimes_D k[P_{\Sigma}] \otimes_D k[P_{M_2}]) \otimes_{\BT} Q_{B_qG} 
 \xrightarrow{\nabla_1\otimes \id- \id \otimes \nabla_2}
  (k[P_{M_1}] \otimes_D k[P_{M_2}])\otimes_{\BT} Q_{B_qG} 
  \\ \xrightarrow{P(\pi)} k[P_{M_1 \cup_{\Sigma} M_2}]\otimes_{\BT} Q_{B_qG} 
\to 0.
\end{multline*}
Using Lemmas \ref{lemma_tenstens} and  \ref{lemma_QFG_QRS}, we get the exact sequence

$$ \begin{tikzcd} \Rep_q^G(\mathbf{M}_1)\otimes_{\Rep_q^G(\mathbf{\Sigma})} \Rep_q^G(\mathbf{M}_2)  \ar[rr,"\nabla_1\otimes \id- \id \otimes \nabla_2"] 
&{}&\Rep_q^G(\mathbf{M}_1)\otimes \Rep_q^G(\mathbf{M}_2) \ar[r, "P(\pi)"] & \Rep_q^G(\mathbf{M}_1\cup_{\mathbf{\Sigma}}\mathbf{M}_2)
\ar[r] & 0, 
\end{tikzcd}$$
which says that $\Rep_q^G(\mathbf{M}_1\cup_{\mathbf{\Sigma}} \mathbf{M}_2)\cong \Rep_q^G(\mathbf{M}_1)\otimes_{\Rep_q^G(\mathbf{\Sigma})}\Rep_q^G(\mathbf{M}_2)$. 
\end{proof}

Putting Corollary \ref{coro_RepSkein} and Theorem \ref{theorem_QVK_FG} together, we obtain the following reformulation of Theorem \ref{theorem2}.

\begin{corollary}\label{coro_QVK_skein} One has an isomorphism $\mathcal{S}_q(\mathbf{M}_1\cup_{\Sigma} \mathbf{M}_2) \cong \mathcal{S}_q(\mathbf{M}_1) \otimes_{\mathcal{S}_q(\mathbf{\Sigma})}\mathcal{S}_q(\mathbf{M}_2)$.
\end{corollary}

One can derive from Corollary \ref{coro_QVK_skein} and Theorem \ref{theorem_surjectivity} an alternative proof of a theorem of Gunningham-Jordan-Safronov in \cite{GunninghamJordanSafranov_FinitenessConjecture} as follows. Let $M$ be a closed compact connected oriented $3$-manifold and suppose that $M=N_1\cup_{\Sigma}N_2$ is obtained by gluing two (compact, oriented, connected) $3$-manifolds along their connected boundary $\Sigma$. Choose a ball $\mathbb{B}^3 \subset M$ which intersects $\Sigma\subset M$ along a disc $\mathbb{D}= \mathbb{B}^3\cap \Sigma$ such that $\mathbb{D}$ cuts $\mathbb{B}^3$ in two hemispheres. Let $\mathbb{D}_M \subset \partial \mathbb{B}^3$ be a disc such that $\mathbb{D}_M$ intersects $\mathbb{D}$ along a closed interval $a$ which cuts $\mathbb{D}_M$ into two half-discs $\mathbb{D}_M^+\subset N_1$ and $\mathbb{D}_M^- \subset N_2$. Choose an oriented homeomorphism $\iota_M: \mathbb{D}^2 \cong \mathbb{D}_M$ sending $\mathbb{D}^{\pm}$ to $\mathbb{D}_M^{\pm}$ and set $\iota_{\pm}: \mathbb{D}_2\cong \mathbb{D}_{\pm} \xrightarrow{\iota_M} \mathbb{D}_M^{\pm}$.
Let $\mathbf{M}:= (M\setminus \mathring{\mathbb{B}}^3, \iota_M) \in \mathcal{M}_{\con}^{(1)}$ and consider $\mathbf{N}_1:= (N_1\setminus \mathring{\mathbb{B}}^3_+, \iota^+)$ and $\mathbf{N}_2:= (N_2 \setminus \mathring{\mathbb{B}}^3_-, \iota^-)$ and write $\mathbf{\Sigma}:= (\Sigma\setminus \mathring{\mathbb{D}}, a) \in \MS_{\con}^{(1)}$ such that $\mathbf{M}= \mathbf{N}_1\cup_{\mathbf{\Sigma}} \mathbf{N}_2$. By Theorem \ref{theorem_surjectivity}, the embedding $\iota: M \to \mathbf{M}$ induces an isomorphism $\iota_*: \mathcal{S}_q(M) \cong (\mathcal{S}_q(\mathbf{M}))^{coinv}$ and by Corollary \ref{coro_QVK_skein}, one has an isomorphism $\mathcal{S}_q(\mathbf{M})\cong \mathcal{S}_q(\mathbf{N}_1)\otimes_{\mathcal{S}_q(\mathbf{\Sigma})} \mathcal{S}_q(\mathbf{N}_2)$. Therefore, using Theorem \ref{theorem_spherical}, we have re-proved the 

\begin{theorem}(Gunningham-Jordan-Safronov \cite[Theorem $2$, Corollary $1$]{GunninghamJordanSafranov_FinitenessConjecture}) We have  isomorphisms:
\begin{align*}
{}& \mathcal{S}_q(M) \cong  \left(\mathcal{S}_q(\mathbf{N}_1)\otimes_{\mathcal{S}_q(\mathbf{\Sigma})} \mathcal{S}_q(\mathbf{N}_2)\right)^{coinv}, \quad \mbox{ and }\\
{}& \mathcal{S}^{rat}_q(M) \cong  \mathcal{S}^{rat}_q(\mathbf{N}_1)\otimes_{\mathcal{S}^{rat}_q(\mathbf{\Sigma})} \mathcal{S}^{rat}_q(\mathbf{N}_2).
\end{align*}
\end{theorem}

Applied to a Heegaard splitting $M=H_g\cup_{\Sigma_g}H_g$ of $M$, this theorem, together with an algebraic analogue of a theorem of Kashiwara-Schapira, was proved in \cite[Theorem $1$]{GunninghamJordanSafranov_FinitenessConjecture} to imply that $\mathcal{S}^{rat}_q(M)$ has finite dimension. Note that, because of Corollary \ref{coro_classical}, $\mathcal{S}_q(M)$ is infinitely generated whenever the character variety of $M$ is not $0$ dimensional so the torsion part of $\mathcal{S}_q(M)$ is infinitely generated in this case.

\subsection{Self-gluing and quantum HNN extensions}\label{sec_selfgluing}

In the last subsections, we have considered the operation of gluing two $3$ manifolds along a surface and proved that $P_{M_1\cup_{\Sigma} M_2}$ is isomorphic to $P_{M_1}\otimes_{P_{\Sigma}}P_{M_2}$. We now consider the operation, giving a $3$ manifold $M$ and two attaching maps $\phi_1: \Sigma \hookrightarrow \partial M$, $\phi_2: \overline{\Sigma} \hookrightarrow \partial M$ of gluing the two copies of $\Sigma$ together to get a new manifold $M_{\phi_1 \# \phi_2}$ and will prove that $P_{M_{\phi_1 \# \phi_2}}$ is isomorphic to $\mathrm{HH}^0(P_{\Sigma}, P_M)$, where $\mathrm{HH}^0$ is defined in a braided sense. We will proceed by identifying $M_{\phi_1 \# \phi_2}$ with $\Sigma\times I \cup_{\overline{\Sigma} \wedge \Sigma} M$ and use the quantum Van Kampen theorem. The proof will then be a braided version of the classical proof of the fact that for $M$ a bimodule over some algebra $A$, then $\mathrm{HH}^0(A, M)$ is isomorphic to $A \otimes_{A^{op}\otimes A} M$.
\vspace{2mm}
\par 
Let $\mathbf{M}\in \mathcal{M}_{\con}^{(1)}$, $\mathbf{\Sigma}\in \MS_{\con}^{(1)}$ and $\phi_1: \Sigma \hookrightarrow \partial M$, $\phi_2: \overline{\Sigma} \hookrightarrow \partial M$ two embeddings such that: $(1)$ the boundary arc $a$ of $\mathbf{\Sigma}$ is sent to $\partial_+ \mathbb{D}_M$ by both $\phi_1$ and $\phi_2$ and $(2)$ the 
 images of $\phi_1$ and $\phi_2$  intersect in a contractible neighborhood of  $\partial_+ \mathbb{D}_M$. We define an embedding $\phi: \overline{\Sigma}\wedge \Sigma \hookrightarrow \partial M$, illustrated in Figure \ref{fig_self_gluing}, as follows.
\par 
 Divide the arc $\partial_+ \mathbb{D}_M$ into two segments $\partial_+ \mathbb{D}_M= \partial_{+,L} \mathbb{D}_M\cup \partial_{+, R} \mathbb{D}_M$ where $\partial_{+,L} \mathbb{D}_M= \iota_M \left( \partial \mathbb{D}^2 \cap \mathbb{R}^{\leq 0} \times \mathbb{R}^{\geq 0}\right)$ and $\partial_{+,R} \mathbb{D}_M= \iota_M \left( \partial \mathbb{D}^2 \cap \mathbb{R}^{\geq 0} \times \mathbb{R}^{\geq 0}\right)$. Write $\{p^+\}=\partial_{+,L} \mathbb{D}_M \cap \partial_{+,R} \mathbb{D}_M$. Isotope $\phi_1$ and $\phi_2$ slightly in the neighborhood of $\partial \mathbb{D}_M$ such that $(1)$ $\phi_1$ maps $a$ to $\partial_{+,L} \mathbb{D}_M$ and $\phi_2$ maps $a$ to $\partial_{+,R} \mathbb{D}_M$ and $(2)$ the intersection of the images of $\phi_1$ and $\phi_2$ is $\{p^+\}$. Define an  embedding $\phi: \overline{\Sigma}\wedge \Sigma \hookrightarrow \partial M$ by pushing slightly $\partial_{+,L} \mathbb{D}_M$  and $\partial_{+,R} \mathbb{D}_M$ inside $\partial M$ and considering an embedding of the  triangle $T$ with edges $\partial_+ \mathbb{D}_M$,  $\partial_{+,L} \mathbb{D}_M$  and $\partial_{+,R} \mathbb{D}_M$. This embedding together with $\phi_1, \phi_2$ glue together to define $\phi$.
\par Consider the cylinder $\mathbf{\Sigma}\times I$ and define an embedding $\psi: \Sigma \wedge \overline{\Sigma} \hookrightarrow \partial (\Sigma\times I)$ as follows. Write $\partial_- \mathbb{D}_{\mathbf{\Sigma}}= \partial_{-, L} \mathbb{D}_{\mathbf{\Sigma}} \cup \partial_{-, R} \mathbb{D}_{\mathbf{\Sigma}}$ as before. Isotope the embedding $\Sigma \cong \Sigma\times \{+1\} \hookrightarrow \partial (\Sigma\times I)$ to get an embedding $\psi_1: \Sigma \hookrightarrow \partial (\Sigma\times I)$ sending $a$ to $\partial_{-, R} \mathbb{D}_{\mathbf{\Sigma}}$. Isotope the embedding $\overline{\Sigma} \cong \Sigma \times \{-1\} \hookrightarrow \partial \Sigma \times I$ to get an embedding $\psi_2: \overline{\Sigma} \hookrightarrow \partial (\Sigma \times I)$ sending $a$ to $\partial_{-, L} \mathbb{D}_{\mathbf{\Sigma}}$. Like before, push slightly $\partial_{-, L} \mathbb{D}_{\mathbf{\Sigma}}$ and $\partial_{+, L} \mathbb{D}_{\mathbf{\Sigma}}$ inside and define an embedding of the triangle with edges $\partial_- \mathbb{D}_{\mathbf{\Sigma}}$, $\partial_{-, L} \mathbb{D}_{\mathbf{\Sigma}}$ and $\partial_{-, R} \mathbb{D}_{\mathbf{\Sigma}}$. This embedding together with $\psi_1$, $\psi_2$ define the embedding $\psi: \Sigma \wedge \overline{\Sigma} \hookrightarrow \partial (\Sigma\times I)$.

 \begin{figure}[!h] 
\centerline{\includegraphics[width=14cm]{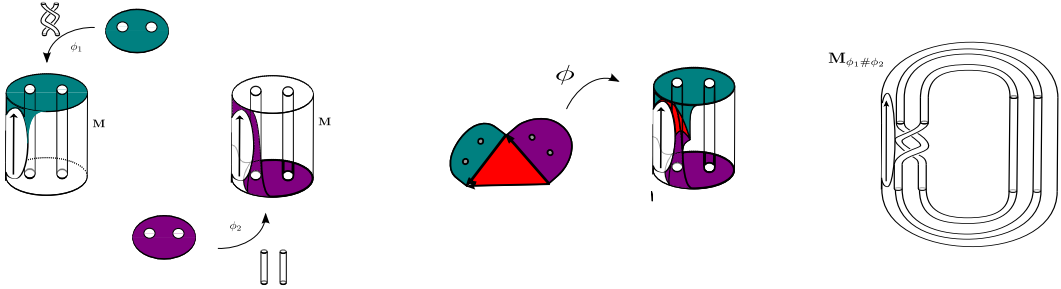} }
\caption{An illustration of the self-gluing operation where $\mathbf{M}= \mathbb{D}_2\times I$, the attaching map $\phi_1$ is given by a braid and $\phi_2$ is given by the identity. The resulting manifold $\mathbf{M}_{\phi_1\# \phi_2}$ is the exterior of the link obtained by closing the braid (here the Hopf link) with a ball removed.}
\label{fig_self_gluing} 
\end{figure}

\begin{definition} The marked $3$ manifold $\mathbf{M}_{\phi_1 \# \phi_2} \in \mathcal{M}_{\con}^{(1)}$ is the gluing $\mathbf{M}_{\phi_1 \# \phi_2} := (\mathbf{\Sigma}\times I)\cup_{\overline{\Sigma}\wedge \Sigma} \mathbf{M}$ defined by the two attaching maps $\phi$ and $\psi$. We denote by $\Psi: \mathbf{M} \to \mathbf{M}_{\phi_1 \# \phi_2}$ the natural embedding.
\end{definition}
 
 While gluing the cylinder $\mathbf{\Sigma}\times I$ to $\mathbf{M}$ using the gluing maps $\phi_1$ and $\psi_2$ we get a marked $3$-manifold $(\mathbf{\Sigma}\times I)\cup_{\Sigma}\mathbf{M}$ isomorphic to $\mathbf{M}$, so the gluing map defines an embedding
  $\nabla_{\phi_1}^L : (\mathbf{\Sigma}\times I) \wedge \mathbf{M} \to \mathbf{M}$. Similarly, while gluing $\mathbf{M}$ with the cylinder $\mathbf{\Sigma}\times I$ using $\phi_2$ and $\psi_1$, we also get a marked $3$-manifold $(\mathbf{\Sigma}\times I)\cup_{\Sigma}\mathbf{M}$ isomorphic to $\mathbf{M}$, so the gluing map defines an embedding  $\nabla_{\phi_2}^R : \mathbf{M} \wedge (\mathbf{\Sigma}\times I)\to \mathbf{M}$. 
  
  
  \begin{definition}\label{def_HH0} Let $(\mathcal{C}, \otimes, c_{\cdot, \cdot}, \theta_{\cdot})$ be a braided balanced category, $A\in \mathcal{C}$ an algebra object and $M \in \mathcal{C}$ an $A$-bimodule object with left and right modules maps $\nabla^L : A\otimes M \to M$ and $\nabla^R: M\otimes A \to M$. The $0$-th Hochschild cohomology group $\mathrm{HH}^0(A,M)$  is  (when it exists) the coequalizer
    $$ \begin{tikzcd}
 A\otimes M
 \ar[rr,shift left=.75ex,"\nabla^L"]
  \ar[rr,shift right=.75ex,swap,"\nabla^R \circ c_{A, M} \circ (\theta_A \otimes \id_M)"] 
&{}& {M}
\ar[r]
&\mathrm{HH}^0(A,M).
\end{tikzcd}$$
  \end{definition}
  
  \begin{remark}\label{remark_twistopposite} Note that if $(M, \nabla^R)$ is a right $A$-module object in $\mathcal{C}$ then $(M, \nabla^R \circ c_{A,M} \circ (\theta_A \otimes \id_M) )$ is a left $A$-module object.
  \end{remark}
  
  The maps $\nabla_{\phi_1}^L$ and $\nabla_{\phi_2}^R$ give $\mathbf{M}$ the structure of a bimodule over $\mathbf{\Sigma}\times I$ in $\mathcal{M}_{\con}^{(1)}$. Composing with $P: \mathcal{M}_{\con}^{(1)}\to \widehat{\BT}$, we obtain that $P_M$ is a $P_{\Sigma}$ bimodule in $\widehat{\BT}$. Similarly composing with $\Rep_q^G$, we obtain that $\Rep_q^G(\mathbf{M})$ is a bimodule over $\Rep_q^G(\mathbf{\Sigma})$ in $\overline{\mathcal{C}_q^G}$.
  
  \begin{theorem}\label{theorem_selfgluing}
  \begin{enumerate}
  \item 
One has $P_{M_{\phi_1\# \phi_2}}= \mathrm{HH}^0(P_{\Sigma}, P_M)$.
\item One has $\Rep_q^G(\mathbf{M}_{\phi_1\# \phi_2}) = \mathrm{HH}^0( \Rep_q^G(\mathbf{\Sigma}), \Rep_q^G(\mathbf{M}))$ and  $\mathcal{S}_q(\mathbf{M}_{\phi_1\# \phi_2}) = \mathrm{HH}^0( \mathcal{S}_q(\mathbf{\Sigma}), \mathcal{S}_q(\mathbf{M}))$.
\end{enumerate}
\end{theorem}

By the quantum Van Kampen Theorem \ref{theorem_QVK_FG}, the image by $P: \mathcal{M}_{\con}^{(1)}\to \widehat{\BT}$ of the following sequence is a coequalizer (here and henceforth we write $\mathbf{\Sigma}$ instead of $\mathbf{\Sigma}\times I$ for simplicity):
 $$ \begin{tikzcd}
 \mathbf{\Sigma} \wedge( \overline{\mathbf{\Sigma}} \wedge \mathbf{\Sigma} ) \wedge \mathbf{M}
 \ar[rr,shift left=.75ex,"\nabla_1\wedge \id"]
  \ar[rr,shift right=.75ex,swap,"\id \wedge \nabla_2"] 
&{}& \mathbf{\Sigma} \wedge \mathbf{M}
\ar[r, "\Psi'"]
&\mathbf{M}_{\phi_1\# \phi_2}.
\end{tikzcd}$$
We want to deduce from this fact that the image by $P$ of the following sequence is a coequalizer too:
 $$ \begin{tikzcd}
 \mathbf{\Sigma}  \wedge \mathbf{M}
 \ar[rr,shift left=.75ex,"\nabla^L_{\phi_1}"]
  \ar[rr,shift right=.75ex,swap,"\nabla_{\phi_2}^R \circ \Psi_{\Sigma, M}"] 
&{}&  \mathbf{M}
\ar[r, "\Psi"]
&\mathbf{M}_{\phi_1\# \phi_2}.
\end{tikzcd}$$

To analyze the morphisms $\nabla_1$ and $\nabla_2$, let us introduce the homeomorphism $\hT_{\Sigma}: \overline{\Sigma}\times I \xrightarrow{\cong} \Sigma \times I$ which consists in performing a half-twist in a tubular neighborhood of the boundary disc $\mathbb{D}_{\Sigma}$ while leaving the rest invariant. The morphism $\hT_{\Sigma}$ does not define a morphism $\overline{\mathbf{\Sigma}} \to \mathbf{\Sigma}$ in $\mathcal{M}_{\con}^{(1)}$ because it reverses the height order of the boundary disc instead of preserving it. However, for any other marked $3$-manifold $\mathbf{N}$ one can define an isomorphism $\id\wedge \hT_{\Sigma} : \mathbf{N} \wedge \overline{\mathbf{\Sigma}} \xrightarrow{\cong} \mathbf{N}\wedge \mathbf{\Sigma}$ in $\mathcal{M}_{\con}^{(1)}$ which consists in performing a half-twist in a neighborhood of $\mathbb{D}_{\overline{\Sigma}}$. Then by definition of the gluing maps, the morphisms $\nabla_1$, $\nabla_2$ decompose as 
$$ \nabla_1: \mathbf{\Sigma}\wedge \overline{\mathbf{\Sigma}} \wedge \mathbf{\Sigma} \xrightarrow{ \id \wedge \hT_{\Sigma}^{-1} \wedge \id} {\mathbf{\Sigma}}\wedge {\mathbf{\Sigma}} \wedge \mathbf{\Sigma}  \xrightarrow{\psi_{\Sigma, \Sigma} \wedge \id} {\mathbf{\Sigma}}\wedge {\mathbf{\Sigma}} \wedge \mathbf{\Sigma} \xrightarrow{\mu^{(2)}_{\Sigma}} \mathbf{\Sigma}$$
and 
$$ \nabla_2 : \overline{\mathbf{\Sigma}} \wedge \mathbf{\Sigma} \wedge \mathbf{M}  \xrightarrow{ \id \wedge \nabla_{\phi_1}^L} \overline{\mathbf{\Sigma}}\wedge \mathbf{M} \xrightarrow{\psi_{\overline{\Sigma}, M}} \mathbf{M}\wedge \overline{\mathbf{\Sigma}} \xrightarrow{\id\wedge \hT_{\Sigma}} \mathbf{M}\wedge \mathbf{\Sigma} \xrightarrow{\nabla_{\phi_2}^R} \mathbf{M}.
$$
Define a morphism $j_1: \mathbf{\Sigma} \wedge (\overline{\mathbf{\Sigma}}\wedge \mathbf{\Sigma}) \wedge \mathbf{M} \to \mathbf{\Sigma} \wedge \mathbf{M}$ as the composition: 
\begin{multline*}
 j_1: \mathbf{\Sigma} \wedge (\overline{\mathbf{\Sigma}}\wedge \mathbf{\Sigma}) \wedge \mathbf{M}
 \xrightarrow{ \id \wedge \hT_{\Sigma}^{-1} \wedge \id \wedge \id}  
 \mathbf{\Sigma} \wedge {\mathbf{\Sigma}}\wedge \mathbf{\Sigma} \wedge \mathbf{M} \\
 \xrightarrow{ \psi_{\Sigma, \Sigma}\wedge \id \wedge \id} 
  \mathbf{\Sigma} \wedge {\mathbf{\Sigma}}\wedge \mathbf{\Sigma} \wedge \mathbf{M} 
  \xrightarrow{\id \wedge \mu_{\Sigma} \wedge \id}
{\mathbf{\Sigma}}\wedge \mathbf{\Sigma} \wedge \mathbf{M} 
\xrightarrow{\id \wedge \nabla_{\phi_1}^L}
\mathbf{\Sigma} \wedge \mathbf{M}, 
\end{multline*}

and set $j_2:= \nabla_{\phi_1}^L$. Then in the diagram

\begin{equation}\label{big_diagram}
\begin{tikzcd}
 \mathbf{\Sigma} \wedge( \overline{\mathbf{\Sigma}} \wedge \mathbf{\Sigma} ) \wedge \mathbf{M}
 \ar[rr,shift left=.75ex,"\nabla_1\wedge \id"]
 \ar[dd, "j_1"] 
  \ar[rr,shift right=.75ex,swap,"\id \wedge \nabla_2"] 
&{}& \mathbf{\Sigma} \wedge \mathbf{M}
\ar[dd, "j_2"]
\ar[rd, "\Psi'"]
&{} \\
{}&{}&{}&
\mathbf{M}_{\phi_1\# \phi_2} \\
\mathbf{\Sigma}  \wedge \mathbf{M}
 \ar[rr,shift left=.75ex,"\nabla^L_{\phi_1}"]
  \ar[rr,shift right=.75ex,swap,"\nabla_{\phi_2}^R \circ \Psi_{\Sigma, M}\circ (\theta_{\Sigma}\wedge \id_M)"] 
&{}&  \mathbf{M}
\ar[ru, "\Psi"]
\end{tikzcd}
\end{equation}
the three following subdiagrams commute (up to isotopy):

\begin{equation}\label{3_diagrams}
\begin{tikzcd}
 \mathbf{\Sigma} \wedge( \overline{\mathbf{\Sigma}} \wedge \mathbf{\Sigma} ) \wedge \mathbf{M}
  \ar[r,"\nabla_1\wedge \id"]
 \ar[d, "j_1"] 
& \mathbf{\Sigma} \wedge \mathbf{M}
\ar[d, "j_2"]
\\
\mathbf{\Sigma}  \wedge \mathbf{M}
 \ar[r,"\nabla^L_{\phi_1}"]
 &  \mathbf{M}
\end{tikzcd}
\quad
\begin{tikzcd}
 \mathbf{\Sigma} \wedge( \overline{\mathbf{\Sigma}} \wedge \mathbf{\Sigma} ) \wedge \mathbf{M}
  \ar[r,"\id \wedge \nabla_2"]
 \ar[d, "j_1"] 
& \mathbf{\Sigma} \wedge \mathbf{M}
\ar[d, "j_2"]
\\
\mathbf{\Sigma}  \wedge \mathbf{M}
 \ar[r,"\nabla_{\phi_2}^R \circ \Psi_{\Sigma, M}\circ (\theta_{\Sigma}\wedge \id_M)"]
 &  \mathbf{M}
\end{tikzcd}
\quad
\begin{tikzcd}
 \mathbf{\Sigma} \wedge \mathbf{M}
\ar[dd, "j_2"] \ar[rd, "\Psi'"]
&{} \\
{}& 
\mathbf{M}_{\phi_1\# \phi_2}\\
\mathbf{M} 
\ar[ru, "\Psi"]
\end{tikzcd}
\end{equation}

\begin{lemma}\label{lemma_diagrams}
Let $\mathcal{C}$ a  category and consider a diagram
$$
\begin{tikzcd}
A 
 \ar[r,shift left=.75ex,"f"]
  \ar[r,shift right=.75ex,swap,"g"] 
\ar[dd, "j_1"] &
B \ar[dd, "j_2"] \ar[rd, "h"] &{} \\
{}&{}&D \\
A'
 \ar[r,shift left=.75ex,"f'"]
  \ar[r,shift right=.75ex,swap,"g'"]
& B'  \ar[ru, "h'"]&{}
\end{tikzcd}
$$

such that:
$(i)$  the following diagrams are commutative: 
$$ 
\begin{tikzcd}
A 
  \ar[r,"f"] 
\ar[dd, "j_1"] &
B \ar[dd,   "j_2"]  \\
{}&{} \\
A'
  \ar[r,"f'"]
& B'  
\end{tikzcd}
\quad
\begin{tikzcd}
A 
  \ar[r,"g"] 
\ar[dd, "j_1"] &
B \ar[dd,   "j_2"]  \\
{}&{} \\
A'
  \ar[r,"g'"]
& B'  
\end{tikzcd}
\quad 
\begin{tikzcd}
B \ar[dd, "j_2"] \ar[rd, "h"] &{} \\
{} & D \\
B' \ar[ru, "h'"] &{}
\end{tikzcd};
$$
\\ $(ii)$ $j_1$ and $j_2$  are epimorphisms
\\ and $(iii)$ the following diagram is a coequalizer:
$$
\begin{tikzcd}
A
 \ar[r,shift left=.75ex,"f"]
  \ar[r,shift right=.75ex,swap,"g"] &
B  \ar[r, "h"] & D 
\end{tikzcd}.
$$
Then the following diagram is a coequalizer:
$$
\begin{tikzcd}
A'
 \ar[r,shift left=.75ex,"f'"]
  \ar[r,shift right=.75ex,swap,"g'"] &
B' \ar[r, "h'"] & D 
\end{tikzcd}.
$$
\end{lemma}

\begin{proof}
First, the equality $h'\circ f'= h'\circ g'$ follows from the fact that $j_1$ is an epimorphism together with the commutativity of the diagrams in $(i)$ and $(iii)$.
Let $k: B' \to E$ be a morphism such that $k\circ f' = k\circ g'$ and let us prove that there exists a unique morphism $m: D \to E$ such that $k=m\circ h'$. Since $k\circ f' \circ j_1= k\circ g'\circ j_1$ by commutativity of the first two diagrams in Hypothesis $(i)$, one has $(k\circ j_2 )\circ f=(k\circ j_2)\circ g$. By $(iii)$, there exists a unique $m:D\to E$ such that $k\circ j_2= m\circ h$. By commutativity of the third diagram in $(i)$, the last identity is equivalent to $k\circ j_2= m\circ h' \circ j_2$ and since $j_2$ is an epimorphism $(ii)$, this is equivalent to $k=m\circ h'$. This proves the claim.
\end{proof}

\begin{proof}[Proof of Theorem \ref{theorem_selfgluing}]
We apply Lemma \ref{lemma_diagrams} to the image by $P$ of Diagram \ref{big_diagram}. The commutativity of Diagrams \ref{3_diagrams} implies hypothesis $(i)$. Hypothesis $(ii)$ follows from the fact that, for $k=1,2$,  the codomain of  $j_k$ retracts by deformation on its image, therefore for every $n\geq 0$, $P_n(j_k) $ is surjective, so $P(j_k)$ is an epimorphism.  Hypothesis $(iii)$ follows from the quantum Van Kampen Theorem \ref{theorem_QVK_FG}. Therefore $P_{M_{\phi_1\# \phi_2}}= \mathrm{HH}^0(P_{\Sigma}, P_M)$. The corresponding results for $\Rep_q^G$ and $\mathcal{S}_q$ are deduced from the fact that $\bullet \otimes_{\BT} Q_{B_qG}$ preserves right exact sequences exactly in the same manner than in the proof of Theorem \ref{theorem_QVK_Rep} so we leave the details to the reader.
\end{proof}

\section{Quantum representation spaces of mapping tori}\label{sec5}

\subsection{The quantum adjoint coaction}\label{sec_QAd}
Let $\mathbf{\Sigma}=(\Sigma, \{a\}) \in \MS^{(1)}_{\con}$ be a connected $1$-marked surface let $\ad_{\Sigma} : \mathbf{\Sigma}\times I \to (\mathbf{\Sigma}\times I) \wedge \mathbf{H}_1$ be the embedding whose image by the quantum fundamental group functor is the morphism $P(\ad_{\Sigma})$ illustrated in Figure \ref{fig_qad}. More precisely, consider the marked surface $\mathbf{\Sigma}\wedge {\mathbb{D}_1}$ obtained from $\mathbf{\Sigma}\bigsqcup \mathbb{D}_1\bigsqcup (\mathbb{D}^2, \{e_1,e_2,e_3\})$ by gluing $e_1$ with $a$ and $e_2$ with the boundary arc of the annulus $\mathbb{D}_1$. In particular, $ (\mathbf{\Sigma}\times I) \wedge \mathbf{H}_1$ is the image by the functor $\cdot \times I : \MS \to \mathcal{M}$ of $\mathbf{\Sigma}\wedge \mathbf{\mathbb{D}_1}$. 
Define an embedding $\ad^0: \Sigma \to \Sigma \wedge \mathbb{D}_1$ whose restriction outside a collar neighborhood $N(a) \cong [0,1]^2$of $a$ is the identity and sending $N(a)$ to a band between $a$ and $e_3$ making one turn around the hole of the annulus $\mathbb{D}_1$. Then $\ad_{\Sigma}$ is the image of $\ad^0$ by  $\cdot \times I : \MS \to \mathcal{M}$ of $\mathbf{\Sigma}\wedge \mathbf{\mathbb{D}_1}$.

 \begin{figure}[!h] 
\centerline{\includegraphics[width=8cm]{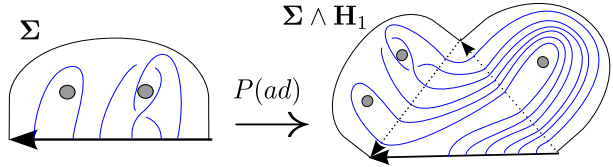} }
\caption{An illustration of the quantum adjoint coaction.} 
\label{fig_qad} 
\end{figure} 

\begin{definition} The \textit{quantum adjoint coactions} are the morphisms $\Ad_{\Sigma}: \Rep_q^G(\mathbf{\Sigma}) \to \Rep_q^G(\mathbf{\Sigma}) \overline{\otimes} B_qG$ and $\Ad_{\Sigma}: \mathcal{S}_q(\mathbf{\Sigma}) \to \mathcal{S}_q(\mathbf{\Sigma}) \overline{\otimes} B_q\SL_2$  obtained by taking the images of $\ad_{\Sigma}$ by $\Rep_q^G$ and $\mathcal{S}_q$ respectively and using the lax monoidality.
\end{definition}

The quantum adjoint coaction equips $\mathcal{S}_q(\mathbf{\Sigma})$  (resp. $\Rep_q^G(\mathbf{\Sigma})$) with a structure of $B_q\SL_2$ (resp. $B_qG$) right comodule object in the braided category $\overline{\mathcal{C}_q^{\SL_2}}$ (resp. $\overline{\mathcal{C}_q^G}$). 
In the particular case where $\mathbf{\Sigma}= \mathbf{H}_1$, we have proved in Theorem \ref{theorem_skein_transmutation} that  $\Ad_{\mathbf{H}_1}$ coincides with Majid's braided adjoint coaction. In particular, $\mathcal{S}_q(\mathbf{\Sigma})$ is an $\mathcal{O}_q[\SL_2]$ right comodule in $\Mod_k$ and a $B_q\SL_2$ right comodule in $\overline{\mathcal{C}_q^{\SL_2}}$. Recall from Section \ref{sec_skein_transmutation} the isomorphism $f:B\mathcal{S}_q(\mathbb{B}) \xrightarrow{\cong} \mathcal{S}_q(\mathbf{H}_1)$ where $B\mathcal{S}_q(\mathbb{B})$ is equal to $\mathcal{S}_q(\mathbb{B})$ as a co-algebra.

\begin{lemma}\label{lemma_braidedcoinv}
The following diagram (in $\Mod_k$) commutes:
$$ \begin{tikzcd}
{} & \mathcal{S}_q(\mathbf{\Sigma})\otimes \mathcal{S}_q(\mathbb{B}) \ar[dd, "\id\otimes f"] \\
\mathcal{S}_q(\mathbf{\Sigma}) \ar[ru, "\Delta^R_a"] \ar[rd, "\Ad_{\Sigma}"] & {} \\
{}& \mathcal{S}_q(\mathbf{\Sigma})\overline{\otimes}\mathcal{S}_q(\mathbf{H}_1)
\end{tikzcd}
$$

In particular, the subspace of $\mathcal{O}_q[\SL_2]$ coinvariant vectors of $\mathcal{S}_q(\mathbf{\Sigma})$ coincides with its subspace of $B_q\SL_2$ coinvariant vectors.
\end{lemma}

\begin{proof} 
Let $[T,s] =  \adjustbox{valign=c}{\includegraphics[width=2cm]{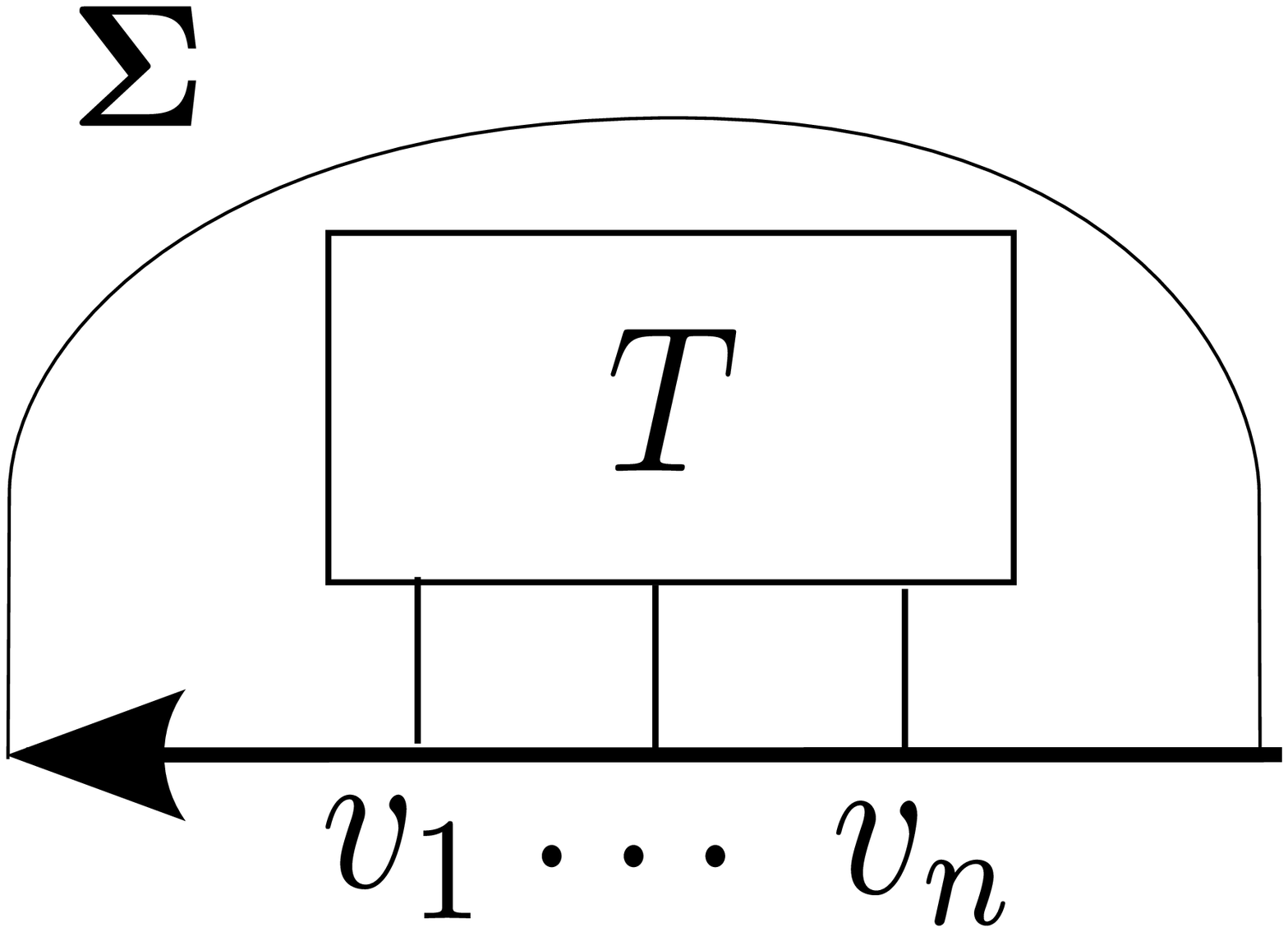}}\in \mathcal{S}_q(\mathbf{\Sigma})$ be the class of an arbitrary stated tangle in $\mathbf{\Sigma}$. Applying the skein relation in Remark \ref{remark_skeinrelations}, one finds:
$$ \mathcal{S}_q(\ad_{\Sigma}) ([T,s]) =\adjustbox{valign=c}{\includegraphics[width=2.5cm]{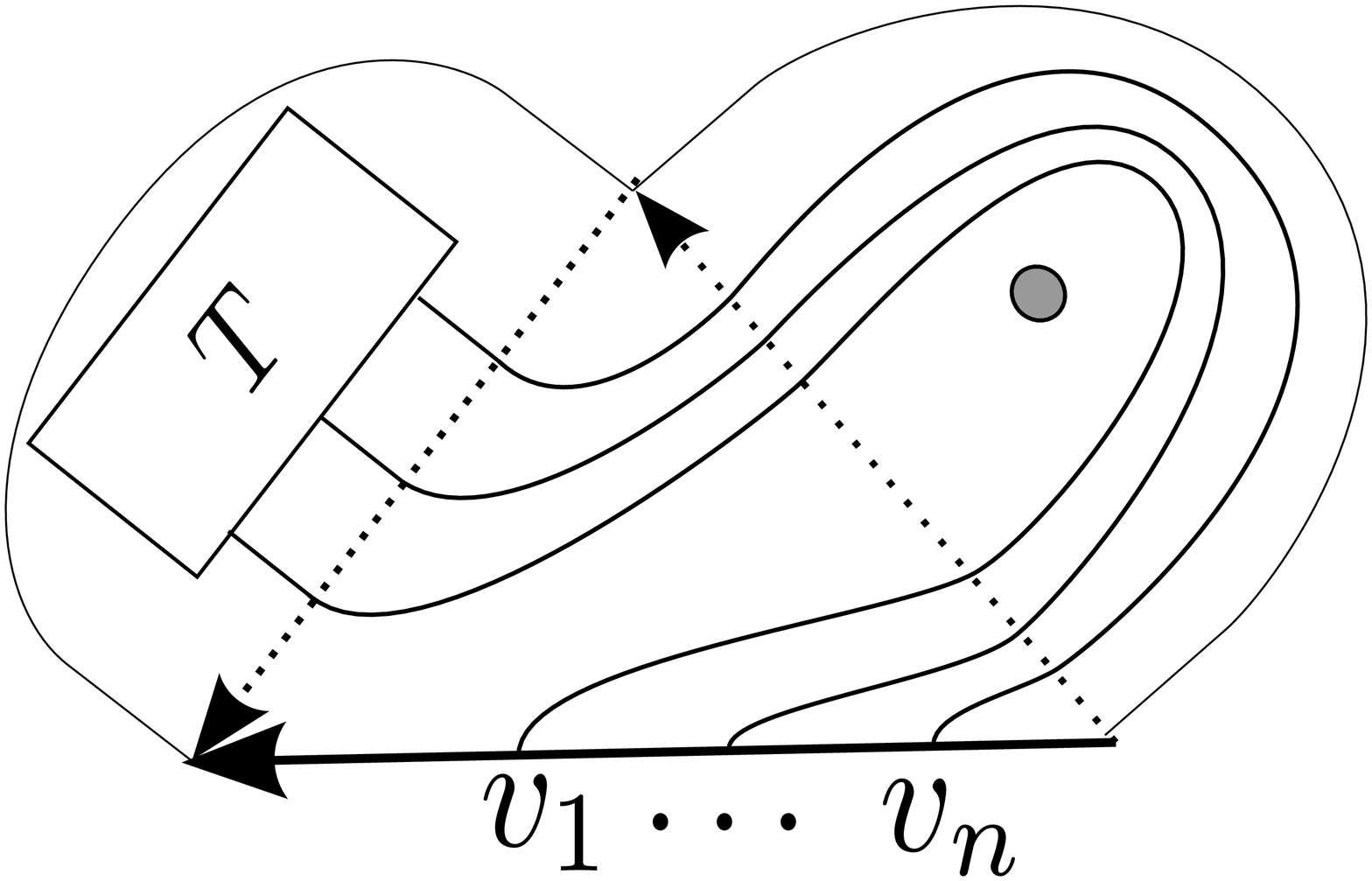}} = \sum_{i_1, \ldots, i_n} \adjustbox{valign=c}{\includegraphics[width=3cm]{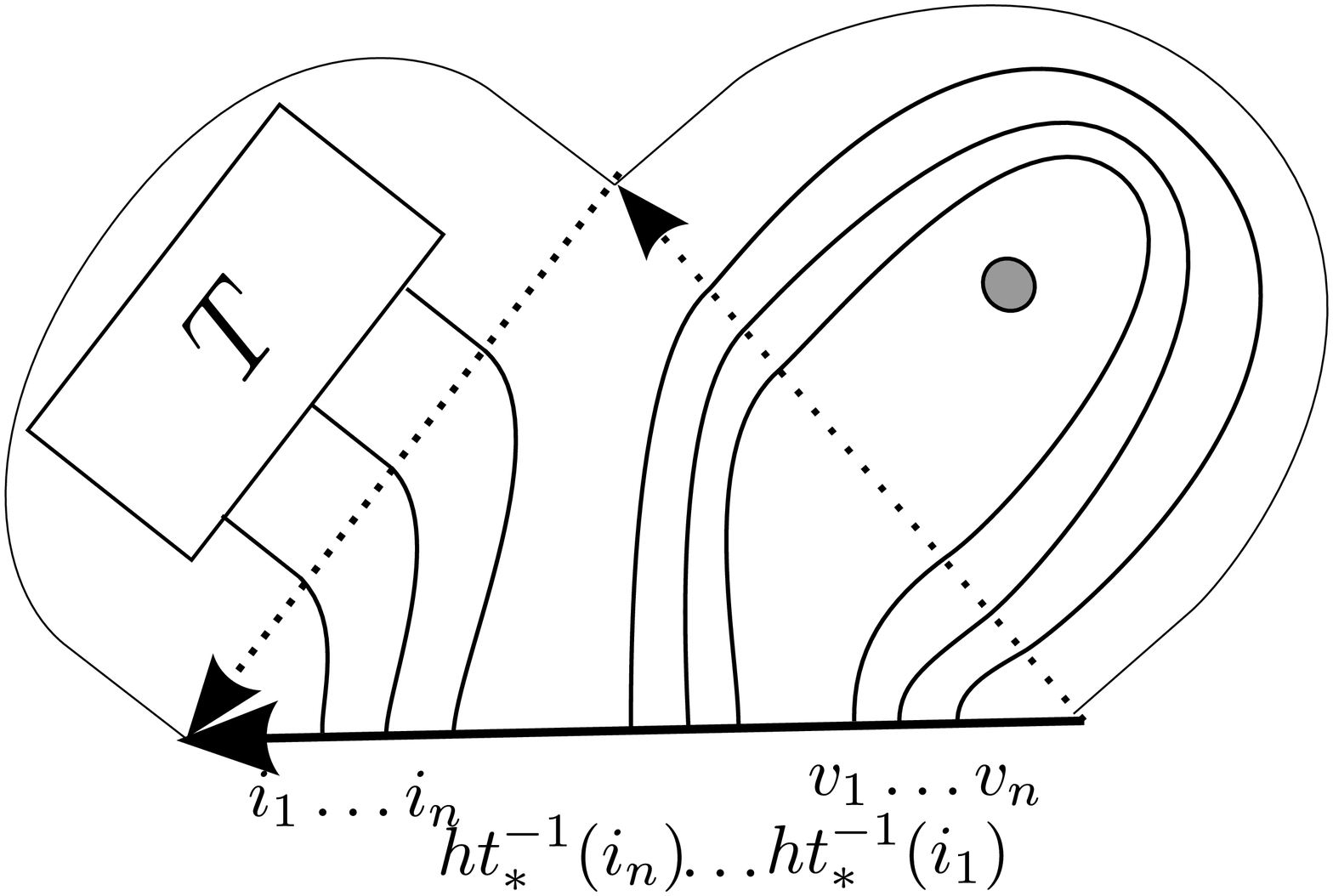}} \xrightarrow[\cong]{\Psi} \sum_{i_1, \ldots, i_n} \adjustbox{valign=c}{\includegraphics[width=4cm]{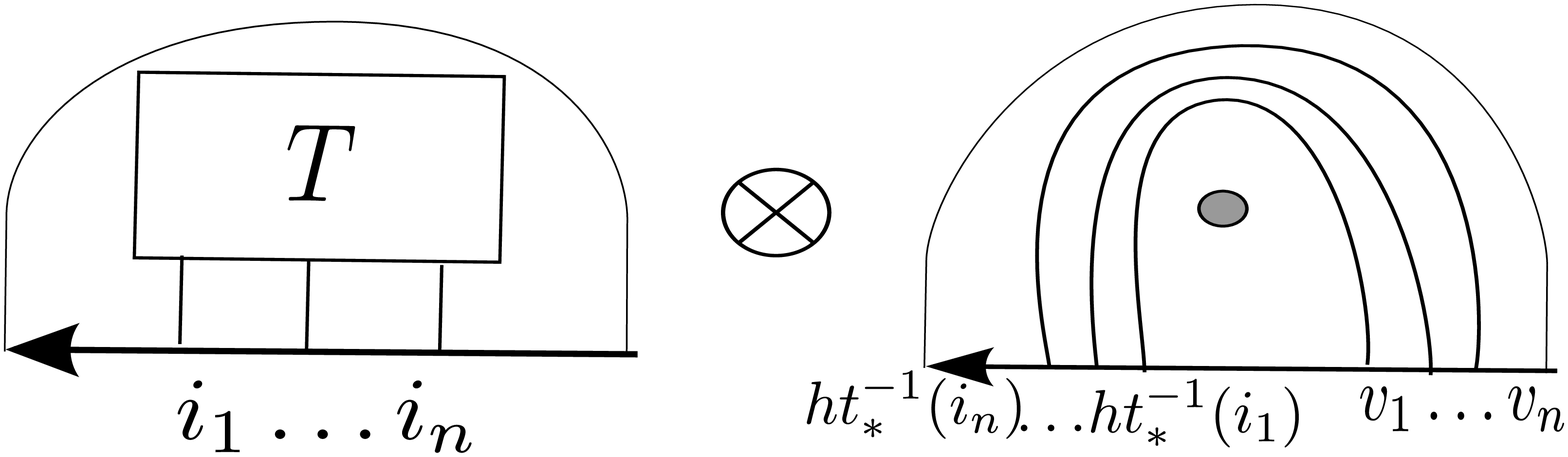}},  $$
where $\Psi: \mathcal{S}_q(\mathbf{\Sigma}\wedge \mathbf{H}_1) \xrightarrow{\cong} \mathcal{S}_q(\mathbf{\Sigma})\overline{\otimes} \mathcal{S}_q(\mathbf{H}_1)$ is the isomorphism of Section \ref{sec_qfusion}.
Applying $\id \otimes f^{-1}$, we obtain:
$$ (\id\otimes f^{-1}) \Ad_{\Sigma} ([T,s]) = \sum_{i_1, \ldots, i_n} \adjustbox{valign=c}{\includegraphics[width=4cm]{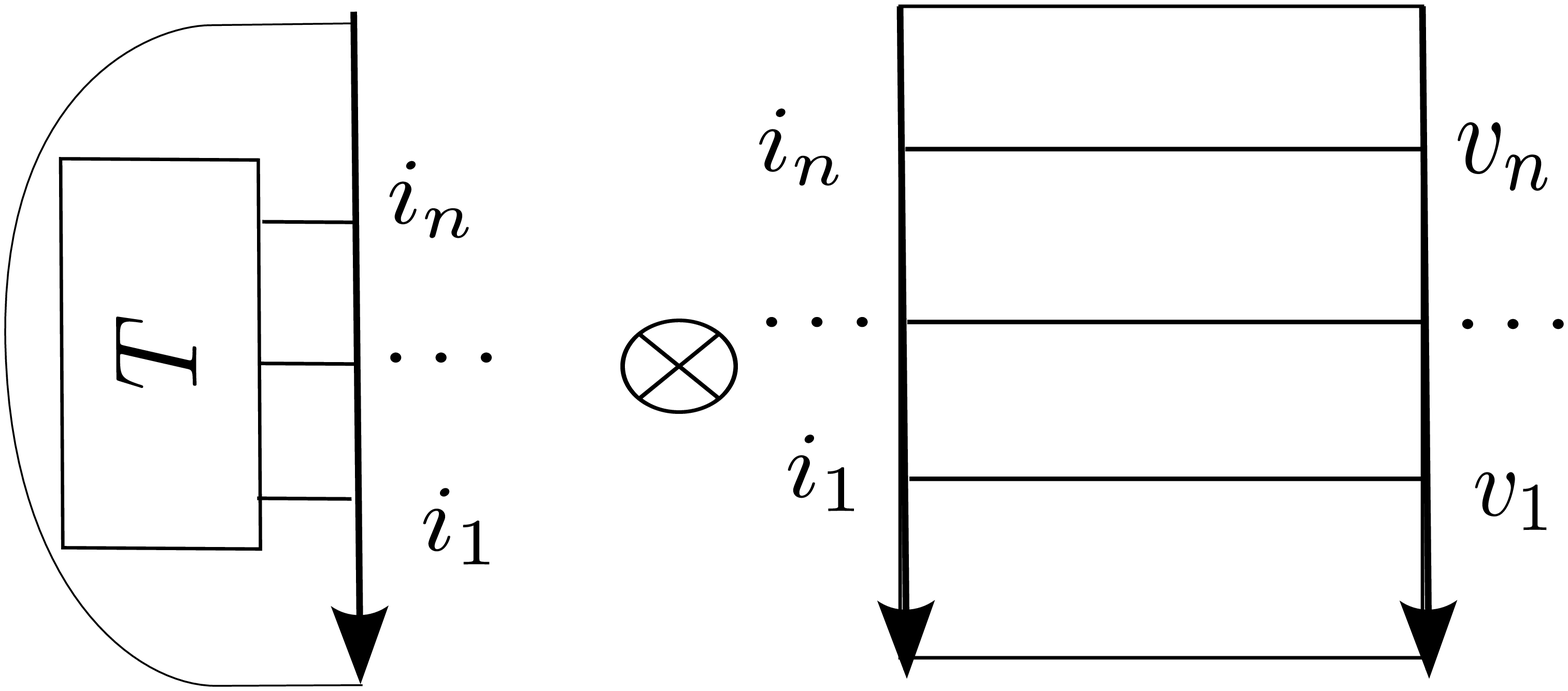}} =\Delta^R_a ([T,s]).$$

\end{proof}

\subsection{Mapping tori}
Let $\mathbf{\Sigma}=(\Sigma, \{a\}) \in \MS^{(1)}_{\con}$ be a connected $1$-marked surface and denote by $\partial$ the connected component of $\partial \Sigma$ containing $a$. Let $\phi: \Sigma \to \Sigma$ be an oriented diffeomorphism whose restriction to $\partial$ is the identity. Let $M_{\phi}:= \quotient{\Sigma \times I}{\left((\phi(x), +1)\sim (x,-1), x\in \Sigma\right)}$ be the associated mapping torus and $\pi: \Sigma\times I \to M_{\phi}$ the quotient map. Consider a parametrized disc $\mathbb{D}_{M_{\phi}} \subset \partial M_{\phi}$ contained in the image of $\partial \times I$ and let $\mathbf{M}_{\phi} := (M_{\phi}, \{\mathbb{D}_{M_{\phi}}\}) \in \mathcal{M}^{(1)}_{\con}$. The goal of this section is to describe $\Rep_q^G(\mathbf{M}_{\phi})$.
\vspace{2mm}
\par Let $\mathbf{M}:= (\mathbf{\Sigma}\times I) \wedge \mathbf{H}_1= (\Sigma\wedge \mathbb{D}_1)\times I$ and consider two attaching maps $\phi_1: \Sigma \hookrightarrow \partial M$ and $\phi_2: \overline{\Sigma} \hookrightarrow \partial M$ defined as follows. 
\begin{itemize}
\item The embedding $\phi_2: \overline{\Sigma} \to \partial M_1$ is the composition $\phi_2: \overline{\Sigma}\xrightarrow{ \overline{\ad_{\Sigma}^0}} \overline{\Sigma \wedge \mathbb{D}_1} \cong (\Sigma\wedge \mathbb{D}_1)\times \{-1\} \subset \partial M$.
\item The embedding $\phi_1: \Sigma \to \partial M_1$ is the composition $\phi_1: \Sigma \xrightarrow{\phi} \Sigma \xrightarrow{\iota_1} \Sigma\wedge \mathbb{D}_1 \cong (\Sigma\wedge \mathbb{D}_1)\times \{+1\} \subset \partial M$.
\end{itemize}
Then $\mathbf{M}_{\phi_1 \# \phi_2}$ is isomorphic to $\mathbf{M}_{\phi}$, so Theorem \ref{theorem_selfgluing} implies that $\Rep_q^G(\mathbf{M}_{\phi}) \cong \mathrm{HH}^0( \Rep_q^G(\mathbf{\Sigma}), \Rep_q^G(\mathbf{M}))$. Note that, since $\mathbf{M}$ is a thickened surface, it has an algebra structure with product $\mu_M$. Write $\mu_M^{top}:= \mu_M \circ \psi_{M, M} \circ (\theta_M\wedge \id_M): M\wedge M \to M$ the \textit{twisted opposite product}. By Remark \ref{remark_twistopposite}, $\mu_M^{top}$ is not really a product since it is not associative but rather equips $M$ with a structure of left module over itself in the sense that one has the equality
$$ \mu^{top}_M (\mu_M \otimes \id_M) = \mu^{top}_M(\id_M \otimes \mu^{top}_M).$$
By definition,  the module morphisms $\nabla_{\phi_1}^L$ and $\nabla_{\phi_2}^R$ decompose as $ \nabla_{\phi_1}^L= \mu_M^{top} \circ (\ad_{\Sigma} \wedge \id)$ and $\nabla_{\phi_2}^R= \mu_M \circ (\iota_1 \phi \wedge \id)$. Therefore we have proved the 

\begin{theorem}
The following sequence is exact: 
$$ \Rep_q^G(\mathbf{\Sigma}) \overline{\otimes} (\Rep_q^G(\mathbf{\Sigma})\overline{\otimes} B_qG) \xrightarrow{ \mu \circ (\iota_1\phi_*\otimes \id) - \mu^{top} \circ (\Ad_{\Sigma} \otimes \id)} \Rep_q^G(\mathbf{\Sigma})\overline{\otimes} B_qG \to \Rep_q^G(\mathbf{M}_{\phi})\to 0.$$
\end{theorem}

See Figure \ref{fig_BraidClosure} for an illustration. 

 \begin{figure}[!h] 
\centerline{\includegraphics[width=12cm]{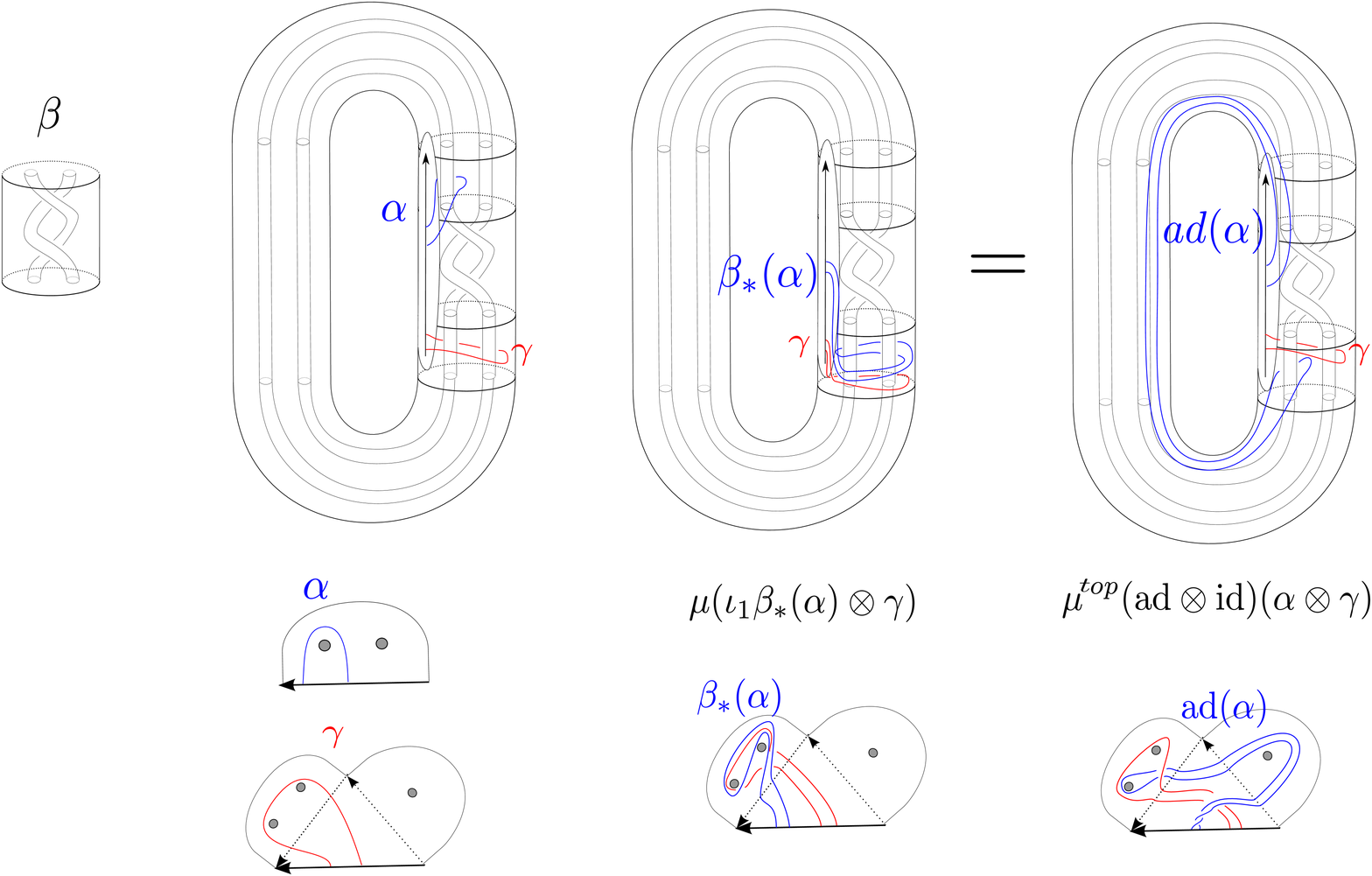} }
\caption{An illustration of the equality 
$P \left(\mu^{top} \circ (\ad_{\Sigma} \wedge \id)\right) (\alpha \otimes \gamma) = P\left( \mu \circ(\iota_1 \beta_* \otimes \id) \right) (\alpha \otimes \gamma)$
 in $P_{M_{\beta}}$ where $\Sigma=\mathbb{D}_2$,  $\beta$ a $2$ strands braid, $\alpha \in  P_1(\mathbb{D}_2)$ and $\beta \in P_1(\mathbb{D}_2\wedge \mathbf{H}_1)$.}
\label{fig_BraidClosure} 
\end{figure}

\section{Quantum representation spaces of links exterior}\label{sec6}

\subsection{Links exterior}

 Let $\beta \in B_n$ be a braid, seen as a mapping class of $\mathbb{D}_n$, and let $L\subset S^3$ be the link obtained by closing $\beta$.  Since $\mathbf{H}_n=\mathbb{D}_n\times I$, the mapping class $\beta$ induces a morphism (still denoted by the same letter) $\beta: \mathbf{H}_n \to \mathbf{H}_n$ in $\mathcal{M}^{(1)}_{\con}$. 
 Let $\mathbf{M}_L\in \mathcal{M}^{(1)}_{\con}$ be the marked $3$ manifold where $M_L= S^3 \setminus(\mathring{N}(L) \cup \mathring{B}^3)$ is obtained by removing from $S^3$ the union of the  interior of a tubular neighborhood of $L$ and an open ball, and with a single boundary disc $\mathbb{D}_L \subset \partial \mathbb{B}^3$. The product $\mu_{\mathbb{D}_n} : \mathbb{D}_n \wedge \mathbb{D}_n \to \mathbb{D}_n$ defines a product $\mu: \Rep_q^G(\mathbb{D}_n)^{\overline{\otimes} 2} \to \Rep_q^G(\mathbb{D}_n)$. Set $\mu^{top}:= \mu \circ \psi \circ (\theta\otimes \id)$.
 
 \begin{theorem}\label{theorem_LinksExt}
One has an exact sequence in $\overline{\mathcal{C}_q^G}$: 
$$ \Rep_q^G(\mathbb{D}_n)\overline{\otimes} \Rep_q^G(\mathbb{D}_n) \xrightarrow{ \mu^{top} - \mu \circ (\beta_*\otimes \id)} \Rep_q^G(\mathbb{D}_n) \to \Rep_q^G(\mathbf{M}_{L}) \to 0.$$
Said differently, identifying $\Rep_q^G(\mathbb{D}_n)$ with $(B_qG)^{\overline{\otimes}n}$, one has 
$$ \Rep_q^G(\mathbf{M}_L) \cong \quotient{ (B_qG)^{\overline{\otimes}n}}{ \left( \mu^{top}(x\otimes y) - \beta_*(x)y \right)}.$$
 \end{theorem}

\begin{proof}
The marked $3$-manifold $\mathbf{M}_L$ is obtained from $\mathbf{M}:= \mathbb{D}_n\times I$ by self-gluing using the attaching maps $\phi_1: \mathbb{D}_n \hookrightarrow \partial M$, $\phi_1: \mathbb{D}_n \xrightarrow{\beta} \mathbb{D}_n \cong \mathbb{D}_n \times \{ +1\} \subset \partial (\mathbb{D}_n \times I)$ and $\phi_2 : \overline{\mathbb{D}_n} \hookrightarrow \partial M$, $\phi_2: \overline{\mathbb{D}_n} \cong \mathbb{D}_n \times\{-1\} \subset \partial (\mathbb{D}_n \times I)$. The left module map $\nabla_{\phi_1}^L$ is the composition $\mu \circ(\beta_* \otimes \id)$ and the right module map $\nabla_{\phi_2}^R$ is just the product $\mu_{\mathbb{D}_n}$ so the results follows from Theorem \ref{theorem_selfgluing}.
\end{proof}

\begin{corollary}
Let $M_L = S^3 \setminus N(L) \in \mathcal{M}^{(0)}$ and $\mathcal{S}_q^{rat}(M_L)$ be the standard (non stated) skein module. Then 
$$ \mathcal{S}_q^{rat}(M_L) \cong \quotient{ (B_q\SL_2 ^{rat})^{\overline{\otimes}n}}{ \left( \mu^{top}(x\otimes y) - \beta_*(x)y \right)}.$$
\end{corollary}

\subsection{Comparison with the constructions in \cite{MurakamiVdV_QRepSpaces} }

We preserve the notations of the last subsection. The following was introduced by the second author together with Van-der-Veen in  \cite{MurakamiVdV_QRepSpaces}:
\begin{definition}
Let $\mathcal{I}_{\beta} \subset (B_q\SL_2)^{\overline{\otimes} n}$ be the right ideal generated by elements $\beta_*(x) - x$. We set 
$$ \mathcal{A}_{\beta} := \quotient{ (B_q\SL_2)^{\overline{\otimes} n} }{\mathcal{I}_{\beta}}.$$
\end{definition}

It is proved in  \cite{MurakamiVdV_QRepSpaces} that $\mathcal{I}_{\beta}$ is a bilateral ideal preserved by the $\mathcal{O}_q\SL_2$ coaction so $\mathcal{A}_{\beta}$ is an algebra in $\overline{\mathcal{C}_q^{\SL_2}}$. Moreover the adjoint coaction $\Ad :  (B_q\SL_2)^{\overline{\otimes} n} \to  (B_q\SL_2)^{\overline{\otimes} n} \overline{\otimes} B_q\SL_2$ sends $\mathcal{I}_{\beta}$ inside $\mathcal{I}_{\beta}\otimes B_q\SL_2$ so induces a comodule morphism (still denoted by the same letter) $\Ad: \mathcal{A}_{\beta} \to \mathcal{A}_{\beta} \otimes B_q\SL_2$ which is an algebra morphism.

\begin{definition}
The subalgebra $\mathcal{A}_{\beta}^{coinv} \subset \mathcal{A}_{\beta}$ is the subalgebra of coinvariant vectors for the $\Ad$ coaction.
\end{definition}

The main result of  \cite{MurakamiVdV_QRepSpaces} is the fact that if $\beta$ and $\beta'$ are two braids which admit the same Markov closure, then $\mathcal{A}_{\beta}\cong \mathcal{A}_{\beta'}$ in $\Alg(\overline{\mathcal{C}_q^{\SL_2}})$ therefore the isomorphism classes of both $\mathcal{A}_{\beta}$ and $\mathcal{A}_{\beta}^{coinv}$ only depends on the link $L$. In  \cite{MurakamiVdV_QRepSpaces}, the algebra $\mathcal{A}_{\beta}$ was called \textit{quantum representation variety} whereas $\mathcal{A}_{\beta}^{coinv}$ was named \textit{quantum character variety} of $L$. The initial motivation for the present paper was to relate these algebras to Habiro's quantum representation and character varieties and to (stated) skein modules. As can be shown by taking a simple example (such as the trivial knot) the two modules $\Rep_q^{\SL_2}(\mathbf{M}_L) $ and $\mathcal{A}_{\beta}$ are not isomorphic. However Theorem \ref{theorem_LinksExt} shows that they are very similar; indeed one has
$$ \mathcal{A}_{\beta}= \quotient{(B_q\SL_2)^{\overline{\otimes} n}}{ \left( \mu(x\otimes y) - \beta_*(x)y\right)} \quad \mbox{ and }\quad \Rep_q^{\SL_2}(\mathbf{M}_L)\cong  \quotient{(B_q\SL_2)^{\overline{\otimes} n}}{ \left( \mu^{top}(x\otimes y) - \beta_*(x)y\right)} $$
so we just have changed $\mu$ to $\mu^{top}$ in the quotient. Moreover, since taking the coinvariant vectors for the $\mathcal{O}_q[\SL_2]$ coaction than for the braided $B_q\SL_2$ quantum coaction by Lemma \ref{lemma_braidedcoinv}, the skein module of $M_L$ and $\mathcal{A}_{\beta}^{coinv}$ are very similar. We expect that the computational techniques developed in \cite{Murakami_RIMS} can be adapted to compute the peripheral ideal of links exteriors using this analogy.



\appendix

\section{Finite presentation of the category of bottom tangles}\label{appendix_presentation_BT}

In this appendix, we explain how the works of Kerler \cite{Kerler_PresTanglesCat, Kerler_AlgCobordisms} and Bobtcheva-Piergallini \cite{BobtchevaPiergallini}  permit to prove the following reformulation of Theorem \ref{theorem_functorBT}:

\begin{theorem}\label{theorem_presBT} The category $\BT$ is presented by the generators $(\mu, \eta, \Delta, \epsilon, S^{\pm 1}, \theta)$ of Figure \ref{fig_BTHopfAlg} and the relations of Figure \ref{fig_BPHopf}. \end{theorem}

Let us fix the terminology used here. Let $\mathcal{C}$ be a PROP, i.e. a monoidal category with set of objects $\mathbb{N}$ and such that $a\otimes b=a+b$, and consider a set $G$ of morphisms in $\mathcal{C}$. For a morphism $\mu$ of $\mathcal{C}$, we write $s(\mu), t(\mu)\in \mathbb{N}$ the integers (source and target) such that $\mu: s(\mu)\to t(\mu)$. The set $G$ is said to \textit{generate} $\mathcal{C}$ if for every $a,b\in \mathbb{N}$, every morphism in $\mathcal{C}(a,b)$ can be expressed as a composition of elements of the form $\mathds{1}_x \otimes g \otimes \mathds{1}_y$ with $g\in G$. Let $\mathcal{F}_G$ be the PROP freely generated by $G$. By definition, the morphisms in $\mathcal{F}_G(a,b)$ are formal expressions of the form $\mu_1\ldots \mu_n$ where $\mu_i= \mathds{1}_{x_i} \otimes g_i \otimes \mathds{1}_{y_i}$ with $g_i \in G$,  such that $s(\mu_i)=t(\mu_{i+1})$, $s(\mu_1)=a$, $t(\mu_n)=b$. Composition is the concatenation of words. Equivalently, we can represent the morphisms in $\mathcal{F}_G(a,b)$ by graphs with $a$ leaves on the top and $b$ leaves on the bottom and possibly some $(s(g), t(g))$ coupons for each $g\in G$. For instance, letting $G$ be the set of morphisms of $\BT$ of Figure \ref{fig_BTHopfAlg}, the morphism $\mu (\mathds{1}\otimes S) \Delta \in \mathcal{F}_G(1,1)$ is represented by the graph $\adjustbox{valign=c}{\includegraphics[width=0.5cm]{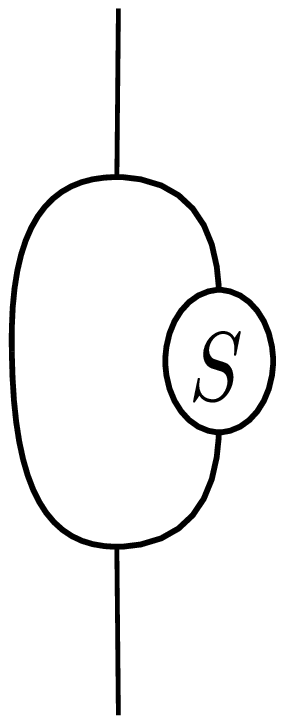}}$. There is a unique monoidal functor $F: \mathcal{F}_G \to \mathcal{C}$ sending $g\in \mathcal{F}_G(s(g), t(g))$ to $g\in \mathcal{C}(s(g),t(g))$ and $G$ is a generating set if and only if $F$ is essentially surjective.

 An \textit{equivalence relation} on a PROP $\mathcal{F}$ is the data for each $a,b\in \mathbb{N}$ of an equivalence relation $\sim$ on the set $\mathcal{F}(a,b)$ in such a way that $\mu\sim \nu$ implies that $\eta\otimes \mu \otimes \kappa \sim \eta\otimes \nu \otimes \kappa$ and that $\mu\sim \nu$ implies that $\mu' \circ \mu \circ \mu'' \sim \mu' \circ \nu \circ \mu''$ (i.e. $\sim$ is preserved by composition and tensor product). In this case we can define a new PROP $\quotient{\mathcal{F}}{\sim}$ by $\quotient{\mathcal{F}}{\sim}(a,b):= \quotient{\mathcal{F}(a,b)}{\sim}$. For instance, if $G$ is a set of morphisms in $\mathcal{C}$ and $F: \mathcal{F}_G\to \mathcal{C}$ the associated functor, we define an equivalence $\sim_{\mathcal{C}}$ in $\mathcal{F}_G$ by setting $\mu \sim_{\mathcal{C}} \nu$ if $F(\mu)=F(\nu)$. When $G$ generates $\mathcal{C}$, then $F: \quotient{\mathcal{F}_G}{\sim_{\mathcal{C}}} \to \mathcal{C}$ is an equivalence of categories. A \textit{relation} in $\mathcal{C}$ is a pair $(\mu, \nu)$ of morphisms in $\mathcal{F}_G$ with the same target and source and such that $F(\mu)=F(\nu)$. For instance $\left( \mu(\mu\otimes \mathds{1}) ,  \mu(\mathds{1}\otimes \mu) \right)$ is the relation in $\BT$ corresponding to the fact that $\mu$ is associative. Define a poset structure $\leq$ on the set $ER$ of equivalence relations on $\mathcal{F}_G$ by setting $\sim_1 \leq \sim_2$ if $\mu \sim_1 \nu$ implies $\mu \sim_2\nu$. The poset $(ER, \leq)$ is clearly filtrant: given $\sim_1$ and $\sim_2$ the relation $\sim$ defined by $\mu \sim \nu$ if we have both $\mu \sim_1 \nu$ and $\mu\sim_2\nu$ clearly satisfies $\sim \leq \sim_1$ and $\sim\leq \sim_2$. Therefore, by Zorn lemma, given a set $Rel$ of relations, we can speak of the smallest equivalence relation $\sim_{Rel}$ containing $Rel$; we say that $Rel$ is a \textit{complete set of relations} if $\sim_{Rel}=\sim_{\mathcal{C}}$.

\begin{definition} For $\mathcal{C}$ a PROP, $G$ a set of morphisms and $Rel$ a set of relations, we say that $(G,Rel)$ is a \textit{presentation} of $\mathcal{C}$ if $G$ generates $\mathcal{C}$ and if $Rel$ is a complete set of relations.
\end{definition}

\begin{theorem}(Habiro \cite[Theorem $5.16$]{Habiro_BottomTangles})\label{theorem_Habiro_generators} The set $G:=\{\mu, \eta, \Delta, \epsilon, S^{\pm 1}, \theta\}$  of morphisms of Figure \ref{fig_BTHopfAlg} generates $\BT$.\end{theorem}

Recall from Remark \ref{remark_CYK} the faithful braided functor 
$$\partial : \BT \to \mathcal{C}^{CYK}$$
 sending $\mathbf{H}_g$ to $\Sigma_{g,1}$ and sending an embedding $\mu : \mathbf{H}_a \to \mathbf{H}_b$ to the cobordism $\mathbf{H}_b \setminus \mu(\mathbf{H}_a)$. Not every cobordism in $\mathcal{C}^{CYK}(a,b)$ is homeomorphic to a cobordism of the form $\mathbf{H}_b \setminus \mu(\mathbf{H}_a)$ (so $\partial$ is not full), however every element of $\mathcal{C}^{CYK}(a,b)$ can be obtained from a cobordism $\mathbf{H}_b \setminus \mu(\mathbf{H}_a)$ by performing some surgery  along a closed framed link $L\subset \mathbf{H}_b \setminus \mu(\mathbf{H}_a)$. We can thus represent pictorially a morphism in $\mathcal{C}^{CYK}(a,b)$ by a $a$-bottom tangle $T$ in $\mathbf{H}_b$ entangled with a closed framed link $L \subset H_b \setminus T$ and think of $T\cup L$ as the cobordism obtained from $\mathbf{H}_b \setminus \mu_T(\mathbf{H}_a)$ by performing a surgery along $L$. The pairs $(T,L)$ are thus only considered up to some Kirby moves (see \cite{Kerler_PresTanglesCat, Kerler_AlgCobordisms, BobtchevaPiergallini} for a precise definition of the Kirby moves involved). For instance, one can consider the two morphisms: 
 $$ \lambda := \adjustbox{valign=c}{\includegraphics[width=1cm]{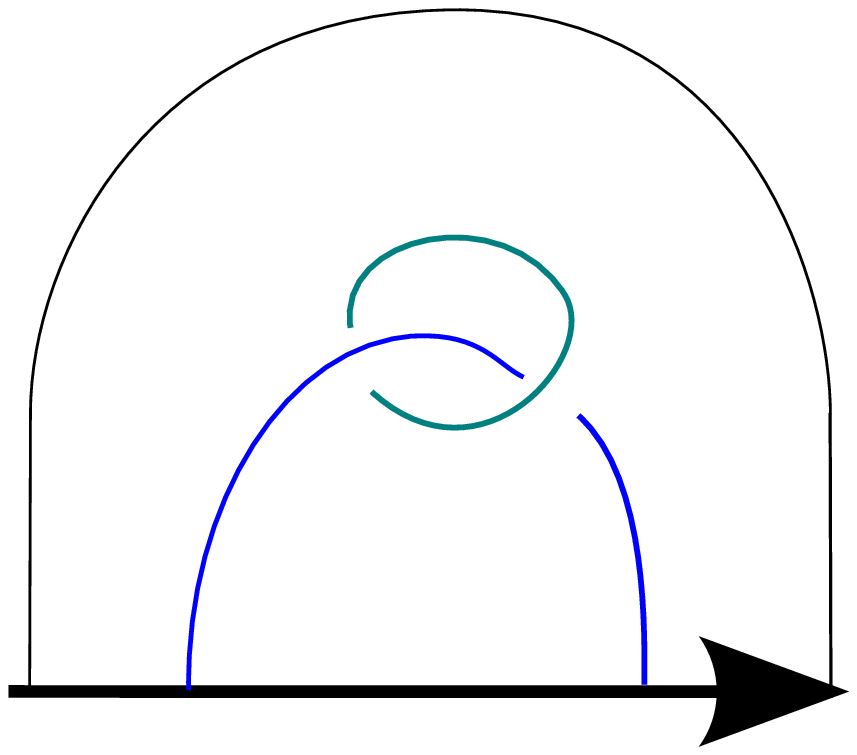}} \in \mathcal{C}^{CYK}(1,0) \quad \mbox{ and }\quad \Lambda := \adjustbox{valign=c}{\includegraphics[width=1cm]{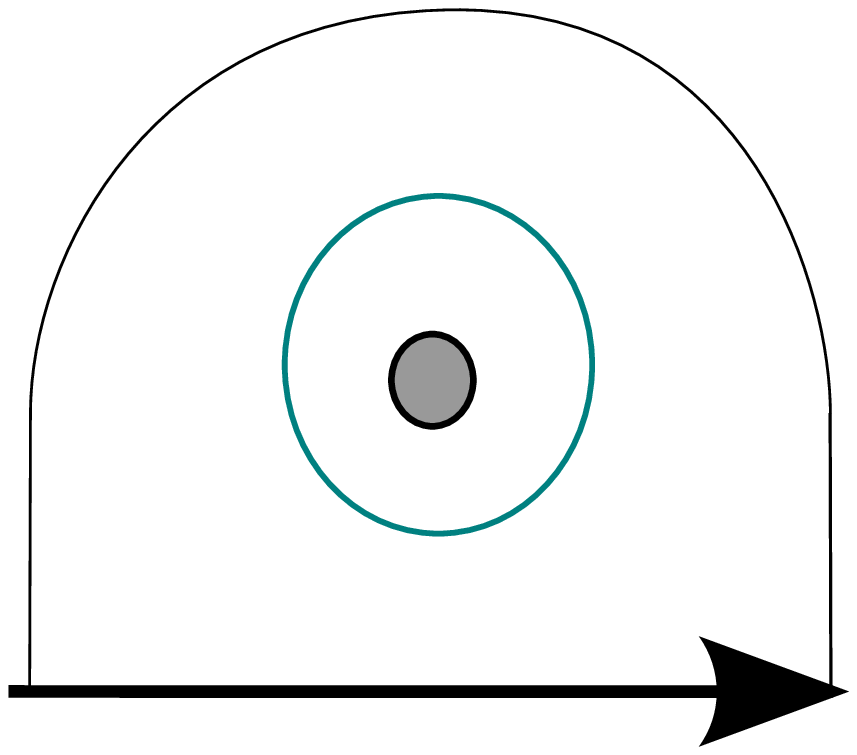}} \in \mathcal{C}^{CYK}(0,1), $$
 which do not belong to the image of $\partial$ (we perform the surgery along the green curve here). They satisfy some relations drawn in Figure \ref{fig_RelationsIntegral} which justify the terminology \textit{integral} and \textit{cointegral} for $\lambda$ and $\Lambda$ respectively.

\begin{figure}[!h] 
\centerline{\includegraphics[width=12cm]{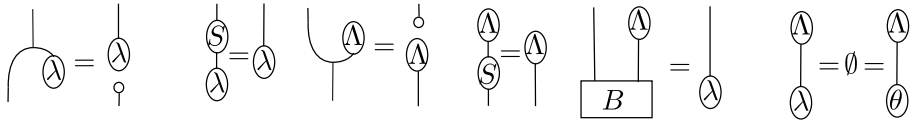} }
\caption{The relations satisfied by the integral and cointegral in $\mathcal{C}^{CYK}$.} 
\label{fig_RelationsIntegral} 
\end{figure}

 By abuse of notations, we identify $\BT$ with its image by $\partial$ inside $\mathcal{C}^{CYK}$ so we consider the set $G$ of morphisms of Figure \ref{fig_BTHopfAlg} both as a set of morphisms in $\BT$ and $\mathcal{C}^{CYK}$. Similarly,  the set $Rel$ of relations defined by Figure \ref{fig_BPHopf} is seen as a set of relations for $\mathcal{C}^{CYK}$ as well.
 
 \begin{theorem}(Bobtcheva-Piergallini \cite[Theorem $5.5.4$]{BobtchevaPiergallini})\label{theorem_BP} The category $\mathcal{C}^{CYK}$ is presented by the set of generators $G \bigsqcup \{ \lambda, \Lambda\}$ and the set of relations which is the union of $Rel$ with the relations of Figure \ref{fig_RelationsIntegral}.
 \end{theorem}

That the set $G \bigsqcup \{ \lambda, \Lambda\}$ generates $\mathcal{C}^{CYK}$ and that they satisfy the above relations (with the exception of the two bottom "BP" relations in Figure \ref{fig_BPHopf}) was proved by Kerler in \cite{Kerler_PresTanglesCat} who conjectured that they form a presentation of $\mathcal{C}^{CYK}$. The two additional BP relations were found by Bobtcheva-Piergallini in \cite{BobtchevaPiergallini}. 

\begin{proof}[Proof of Theorem \ref{theorem_presBT}]
Let $\mathcal{D}$ be the category presented by $(G,Rel)$ and $f: \mathcal{D}\to \mathcal{C}^{CYK}$ be the inclusion functor. It suffices to prove that $f$ is faithful; indeed, since $\partial: \BT \to \mathcal{C}^{CYK}$ is faithful and since $\BT$ is generated by $G$ by Theorem \ref{theorem_Habiro_generators}, this will prove that $\BT$ is equivalent to $\mathcal{D}$ and conclude the proof. Write $G':= G\bigsqcup \{\lambda, \Lambda\}$, let $Rel''$ be the set of relations (in $\mathcal{F}_{G'}$) of Figure \ref{fig_RelationsIntegral} and $Rel':= Rel \bigsqcup Rel''$ so that $\mathcal{C}^{CYK}$ is presented by $(G', Rel')$ (in virtue of Theorem \ref{theorem_BP}).
\par The faithfulness of $f$ follows from the following remark: for each relation in $Rel''$ of Figure \ref{fig_RelationsIntegral}, if we replace $\lambda$ by the counit $\epsilon$ and $\Lambda$ by the unit $\eta$, then get a relation in $Rel$. Therefore the quotient category $\quotient{ \mathcal{C}^{CYK}}{\left( \adjustbox{valign=c}{\includegraphics[width=1.5cm]{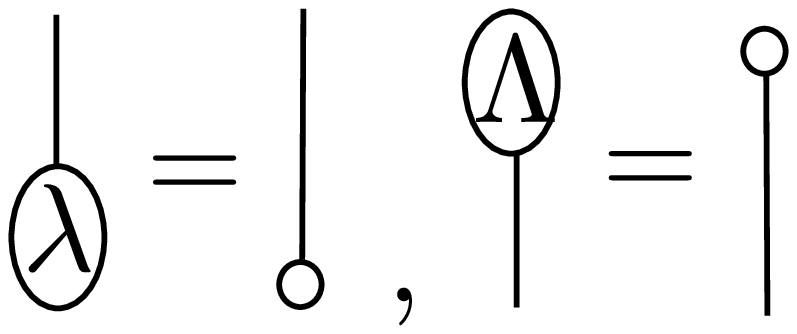}} \ \right)}$ is equivalent to $\mathcal{D}$ and the quotient functor
$$g : \mathcal{C}^{CYK} \to  \quotient{ \mathcal{C}^{CYK}}{\left( \adjustbox{valign=c}{\includegraphics[width=1.5cm]{Integral_quotient.eps}} \right)} \cong \mathcal{D}$$ 
satisfies $g\circ f = \id_{\mathcal{D}}$. This proves that $f$ is faithful and concludes the proof.

\end{proof}

\bibliographystyle{amsalpha}
\bibliography{biblio}

\end{document}